\documentclass[12pt]{amsart}
\usepackage{amsmath}
\usepackage{amssymb}
\usepackage{mathrsfs}
\usepackage{xcolor}
\usepackage{mathdots}
\usepackage{tikz}
\usetikzlibrary{cd}
\usetikzlibrary{knots}
\usetikzlibrary{calc}
\usetikzlibrary{decorations,decorations.pathreplacing,decorations.markings}
\usetikzlibrary{positioning}

\tikzstyle{annular}=[scale=.6, inner sep=1mm, baseline]
\tikzstyle{STrain}=[baseline=0,scale=2]
\tikzstyle{shaded}=[fill=red!10!blue!20!gray!40!white]
\tikzset{string/.style={ultra thick}}
\tikzset{smallstring/.style={thick,scale=0.75,every node/.style={transform shape}}}
\tikzset{
    triple/.style args={[#1] in [#2] in [#3]}{
        #1,preaction={preaction={draw,#3},draw,#2}
    }
}
\tikzset{
    quadruple/.style args={[#1] in [#2] in [#3] in [#4]}{
        #1,preaction={preaction={preaction={draw,#4},draw,#3}, draw,#2}
    }
} 
\tikzset{
	super thick/.style={line width=3pt},
	more thick/.style={line width=1pt},
}

\tikzstyle{box} = [rectangle,draw,rounded corners=5pt,very thick]

\newcommand{\roundNbox}[6]{
	\draw[rounded corners=5pt, very thick, #1] ($#2+(-#3,-#3)+(-#4,0)$) rectangle ($#2+(#3,#3)+(#5,0)$);
	\coordinate (ZZa) at ($#2+(-#4,0)$);
	\coordinate (ZZb) at ($#2+(#5,0)$);
	\node at ($1/2*(ZZa)+1/2*(ZZb)$) {#6};
}
\newcommand{\ncircle}[5]{
 \draw[very thick, #1] #2 circle (#3);
 \node at #2 {#5};
 \node at ($#2+(#4:.15cm)+(#4:#3cm)$) {\scriptsize{$\star$}};
}

\tikzstyle{mid>}=[decoration={markings, mark=at position 0.5 with {\arrow{>}}}, postaction={decorate}]
\tikzstyle{mid<}=[decoration={markings, mark=at position 0.5 with {\arrow{<}}}, postaction={decorate}]
\tikzstyle{upper>}=[decoration={markings, mark=at position 0.8 with {\arrow{>}}}, postaction={decorate}]
\tikzstyle{upper<}=[decoration={markings, mark=at position 0.8 with {\arrow{<}}}, postaction={decorate}]
\tikzstyle{lower>}=[decoration={markings, mark=at position 0.25 with {\arrow{>}}}, postaction={decorate}]
\tikzstyle{lower<}=[decoration={markings, mark=at position 0.25 with {\arrow{<}}}, postaction={decorate}]

\usepackage{graphicx}
\usepackage{fullpage}
\usepackage[utf8]{inputenc}
\usepackage[T1]{fontenc}
\usepackage[final]{microtype}
\usepackage{ifthen}
\usepackage{libertine}
\usepackage[libertine]{newtxmath}
\usepackage[all]{xy}

\usepackage{array, longtable}
\newcolumntype{C}{>{$}c<{$}}
\usepackage{capt-of}

\usepackage{xcolor}
\definecolor{dark-red}{rgb}{0.7,0.25,0.25}
\definecolor{dark-blue}{rgb}{0.15,0.15,0.55}
\definecolor{medium-blue}{rgb}{0,0,0.65}
\definecolor{DarkGreen}{RGB}{0,150,0}

\newcommand{\ShadeOne}{red!20}
\newcommand{\ShadeTwo}{blue!20}
\newcommand{\ShadeThree}{DarkGreen!20}

\newcommand{\ColorDot}[1]{
\begin{tikzpicture}[baseline=-.1cm]
	\filldraw[#1] (0,0) circle (.08cm);
\end{tikzpicture}
}

\usepackage[pdftex,plainpages=false,hypertexnames=false,pdfpagelabels,breaklinks]{hyperref}
\hypersetup{
   colorlinks, linkcolor={purple},
   citecolor={medium-blue}, urlcolor={medium-blue}
}
\usepackage[utf8]{inputenc}
\usepackage[backend=biber,bibencoding=utf8, style=alphabetic,doi=false,isbn=false,url=false,minalphanames=6,maxalphanames=6,maxbibnames=99]{biblatex}
\renewbibmacro{in:}{%
  \ifentrytype{article}{}{\printtext{\bibstring{in}\intitlepunct}}}
\renewbibmacro*{volume+number+eid}{%
  \printtext{vol.}
  \printfield{volume}%
  \setunit*{\addnbspace}% NEW (optional); there's also \addnbthinspace
  \printfield{number}%
  \setunit{\addcomma\space}%
  \printfield{eid}}
\DeclareFieldFormat[article]{number}{\mkbibparens{#1}}
\usepackage{silence}
\makeatletter
\newcommand{\hashdef}[2]{\@namedef{#1}{#2}}
\newcommand{\hashlookup}[1]{\@nameuse{#1}}
\makeatother

\newcommand{\doi}[1]{\href{https://dx.doi.org/#1}{{\small DOI:#1}}}

\newcommand{\googlebooks}[1]{(preview at \href{https://books.google.com/books?id=#1}{google books})}

\newcommand{\numdam}[1]{}

\theoremstyle{plain}
\newtheorem{prop}{Proposition}[section]

\newtheorem{thm}[prop]{Theorem}
\newtheorem{lem}[prop]{Lemma}
\newtheorem{theorem}[prop]{Theorem}
\newtheorem{lemma}[prop]{Lemma}
\newtheorem{cor}[prop]{Corollary}
\newtheorem{corollary}[prop]{Corollary}
\numberwithin{equation}{section}

\theoremstyle{remark}
\newtheorem{example}[prop]{Example}

\newtheorem{remark}[prop]{Remark}
\newtheorem{warn}[prop]{Warning}
\newtheorem{problem}[prop]{Problem}

\theoremstyle{definition}
\newtheorem{defn}[prop]{Definition}         % numbered definition
\newtheorem{fact}[prop]{Fact}
\newtheorem{nota}[prop]{Notation}

\newcommand{\sslash}{\mathbin{/\mkern-6mu/}}

\DeclareMathOperator{\ev}{ev}
\DeclareMathOperator{\coeff}{coeff}
\DeclareMathOperator{\coev}{coev}
\DeclareMathOperator{\Dim}{Dim}
\DeclareMathOperator{\End}{End}

\DeclareMathOperator{\Hom}{Hom}
\DeclareMathOperator{\id}{id}

\DeclareMathOperator{\Irr}{Irr}

\DeclareMathOperator{\Tr}{Tr}
\DeclareMathOperator{\tr}{tr}
\newcommand{\Bim}{{\mathsf {Bim}}}
\newcommand{\Hilb}{{\mathsf {Hilb}}}
\newcommand{\Fun}{{\mathsf {Fun}}}

\newcommand{\Mod}{{\mathsf {Mod}}}

\renewcommand{\Vec}{{\mathsf {Vec}}}
\newcommand{\Rep}{{\mathsf {Rep}}}

\newcommand{\jw}[1]{f^{(#1)}}
\newcommand{\JW}[1]{f^{(#1)}}
\newcommand{\Cstar}{{\rm C^*}}
\newcommand{\scrC}{{\mathscr C}}

\def\semicolon{;}
\def\applytolist#1{
    \expandafter\def\csname multi#1\endcsname##1{
        \def\multiack{##1}\ifx\multiack\semicolon
            \def\next{\relax}
        \else
            \csname #1\endcsname{##1}
            \def\next{\csname multi#1\endcsname}
        \fi
        \next}
    \csname multi#1\endcsname}

\def\calc#1{\expandafter\def\csname c#1\endcsname{{\mathcal #1}}}
\applytolist{calc}QWERTYUIOPLKJHGFDSAZXCVBNM;
\def\bbc#1{\expandafter\def\csname bb#1\endcsname{{\mathbb #1}}}
\applytolist{bbc}QWERTYUIOPLKJHGFDSAZXCVBNM;
\def\bfc#1{\expandafter\def\csname bf#1\endcsname{{\mathbf #1}}}
\applytolist{bfc}QWERTYUIOPLKJHGFDSAZXCVBNM;

\makeatletter
\newlength{\L@UnitsRaiseDisplaystyle}
\newlength{\L@UnitsRaiseTextstyle}
\newlength{\L@UnitsRaiseScriptstyle}
\DeclareRobustCommand*{\@UnitsNiceFrac}[3][]{%
  \ifthenelse{\boolean{mmode}}{%
    \settoheight{\L@UnitsRaiseDisplaystyle}{%
      \ensuremath{\displaystyle#1{M}}%
    }%
    \settoheight{\L@UnitsRaiseTextstyle}{%
      \ensuremath{\textstyle#1{M}}%
    }%
    \settoheight{\L@UnitsRaiseScriptstyle}{%
      \ensuremath{\scriptstyle#1{M}}%
    }%
    \settoheight{\@tempdima}{%
      \ensuremath{\scriptscriptstyle#1{M}}%
    }%
    \addtolength{\L@UnitsRaiseDisplaystyle}{%
      -\L@UnitsRaiseScriptstyle%
    }%
    \addtolength{\L@UnitsRaiseTextstyle}{%
      -\L@UnitsRaiseScriptstyle%
    }%
    \addtolength{\L@UnitsRaiseScriptstyle}{-\@tempdima}%
    \mathchoice
      {%
        \raisebox{\L@UnitsRaiseDisplaystyle}{%
          \ensuremath{\scriptstyle#1{#2}}%
        }%
      }%
      {%
        \raisebox{\L@UnitsRaiseTextstyle}{%
          \ensuremath{\scriptstyle#1{#2}}%
        }%
      }%
      {%
        \raisebox{\L@UnitsRaiseScriptstyle}{%
          \ensuremath{\scriptscriptstyle#1{#2}}%
        }%
      }%
      {%
        \raisebox{\L@UnitsRaiseScriptstyle}{%
          \ensuremath{\scriptscriptstyle#1{#2}}%
        }%
      }%
    \mkern-2mu{\sslash}\mkern-1mu%
    \bgroup
      \mathchoice
        {\scriptstyle}%
        {\scriptstyle}%
        {\scriptscriptstyle}%
        {\scriptscriptstyle}%
      #1{#3}%
    \egroup
  }%
  {%
    \settoheight{\L@UnitsRaiseTextstyle}{#1{M}}%
    \settoheight{\@tempdima}{%
      \ensuremath{%
        \mbox{\fontsize\sf@size\z@\selectfont#1{M}}%
      }%
    }%
    \addtolength{\L@UnitsRaiseTextstyle}{-\@tempdima}%
    \raisebox{\L@UnitsRaiseTextstyle}{%
      \ensuremath{%
        \mbox{\fontsize\sf@size\z@\selectfont#1{#2}}%
      }%
    }%
    \ensuremath{\mkern-2mu}{\sslash}\ensuremath{\mkern-1mu}%
    \ensuremath{%
      \mbox{\fontsize\sf@size\z@\selectfont#1{#3}}%
    }%
  }%
}

\DeclareRobustCommand*{\nicefrac}{\@UnitsNiceFrac}%
\makeatother

\newcommand{\drawS}[3]{%
	\filldraw[fill=white,thick] (#1,#2) ellipse (3mm and 3mm);
	\node at (#1,#2) {\Large $S$};
	\path(#1,#2) ++(#3:0.37) node {$\star$};
}

\newcommand{\RainbowOne}{
	\fill[shaded] (-0.8,0) -- (-0.8,0.6) arc (180:0:0.8) -- (0.8,0) -- (0.2,0) -- (0.2,1) -- (-0.2,1) -- (-0.2,0);
	\draw (-0.8,0) -- (-0.8,0.6) arc (180:0:0.8) -- (0.8,0);
	\draw (-0.2,0) -- (-0.2,1);
	\draw (0.2,0) -- (0.2,1);
	\node at (0.3,0.5) {\footnotesize$15$};
	\drawS{0}{1}{-90}
}

\newcommand{\RainbowTwo}{
	\draw (0,0) -- (0,1);
	\node at (0.15,0.5) {\footnotesize$16$};
	\drawS{0}{1}{180}
	\filldraw[shaded] (-1,0) -- (-1,0.5) arc (180:0:1) -- (1,0)  -- (0.9,0) -- (0.9,0.5) arc (0:180:0.9) -- (-0.9,0);
}

\newcommand{\JWPlusTwo}{%
	\filldraw[fill=white,thick] (-1,-0.2) rectangle (1,0.2);
	\node at (0,0) {\Large$\JW{18}$};
}

\newcommand{\JWPlusFour}{%
	\filldraw[fill=white,thick] (-1.2,-0.2) rectangle (1.2,0.2);
	\node at (0,0) {\Large$\JW{20}$};
}

\newcommand{\STrainOneOne}{%
	\foreach \x in {-1,0} {
		\node[anchor=south] at (\x+0.5,1) {\footnotesize$7$};
		\draw (\x,1) -- (\x + 1, 1);
	}
	\foreach \x in {-1,0,1} {
		\drawS{\x}{1}{90}
	}
}

\newcommand{\STrainStrings}[2]{%
	\fill[shaded] (-0.5,0) rectangle (0.5,1);
	\draw (-0.5,1) -- (-0.5,0);
	\node[anchor=west] at (-0.5,0.5) {\footnotesize#1};
	\draw (0.5,1) -- (0.5,0);
	\node[anchor=west] at (0.5,0.5) {\footnotesize#2};
}

\newcommand{\STrainOne}{%
	\node[anchor=south] at (0,1) {\footnotesize$7$};
	\draw (-0.5,1) -- (0.5, 1);
	\foreach \x in {-0.5,0.5} {
		\drawS{\x}{1}{90}
	}
}

\newcommand{\STrainThreeStrings}[3]{%
	\fill[shaded] (-1,1) rectangle (1,0);
	\foreach \x in {-1,0,1} {
		\draw (\x,0) -- (\x,1);
	}
        \node [anchor=west] at (-1,0.5) {\footnotesize#1};
        \node [anchor=west] at (0,0.5) {\footnotesize#2};
	\node [anchor=west] at (1,0.5) {\footnotesize#3};
}

\begin{document}

\title{The Extended Haagerup fusion categories}
\author{Pinhas Grossman, Scott Morrison, David Penneys, Emily Peters, and Noah Snyder}
\date{\today}
\begin{abstract}
In this paper we construct two new fusion categories and  many new subfactors related to the exceptional Extended Haagerup subfactor.  

The Extended Haagerup subfactor has two even parts $\cE\cH_1$ and $\cE\cH_2$.  
These fusion categories are mysterious and are the only known fusion categories which appear to be unrelated to finite groups, quantum groups, or Izumi quadratic categories.  
One key technique which has previously revealed hidden structure in fusion categories is to study all other fusion categories in the Morita equivalence class, and hope that one of the others is easier to understand.  
In this paper we show that there are exactly four categories ($\cE\cH_1, \cE\cH_2, \cE\cH_3, \cE\cH_4$) in the Morita equivalence class of Extended Haagerup, and that there is a unique Morita equivalence between each pair.  
The existence of $\cE\cH_3$ and $\cE\cH_4$ gives a number of interesting new subfactors.  
Neither $\cE\cH_3$ nor $\cE\cH_4$ appears to be easier to understand than the Extended Haaerup subfactor, providing further evidence that Extended Haagerup does not come from known constructions.  
We also find several interesting intermediate subfactor lattices related to Extended Haagerup. 

The method we use to construct $\cE\cH_3$ and $\cE\cH_4$ is interesting in its own right and gives a general computational recipe for constructing fusion categories in the Morita equivalence class of a subfactor.  We show that pivotal module $\rm C^*$ categories over a given subfactor correspond exactly to realizations of that subfactor planar algebra as a planar subalgebra of a graph planar algebra.  This allows us to construct $\cE\cH_3$ and $\cE\cH_4$ by realizing the Extended Haagerup subfactor planar algebra inside the graph planar algebras of two new graphs.  This technique also answers a long-standing question of Jones: which graph planar algebras contain a given subfactor planar algebra?
%This is the submitted version of \arXiv{...}.
\end{abstract}
\maketitle

%%%%%%%%%%%%%%%%%%%%%%%%%%%%%%%%%%%%%%%%%%%%%%%%%%%%%%%%%%%%
%%%%%%%%%%%%%%%%%%%%%%%%%%%%%%%%%%%%%%%%%%%%%%%%%%%%%%%%%%%%
%%%%%%%%%%%%%%%%%%%%%%%%%%%%%%%%%%%%%%%%%%%%%%%%%%%%%%%%%%%%

%auto-ignore
%this ensures the arxiv doesn't try to start TeXing here.
%!TEX root =../EH3.tex

%%%%%%%%%%%%%%%%%%%%%%%%%%%%%%%%%%%%%%%%%%%%%%%%%%%%%%%%%%%%
%%%%%%%%%%%%%%%%%%%%%%%%%%%%%%%%%%%%%%%%%%%%%%%%%%%%%%%%%%%%
%%%%%%%%%%%%%%%%%%%%%%%%%%%%%%%%%%%%%%%%%%%%%%%%%%%%%%%%%%%%
\section{Introduction}

The Extended Haagerup subfactor gives a Morita equivalence between two unitary fusion categories called $\cE\cH_1$ and $\cE\cH_2$.  The goal of this paper is to find all fusion categories Morita equivalent to these fusion categories.  We find two new fusion categories, seven new subfactors (along with their duals and reduced subfactors), and several interesting new intermediate subfactor lattices.

\begin{thm}
\label{thm:Main}
There are exactly two further fusion categories in the Morita equivalence class of $\mathcal{EH}_1$ and $\mathcal{EH}_2$, which we call $\mathcal{EH}_3$ and $\mathcal{EH}_4$.  Between any two of these four fusion categories, there is exactly one Morita equivalence.
\end{thm}

For every choice of simple object in each of these Morita equivalences, we get a subfactor.  In addition to the original $7$-supertransitive Extended Haagerup subfactor, we get two new $3$-supertransitive subfactors: one is self-dual and comes from the Morita auto-equivalence of $\cE\cH_3$ and the other comes from the Morita equivalence between $\cE\cH_3$ and $\cE\cH_4$.  Their principal graphs are:
$$
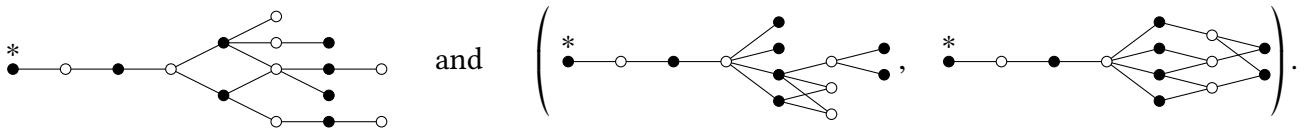
\begin{tikzpicture}[baseline=0mm, scale=.7]
	\draw (0,0) -- (3,0);
	\draw (3,0) -- (4,.5);
	\draw (3,0) -- (4,-.5);
	\draw (4,.5) -- (5,1);
	\draw (4,.5) -- (5,.5);
	\draw (4,.5) -- (5,0);	
	\draw (4,-.5) -- (5,0);
	\draw (4,-.5) -- (5,-1);
	\draw (5,.5) -- (6,.5);
	\draw (5,0) -- (6,0);
	\draw (5,0) -- (6,-.5);	
	\draw (5,-1) -- (6,-1);	
	\draw (6,0) -- (7,0);
	\draw (6,-1) -- (7,-1);
		
	\filldraw              (0,0) node [above] {$*$} circle (1mm);
	\filldraw [fill=white] (1,0) circle (1mm);
	\filldraw              (2,0) circle (1mm);
	\filldraw [fill=white] (3,0) circle (1mm);
	\filldraw              (4,.5) circle (1mm);	
	\filldraw              (4,-.5) circle (1mm);
	\filldraw [fill=white] (5,1) circle (1mm);
	\filldraw [fill=white] (5,.5) circle (1mm);
	\filldraw [fill=white] (5,0) circle (1mm);
	\filldraw [fill=white] (5,-1) circle (1mm);
	\filldraw              (6,.5) circle (1mm);
	\filldraw              (6,0) circle (1mm);
	\filldraw              (6,-.5) circle (1mm);
	\filldraw              (6,-1) circle (1mm);
	\filldraw [fill=white] (7,0) circle (1mm);
	\filldraw [fill=white] (7,-1) circle (1mm);
\end{tikzpicture}
\hspace{.25in}
\text{and}
\hspace{.25in}
\left(\begin{tikzpicture}[baseline=-1mm, scale=.35]
	\draw (0,0) -- (2,0);
	\draw (2,0) -- (4,0);
	\draw (4,0) -- (6,0);
	\draw (6,0) -- (8,1.5);
	\draw (6,0) -- (8,.5);	
	\draw (6,0) -- (8,-.5);
	\draw (6,0) -- (8,-1.5);

	\draw (8,-.5) -- (10,0);
	\draw (8,-.5) -- (10,-1);
	\draw (8,-.5) -- (10,-2);
	\draw (8,-1.5) -- (10,-1);
	\draw (8,-1.5) -- (10,-2);
	
	\draw (10,0) -- (12,.5);
	\draw (10,0) -- (12,-.5);
		
	\filldraw (0,0) node [above] {$*$} circle (2mm);
	\filldraw [fill=white] (2,0) circle (2mm) ;
	\filldraw  (4,0) circle (2mm);
	\filldraw [fill=white] (6,0) circle (2mm) ;	
	\filldraw  (8,1.5) circle (2mm);
	\filldraw  (8,.5) circle (2mm);
	\filldraw  (8,-.5) circle (2mm);
	\filldraw  (8,-1.5) circle (2mm);	
	\filldraw [fill=white] (10,0) circle (2mm) ;		
	\filldraw [fill=white] (10,-1) circle (2mm) ;		
	\filldraw [fill=white] (10,-2) circle (2mm) ;		
	\filldraw  (12,.5) circle (2mm);
	\filldraw  (12,-.5) circle (2mm);	
\end{tikzpicture}
\;,\hspace{.1in}
\begin{tikzpicture}[baseline=-1mm, scale=.35]
	\draw (0,0) -- (2,0);
	\draw (2,0) -- (4,0);
	\draw (4,0) -- (6,0);

	\draw (6,0) -- (8,1.5);
	\draw (6,0) -- (8,.5);	
	\draw (6,0) -- (8,-.5);
	\draw (6,0) -- (8,-1.5);

	\draw (8,1.5) -- (10,1);
	\draw (8,.5) -- (10,0);
	\draw (8,-.5) -- (10,0);
	\draw (8,-.5) -- (10,-1);	
	\draw (8,-1.5) -- (10,-1);
		
	\draw (10,1) -- (12,.5);
	\draw (10,1) -- (12,-.5);
	\draw (10,0) -- (12,.5);
	\draw (10,-1) -- (12,-.5);
		
	\filldraw (0,0) node [above] {$*$} circle (2mm);
	\filldraw [fill=white] (2,0) circle (2mm) ;
	\filldraw  (4,0) circle (2mm) ;
	\filldraw [fill=white] (6,0) circle (2mm) ;	
	\filldraw  (8,1.5) circle (2mm);
	\filldraw  (8,.5) circle (2mm);
	\filldraw  (8,-.5) circle (2mm);
	\filldraw  (8,-1.5) circle (2mm);	
	\filldraw [fill=white] (10,1) circle (2mm) ;		
	\filldraw [fill=white] (10,0) circle (2mm) ;		
	\filldraw [fill=white] (10,-1) circle (2mm) ;		
	\filldraw  (12,.5) circle (2mm);
	\filldraw  (12,-.5) circle (2mm);	
\end{tikzpicture}\right).
$$

The structures of $\mathcal{EH}_3$ and $\mathcal{EH}_4$ are explained in more detail in Section \ref{sec:StructureOfEH3}.  Neither appears to be easily understood using any general techniques, but we encourage the reader to look for a new way to construct them which could give a better understanding of the Extended Haagerup subfactor.  

At first it was thought that all fusion categories might come from groups or quantum groups at roots of unity.  
The first ``exotic'' examples which appeared to be unrelated to finite groups or quantum groups came from Haagerup's small index classification program \cite{MR1317352}, namely the Haagerup and Asaeda--Haagerup subfactors constructed in \cite{MR1686551} and the Extended Haagerup subfactor constructed in \cite{MR2979509} (after numerical evidence for existence was given by \cite{MR1633929}).
However, with time, the first two of these three examples have been seen to be related to Izumi quadratic categories.  
Izumi generalized the Haagerup subfactor to a possibly infinite family of quadratic $3^G$ subfactors \cite{MR1832764,MR2837122}.  
Recently in \cite{MR3859276}, Grossman-Izumi-Snyder found all fusion categories Morita equivalent to the Asaeda-Haagerup fusion categories, and discovered that one is an Izumi quadratic category.  The only remaining fusion categories which appear unrelated to groups, quantum groups at roots of unity, or Izumi quadratic categories are the Extended Haagerup fusion categories.
Unlike in the Asaeda--Haagerup case where one of the new categories was a quadratic category, neither $\cE\cH_3$ nor $\cE\cH_4$ has nontrivial invertible objects and so neither can be a quadratic category.  

The proof of the main theorem has two parts.  On the one hand we need to limit the possible fusion categories and Morita equivalences, and on the other hand we need to construct the remaining possibilities.  The former is an application of the techniques introduced in \cite{MR3449240}, using combinatorial restrictions for compatible fusion rules for the hypothetical fusion categories and bimodule categories.

We construct $\cE\cH_3$ and $\cE\cH_4$ using a general graph planar algebra \cite{MR1865703} technique for finding module categories over any fusion category where we have a good skein theoretic description.  
This technique can be viewed as a generalization of the Ocneanu cell technique for $SU(n)$ (\cite{MR1907188}, \cite{MR2545609}, \cite{MR1839381}) to arbitrary tensor categories with good skein theoretic descriptions.  
From our combinatorial calculation we know that there is at most one module category over each of $\cE\cH_1$ and $\cE\cH_2$ whose dual can be $\cE\cH_3$ (or $\cE\cH_4$).  
So if we can construct a module category with the correct fusion rules, we will have a construction of $\cE\cH_3$ (or $\cE\cH_4$) as the commutant category.  

We can package  $\cE\cH_1$ and $\cE\cH_2$ together with their Morita equivalence into a single multifusion category which is the Extended Haagerup planar algebra.  The fusion rules for tensoring simple objects in a module with the single strand in Extended Haagerup give a bipartite graph $\Gamma$.  
We prove a graph planar algebra embedding theorem for module categories,
which shows such module categories for this planar algebra correspond (up to gauging and graph automorphism) to embeddings of the Extended Haagerup subfactor planar algebra inside the graph planar algebra of $\Gamma$.     This generalizes the original graph planar algebra embedding theorem \cite{MR2812459} which only applied to the principal graphs.

\begin{thm}
\label{cor:EmbeddingGivesModule}
Suppose $\cP_\bullet$ is a finite depth subfactor planar algebra. Let $\scrC$ denote the unitary multifusion category of projections in $\cP_\bullet$, with distinguished object $X=\id_{1,+}\in\cP_{1,+}$ and the standard unitary pivotal structure with respect to $X$. There is an equivalence between:
\begin{enumerate}
\item
Planar $\dag$-algebra embeddings $\cP_\bullet \to \cG\cP\cA(\Gamma)_\bullet$, where $\Gamma$ is a finite connected bipartite graph, and
\item
indecomposable finitely semisimple pivotal left $\scrC$-module $\Cstar$ categories $\cM$ whose fusion graph with respect to $X$ is $\Gamma$.
\end{enumerate}
\end{thm}
This theorem answers a long-standing question of Vaughan Jones: given a finite depth subfactor planar algebra $\cP_\bullet$, determine all bipartite graphs $\Gamma$ for which $\cP_\bullet$ embeds in $\cG\cP\cA(\Gamma)_\bullet$.

By the skein theoretic description of Extended Haagerup in \cite{MR2979509}, in order to construct a map from Extended Haagerup into a graph planar algebra, we need only specify a number for each loop of length $16$ and check a large number of linear and quadratic equations in these numbers.  

\begin{thm}
\label{thm:GraphsForModules}
The Extended Haagerup planar algebra can be embedded into the graph planar algebras of each of the following bipartite graphs:
$$
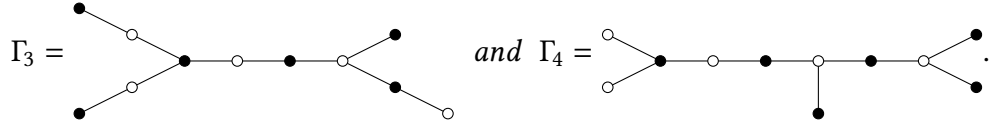
\Gamma_3 = 
\begin{tikzpicture}[baseline, scale=.7]
	\draw (-2,-1) -- (0,0);
	\draw (-2,1)--(0,0);
	\draw (0,0)--(3,0);
	\draw (3,0)--(4,.5);
	\draw (3,0)--(5,-1);
	\filldraw [fill]  (-2,-1) circle (1mm);
	\filldraw [fill]  (-2,1) circle (1mm);
	\filldraw [fill=white] (-1,-.5) circle (1mm);
	\filldraw [fill=white] (-1,.5) circle (1mm);
	\filldraw [fill]  (0,0) circle (1mm);
	\filldraw [fill=white] (1,0) circle (1mm);
	\filldraw [fill]  (2,0) circle (1mm);
	\filldraw [fill=white] (3,0) circle (1mm);
	\filldraw [fill]  (4,-.5) circle (1mm);
	\filldraw [fill]  (4,.5) circle (1mm);
	\filldraw [fill=white] (5,-1) circle (1mm);
\end{tikzpicture}
\hspace{.1in} \text{and} \hspace{.1in}
\Gamma_4 = 
\begin{tikzpicture}[baseline, scale=.7]
	\draw (-1,-.5) -- (0,0);
	\draw (-1,.5)--(0,0);
	\draw (0,0)--(5,0);
	\draw (5,0)--(6,.5);
	\draw (5,0)--(6,-.5);
	\draw (3,0)--(3,-1);
	\filldraw [fill=white] (-1,-.5) circle (1mm);
	\filldraw [fill=white] (-1,.5) circle (1mm);
	\filldraw [fill]  (0,0) circle (1mm);
	\filldraw [fill=white] (1,0) circle (1mm);
	\filldraw [fill]  (2,0) circle (1mm);
	\filldraw [fill=white] (3,0) circle (1mm);
	\filldraw [fill]  (4,0) circle (1mm);
	\filldraw [fill=white] (5,0) circle (1mm);
	\filldraw [fill]  (6,-.5) circle (1mm);
	\filldraw [fill]  (6,.5) circle (1mm);
	\filldraw [fill] (3,-1) circle (1mm);
\end{tikzpicture}
.$$
\end{thm}

Thus, the existence of $\cE\cH_3$ and $\cE\cH_4$ is now a corollary to Theorems \ref{cor:EmbeddingGivesModule} and \ref{thm:GraphsForModules}.

From our complete classification of module categories, we see that there are exactly four graphs whose graph planar algebras take maps from the Extended Haagerup planar algebra: the principal and dual principal graphs for the original subfactor, and the two graphs in the previous theorem.  
(One may think of these embeddings as giving four independent constructions of the Extended Haagerup subfactor planar algebra.)

Theorem \ref{cor:EmbeddingGivesModule} also connects the results of \cite{MR2679382} and \cite{MR2909758} to complete the classification of graph planar algebra embeddings for the Haagerup planar algebra.  
In the last section of \cite{MR2679382}, three embeddings of the Haagerup planar algebra into graph planar algebras were found, corresponding to the two principal graphs and the `broom' graph.  However, it was not proven there could not be others.
The main result of \cite{MR2909758} shows there are exactly three module categories over the Haagerup subfactor planar algebra.  Thus we have:
\begin{cor} 
\label{cor:hrp_graphs}
The Haagerup subfactor planar algebra embeds into $\cG\cP\cA(\Gamma)_{\bullet}$ if and only if $\Gamma$ is one of the following:
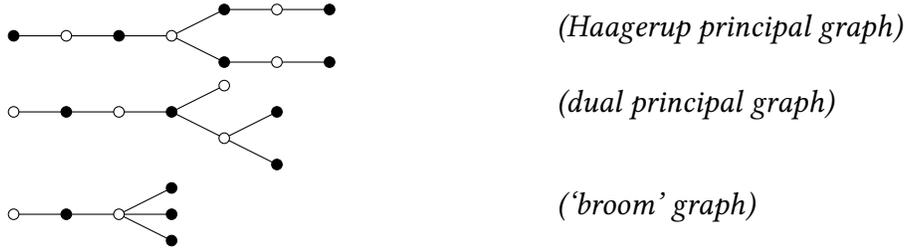
\begin{align*}
&
\begin{tikzpicture}[baseline, scale=.7]
	\draw (0,0)--(3,0);
	\draw (3,0)--(4,.5) -- (6,.5);
	\draw (3,0)--(4,-.5) -- (6,-.5);
	\filldraw [fill]  (0,0) circle (1mm);
	\filldraw [fill=white] (1,0) circle (1mm);
	\filldraw [fill]  (2,0) circle (1mm);
	\filldraw [fill=white] (3,0) circle (1mm);
	\filldraw [fill]  (4,-.5) circle (1mm);
	\filldraw [fill]  (4,.5) circle (1mm);
	\filldraw [fill=white] (5,-.5) circle (1mm);
	\filldraw [fill=white] (5,.5) circle (1mm);
	\filldraw [fill] (6,-.5) circle (1mm);
	\filldraw [fill] (6,.5) circle (1mm);
\end{tikzpicture}
&&
\text{(Haagerup principal graph)}
\\
&
\begin{tikzpicture}[baseline, scale=.7]
	\draw (0,0)--(3,0);
	\draw (3,0)--(4,.5);
	\draw (3,0)--(4,-.5) -- (5,0);
	\draw (4,-.5) -- (5,-1);
	\filldraw [fill=white]  (0,0) circle (1mm);
	\filldraw [fill] (1,0) circle (1mm);
	\filldraw [fill=white]  (2,0) circle (1mm);
	\filldraw [fill] (3,0) circle (1mm);
	\filldraw [fill=white]  (4,-.5) circle (1mm);
	\filldraw [fill=white]  (4,.5) circle (1mm);
	\filldraw [fill] (5,0) circle (1mm);
	\filldraw [fill] (5,-1) circle (1mm);
\end{tikzpicture}
&&
\text{(dual principal graph)}
\\
&
\begin{tikzpicture}[baseline, scale=.7]
	\draw (1,0)--(4,0);
	\draw (4,.5)--(3,0)--(4,-.5);
	\filldraw [fill=white] (1,0) circle (1mm);
	\filldraw [fill]  (2,0) circle (1mm);
	\filldraw [fill=white] (3,0) circle (1mm);
	\filldraw [fill]  (4,-.5) circle (1mm);
	\filldraw [fill]  (4,.5) circle (1mm);
	\filldraw [fill] (4,0) circle (1mm);
\end{tikzpicture}
&&
\text{(`broom' graph)}
\end{align*}
Here, the two unshaded vertices of the `broom' graph correspond to the third $\cH_2$-module from \cite[Cor.~3.16]{MR2909758}.
\end{cor}

Section \ref{sec:StructureOfEH3} summarizes the combinatorial structure of the four Extended Haagerup fusion categories and the Morita equivalences between them.  In particular, we describe the fusion rules for each fusion category, and give the principal and dual principal graphs for all the subfactors coming from small objects in the bimodule categories.  This section concludes with a table of all lattices of intermediate subfactors coming from these fusion categories, which can be read off from the fusion rules of the bimodule categories following \cite[Cor. 2.4]{MR2670925}.  There are several particularly interesting examples: a (3,3)-supertransitive non-commuting but cocommuting quadrilateral with indices $(7.0283\ldots, 8.0283\ldots)$, a (2,2)-supertransitive non-commuting but cocommuting quadrilateral with indices $(13.3305\ldots, 14.3305\ldots)$, and a hexagonal intermediate subfactor lattice whose lower and upper inclusions are the $7$-supertransitive index $4.3772\ldots$ extended Haagerup subfactor, and whose middle inclusions are $2$-supertransitive with index $13.3305 \ldots$  The first of these quadrilaterals is the smallest known example of index above $4$ in Class II of the Grossman-Izumi classification \cite{MR2418197} of highly supertransitive non-commuting quadrilaterals.  It is striking that the smallest examples in Class III and Class IV are respectively the Asaeda--Haagerup and Haagerup subfactors.

The main goal of Section \ref{sec:ModuleGPA} is to prove Theorem \ref{cor:EmbeddingGivesModule}.  
We begin by recalling some key background about module categories, bimodule categories, Brauer--Picard groupoids, and the Maximal Atlas.  
This background is used throughout the paper.  We also recall  the relationship between module categories and functors $\scrC \rightarrow \End(M)$.  We then relate $\End(M)$ to a version of graph planar algebra. 
We hope this exposition will make graph planar algebra techniques accessible to readers with a background in tensor categories.  
In particular, we prove a purely algebraic analogue of Theorem \ref{cor:EmbeddingGivesModule} directly from the action map $\scrC \rightarrow \End(M)$.  
In order to adapt this simple algebraic argument to prove Theorem \ref{cor:EmbeddingGivesModule}, we recall some technical background on the definition and classification of unitary pivotal structures from \cite{1808.00323}, and the correct unitary pivotal analogues of module categories (analogous to \cite{MR3420332} in the unitary non-pivotal case and \cite{MR3019263} in the non-unitary pivotal case).  
Note that the unitary pivotal version of this result is essential to our main results, because the characterization of maps out of the Extended Haagerup planar algebra in \cite{MR2979509} relies on positive definiteness in order to check more complicated skein relations based on only a few simple skein relations.

In Section \ref{sec:combinatorics} we show that there are at most four fusion categories in the Extended Haagerup Morita equivalence class and exactly one Morita equivalence between each pair of categories in this class.  Furthermore we determine all the fusion rules between objects in each of these potential fusion categories and bimodule categories.  This closely follows the techniques introduced in \cite{MR3449240} by computing compatible fusion rules for the hypothetical fusion categories and bimodule categories.

In Section \ref{sec:construction} we recall and slightly modify the characterization of maps out of the Extended Haagerup planar algebra proved in \cite{MR2979509}.  We then prove Theorem \ref{thm:GraphsForModules}, again following the outline of \cite{MR2979509}.

In the first appendix we give an alternate construction of $\cE\cH_3$ by directly constructing a Q-system on $1 \oplus f^{(6)}$ in the Extended Haagerup planar algebra using its explicit embedding into the graph planar algebra of its principal graph.  
This approach does not work for $\cE\cH_4$ because the smallest Q-system yielding $\cE\cH_4$ lives in too large a box space for computer calculations to be feasible.  
In the second appendix we outline one promising skein theoretic approach which could give a more natural description of $\cE\cH_3$ and $\cE\cH_4$.  These appendices are not intended to appear in the published version.

Throughout the paper we will use the notation $\scrC(X \rightarrow Y)$ to denote the morphisms between objects $X$ and $Y$ in the category $\scrC$.

\subsection{Acknowledgements}

This work was completed at the 2016 and 2017 AIM SQuaRE ``Classifying fusion categories.''
The authors would like to thank AIM for their hospitality.  
NS and DP would like to thank Andr\'e Henriques and Corey Jones for helpful conversations.  In particular, Andr\'e had an immense impact on the ideas and results in Section \ref{sec:ModuleGPA}.
PG was supported by ARC grants DP140100732 and DP170103265.
SM was supported by Discovery
Projects ‘Subfactors and symmetries’ DP140100732 and ‘Low dimensional categories’ DP160103479,
and a Future Fellowship ‘Quantum symmetries’ FT170100019 from the Australian Research Council.
DP was supported by NSF DMS grants 1500387/1655912 and 1654159.
EP was supported by NSF DMS grant 1501116. 
NS was supported by NSF DMS grant 1454767.
%auto-ignore
%this ensures the arxiv doesn't try to start TeXing here.
%!TEX root =../EH3.tex

%%%%%%%%%%%%%%%%%%%%%%%%%%%%%%%%%%%%%%%%%%%%%%%%%%%%%%%%%%%%
%%%%%%%%%%%%%%%%%%%%%%%%%%%%%%%%%%%%%%%%%%%%%%%%%%%%%%%%%%%%
%%%%%%%%%%%%%%%%%%%%%%%%%%%%%%%%%%%%%%%%%%%%%%%%%%%%%%%%%%%%
\section{Facts about the Extended Haagerup fusion categories}
\label{sec:StructureOfEH3}

In this section, we describe the Extended Haagerup fusion categories from Theorem \ref{thm:Main}. 
The logic here is somewhat convoluted; the statements of this section logically depend on the later sections (and we're careful not to use the statements here in those sections!). We have decided to put this summary first in order to make the structure of these new fusion categories as accessible as possible.

Recall that by the main results of this paper, there are exactly four unitary fusion categories $\cE\cH_1$, $\cE\cH_2$, $\cE\cH_3$, and $\cE\cH_4$ in the Extended Haagerup Morita equivalence class, and between any two of these four, there is exactly one Morita equivalence.  
Furthermore, because the center of the Extended Haagerup fusion category has no nontrivial invertible objects, there is a unique composition for each tensor product of bimodules $$\Phi:
{}_\mathcal{A} \mathcal{K} {}_\mathcal{B} \boxtimes_{\mathcal{B}} {}_\mathcal{B} \mathcal{L} {}_\mathcal{C} 
\cong  
{}_\mathcal{A} \mathcal{M} {}_\mathcal{C}.
$$

\begin{nota}
\label{nota:ExtendedHaagerupMultifusionCategory}
For $1\leq i,j\leq 4$, we denote by $\cE\cH_{ij}$ the unique invertible $\cE\cH_i - \cE\cH_j$ 
bimodule category, and notice that $\cE\cH_{ii} = \cE\cH_i$.
One may view $\cE\cH = (\cE\cH_{ij})_{i,j=1}^4$ as a single $4\times 4$ unitary multifusion category.

We interpret the fusion ring $EH:=K_0(\cE\cH)$ for $\cE\cH$ as a single ring whose basis consists of the disjoint union of a set of representatives of simple objects $\Irr(\cE\cH_{ij})$ for each $\cE\cH_{ij}$.
We denote by $EH_{ij}$ the span of the set of representatives of simple objects in $\Irr(\cE\cH_{ij})$, and notice that $EH_{ii} = K_0(\cE\cH_i)$.
Of course the products of objects which are not composable are declared to be zero; that is, $EH$ is faithfully graded by the standard system of matrix units for $M_4(\bbC)$.

Within each $\cE\cH_{ij}$ we order simple objects by increasing dimension, so $O_{ij}^k$ denotes the $k$th smallest simple object in $\cE\cH_{ij}$ (or, abusing notation, the corresponding basis element in the Grothendieck ring).  
When there are duplicate dimensions the ties are broken arbitrarily.

We describe the fusion ring $EH$ in the Mathematica notebook \texttt{EHmult.nb}, which is a wrapper for the data file \texttt{EHmult.txt}, both of which are bundled with the arXiv sources of this article.
Therein, we supply a 6-dimensional tensor $T$ whose $(i,j,k,x,y,z)$-entry is the coefficient of 
$z$-th basis element of $EH_{ik}$ in the product of the
$x$-th basis element of $EH_{ij}$ 
and the
$y$-th basis element of $EH_{jk}$. 
On the level of categories, these coefficients are the dimensions of $\text{Hom} $ spaces between simple objects and tensor products of pairs of simple objects.
That is,
$$
O_{ij}^x \otimes O_{jk}^y \cong \bigoplus_z T(i,j,k,x,y,z) O_{ik}^z.
$$ 
\end{nota}

In this section, we gather information on all the Extended Haagerup fusion categories, including fusion rules, the simplest Q-systems, and intermediate subfactor lattices.

\begin{nota}
\label{notation:TensorOnRight}
Our convention for principal graphs of subfactors and fusion graphs of fusion categories is that we always tensor on the \emph{right}.
In particular, the fusion graph for $X$ has $\dim(\cE\cH(A\otimes X \to B))$ oriented edges between simples $A$ and $B$.
Later, in \S\ref{sec:FusionGraphsForEH2Modules}, we will discuss fusion graphs for left module categories, which use the opposite convention.
\end{nota}

\subsection{Structure of \texorpdfstring{$\cE\cH_1$}{EH1}}
\label{sec:StructureOfEH1}

The fusion rules for $\cE\cH_1$ (which is the dual even half of the Extended Haagerup subfactor) are given by 
\begin{center}
\scalebox{.73}{
\begin{tabular}{c||c|c|c|c|c|}
 $\otimes$ & $\jw{2}$ & $\jw{4}$ & $\jw{6}$ & $P'$ & $Q'$ \\ \hline \hline
 $\jw{2}  $ & $1 \oplus  \jw{2} \oplus  \jw{4}  $ & $ \jw{2} \oplus  \jw{4} \oplus  \jw{6}  $ & $ \jw{4} \oplus  \jw{6} \oplus  P' \oplus  Q'  $ & $ \jw{6} \oplus  2 P' \oplus  Q'  $ & $ \jw{6} \oplus  P' $\\ \hline
$ \jw{4}  $ & $ \jw{2} \oplus  \jw{4} \oplus  \jw{6}  $ & $\begin{array}{c} 1 \oplus  \jw{2} \oplus  \jw{4} \\  \oplus  \jw{6} \oplus  P' \oplus  Q' \end{array}$ & $\begin{array}{c} \jw{2} \oplus  \jw{4} \oplus  2 \jw{6} \\ \oplus  3 P' \oplus  Q'  \end{array}$ & $ \begin{array}{c} \jw{4} \oplus  3 \jw{6} \\ \oplus  3 P' \oplus  2 Q'  \end{array}$ & $\begin{array}{c} \jw{4} \oplus  \jw{6} \\ \oplus  2 P' \oplus  Q' \end{array}$ \\ \hline
$\jw{6}$ & $\jw{4} \oplus  \jw{6} \oplus  P' \oplus  Q'  $ & $\begin{array}{c} \jw{2} \oplus  \jw{4} 2 \jw{6} \\ \oplus  3 P' \oplus  Q'  \end{array}$ & $\begin{array}{c} 1 \oplus \jw{2} \oplus  2 \jw{4}  \\ \oplus  4 \jw{6} \oplus  5 P' \oplus  3 Q'  \end{array}$ & $\begin{array}{c} \jw{2} \oplus  3 \jw{4} \oplus  5 \jw{6}  \\ \oplus  6 P' \oplus  3 Q' \end{array} $ & $\begin{array}{c} \jw{2} \oplus  \jw{4} \oplus  3 \jw{6} \\ \oplus  3 P' \oplus  2 Q' \end{array}$ \\ \hline
$P'$ & $ \jw{6} \oplus  2 P' \oplus  Q' $ & $\begin{array}{c} \jw{4} \oplus  3 \jw{6} \\ \oplus  3 P' \oplus  2 Q'\end{array}$ & $\begin{array}{c} \jw{2} \oplus  3 \jw{4} \oplus  5 \jw{6} \\ \oplus  6 P' \oplus  3 Q'  \end{array}$ & $\begin{array}{c} 1 \oplus  2 \jw{2} \oplus  3 \jw{4} \\ \oplus  6 \jw{6} \oplus  7 P' \oplus  4 Q'   \end{array}$ & $ \begin{array}{c} \jw{2} \oplus  2 \jw{4} \oplus  3 \jw{6} \\ \oplus  4 P' \oplus  2 Q' \end{array}$\\ \hline
 $Q'  $ & $ \jw{6} \oplus  P'  $ & $\begin{array}{c} \jw{4} \oplus  \jw{6} \\ \oplus  2 P' \oplus  Q'  \end{array}$ & $\begin{array}{c} \jw{2} \oplus  \jw{4} \oplus  3 \jw{6} \\ \oplus  3 P' \oplus  2 Q' \end{array}$ & $\begin{array}{c} \jw{2} \oplus  2 \jw{4} \oplus  3 \jw{6} \\ \oplus  4 P' \oplus  2 Q' \end{array}$ & $\begin{array}{c}1 \oplus  \jw{4} \oplus  2 \jw{6}  \\ \oplus  2 P' \oplus  Q' \end{array}$ \\ \hline
\end{tabular}}
\end{center}

Here we have given more informative names to the objects, corresponding to those used in \cite{MR2979509} rather than merely naming them $O_{11}^x$. 
The dimensions of the objects $(\jw{2}, \jw{4}, \jw{6}, P', Q')$ are roughly $(3.4, 7.0, 13.3, 16.0, 8.7)$, and this determines the ordering used in the $O_{11}^x$ notation. 
In particular $O_{11}^1 = \mathbf{1}$, $O_{11}^2 = f^{(2)}$, $O_{11}^3 = f^{(4)}$, $O_{11}^4 = Q'$ (as it has the next smallest dimension), $O_{11}^5 = f^{(6)}$, and $O_{11}^6 = P'$.

A Morita equivalence between $\scrC$ and $\cD$ gives a braided equivalence $Z(\scrC) \cong Z(\cD)$.  Any such $\cD$ is of the form $A$-mod for $A$ a Lagrangian algebra in $Z(\scrC)$  \cite{MR1822847, MR2677836, MR3039775}.  In general there might be several Lagrangian algebras yielding a given $\cD$, but in our case since the Brauer--Picard group is trivial there is a unique Lagrangian algebra for each $\cD$.  Using the notation of \cite{MR3719545} for the objects in the center $Z(\cE\cH)$, the Lagrangian algebra giving $\cE\cH_1$ has underlying object:
$$\omega_0\oplus \omega_1\oplus \omega_2 \oplus \alpha_1 \oplus \alpha_2 \oplus \alpha_3.$$

%%%%%%%%%%%%%%%%%%%%%%%%%%%%%%%%%%%%%%%%%%%%%%%%%%%%%%%%%%%%%%%
\subsection{Structure of \texorpdfstring{$\cE\cH_2$}{EH2}}
\label{sec:StructureOfEH2}

The fusion rules for $\cE\cH_2$ (which is the principal even half of the Extended Haagerup subfactor) are given by 
\begin{center}
\scalebox{0.7}{
\begin{tabular}{c||c|c|c|c|c|c|c|}
$\otimes$  & $\jw{2}$ & $\jw{4}$ & $\jw{6}$ & $P$ & $Q$ & $A$ & $B$ \\
\hline\hline
 $\jw{2}$ & $1\oplus  \jw{2}\oplus \jw{4}$ & $\jw{2}\oplus \jw{4}\oplus \jw{6}$ & $\jw{4} \oplus  W$ & $B \oplus  W$ & $A\oplus  W$ & $Q$ & $P$ \\
 \hline
 $\jw{4}$ & $\jw{2} \oplus  \jw{4} \oplus  \jw{6}$ & $\begin{array}{c}1 \oplus  \jw{2} \oplus  \\ \jw{4} \oplus  W \end{array}$ & $\begin{array}{c} \jw{2} \oplus  \jw{4} \oplus  \\ A \oplus  B \oplus  2 W \end{array}$ & $\jw{4} \oplus  A \oplus  2 W$ & $\jw{4} \oplus  B \oplus  2 W$ & $\jw{6} \oplus  P$ & $\jw{6} \oplus  Q$ \\
 \hline
 $\jw{6}$ & $\jw{4} \oplus  W$ & $\begin{array}{c}\jw{2} \oplus  \jw{4} \oplus  \\ A \oplus  B \oplus  2 W \end{array}$ & $1 \oplus  W \oplus  Z$ & $\jw{6} \oplus  Q \oplus  Z$ & $\jw{6} \oplus  P \oplus  Z$ & $\jw{4} \oplus  B \oplus  W$ & $\jw{4} \oplus  A \oplus  W$ \\
 \hline
 $P$ & $A \oplus  W$ & $\jw{4} \oplus  B \oplus  2 W$ & $\jw{6} \oplus  Q \oplus  Z$ & $1 \oplus  P \oplus  Z$ & $\jw{6} \oplus  Z$ & $\jw{2} \oplus  A \oplus  W$ & $\jw{4} \oplus  W$ \\
 \hline
 $Q$ & $B \oplus  W$ & $ \jw{4} \oplus  A \oplus  2 W$ & $\jw{6} \oplus  P \oplus  Z$ & $\jw{6} \oplus  Z$ & $1 \oplus  Q \oplus  Z$ & $\jw{4} \oplus  W$ & $\jw{2} \oplus  B \oplus  W$ \\
 \hline
 $A$ & $P$ & $\jw{6} \oplus  Q$ & $ \jw{4} \oplus  B \oplus  W$ & 
  $\jw{4} \oplus  W$ & $\jw{2} \oplus  A \oplus  W$ &
  $\jw{6}$ & $1 \oplus  P$ \\
 \hline
 $B$ & $Q$ & $\jw{6} \oplus  P$ & $\jw{4} \oplus  A \oplus  W$ & 
  $\jw{2} \oplus  B \oplus  W$ & $\jw{4} \oplus  W$ &
 $1 \oplus  Q$ & $\jw{6}$ \\ \hline
\end{tabular}
}
\end{center}
\noindent
using the abbreviations $W  = \jw{6}+P+Q$ and $Z  =A + B + \jw{2} + 2 \jw{4} + 3 \jw{6} + 3 P + 3 Q$.

\begin{remark}
In the first appendix we give an alternate construction of $\cE\cH_3$ by constructing a $Q$-system on $1+\jw{6}$ in $\cE\cH_2$ whose dual category is $\cE\cH_3$.  This construction is viable because $$\dim(\Hom(\jw{6}\otimes \jw{6} \to \jw{6})) = 4$$ is not too large.
\end{remark}

The dimensions of $(\jw{2}, \jw{4}, \jw{6}, P, Q, A, B)$ are roughly $(3.4, 7.0, 13.3, 12.3, 12.3, 3.7, 3.7)$, which again determines the ordering in the $O_{22}^x$ naming.

Since $A$ and $B = \overline{A}$ have small dimension we also record the fusion graph for tensoring with $A$ and the principal graph for the corresponding subfactor. The fusion graph is:

$$
%Fusion graph for tensoring on the right with $A$, checked 7/2/18
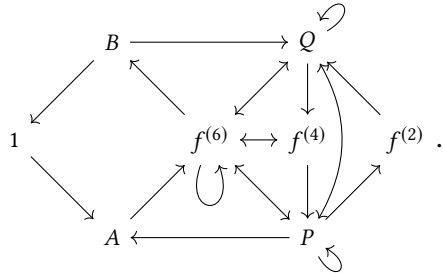
\begin{tikzpicture}[baseline=-1mm, scale=.65]
	\node (1) at (0,0)  {\scriptsize{$1$}};
	\node (A) at (2,-2)  {\scriptsize{$A$}};
	\node (B) at (2,2)  {\scriptsize{$B$}};
	\node (6) at (4,0)  {\scriptsize{$\jw{6}$}};
	\node (4) at (6,0)  {\scriptsize{$\jw{4}$}};
	\node (2) at (8,0)  {\scriptsize{$\jw{2}$}};
	\node (Q) at (6,2)  {\scriptsize{$Q$}};
	\node (P) at (6,-2)  {\scriptsize{$P$}};
	\draw[->] (1) -- (A);
	\draw[->] (A) -- (6);
	\draw[->] (6) -- (B);
	\draw[->] (B) -- (1);
	\draw[<->] (6) -- (4);
	\draw[<->] (6) -- (P);
	\draw[<->] (6) -- (Q);
	\draw[->] (P) -- (2);	
	\draw[->] (2) -- (Q);
	\draw[->] (Q) -- (4);
	\draw[->] (4) -- (P);
	\draw[->] (B) -- (Q);
	\draw[->] (P) -- (A);	
	\draw[->] (6) to [out=250,in=290,looseness=8] (6);
	\draw[->] (P) to [out=300,in=330,looseness=8] (P);
	\draw[->] (Q) to [out=30,in=60,looseness=8] (Q);
	\draw[<->] (P) to [out=60,in=300] (Q);
\end{tikzpicture}.
$$

The corresponding subfactor $1 \subset A\overline{A}$ has index roughly $13.3$.  
It has the following principal and dual principal graph.
Notice that every simple appears twice --- once as an even vertex and once as an odd vertex.

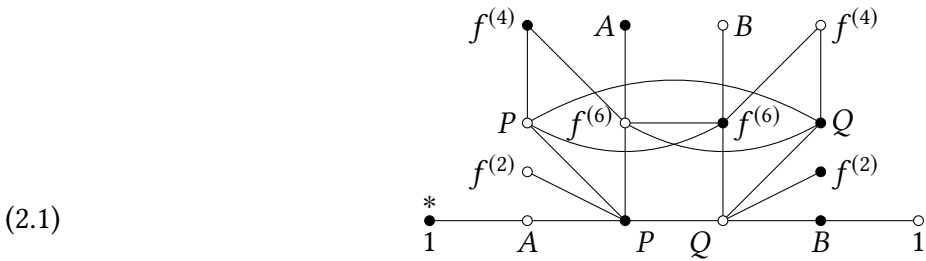
\begin{equation} \label{eq:S13}
%Graph for alternating subfactor associated $A, \bar A$, checked 7/2/18
\begin{tikzpicture}[baseline=-1mm, scale=.65]
	\draw (0,0) -- (2,0);
	\draw (2,0) -- (4,0);
	
	\draw (4,0) -- (6,0);
	\draw (4,0) -- (2,2);
	\draw (4,0) -- (4,2);
	
	\draw (6,0) -- (8,0);
	\draw (6,0) -- (6,2);
	\draw (6,0) -- (8,2);

	\draw(8,0) -- (10,0);

	\draw(2,2) -- (2,4);	
	\draw(4,2) -- (2,4);

	\draw(6,2) -- (8,4);	
	\draw(8,2) -- (8,4);

	\draw(2,2) to [out=30, in=150] (8,2);
	\draw(2,2) to [out=-30, in=-150] (6,2);
	\draw(4,2) to [out=-30, in=-150] (8,2);
	\draw(4,2) -- (6,2);

	\draw(2,1) -- (4,0);
	\draw(8,1) -- (6,0);

	\draw(4,2) -- (4,4);
	\draw(6,2) -- (6,4);	
							
	\filldraw (0,0) node [above] {$*$} node [below] {$1$} circle (1mm);
	\filldraw [fill=white] (2,0) node [below] {$A$} circle (1mm);
	\filldraw (4,4) node [left] {$A$} circle (1mm);
	\filldraw (4,0) node [below right] {$P$} circle (1mm);
	\filldraw (8,1) node [right] {$f^{(2)}$} circle (1mm);
	\filldraw [fill=white] (6,0) node [below left] {$Q$} circle (1mm);
	\filldraw [fill=white] (2,1) node [left] {$f^{(2)}$} circle (1mm);
	\filldraw (8,0) node [below] {$B$} circle (1mm);
	\filldraw [fill=white] (6,4) node [right] {$B$} circle (1mm);
	\filldraw [fill=white] (10,0) node [below] {$1$} circle (1mm);;
	\filldraw [fill=white] (2,2) node [left] {$P$} circle (1mm);
	\filldraw [fill=white] (4,2) node [left] {$f^{(6)}$} circle (1mm);
	\filldraw (6,2) node [right] {$\jw{6}$} circle (1mm);
	\filldraw (8,2) node [right] {$Q$} circle (1mm);
	\filldraw (2,4) node [left] {$\jw{4}$} circle (1mm);	
	\filldraw [fill=white] (8,4) node [right] {$f^{(4)}$} circle (1mm);
\end{tikzpicture}
\end{equation}

The Extended Haagerup subfactor planar algebra constructed in \cite{MR2979509} (see \S\ref{sec:construction}) provides a Morita equivalence between $\cE\cH_1$ and $\cE\cH_2$.  
The generating object $X$ in the Morita equivalence is the smallest object $O_{12}^1$ in the unique invertible bimodule category between these two fusion categories.
The principal graphs are
$$
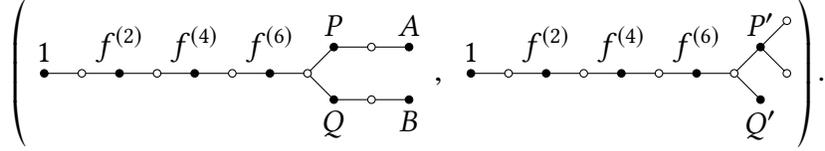
\left(
\begin{tikzpicture}[baseline=-1mm, scale=.5]
	\draw (0,0) -- (7,0);
	\draw (7,0)--(7.7,.7)--(9.7,.7);
	\draw (7,0)--(7.7,-.7)--(9.7,-.7);

	\filldraw              (0,0) node [above] {$1$} circle (1mm);
	\filldraw [fill=white] (1,0) node {} circle (1mm);
	\filldraw              (2,0) node [above] {$\jw{2}$}  circle (1mm);
	\filldraw [fill=white] (3,0) node {}  circle (1mm);
	\filldraw              (4,0) node [above] {$\jw{4}$}  circle (1mm);
	\filldraw [fill=white] (5,0) node {}  circle (1mm);
	\filldraw              (6,0) node [above] {$\jw{6}$}  circle (1mm);
	\filldraw [fill=white] (7,0) node {}  circle (1mm);

	\filldraw              (7.7,.7) node [above] {$P$}  circle (1mm);
	\filldraw [fill=white] (8.7,.7) node [above] {}  circle (1mm);
	\filldraw              (9.7,.7) node [above] {$A$}  circle (1mm);
	\filldraw              (7.7,-.7) node [below] {$Q$}  circle (1mm);
	\filldraw [fill=white] (8.7,-.7) node [above] {}  circle (1mm);
	\filldraw              (9.7,-.7) node [below] {$B$}  circle (1mm);
\end{tikzpicture}
\,,\,
\begin{tikzpicture}[baseline=-1mm, scale=.5]
	\draw (0,0) -- (7,0);
	\draw (7,0)--(7.7,.7);
	\draw (7,0)--(7.7,-.7);
	\draw (7.7,.7)--(8.4,1.4);
	\draw (7.7,.7)--(8.4,0);

	\filldraw              (0,0) node [above] {$1$} circle (1mm);
	\filldraw [fill=white] (1,0) node {} circle (1mm);
	\filldraw              (2,0) node [above] {$\jw{2}$}  circle (1mm);
	\filldraw [fill=white] (3,0) node {}  circle (1mm);
	\filldraw              (4,0) node [above] {$\jw{4}$}  circle (1mm);
	\filldraw [fill=white] (5,0) node {}  circle (1mm);
	\filldraw              (6,0) node [above] {$\jw{6}$}  circle (1mm);
	\filldraw [fill=white] (7,0) node {}  circle (1mm);

	\filldraw              (7.7,.7) node [above] {$P'$}  circle (1mm);
	\filldraw              (7.7,-.7) node [below] {$Q'$}  circle (1mm);
	\filldraw [fill=white] (8.4,1.4) node [above] {}  circle (1mm);
	\filldraw [fill=white] (8.4,0) node [above] {}  circle (1mm);
\end{tikzpicture}
\right).
$$
Notice all even vertices are self dual except for $\overline{A} \cong B$.

The Lagrangian algebra giving $\cE\cH_2$ has underlying object:
$$\omega_0\oplus \omega_1\oplus \omega_2 \oplus 2\alpha_1 \oplus \alpha_2.$$

%%%%%%%%%%%%%%%%%%%%%%%%%%%%%%%%%%%%%%%%%%%%%%%%%%%%%%%%%%%%%%%
\subsection{Structure of \texorpdfstring{$\cE\cH_3$}{EH3}}
The fusion rules for $\cE\cH_3$ are as follows.

\begin{center}
\scriptsize{
\begin{longtable}{c||c|c|c|c|c|c|c|}
$\otimes$ 
&
 $X$ 
&
 $\overline{X}$ 
&
 $Y_ 1$ 
&
 $Y_ 2$ 
&
 $U$ 
&
 $V$ 
&
 $W$ 
\\\hline \hline
$X$ 
&
 $U$ 
&
 $1 {\oplus} Y_ 2$ 
&
 $V {\oplus} X$ 
&
 $W$ 
&
 $\overline{X} {\oplus} W$ 
&
 $V {\oplus} W {\oplus} Y_ 2$ 
&
 $U {\oplus} V {\oplus} W {\oplus} Y_ 1$ 
\\\hline 
$\overline{X}$ 
&
 $1 {\oplus} Y_ 1$ 
&
 $U$ 
&
 $W$ 
&
 $\overline{X} {\oplus} V$ 
&
 $W {\oplus} X$ 
&
 $V {\oplus} W {\oplus} Y_ 1$ 
&
 $U {\oplus} V {\oplus} W {\oplus} Y_ 2$ 
\\\hline 
$Y_ 1$ 
&
 $W$ 
&
 $\overline{X} {\oplus} V$ 
&
 $1 {\oplus} V {\oplus} W {\oplus} Y_ 1$ 
&
 $U {\oplus} V {\oplus} W$ 
&
 $U {\oplus} V {\oplus} W {\oplus} Y_ 2$ 
&
 $\begin{array}{c}
 \overline{X} {\oplus} U {\oplus} 2 V {\oplus}
\\
 2 W {\oplus} Y_ 1 {\oplus} Y_ 2
\end{array}$ 
&
 $\begin{array}{c}
 U {\oplus} 2 V {\oplus} 3 W {\oplus}
\\
 X {\oplus} Y_ 1 {\oplus} Y_ 2
\end{array}$ 
\\\hline 
$Y_ 2$ 
&
 $V {\oplus} X$ 
&
 $W$ 
&
 $U {\oplus} V {\oplus} W$ 
&
 $1 {\oplus} V {\oplus} W {\oplus} Y_ 2$ 
&
 $U {\oplus} V {\oplus} W {\oplus} Y_ 1$ 
&
 $\begin{array}{c}
 U {\oplus} 2 V {\oplus} 2 W {\oplus}
\\
 X {\oplus} Y_ 1 {\oplus} Y_ 2
\end{array}$ 
&
 $\begin{array}{c}
 \overline{X} {\oplus} U {\oplus} 2 V {\oplus}
\\
 3 W {\oplus} Y_ 1 {\oplus} Y_ 2
\end{array}$ 
\\\hline 
$U$ 
&
 $\overline{X} {\oplus} W$ 
&
 $W {\oplus} X$ 
&
 $U {\oplus} V {\oplus} W {\oplus} Y_ 2$ 
&
 $U {\oplus} V {\oplus} W {\oplus} Y_ 1$ 
&
 $\begin{array}{c}
 1 {\oplus} U {\oplus} V {\oplus}
\\
 W {\oplus} Y_ 1 {\oplus} Y_ 2
\end{array}$ 
&
 $U {\oplus} 2 V {\oplus} 3 W {\oplus} Y_ 1 {\oplus} Y_ 2$ 
&
 $\begin{array}{c}
 \overline{X} {\oplus} U {\oplus} 3 V {\oplus}
\\
 3 W {\oplus} X {\oplus} Y_ 1 {\oplus} Y_ 2
\end{array}$ 
\\\hline 
$V$ 
&
 $V {\oplus} W {\oplus} Y_ 1$ 
&
 $V {\oplus} W {\oplus} Y_ 2$ 
&
 $\begin{array}{c}
 U {\oplus} 2 V {\oplus} 2 W {\oplus}
\\
 X {\oplus} Y_ 1 {\oplus} Y_ 2
\end{array}$ 
&
 $\begin{array}{c}
 \overline{X} {\oplus} U {\oplus} 2 V {\oplus}
\\
 2 W {\oplus} Y_ 1 {\oplus} Y_ 2
\end{array}$ 
&
 $U {\oplus} 2 V {\oplus} 3 W {\oplus} Y_ 1 {\oplus} Y_ 2$ 
&
 $\begin{array}{c}
 1 {\oplus} \overline{X} {\oplus} 2 U {\oplus} 4 V {\oplus}
\\
 5 W {\oplus} X {\oplus} 2 Y_ 1 {\oplus} 2 Y_ 2
\end{array}$ 
&
 $\begin{array}{c}
 \overline{X} {\oplus} 3 U {\oplus} 5 V {\oplus}
\\
 6 W {\oplus} X {\oplus} 2 Y_ 1 {\oplus} 2 Y_ 2
\end{array}$ 
\\\hline 
$W$ 
&
 $U {\oplus} V {\oplus} W {\oplus} Y_ 2$ 
&
 $U {\oplus} V {\oplus} W {\oplus} Y_ 1$ 
&
 $\begin{array}{c}
 \overline{X} {\oplus} U {\oplus} 2 V {\oplus}
\\
 3 W {\oplus} Y_ 1 {\oplus} Y_ 2
\end{array}$ 
&
 $\begin{array}{c}
 U {\oplus} 2 V {\oplus} 3 W {\oplus}
\\
 X {\oplus} Y_ 1 {\oplus} Y_ 2
\end{array}$ 
&
 $\begin{array}{c}
 \overline{X} {\oplus} U {\oplus} 3 V {\oplus}
\\
 3 W {\oplus} X {\oplus} Y_ 1 {\oplus} Y_ 2
\end{array}$ 
&
 $\begin{array}{c}
 \overline{X} {\oplus} 3 U {\oplus} 5 V {\oplus}
\\
 6 W {\oplus} X {\oplus} 2 Y_ 1 {\oplus} 2 Y_ 2
\end{array}$ 
&
 $\begin{array}{c}
 1 {\oplus} \overline{X} {\oplus} 3 U {\oplus} 6 V {\oplus}
\\
 7 W {\oplus} X {\oplus} 3 Y_ 1 {\oplus} 3 Y_ 2
\end{array}$ 
\\
\hline
\end{longtable}
%  \captionof{figure}{Fusion rules for $\cE\cH_3$}%
%  \addtocounter{table}{-1}%
}\end{center}

The dimensions of $(X,\overline{X}, Y_1, Y_2, U, V, W)$ are approximately $(2.6, 2.6, 6.0, 6.0, 7.0, 13.3, 15.9)$.  They are listed in the order $O_{33}^2$, $O_{33}^3$, etc.

The fusion graph for $X\in \cE\cH_3$ is given by
$$
%Fusion graph for tensoring on the right with $X$, checked 7/2/18
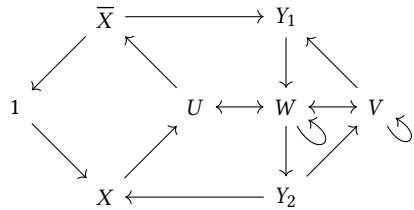
\begin{tikzpicture}[baseline=-1mm, scale=.6]
	\node (1) at (0,0)  {\scriptsize{$1$}};
	\node (X) at (2,-2)  {\scriptsize{$X$}};
	\node (Xb) at (2,2)  {\scriptsize{$\overline{X}$}};
	\node (U) at (4,0)  {\scriptsize{$U$}};
	\node (W) at (6,0)  {\scriptsize{$W$}};
	\node (V) at (8,0)  {\scriptsize{$V$}};
	\node (Y1) at (6,2)  {\scriptsize{$Y_1$}};
	\node (Y2) at (6,-2)  {\scriptsize{$Y_2$}};
	\draw[->] (1) -- (X);
	\draw[->] (X) -- (U);
	\draw[->] (U) -- (Xb);
	\draw[->] (Xb) -- (1);
	\draw[<->] (U) -- (W);
	\draw[->] (Xb) -- (Y1);
	\draw[->] (Y1) -- (W);
	\draw[<->] (W) -- (V);
	\draw[<-] (V) -- (Y2);
	\draw[->] (W) -- (Y2);
	\draw[->] (V) -- (Y1);
	\draw[->] (Y2) -- (X);
	\draw[->] (W) to [out=300,in=330,looseness=8] (W);
	\draw[->] (V) to [out=300,in=330,looseness=8] (V);
\end{tikzpicture}
$$
The principal graph of the subfactor corresponding to the algebra object $X\otimes \overline{X} \in \cE\cH_3$ is given below.  
It has index roughly $7.0$.  
Notice that each simple of $\cE\cH_3$ appears twice --- once as an even vertex and once as an odd vertex.

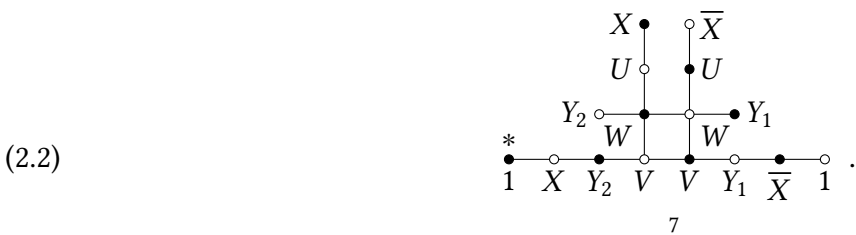
\begin{equation} \label{eq:S7}
%Graph for alternating subfactor associated $X, \bar X$, not checked 7/2/18
\begin{tikzpicture}[baseline=-1mm, scale=.6]
	\draw (0,0) -- (7,0);
	\draw (3,0)--(3,3);
	\draw (4,0)--(4,3);
	\draw (2,1)--(5,1);

	\filldraw              (0,0) node [above] {$*$} node [below] {$1$} circle (1mm);
	\filldraw [fill=white] (1,0) circle (1mm) node [below] {$X$};
	\filldraw              (2,0) circle (1mm) node [below] {$Y_2$};
	\filldraw [fill=white] (3,0) circle (1mm) node [below] {$V$};
	\filldraw              (4,0) circle (1mm) node [below] {$V$};
	\filldraw [fill=white] (5,0) circle (1mm) node [below] {$Y_1$};
	\filldraw              (6,0) circle (1mm) node [below] {$\overline{X}$};
	\filldraw [fill=white] (7,0) circle (1mm) node [below] {$1$};

	\filldraw [fill=white] (2,1) circle (1mm) node [left] {$Y_2$};
	\filldraw              (3,1) circle (1mm) node [below left] {$W$};
	\filldraw [fill=white] (4,1) circle (1mm) node [below right] {$W$};;
	\filldraw              (5,1) circle (1mm) node [right] {$Y_1$};

	\filldraw [fill=white] (3,2) circle (1mm) node [left] {$U$};
	\filldraw              (4,2) circle (1mm) node [right] {$U$};

	\filldraw              (3,3) circle (1mm) node [left] {$X$};
	\filldraw [fill=white] (4,3) circle (1mm) node [right] {$\overline{X}$};
\end{tikzpicture}
\,.
\end{equation}

This paper establishes that $\cE\cH_3$ is a categorification of the above ring, but we don't know that this is the only such categorification.  Thus in order to construct the Extended Haagerup subfactor from an alternative proposed construction of $\cE\cH_3$ one would need to do additional work.
Since the even parts of the Extended Haagerup subfactor are the only categorifications of their fusion rings, it would be enough to construct a categorification of the $\cE\cH_3$ fusion ring plus an algebra structure on $1\oplus U$, and check that the fusion ring of the commutant category corresponding to the algebra $1\oplus U$ is the fusion ring of $\cE\cH_2$. 

We are able to list all Q-systems in $\cE\cH_3$.  The Q-system corresponding to $O_{23}^1$ has relatively small dimension and its underlying object is $1 \oplus U$.  The dual algebra in this case is $1 \oplus f^{(6)}$ in $\cE\cH_2$. 
The index of this subfactor is approximately $14.3$, and the principal graphs are:

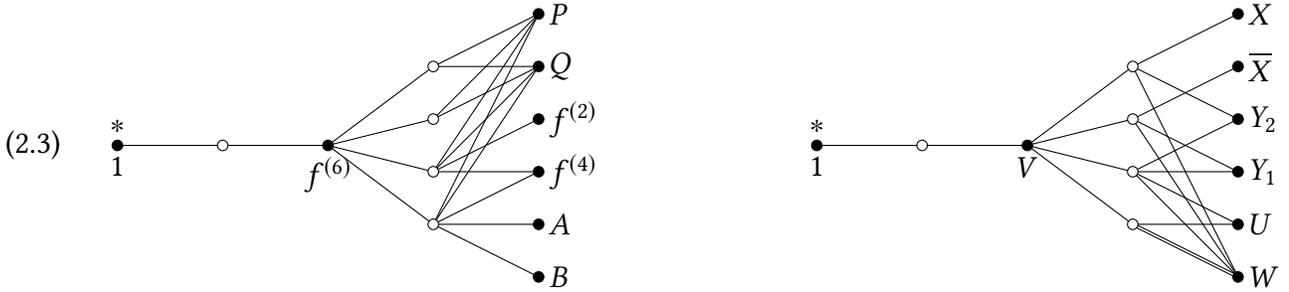
\begin{equation}\label{eq:S14}
%Graph for subfactor associated $\cE\cH_{22},\cE\cH_{23}$, tensoring on right, checked 7/2/18
\begin{tikzpicture}[baseline=-1mm, scale=.7]
	\draw (0,0) -- (2,0);
	\draw (2,0) -- (4,0);
	\draw (4,0) -- (6,1.5);
	\draw (4,0) -- (6,.5);
	\draw (4,0) -- (6,-.5);
	\draw (4,0) -- (6,-1.5);
	\draw (6,1.5)--(8,2.5);
	\draw (6,.5)--(8,2.5);
	\draw (6,-.5)--(8,2.5);
	\draw (6,-1.5)--(8,2.5);
	\draw (6,1.5)--(8,1.5);
	\draw (6,.5)--(8,1.5);
	\draw (6,-.5)--(8,1.5);
	\draw (6,-1.5)--(8,1.5);
	\draw (6,-.5)--(8,.5);
	\draw (6,-.5)--(8,-.5);
	\draw (6,-1.5)--(8,-.5);
	\draw (6,-1.5)--(8,-1.5);
	\draw (6,-1.5)--(8,-2.5);

	\filldraw              (0,0) node [above] {$*$} node [below] {$1$} circle (1mm);
	\filldraw [fill=white] (2,0) circle (1mm) ;
	\filldraw  (4,0) circle (1mm) node [below] {$f^{(6)}$};
	\filldraw [fill=white] (6,1.5) circle (1mm) ;	
	\filldraw [fill=white] (6,.5) circle (1mm) ;	
	\filldraw [fill=white] (6,-.5) circle (1mm) ;	
	\filldraw [fill=white] (6,-1.5) circle (1mm) ;	
	\filldraw  (8,2.5) circle (1mm) node [right] {$P$} ;
	\filldraw  (8,1.5) circle (1mm) node [right] {$Q$} ;
	\filldraw  (8,.5) circle (1mm) node [right] {$f^{(2)}$} ;
	\filldraw  (8,-.5) circle (1mm) node [right] {$f^{(4)}$} ;	
	\filldraw  (8,-1.5) circle (1mm) node [right] {$A$} ;
	\filldraw  (8,-2.5) circle (1mm) node [right] {$B$} ;	
\end{tikzpicture}
\hspace{1in}
\begin{tikzpicture}[baseline=-1mm, scale=.7]
	\draw (0,0) -- (2,0);
	\draw (2,0) -- (4,0);
	\draw (4,0) -- (6,1.5);
	\draw (4,0) -- (6,.5);
	\draw (4,0) -- (6,-.5);
	\draw (4,0) -- (6,-1.5);

	\draw (6,1.5)--(8,2.5);

	\draw (6,.5)--(8,1.5);

	\draw (6,1.5)--(8,.5);
	\draw (6,-.5)--(8,.5);

	\draw (6,.5)--(8,-.5);
	\draw (6,-.5)--(8,-.5);

	\draw (6,-.5)--(8,-1.5);
	\draw (6,-1.5)--(8,-1.5);
	
	\draw[double] (6,-1.5)--(8,-2.5);
	\draw (6,-.5)--(8,-2.5);
	\draw (6,.5)--(8,-2.5);
	\draw (6,1.5)--(8,-2.5);

	\filldraw (0,0) node [above] {$*$} node [below] {$1$} circle (1mm);
	\filldraw [fill=white] (2,0) circle (1mm) ;
	\filldraw  (4,0) circle (1mm) node [below] {$V$};
	\filldraw [fill=white] (6,1.5) circle (1mm) ;	
	\filldraw [fill=white] (6,.5) circle (1mm) ;	
	\filldraw [fill=white] (6,-.5) circle (1mm) ;	
	\filldraw [fill=white] (6,-1.5) circle (1mm) ;	
	\filldraw  (8,2.5) circle (1mm) node [right] {$X$} ;
	\filldraw  (8,1.5) circle (1mm) node [right] {$\overline{X}$} ;
	\filldraw  (8,.5) circle (1mm) node [right] {$Y_2$} ;
	\filldraw  (8,-.5) circle (1mm) node [right] {$Y_1$} ;	
	\filldraw  (8,-1.5) circle (1mm) node [right] {$U$} ;
	\filldraw  (8,-2.5) circle (1mm) node [right] {$W$} ;	
\end{tikzpicture}
\,.
\end{equation}

%Checked in EHmult.nb on 7/2/18 for tensoring convention:
% \Gamma_+ --> -\otimes X and \Gamma_- --> -\otimes \overline{X}.

The Lagrangian algebra in $Z(\cE\cH)$ giving $\cE\cH_3$ has underlying object:
$$\omega_0\oplus \omega_1\oplus \omega_2 \oplus 2\alpha_1 \oplus \alpha_2.$$
Note that this is the same underlying object as the Lagrangian algebra corresponding to $\cE\cH_2$; the algebra structures must be different.

%%%%%%%%%%%%%%%%%%%%%%%%%%%%%%%%%%%%%%%%%%
\subsection{Structure of \texorpdfstring{$\cE\cH_4$}{EH4}}

We now turn to describing some of the combinatorial structure of $\cE\cH_4$.  The fusion rules for $\cE\cH_4$ are as follows.

{\tiny{
\begin{longtable}{c||c|c|c|c|c|c|c|}
$\otimes$ 
&
 $Z$ 
&
 $\overline{Z}$ 
&
 $G$ 
&
 $H$ 
&
 $K_ 1$ 
&
 $K_ 2$ 
&
 $L$ 
\\\hline \hline
$Z$ 
&
 $G {\oplus} K_ 1 {\oplus} K_ 2 {\oplus} L$ 
&
 $\begin{array}{c}
 1 {\oplus} G {\oplus}
\\
 H {\oplus} K_ 2 {\oplus} L
\end{array}$ 
&
 $\begin{array}{c}
 H {\oplus} K_ 2 {\oplus}
\\
 L {\oplus} \overline{Z} {\oplus} Z
\end{array}$ 
&
 $\begin{array}{c}
 G {\oplus} H {\oplus} K_ 1 {\oplus}
\\
 K_ 2 {\oplus} L {\oplus} Z
\end{array}$ 
&
 $\begin{array}{c}
 G {\oplus} H {\oplus} K_ 1 {\oplus}
\\
 K_ 2 {\oplus} L {\oplus} \overline{Z} {\oplus} Z
\end{array}$ 
&
 $\begin{array}{c}
 H {\oplus} K_ 1 {\oplus}
\\
 K_ 2 {\oplus} 2 L {\oplus} \overline{Z}
\end{array}$ 
&
 $\begin{array}{c}
 G {\oplus} H {\oplus} 2 K_ 1 {\oplus}
\\
 K_ 2 {\oplus} 2 L {\oplus} \overline{Z} {\oplus} Z
\end{array}$ 
\\\hline 
$\overline{Z}$ 
&
 $\begin{array}{c}
 1 {\oplus} G {\oplus}
\\
 H {\oplus} K_ 1 {\oplus} L
\end{array}$ 
&
 $G {\oplus} K_ 1 {\oplus} K_ 2 {\oplus} L$ 
&
 $\begin{array}{c}
 H {\oplus} K_ 1 {\oplus}
\\
 L {\oplus} \overline{Z} {\oplus} Z
\end{array}$ 
&
 $\begin{array}{c}
 G {\oplus} H {\oplus} K_ 1 {\oplus}
\\
 K_ 2 {\oplus} L {\oplus} \overline{Z}
\end{array}$ 
&
 $\begin{array}{c}
 H {\oplus} K_ 1 {\oplus}
\\
 K_ 2 {\oplus} 2 L {\oplus} Z
\end{array}$ 
&
 $\begin{array}{c}
 G {\oplus} H {\oplus} K_ 1 {\oplus}
\\
 K_ 2 {\oplus} L {\oplus} \overline{Z} {\oplus} Z
\end{array}$ 
&
 $\begin{array}{c}
 G {\oplus} H {\oplus} K_ 1 {\oplus}
\\
 2 K_ 2 {\oplus} 2 L {\oplus} \overline{Z} {\oplus} Z
\end{array}$ 
\\\hline 
$G$ 
&
 $\begin{array}{c}
 H {\oplus} K_ 1 {\oplus}
\\
 L {\oplus} \overline{Z} {\oplus} Z
\end{array}$ 
&
 $\begin{array}{c}
 H {\oplus} K_ 2 {\oplus}
\\
 L {\oplus} \overline{Z} {\oplus} Z
\end{array}$ 
&
 $\begin{array}{c}
 1 {\oplus} G {\oplus} H {\oplus}
\\
 K_ 1 {\oplus} K_ 2 {\oplus} L
\end{array}$ 
&
 $\begin{array}{c}
 G {\oplus} H {\oplus} K_ 1 {\oplus}
\\
 K_ 2 {\oplus} L {\oplus} \overline{Z} {\oplus} Z
\end{array}$ 
&
 $\begin{array}{c}
 G {\oplus} H {\oplus} K_ 1 {\oplus}
\\
 K_ 2 {\oplus} 2 L {\oplus} Z
\end{array}$ 
&
 $\begin{array}{c}
 G {\oplus} H {\oplus} K_ 1 {\oplus}
\\
 K_ 2 {\oplus} 2 L {\oplus} \overline{Z}
\end{array}$ 
&
 $\begin{array}{c}
 G {\oplus} H {\oplus} 2 K_ 1 {\oplus}
\\
 2 K_ 2 {\oplus} 2 L {\oplus} \overline{Z} {\oplus} Z
\end{array}$ 
\\\hline 
$H$ 
&
 $\begin{array}{c}
 G {\oplus} H {\oplus} K_ 1 {\oplus}
\\
 K_ 2 {\oplus} L {\oplus} Z
\end{array}$ 
&
 $\begin{array}{c}
 G {\oplus} H {\oplus} K_ 1 {\oplus}
\\
 K_ 2 {\oplus} L {\oplus} \overline{Z}
\end{array}$ 
&
 $\begin{array}{c}
 G {\oplus} H {\oplus} K_ 1 {\oplus}
\\
 K_ 2 {\oplus} L {\oplus} \overline{Z} {\oplus} Z
\end{array}$ 
&
 $\begin{array}{c}
 1 {\oplus} G {\oplus} H {\oplus} K_ 1 {\oplus}
\\
 K_ 2 {\oplus} 2 L {\oplus} \overline{Z} {\oplus} Z
\end{array}$ 
&
 $\begin{array}{c}
 G {\oplus} H {\oplus} 2 K_ 1 {\oplus}
\\
 K_ 2 {\oplus} 2 L {\oplus} \overline{Z} {\oplus} Z
\end{array}$ 
&
 $\begin{array}{c}
 G {\oplus} H {\oplus} K_ 1 {\oplus}
\\
 2 K_ 2 {\oplus} 2 L {\oplus} \overline{Z} {\oplus} Z
\end{array}$ 
&
 $\begin{array}{c}
 G {\oplus} 2 H {\oplus} 2 K_ 1 {\oplus}
\\
 2 K_ 2 {\oplus} 3 L {\oplus} \overline{Z} {\oplus} Z
\end{array}$ 
\\\hline 
$K_ 1$ 
&
 $\begin{array}{c}
 H {\oplus} K_ 1 {\oplus}
\\
 K_ 2 {\oplus} 2 L {\oplus} \overline{Z}
\end{array}$ 
&
 $\begin{array}{c}
 G {\oplus} H {\oplus} K_ 1 {\oplus}
\\
 K_ 2 {\oplus} L {\oplus} \overline{Z} {\oplus} Z
\end{array}$ 
&
 $\begin{array}{c}
 G {\oplus} H {\oplus} K_ 1 {\oplus}
\\
 K_ 2 {\oplus} 2 L {\oplus} \overline{Z}
\end{array}$ 
&
 $\begin{array}{c}
 G {\oplus} H {\oplus} 2 K_ 1 {\oplus}
\\
 K_ 2 {\oplus} 2 L {\oplus} \overline{Z} {\oplus} Z
\end{array}$ 
&
 $\begin{array}{c}
 1 {\oplus} G {\oplus} 2 H {\oplus} K_ 1 {\oplus}
\\
 2 K_ 2 {\oplus} 2 L {\oplus} \overline{Z} {\oplus} Z
\end{array}$ 
&
 $\begin{array}{c}
 G {\oplus} H {\oplus} 2 K_ 1 {\oplus}
\\
 2 K_ 2 {\oplus} 2 L {\oplus} \overline{Z} {\oplus} Z
\end{array}$ 
&
 $\begin{array}{c}
 2 G {\oplus} 2 H {\oplus} 2 K_ 1 {\oplus}
\\
 2 K_ 2 {\oplus} 3 L {\oplus} \overline{Z} {\oplus} 2 Z
\end{array}$ 
\\\hline 
$K_ 2$ 
&
 $\begin{array}{c}
 G {\oplus} H {\oplus} K_ 1 {\oplus}
\\
 K_ 2 {\oplus} L {\oplus} \overline{Z} {\oplus} Z
\end{array}$ 
&
 $\begin{array}{c}
 H {\oplus} K_ 1 {\oplus}
\\
 K_ 2 {\oplus} 2 L {\oplus} Z
\end{array}$ 
&
 $\begin{array}{c}
 G {\oplus} H {\oplus} K_ 1 {\oplus}
\\
 K_ 2 {\oplus} 2 L {\oplus} Z
\end{array}$ 
&
 $\begin{array}{c}
 G {\oplus} H {\oplus} K_ 1 {\oplus}
\\
 2 K_ 2 {\oplus} 2 L {\oplus} \overline{Z} {\oplus} Z
\end{array}$ 
&
 $\begin{array}{c}
 G {\oplus} H {\oplus} 2 K_ 1 {\oplus}
\\
 2 K_ 2 {\oplus} 2 L {\oplus} \overline{Z} {\oplus} Z
\end{array}$ 
&
 $\begin{array}{c}
 1 {\oplus} G {\oplus} 2 H {\oplus} 2 K_ 1 {\oplus}
\\
 K_ 2 {\oplus} 2 L {\oplus} \overline{Z} {\oplus} Z
\end{array}$ 
&
 $\begin{array}{c}
 2 G {\oplus} 2 H {\oplus} 2 K_ 1 {\oplus}
\\
 2 K_ 2 {\oplus} 3 L {\oplus} 2 \overline{Z} {\oplus} Z
\end{array}$ 
\\\hline 
$L$ 
&
 $\begin{array}{c}
 G {\oplus} H {\oplus} K_ 1 {\oplus}
\\
 2 K_ 2 {\oplus} 2 L {\oplus} \overline{Z} {\oplus} Z
\end{array}$ 
&
 $\begin{array}{c}
 G {\oplus} H {\oplus} 2 K_ 1 {\oplus}
\\
 K_ 2 {\oplus} 2 L {\oplus} \overline{Z} {\oplus} Z
\end{array}$ 
&
 $\begin{array}{c}
 G {\oplus} H {\oplus} 2 K_ 1 {\oplus}
\\
 2 K_ 2 {\oplus} 2 L {\oplus} \overline{Z} {\oplus} Z
\end{array}$ 
&
 $\begin{array}{c}
 G {\oplus} 2 H {\oplus} 2 K_ 1 {\oplus}
\\
 2 K_ 2 {\oplus} 3 L {\oplus} \overline{Z} {\oplus} Z
\end{array}$ 
&
 $\begin{array}{c}
 2 G {\oplus} 2 H {\oplus} 2 K_ 1 {\oplus}
\\
 2 K_ 2 {\oplus} 3 L {\oplus} 2 \overline{Z} {\oplus} Z
\end{array}$ 
&
 $\begin{array}{c}
 2 G {\oplus} 2 H {\oplus} 2 K_ 1 {\oplus}
\\
 2 K_ 2 {\oplus} 3 L {\oplus} \overline{Z} {\oplus} 2 Z
\end{array}$ 
&
 $\begin{array}{c}
 1 {\oplus} 2 G {\oplus} 3 H {\oplus} 3 K_ 1 {\oplus}
\\
 3 K_ 2 {\oplus} 4 L {\oplus} 2 \overline{Z} {\oplus} 2 Z
\end{array}$ 
\\
\hline
\end{longtable}
}}
The dimensions of the objects $(Z, \overline{Z}, G, H, K_1, K_2, L)$ are approximately $(6.3, 6.3, 7.0, 8.6, 9.6, 9.6, 13.3)$.  These are already in increasing order, so the objects $O_{44}^x$ appear in this order.

None of the objects in $\cE\cH_4$ is small enough to have a nice fusion graph.  The Q-system corresponding to $O_{34}^1$ does give a $3$-supertransitive subfactor.  The underlying object of this Q-system is $1 \oplus G$ and the dual Q-system is $1 \oplus U$ in $\cE\cH_3$, and the index is roughly $8.0$.  The principal graphs for this subfactor are given by:

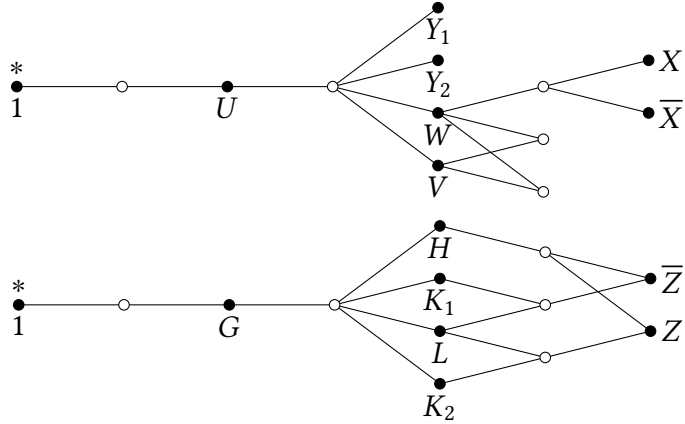
\begin{align}\label{eq:S8}
%Subfactor graph checked on 7/2/18, tensoring on right.
\begin{tikzpicture}[baseline=-1mm, scale=.7]
	\draw (0,0) -- (2,0);
	\draw (2,0) -- (4,0);
	\draw (4,0) -- (6,0);
	\draw (6,0) -- (8,1.5);
	\draw (6,0) -- (8,.5);	
	\draw (6,0) -- (8,-.5);
	\draw (6,0) -- (8,-1.5);
	\draw (8,-.5) -- (10,0);
	\draw (8,-.5) -- (10,-1);
	\draw (8,-.5) -- (10,-2);
	\draw (8,-1.5) -- (10,-1);
	\draw (8,-1.5) -- (10,-2);
	\draw (10,0) -- (12,.5);
	\draw (10,0) -- (12,-.5);
	\filldraw (0,0) node [above] {$*$} node [below] {$1$} circle (1mm);
	\filldraw [fill=white] (2,0) circle (1mm) ;
	\filldraw  (4,0) circle (1mm) node [below] {$U$};
	\filldraw [fill=white] (6,0) circle (1mm) ;	
	\filldraw  (8,1.5) circle (1mm) node [below] {$Y_1$} ;
	\filldraw  (8,.5) circle (1mm) node [below] {$Y_2$} ;
	\filldraw  (8,-.5) circle (1mm) node [below] {$W$} ;
	\filldraw  (8,-1.5) circle (1mm) node [below] {$V$} ;	
	\filldraw [fill=white] (10,0) circle (1mm) ;		
	\filldraw [fill=white] (10,-1) circle (1mm) ;		
	\filldraw [fill=white] (10,-2) circle (1mm) ;		
	\filldraw  (12,.5) circle (1mm) node [right] {$X$} ;
	\filldraw  (12,-.5) circle (1mm) node [right] {$\overline{X}$} ;	
\end{tikzpicture}
\displaybreak[1]\\
\begin{tikzpicture}[baseline=-1mm, scale=.7]
	\draw (0,0) -- (2,0);
	\draw (2,0) -- (4,0);
	\draw (4,0) -- (6,0);
	\draw (6,0) -- (8,1.5);
	\draw (6,0) -- (8,.5);	
	\draw (6,0) -- (8,-.5);
	\draw (6,0) -- (8,-1.5);
	\draw (8,1.5) -- (10,1);
	\draw (8,.5) -- (10,0);
	\draw (8,-.5) -- (10,0);
	\draw (8,-.5) -- (10,-1);	
	\draw (8,-1.5) -- (10,-1);
	\draw (10,1) -- (12,.5);
	\draw (10,1) -- (12,-.5);
	\draw (10,0) -- (12,.5);
	\draw (10,-1) -- (12,-.5);
	\filldraw (0,0) node [above] {$*$} node [below] {$1$} circle (1mm);
	\filldraw [fill=white] (2,0) circle (1mm) ;
	\filldraw  (4,0) circle (1mm) node [below] {$G$};
	\filldraw [fill=white] (6,0) circle (1mm) ;	
	\filldraw  (8,1.5) circle (1mm) node [below] {$H$} ;
	\filldraw  (8,.5) circle (1mm) node [below] {$K_1$} ;
	\filldraw  (8,-.5) circle (1mm) node [below] {$L$} ;
	\filldraw  (8,-1.5) circle (1mm) node [below] {$K_2$} ;	
	\filldraw [fill=white] (10,1) circle (1mm) ;		
	\filldraw [fill=white] (10,0) circle (1mm) ;		
	\filldraw [fill=white] (10,-1) circle (1mm) ;		
	\filldraw  (12,.5) circle (1mm) node [right] {$\overline{Z}$} ;
	\filldraw  (12,-.5) circle (1mm) node [right] {$Z$} ;	
\end{tikzpicture}
\notag
\end{align}

%Checked 7/2/18

As with $\cE\cH_2$ and $\cE\cH_3$, the Lagrangian algebra giving $\cE\cH_4$ has underlying object:
$$\omega_0\oplus \omega_1\oplus \omega_2 \oplus 2\alpha_1 \oplus \alpha_2.$$  We thus see that this object must have three distinct algebra structures on it.

%%%%%%%%%%%%%%%%%%%%%%%%%%%%%%%%%%%%%%%%%%%%%%%%%%%%%%%%%%%%
\subsection{Intermediate subfactors}
\label{sec:intermediate}

In this section we describe all the lattices of intermediate subfactors for subfactors coming from the objects in the Extended Haagerup bimodule categories.  
These can be read directly from the fusion rules for the bimodules following \cite[Cor. 2.4]{MR2670925}.  
The relationship between bimodule fusion rules and intermediate subfactors is particularly simple for Extended Haagerup because the categories have no nontrivial outer automorphisms and no invertible objects. 
In such a case intermediate subfactors correspond exactly to triples of objects $X$, $Y$, $Z$ in the bimodule categories such that $X \otimes Y \cong Z$.  
The corresponding subfactors are $1 \subset X \otimes \overline{X} \subset  X \otimes (Y \otimes \overline{Y}) \otimes \overline{X} = Z \otimes \overline{Z}$.  
The small subfactors $1 \subset X \otimes \overline{X}$ and $X \otimes \overline{X} \subset X \otimes (Y \otimes \overline{Y}) \otimes \overline{X}$ are irreducible exactly when $X$ and $Y$ are, and the large subfactor $1 \subset Z \otimes \overline{Z}$ is irreducible exactly when $X \otimes Y$ is simple.  
Note that it is rare to have a pair of simple objects $X, Y$ such that the product $X \otimes Y$ is simple, and this explains why 
irreducible subfactors typically have few intermediates.

Each row of the following table lists triples of simple objects $X = O_{ij}^x$, $Y = O_{jk}^y$, and their simple product $X \otimes Y = O_{ik}^z$.
The columns labelled `ST' and `Index' indicate the supertransitivity and index of the corresponding subfactors.
We only list one representative of each dual pair, so in the tables we always have $i<j$.  
We've grouped rows together according to the identity of the large object $Z$, as these rows are all intermediate subfactors of the same large subfactor.

We recall that:
\begin{itemize}
\item $O_{21}^1$ and $O_{12}^1$ correspond to the Extended Haagerup subfactor and its dual,
\item $O_{33}^2$ corresponds to the subfactor with principal graph shown in Equation \eqref{eq:S7}, 
\item $O_{34}^1$ and $O_{43}^1$ correspond to the subfactor with principal graph shown in Equation \eqref{eq:S8},
\item $O_{22}^3$ and  $O_{22}^4$ both correspond to the subfactor with principal graph shown in Equation \eqref{eq:S13}, and 
\item $O_{23}^1$ and $O_{32}^1$ correspond to the subfactor with principal graph shwon in Equation \eqref{eq:S14}.
\end{itemize}

\begin{longtable}{C|C||C|C|C||C|C|C}
Z& \mathrm{Index}(Z)& X& \mathrm{ST}(X)& \mathrm{Index}(X)& Y& \mathrm{ST}(Y)& \mathrm{Index}(Y) \\ \hline
O_{11}^6 & 255.411 & O_{12}^4& 1 & 58.3502& O_{21}^1 & 7 & 4.3772 \\
 & & O_{12}^3 & 1 & 58.3502 & O_{21}^1 & 7 & 4.3772 \\
  & & O_{12}^1 & 7 & 4.3772 & O_{21}^3& 1 & 58.3502 \\
  & & O_{12}^1 & 7 & 4.3772 & O_{21}^4& 1 & 58.3502 \\ \hline
O_{12}^3 & 58.3502 & O_{12}^1& 7 & 4.3772 & O_{22}^4 & 2 & 13.3305 \\ \hline
O_{12}^4 & 58.3502 & O_{12}^1& 7 & 4.3772 & O_{22}^3 & 2 & 13.3305 \\ \hline
O_{12}^6 & 329.743 & O_{13}^1 & 1 & 23.0099 &  O_{32}^1 & 2 & 14.3305\\
 &  & O_{13}^2 & 1 & 23.0099 &  O_{32}^1 & 2 & 14.3305\\
 &  & O_{12}^2 & 1 & 24.736  & O_{22}^3 & 2 & 13.3305\\
 &  & O_{12}^2 & 1 & 24.736  & O_{22}^4 & 2 & 13.3305\\
 &  & O_{11}^4 & 1 & 75.3318 & O_{12}^1 & 7 & 4.3772\\ \hline
O_{13}^4 & 62.7274 & O_{12}^1 & 7 & 4.3772 & O_{23}^1 & 2 & 14.3305 \\ \hline
O_{13}^5 & 161.72 & O_{13}^1 & 1 & 23.0099 & O_{33}^2 & 3 & 7.0283 \\
 &  & O_{13}^2 & 1 & 23.0099 & O_{33}^3 & 3 & 7.0283 \\ \hline
O_{13}^6 & 262.439 & O_{14}^1 & 1 & 32.6893 & O_{43}^1 & 3 & 8.0283 \\
 & & O_{14}^2 & 1 & 32.6893 & O_{43}^1 & 3 & 8.0283 \\
 & & O_{13}^3 & 1 & 37.3404 & O_{33}^2 & 3 & 7.0283 \\ 
 & & O_{13}^3 & 1 & 37.3404 & O_{33}^3 & 3 & 7.0283 \\ 
 & & O_{11}^2 & 1 & 11.4055 & O_{13}^1 & 1 & 23.0099 \\
 & & O_{11}^2 & 1 & 11.4055 & O_{13}^2 & 1 & 23.0099 \\ \hline
O_{22}^6 & 152.041 & O_{22}^2 & 1 & 11.4055 & O_{22}^4 & 2 & 13.3305 \\
 & & O_{22}^3 & 2 & 13.3305 & O_{22}^2 & 1 & 11.4055 \\ \hline
O_{22}^7 & 152.041 & O_{22}^2 & 1 & 11.4055 & O_{22}^3 & 2 & 13.3305 \\
 & & O_{22}^4 & 2 & 13.3305 & O_{22}^2 & 1 & 11.4055 \\ \hline
O_{22}^8 & 177.702 & O_{22}^3& 2 & 13.3305 & O_{22}^3 & 2 & 13.3305\\
 &  & O_{22}^4 & 2 & 13.3305 & O_{22}^4 & 2 & 13.3305\\ \hline
O_{23}^2 & 100.719 & O_{23}^1 & 2 & 14.3305 & O_{33}^3 & 3 & 7.0283 \\
 & & O_{21}^1 & 7 & 4.3772 & O_{13}^2 & 1 & 23.0099\\ \hline
O_{23}^3 & 100.719 & O_{23}^1 & 2 & 14.3305 & O_{33}^2 & 3 & 7.0283 \\
 & & O_{21}^1 & 7 & 4.3772 & O_{13}^1 & 1 & 23.0099\\ \hline
 O_{23}^4 & 163.446 & O_{24}^1 & 1 & 20.3588 & O_{43}^1 & 3 & 8.0283\\
& & O_{24}^2 & 1 & 20.3588 & O_{43}^1 & 3 & 8.0283\\
& & O_{22}^2 & 1 & 11.4055 & O_{23}^1 & 2 & 14.3305 \\
& & O_{21}^1 & 7 & 4.3772 & O_{13}^3 & 1 & 37.3404\\ \hline
O_{23}^5 & 191.032 & O_{22}^3 & 2 & 13.3305 & O_{23}^1 & 2 & 14.3305 \\
 &  & O_{22}^4 & 2 & 13.3305 & O_{23}^1 & 2 & 14.3305 \\
O_{24}^3 & 115.049 & O_{23}^1 & 2 & 14.3305 & O_{34}^1 & 3 & 8.0283 \\ \hline
O_{24}^4 & 143.088 & O_{21}^1 & 7 & 4.3772 & O_{14}^1 & 1 & 32.6893 \\
 &  & O_{21}^1 & 7 & 4.3772 & O_{14}^2 & 1 & 32.6893 \\ \hline
O_{24}^5 & 271.392 & O_{22}^4 & 2 & 13.3305 & O_{24}^1 & 1 & 20.3588 \\
 &  & O_{22}^3 & 2 & 13.3305 & O_{24}^2 & 1 & 20.3588 \\
 & & O_{21}^1 & 7 & 4.3772 & O_{14}^3 &1 & 62.0013 \\ \hline
O_{33}^6 & 49.3969 & O_{33}^2 & 3 & 7.0283 & O_{33}^2 & 3 & 7.0283 \\
 &  & O_{33}^3 & 3 & 7.0283 & O_{33}^3 & 3 & 7.0283 \\ \hline
 O_{33}^8 & 255.411 & O_{33}^5 & 1 & 36.3404 & O_{33}^3 & 3 & 7.0283\\
  & & O_{33}^4 & 1 & 36.3404 & O_{33}^2 & 3 & 7.0283 \\
  & & O_{33}^3 & 3 & 7.0283 & O_{33}^4 & 1 & 36.3404 \\ 
  & & O_{33}^2 & 3 & 7.0283 & O_{33}^5 & 1 & 36.3404 \\ \hline\newpage\hline
 O_{34}^2 & 56.4252 & O_{33}^2 & 3 & 7.0283 & O_{34}^1 & 3 & 8.0283\\
 &  & O_{33}^3 & 3 & 7.0283 & O_{34}^1 & 3 & 8.0283\\ \hline
O_{34}^5 & 291.751 & O_{33}^4 & 1 & 36.3404 & O_{34}^1 & 3 & 8.0283\\
 &  & O_{33}^5 & 1 & 36.3404 & O_{34}^1 & 3 & 8.0283\\
 & & O_{32}^1 & 2 & 14.3305 & O_{24}^1 & 1 & 20.3588 \\
 & & O_{32}^1 & 2 & 14.3305 & O_{24}^2 & 1 & 20.3588 \\
\end{longtable}

Let us briefly summarise the interesting subfactor lattices encoded in the above table.  

There are four lines with $Z = O_{11}^6$, so there are four intermediate subfactors of the index $255.411\ldots$ subfactor corresponding to $O_{11}^6$.  
Note that $O_{11}^6$ denotes the 6th smallest object in $\cE\cH_1$, which is $P'$, so this subfactor is the reduced subfactor corresponding to $P'$.  
For two of these intermediates the lower inclusion is Extended Haagerup while the upper inclusions have index $58.3502\ldots$ and come from the reduced subfactor construction for $O_{12}^3$ or $O_{12}^4$ (these are the odd vertices near the ends of the Extended Haagerup principal graph).  For the other two, the lower and upper parts are switched.  
We next want to see how these fit together into a lattice.  From the next two lines in the table we see that the index $58.3502\ldots$ subfactors themselves each have a single intermediate, with one inclusion being Extended Haagerup and the other being the $2$-supertransitive subfactor of index $13.3305\ldots$ with principal graph shown in Equation \eqref{eq:S13}.  
It follows that the lattice is a hexagon, where the upper and lower edges are the Extended Haagerup subfactor and the middle edges are the index $13.3305\ldots$ subfactors.

Note that since none of the other entries in the $X$ or $Y$ columns also occurs in the $Z$ column, other than the hexagon every lattice will just be $M_n$, the lattice with one maximal element, one minimal element, and $n$ incomparable elements between them.  The number of such incomparable entries is simply the number of rows with that $Z$; for example, $O_{13}^6$ has intermediate subfactor lattice $M_6$.

In addition to the hexagon there are a few notable examples of quadrilaterals where all inclusions are at least $2$-supertransitive.

\begin{itemize}
\item A quadrilateral from $O_{33}^6 = U$ which follows from $X \otimes X \cong U \cong \overline{X} \otimes \overline{X}$. This quadrilateral is commuting and cocommuting.  This quadrilateral suggests a possible diagrammatic presentation for $\cE\cH_3$, described in Appendix \ref{sec:YBR}.

\item A quadrilateral from $O_{22}^8 = \jw{6}$ which follows from $A \otimes A \cong \jw{6} \cong \jw{6} \cong B \otimes B$ in $\cE\cH_2$.  This quadrilateral is commuting and cocommuting.  As for the previous quadrilateral, this may lead to a diagrammatic presentation for $\cE\cH_2$.

\item A quadrilateral from $O_{23}^5$ whose dual is a non-commuting but cocommuting $(2,2)$-supertransitive quadrilateral where the subfactor is the index $14.3305\ldots$ one from Equation \eqref{eq:S14} and the upper quadrilateral is the index $13.3305\ldots$ from Equation \eqref{eq:S13}.

\item A quadrilateral from $O_{34}^2$ whose dual is a non-commuting but cocommuting $(3,3)$-supertransitive quadrilateral where the subfactor is the index $8.0283\ldots$ one from Equation \eqref{eq:S8} and the upper quadrilateral is the index $7.0283\ldots$ from Equation \eqref{eq:S7}.
\end{itemize}

In \cite{MR2418197} non-commuting cocommuting supertransitive quadrilaterals are divided into three cases (called Classes II, III, and IV, while Class I referred to non-cocommuting) based on whether the Galois group is trivial, cyclic of order $2$, or cyclic of order $3$.  
The third and fourth examples above both fit into Class II.  
We expect that the fourth example is the smallest quadrilateral in Class II with indices above $4$.  
Note that the smallest example with indices above $4$ in Class III comes from the Asaeda--Haagerup subfactor, while the smallest example with indices above $4$ in Class IV comes from the Haagerup subfactor.    
It is striking that Extended Haagerup appears to be the smallest example in the remaining case.

%auto-ignore
%this ensures the arxiv doesn't try to start TeXing here.
%!TEX root =../EH3.tex

%%%%%%%%%%%%%%%%%%%%%%%%%%%%%%%%%%%%%%%%%%%%%%%%%%%%%%%%%%%%
%%%%%%%%%%%%%%%%%%%%%%%%%%%%%%%%%%%%%%%%%%%%%%%%%%%%%%%%%%%%
%%%%%%%%%%%%%%%%%%%%%%%%%%%%%%%%%%%%%%%%%%%%%%%%%%%%%%%%%%%%
\section{Module categories and graph planar algebra embeddings} \label{sec:ModuleGPA}

The graph planar algebra embedding theorem from \cite{MR2812459} states that any subfactor planar algebra embeds in the graph planar algebra \cite{MR1865703} of either of its principal graphs.  Peters observed in \cite{MR2679382} that it is possible for a subfactor planar algebra to embed in the graph planar algebra of other graphs; in particular she
found that the Haagerup planar algebra embeds in the graph planar algebra of a third graph, called the `broom'. 
In this section we strengthen the graph planar algebra embedding theorem, to obtain a classification of embeddings in graph planar algebras. In particular, we show that a subfactor planar algebra embeds into the graph planar algebra of a bipartite graph \emph{if and only if} the graph is the fusion graph of a unitary module category with a compatible trace.

We begin with the simple observation that a module category $\cM$ for a tensor category $\scrC$
is exactly the data of a tensor functor $\scrC \rightarrow \End(\cM)$.
As we proceed through this section, we elaborate this observation in various directions, eventually obtaining our theorem.
This involves four adjustments:
\begin{itemize}
\item describing endofunctors in $\End(\cM)$ as graphs,
\item adapting to the shaded setting required for subfactor planar algebras,
\item working in the unitary setting, and
\item understanding the additional data corresponding to pivotal structures.
\end{itemize}
Note that in order to be able to characterize maps out of the Extended Haagerup subfactor planar algebra (and hence characterize modules for the module categories),
we will rely on the unitary pivotal structure (see Remark \ref{rem:UnitaryEssential}). Thus even if the reader is only interested in the algebraic classification of modules over the Extended Haagerup fusion categories, they still need to understand the unitary pivotal version of the graph planar algebra embedding theorem!

We will assume that the reader is familiar with tensor categories following \cite{MR3242743}, 
but we will not assume previous familiarity with graph planar algebras. 
We take this pedagogical approach for several reasons.
First, it was this algebraic perspective that allowed us to see that one should expect a GPA embedding theorem for modules.
Thus this approach unifies
(unitary) module category classification results like \cite{MR2046203,MR3420332} 
and 
GPA embedding constructions like \cite{MR1929335,MR2679382,MR2979509}, which will hopefully make GPA embeddings more accessible to algebraists.
Second, an independent purely subfactor theoretic proof of our classification of embeddings will appear soon in \cite{ModuleEmbedding}, using towers of algebras.
Subfactor experts may prefer to read that paper as a replacement for this section.

Here is a more detailed breakdown of this section.  In \S \ref{sec:EndEmbedding} we recall some background on module categories, Morita equivalences, and the endofunctor embedding theorem (that giving a module category $\cM$ over $\scrC$ is the same as giving a functor $\scrC \rightarrow \End(\cM)$).  In \S \ref{sec:AlgebraicGPA} we introduce an ``unbiased" definition of monoidal categories which we call monoidal algebras.  Monoidal algebras are an analogue of planar algebras without rotational symmetry.  In Section \ref{sec:MonoidalAlgebraEmbedding} we introduce the graph monoidal algebra (which is an analogue of the graph planar algebra), explain its relationship to $\End(\cM)$, and see that the endomorphism embedding theorem yields a graph monoidal algebra embedding theorem for module categories.  This is the simplest non-technical version of our main result, and contains the major idea of this section.

The rest of the section is dedicated to adapting the graph monoidal algebra embedding theorem for module categories to the pivotal and unitary pivotal settings where it becomes the graph planar algebra embedding theorem for appropriate pivotal and unitary pivotal analogues of module categories.  These analogues of module categories  involve both {\em structure} on $\cM$ and {\em compatibility} of that structure with the module action.  
In the semisimple pivotal setting Schaumann \cite{MR3019263} showed that the appropriate structure is a choice of trace on $\cM$.
In \S \ref{sec:PlanarAlgebras}, we recall the definitions of planar algebra, unitary dual functors, and unitary pivotal structure, and we explain the relationship between planar algebras and unitary pivotal fusion categories.  
In \S \ref{sec:GPAModuleEmbedding},  we recall Schaumann's notion of trace and modify this notion to the unitary setting, we then define (unitary) pivotal modules, prove a (unitary) pivotal analogue of endofunctor embedding, and translate that into the desired graph planar algebra embedding theorem.

%%%%%%%%%%%%%%%%%%%%%%%%%%%%%%%%%%%%%%%%%%%%%%%%%%%%%%%%%%%%
\subsection{Module categories, Morita equivalences, and endofunctor embedding} \label{sec:EndEmbedding}

Recall that a \emph{multitensor category} over $\mathbb{C}$ is a $\mathbb{C}$-linear abelian category together with a tensor product $\otimes$, a unit object $1_\scrC$, and unitors and associators satisfying natural axioms, where every object $X$ has a dual $X^\vee$.  
We call $\scrC$ a \emph{tensor category} if $1_\scrC$ is simple.  
A \emph{multifusion category} is a semisimple multitensor category with finitely many isomorphism classes of simples, and a \emph{fusion category} is a multifusion category with $1_\scrC$ simple.  
For further details see the book \cite{MR3242743}.

Tensor categories can be thought of as categorical analogues of ordinary algebras.  
Many ordinary algebraic notions have analogues for tensor categories, and in particular the analogues of modules, bimodules, and Morita equivalences play a key role in studying tensor categories, as pioneered by Ocneanu, M\"uger, Ostrik, and others \cite{MR1907188,MR1966524,MR1976459}.  
Again see \cite[\S7]{MR3242743} for further details.

In particular, we have the following two important problems about the ``representation theory'' of fusion categories.

\begin{problem}[Classification of Modules]
\label{problem:Modules}
Classify all indecomposable semisimple module categories over a given fusion category $\scrC$.
\end{problem}

\begin{problem}[Morita Equivalence]
\label{problem:Morita}
Classify all fusion categories $\cD$ (up to tensor equivalence) which are Morita equivalent to $\scrC$, and all the Morita equivalences between them (up to bimodule equivalence).
Furthermore, understand the Brauer-Picard groupoid, which describes the compositions of these Morita equivalences under balanced tensor product ${}_\scrC \cM \boxtimes_\cD \cN_\cE$.
\end{problem}

From a higher categorical perspective it is somewhat unnatural to only study equivalence classes, and it is more natural to consider Etingof-Nikshych-Ostrik's Brauer-Picard $3$-groupoid \cite{MR2677836} which consists of fusion categories Morita equivalent to $\scrC$, Morita equivalences between them, bimodule equivalences between these Morita equivalences, and bimodule natural isomorphisms.  
The higher structure of this $3$-groupoid is essential for classifying $G$-graded extensions of fusion categories.  
The Morita equivalence problem asks for the fundamental $1$-groupoid of this $3$-groupoid.  
As it turns out, for the examples considered in this paper, the higher structure of the $3$-groupoid is trivial (see Corollary \ref{cor:HomotopyTypeOfBPG}).

The following theorem shows that the two problems above are closely related.  
Recall that if $X$ is an invertible object, then conjugation by $X$ is an inner autoequivalence.

\begin{thm}[{\cite[Prop.~4.2 and \S 4.3]{MR2677836}, \cite[\S 7.12]{MR3242743}}]
If $\scrC$ is a fusion category and $\cM$ is a semisimple $\scrC$-module category, 
then $\cC-\cD$ bimodule category structures on $\cM$ which extend the $\cC$-module structure correspond exactly to functors $\cF: \cD \rightarrow \End_{\scrC}(\cM)$, and such a bimodule is a Morita equivalence if and only if $\cF$ is an equivalence of multitensor categories.  
Two such bimodule categories are equivalent if and only if the functors differ by an inner autoequivalence.  
Furthermore, $\End_{\scrC}(\cM)$ is a tensor category (with simple unit object) if and only if $\cM$ is indecomposable.
\end{thm}

In particular, in order to solve the Problem \ref{problem:Morita} about Morita equivalence, it is enough to solve Problem \ref{problem:Modules} about modules, and further solve the following.

\begin{problem}[Outer Automorphisms]
For each $\cD$ in the Morita equivalence class of $\scrC$, find the outer automorphism group of $\cD$.
\end{problem}

None of the fusion categories we study in this article have outer automorphisms.
Thus classifying modules and Morita equivalences are essentially the same.  
However, the reader should note that given $\scrC$ and $\cM$, actually calculating the structure of the dual category $\End_{\scrC}(\cM)$ may be quite difficult.  
The dual categories $\End_{\scrC}(\cM)$ are essentially the same thing as the dual parts of GHJ subfactors \cite{MR999799}.
We refer the reader to \cite{MR1355948} for a notable concrete example where understanding the detailed structure of the dual category is difficult.

All of the above problems are ``external'' problems, relating $\scrC$ to other tensor categories and module categories.  
However, they are closely related by a theorem of Ostrik \cite{MR1976459} to ``internal'' problems about algebra objects or Q-systems inside $\scrC$.  
Two such algebras are internally Morita equivalent if there is an invertible bimodule object between them.  

\begin{thm}[\cite{MR1976459}]
Given $A$, a connected semisimple algebra, $\Mod_\scrC(A)$ is an indecomposable module category. Moreover every indecomposable $\cC$-module category $\cM$ is equivalent to one of this form, by taking $A = \underline{\End}_\cC(m)$ for any simple $m \in \cM$.

The collection of connected semisimple algebras $\left \{B \mid \Mod_\scrC(B) \cong \Mod_\scrC(A) \right \}$ is exactly the internal Morita equivalence class of $A$.

The dual category $\End_{\scrC}(\Mod_\scrC(A))$ is canonically identified with the category of $A-A$ bimodules in $\cC$.
\end{thm}

This theorem shows that the above problems are closely related to Ocneanu's ``maximal atlas'' \cite{MR1907188}.

\begin{defn}
Let $\scrC$ be a fusion category.
A \emph{maximal atlas} for $\scrC$ is a choice of a semisimple connected algebra $A$ in each internal Morita equivalence class.
From such a maximal atlas, one gets a collection of fusion categories $\Bim_\scrC(A,A)$ and Morita equivalences $\Bim_\scrC(A,B)$.\footnote{
The distinction between thinking of the maximal atlas as a collection of algebras and bimodules or as a collection of tensor categories and Morita equivalences is often elided.
}
\end{defn}

In general, a maximal atlas will contain less information than the Brauer-Picard groupoid, because it does not remember the tensor equivalences between the fusion categories $\Bim_\scrC(A,A)$ 

\begin{example}  
For $\scrC = \Vec(\mathbb{Z}/3\mathbb{Z})$, a maximal atlas is given by $1$ and the group algebra $A = \mathbb{C}[\mathbb{Z}/3\mathbb{Z}]$ (with each group element in its own grade).  
The category of bimodules $\Bim_\scrC(A,A)$ is $\Rep(\mathbb{Z}/3\mathbb{Z})$, which is (non-canonically!) equivalent to $\Vec(\mathbb{Z}/3\mathbb{Z})$.  
The outer automorphism group of  $\Vec(\mathbb{Z}/3\mathbb{Z})$ is the group of units $(\mathbb{Z}/3\mathbb{Z})^\times$ acting by permuting simple objects, so we get two distinct equivalences $\Rep(\mathbb{Z}/3\mathbb{Z}) \cong\Vec(\mathbb{Z}/3\mathbb{Z})$.  
Thus the aforementioned maximal atlas of $\scrC$ consists of two tensor categories (which happen to be tensor equivalent) and a single bimodule between the two, while the Brauer-Picard groupoid consists of one tensor category and four Morita autoequivalences.  

One can then determine the group structure of this set of four autoequivalences.
By a result of Etingof-Nikshych-Ostrik \cite[Cor.~1.2]{MR2677836}, this Brauer-Picard group must be the split orthogonal group $O_2(\mathbb{F}_3 \oplus \mathbb{F}_3^*)$, which is the Klein four group.  
Note that in the maximal atlas formalism one can not even ask about the structure of this group.
In a sense the maximal atlas is a ``universal cover'' of the Brauer-Picard groupoid, and has lost all the interesting topological information about the latter (while still retaining the combinatorial information).
However, for all examples in this article, the Brauer-Picard group is trivial, and so these subtleties between the Brauer-Picard groupoid and the maximal atlas do not play an important role.
(In contrast, this distinction was critical in the study of the Asaeda--Haagerup subfactor \cite{MR3449240}, which has Brauer-Picard group the Klein four group.)
\end{example}

Just as a module $M$ over an algebra $A$ is equivalent to a homomorphism $A \to \End(M)$,
module categories $\cM$ over $\scrC$ are equivalent to tensor functors $\scrC \rightarrow \End(\cM)$ \cite[Prop. 7.1.3.]{MR3242743}.  Thus the module classification problem is equivalent to the following.

\begin{problem}[Endofunctor embedding]
Classify all semisimple categories $\cM$ and all tensor functors $\scrC \rightarrow \End(\cM)$, up to conjugation by an autoequivalence of $\cM$.
\end{problem}

The following omnibus theorem summarizes much of the above.

\begin{thm}\label{thm:AlgOmnibus}
Suppose that $\scrC$ is a fusion category.
\begin{itemize}
\item Module category structures on a semisimple category $\cM$ correspond exactly to tensor functors $\scrC \rightarrow \End(\cM)$.
\item A fusion category $\cD$ is Morita equivalent to $\scrC$ if and only if there is an indecomposable semisimple $\cC$-module category $\cM$ such that $\cD$ is tensor equivalent to $\End_\scrC(\cM)$.  
Furthermore the Morita equivalences ${}_\scrC\cN_\cD$ such that ${}_\scrC\cN$ is equivalent to ${}_\scrC \cM$ are a torsor for the group of outer automorphisms $\mathrm{Out}(\cD)$.
\item 
Pairs $(\cM, m)$, where $\cM$ is a semisimple indecomposable $\scrC$-module category and $m \in \cM$ is a simple object, correspond exactly to connected semisimple algebras $A$ in $\scrC$, via $A\mapsto \Mod_{\scrC}(A)$ and $(\cM,m) \mapsto \underline{\End}_{\scrC}(m)$. 
The dual category $\End_\cC(\cM)$ corresponding to $A$ is the category of $A$-$A$ bimodules in $\scrC$.
\end{itemize}
\end{thm}

%%%%%%%%%%%%%%%%%%%%%%%%%%%%%%%%%%%%%%%%%%%%%%%%%%%%%%%%
\subsubsection{Modules for multifusion categories}
\label{sec:ModulesForMultifusionCats}

Recall that a multifusion category $\scrC$ is like a fusion category, except $1_\scrC$ is no longer simple.
Since $\scrC$ is semisimple and $\scrC(1_\scrC \to 1_\scrC)$ is a commutative algebra, $1_\scrC$ breaks up as a sum of $r$ distinct simple objects $1_\scrC = \bigoplus_{i=1}^r 1_i$.
We call such a multifusion category \emph{$r$-shaded}.
We denote by $\scrC_{ij}$ the summand $1_i\otimes \scrC \otimes 1_j$.

\begin{prop} \label{prop:MultifusionReconstruction}
If $\scrC$ is an $r$-shaded multifusion category, then each $\scrC_{ii}$ is a fusion category. When $\scrC$ is indecomposable as a multifusion category each $\scrC_{ij}$ is a Morita equivalence between $\scrC_{ii}$ and $\scrC_{jj}$.  Furthermore, the tensor product map $\scrC_{ij} \boxtimes_{\scrC_{jj}} \scrC_{jk} \rightarrow \scrC_{ik}$ is an equivalence.

Conversely, given fusion categories $\cD_{11},\dots, \cD_{rr}$ and a Morita equivalence $\cD_{1j}$ between $\cD_{11}$ and $\cD_{jj}$ for each $1< j\leq r$, we define $\cD_{ik}:= \cD_{1i}^{-1}\boxtimes_{\cD_{11}}  \cD_{1k}$ for each $i,k\in \{1,\dots, r\}$ to get an indecomposable multifusion category $\cD = \bigoplus_{i,k=1}^r (\cD_{ik})$. These constructions are mutually inverse.
\footnote{
\label{footnote:MultifusionFromSumOfAlgebras}
One can avoid the relative tensor product to obtain a multifusion category equivalent to $\cD$ as follows.
First, choose a simple object $d_i\in \cD_{1i}$ for all $i=1,\dots, r$, and consider the connected algebra objects $A_i = \underline{\End}_{\cD_{11}}(d_i)$.
Then $\cD$ is equivalent to the category of $A-A$ bimodules internal to $\cD_{11}$ where $A=\bigoplus_{i=1}^r A_i$.
}
\end{prop}
\begin{proof}
The forward direction is \cite[Thm. 6.1]{MR2677836} where instead of a grading group we have a grading by the groupoid of standard matrix units $E_{ij}$.  The proof for groupoids is parallel to the proof for groups.  (See also \cite[Prop.~7.17.5]{MR3242743} which shows the first two parts.)

For the converse direction $\bigoplus_{i,k=1}^r  \cD_{1i}^{-1}\boxtimes_{\cD_{11}}  \cD_{1k}$ has a monoidal structure given by 
\begin{align*}
\left(\cD_{1i}^{-1}\boxtimes_{\cD_{11}}  \cD_{1k}\right) \boxtimes \left(\cD_{1k}^{-1}\boxtimes_{\cD_{11}}  \cD_{1\ell}\right) 
& \rightarrow \cD_{1i}^{-1}\boxtimes_{\cD_{11}}  \cD_{1k} \boxtimes_{\cD_{kk}} \cD_{1k}^{-1}\boxtimes_{\cD_{11}}  \cD_{1\ell} \displaybreak[1] \\
& \rightarrow \cD_{1i}^{-1}\boxtimes_{\cD_{11}}  \cD_{11} \boxtimes_{\cD_{11}}  \cD_{1\ell}   \displaybreak[1] \\
& \rightarrow \cD_{1i}^{-1}\boxtimes_{\cD_{11}} \cD_{1\ell}
\end{align*}
That this monoidal category is rigid follows from the identification of the inverse of a Morita equivalence $\cD_{ij}$ with the opposite category $\cD_{ij}^{\text{op}}$ from \cite[Prop. 4.2]{MR2677836}, where the dual of an object in $\cD_{ij}$ is the same object thought of in $\cD_{ij}^{\text{op}} \cong \cD_{ji}$.
\end{proof}

\begin{remark}
To each multifusion category $\scrC$ there is a corresponding rigid $2$-category whose objects are the indices, whose 1-morphisms are the objects in the $\scrC_{ij}$, and whose $2$-morphisms are the morphisms in $\scrC_{ij}$.  There is not an important difference between this $2$-category and the multifusion category, but in this paper, we use the multifusion language to align with the results of \cite{MR3242743,1808.00323}.
\end{remark}

The non-pivotal algebraic analogue of an irreducible finite depth subfactor $N \subset M$ is a pair $(\scrC, A)$ where $\scrC$ is a fusion category (which corresponds to the $N-N$ bimodules generated by $M$) and $A\in \scrC$ is a semisimple connected algebra object (which corresponds to $M$).
Given such an $A \in \scrC$, we get a Morita equivalence $\Mod_\scrC(A)= \Bim_\scrC(1_\scrC, A)$ between $\scrC$ and $\Bim_\scrC(A)$, and we get a 2-shaded indecomposable multifusion category as in Footnote \ref{footnote:MultifusionFromSumOfAlgebras} by
$$
\Bim_\scrC(1_\scrC \oplus A , 1_\scrC \oplus A)
=
\begin{pmatrix}
\Bim_\scrC(1_\scrC, 1_\scrC) & \Bim_\scrC(1_\scrC, A)
\\
\Bim_\scrC(A, 1_\scrC) & \Bim_\scrC(A,A)
\end{pmatrix}
$$
with tensor product given by $\otimes_\scrC$, $\otimes_A$, or zero as appropriate.
Notice that $\Bim_\scrC(1_\scrC, 1_\scrC) = \scrC$.

There's an analogue of Prop. \ref{prop:MultifusionReconstruction} for module categories.

\begin{prop}
Suppose that $\cD$ is an $r$-shaded multifusion category with components $\cD_{ij}$.  Suppose that $\cM$ is an indecomposable module category over $\cD$.  Then $\cM = \bigoplus_{j=1}^r \cM_j$ where $\cM_i =1_i \vartriangleright \cM$.  Furthermore, the action maps $\cD_{ij} \boxtimes_{\cD_{jj}} \cM_j \rightarrow \cM_i$ are equivalences.

Conversely, given an indecomposable module category $\cM_1$ over $\cD_{11}$, we define $$\cF(\cM_1) := \bigoplus_{i=1}^r \cD_{i1} \boxtimes_{\cD_{11}} \cM_1.$$  We can endow $\cF(\cM_1)$ with the structure of a $\cD$-module category via 
$$\cD_{kj} \boxtimes \cF(\cM_1) \cong \bigoplus_{i=1}^r \cD_{kj} \boxtimes \cD_{i1} \boxtimes_{\cD_{11}} \cM_1 \rightarrow \cD_{kj} \boxtimes_{\cD_{jj}} \cD_{j1} \boxtimes_{\cD_{11}} \cM_1 \cong \cD_{k1} \boxtimes_{\cD_{11}} \cM_1 \subseteq \cF(\cM_1).$$
 These constructions are mutually inverse with the isomorphism $\cF(\cM_1)\rightarrow \cM$ being the direct sum of the action maps $\cD_{i1} \boxtimes_{\cD_{11}} \cM_1 \cong \cM_i$.
\end{prop}
\begin{proof}
The only nontrivial step is that $\cD_{ij} \boxtimes_{\cD_{jj}} \cM_j \rightarrow \cM_i$ is an equivalence.  This follows either by the techniques of \cite[Thm. 6.1]{MR2677836} or of \cite[Prop.~7.17.5]{MR3242743}. 
Choose a simple object $m_j$ in $\cM_j$, and let $B_j = \underline{\End}_{\cD_{jj}}(m_j)$ be the internal endomorphisms of $m_j$ in $\cD_{jj}$.  
Similarly, choose a simple object $x_{ij}$ in $\cD_{ij}$ and let $A_{ij}={}^\vee x_{ij} \otimes x_{ij}$ be its internal endomorphism algebra in $\cD_{jj}$.  
We have an equivalence:
$$
\cM \rightarrow \Bim_{\cD}(A_{ij},B_j)
$$ 
via $m \mapsto {}^\vee x_{ij} \otimes \underline{\Hom}_{\cD}(m,m_j)$.  
The restriction of this functor to $\cM_{i}$ then gives an inverse to the map  $\cD_{ij} \boxtimes_{\cD_{jj}} \cM_j \rightarrow \cM_i$.
\end{proof}

Thus classifying modules for $\cD_{11}$ (answering Problem \ref{problem:Modules}) is equivalent to classifying modules for $\cD$.
In particular, given an algebraic analogue of a subfactor $A \in \scrC$, we can instead solve the module problem over the corresponding indecomposable $2$-shaded multifusion category $\Bim_\scrC(1_\scrC \oplus A, 1_\scrC \oplus A)$ which is the purely algebraic, non-pivotal analogue of the subfactor planar algebra.
That is, we construct module categories for $\cE\cH_1$ by constructing module categories over the indecomposable $2$-shaded multifusion category which combines $\cE\cH_1$ and $\cE\cH_2$. This strategy is successful because the extended Haagerup subfactor planar algebra has a better skein theoretic description than either of the fusion categories $\cE\cH_1$ and $\cE\cH_2$ individually.

%%%%%%%%%%%%%%%%%%%%%%%%%%%%%%%%%%%%%%%%%%%%%%%%%%%%%%%%
\subsubsection{Module \texorpdfstring{$\Cstar$}{Cstar} categories for unitary multitensor categories}

In the nomenclature of \cite{1808.00323},
a \emph{unitary multitensor category} $\scrC$ is a Cauchy complete rigid tensor $\Cstar$ category, which is semisimple by \cite{MR1444286}.
We call $\scrC$ a \emph{unitary tensor category} if $1_\scrC$ is simple.
Similar to the above characterization of module categories, given a (Cauchy complete) $\Cstar$ category $\cM$, endowing $\cM$ with the structure of a $\scrC$-module $\Cstar$ category is equivalent to supplying a dagger tensor functor $\scrC \to \End^\dag(\cM)$, the $\Cstar$ category of dagger endofunctors of $\cM$.\footnote{
In order for $\End^\dag(\cM)$ to be $\Cstar$, we only work with bounded natural transformations, i.e., those $\theta: F \Rightarrow G$ such that $\sup_{c\in \scrC} \|\theta_c\| <\infty$.
One then defines $\theta^\dag$ component-wise: $(\theta^\dag)_c := (\theta_c)^\dag$.
}
We provide the following lemma for those less familiar with $\Cstar$ categories, which also appears as \cite[Lem.~A.4.1]{MR3121622}. 

\begin{lem}
\label{lem:DaggerModules}
Suppose $\scrC$ is a unitary multitensor category and $\cM$ is a $\Cstar$ category.
Equipping $\cM$ with the structure of a $\scrC$-module $\Cstar$ category is equivalent to supplying a dagger tensor functor $(\Psi,\mu):\scrC \to \End^\dag(\cM)$.
\end{lem}
\begin{proof}
We show how each structure induces the other, and we leave it to the reader to check these two processes are mutually inverse (up to dagger equivalence).

Suppose $\cM$ is a $\scrC$-module dagger category.
Note that $c\vartriangleright - $ is a dagger functor in $\End^\dag(\cM)$ for each $c\in \scrC$.
Moreover, if $f\in \scrC(a\to b)$, then $(f\vartriangleright - )^\dag = f^\dag \vartriangleright -$.
Hence $\Psi:\scrC \to \End^\dag(\cM)$ given by $\Psi(c) = c\vartriangleright -$ and $\Psi(f) = f\vartriangleright -$ defines a dagger functor.
Now defining 
$$
\mu^{a,b} : \Psi(a) \circ \Psi(b) = a\vartriangleright b\vartriangleright - \Rightarrow a\otimes b\vartriangleright - = \Psi(a\otimes b)
$$ 
by $\mu^{a,b}_m := \alpha_{a,b,m} : a\vartriangleright b\vartriangleright m \to a\otimes b \vartriangleright m$ defines a unitary natural isomorphism, equipping $\Psi$ with the structure of a dagger tensor functor.

Conversely, suppose $(\Psi,\mu) : \scrC \to \End^\dag(\cM)$ is a dagger tensor functor.
For $c\in \scrC$ and $m\in \cM$, define $c\vartriangleright m := \Psi(c)(m)$.
For $c\in\scrC$ and $g\in \cM(m\to n)$, define $\id_c \vartriangleright g := \Psi(c)(f)$.
For $f\in \scrC(a\to b)$ and $m\in \cM$, define $f\vartriangleright \id_m := \Psi(f)_m$.
To show that $\vartriangleright : \scrC \times \cM \to \cM$ defines a bifunctor, it suffices to prove the exchange relation, which follows immediately from naturality.
That is, for $f\in \scrC(a\to b)$ and $g\in \cM(m\to n)$, the following diagrams commute:
\begin{equation*}
\begin{tikzcd}%[column sep]
\Psi(c)(m)\ar[r, "\Psi(f)_m"] \ar[d, "\Psi(c)(g)"]  & \Psi(d)(m)\ar[d, "\Psi(d)(g)"] 
\\
\Psi(c)(n)\ar[r, "\Psi(f)_n"]  & \Psi(d)(n)
\end{tikzcd}
\qquad
=
\qquad
\begin{tikzcd}%[column sep]
c\vartriangleright m \ar[r, "f\vartriangleright \id_m"] \ar[d, "\id_c \vartriangleright g"]  &d \vartriangleright m\ar[d, "\id_d \vartriangleright g"] 
\\
c \vartriangleright m\ar[r, "f\vartriangleright \id_n"]  &d \vartriangleright n
\end{tikzcd}
%(\id_c\vartriangleright g)\circ (f \vartriangleright \id_m) 
%=
%(f \vartriangleright \id_m) \circ(\id_c\vartriangleright g).
\end{equation*}
W define the natural unitary associator isomorphism $\alpha_{a,b,m} \in \cM( a\vartriangleright b \vartriangleright m \to a\otimes b \vartriangleright m)$ by $\alpha_{a,b,m} := \mu^{a.b}_m : [\Psi(a) \circ \Psi(b)](m) \to \Psi(a\otimes b)(m)$.
Notice that $\mu^{a,b} : \Psi(a)\circ \Psi(b) \Rightarrow \Psi(a\otimes b)$ is unitary if and only if $\mu^{a,b}_m$ is unitary for all $m\in \cM$.
Now one calculates
$(f \vartriangleright \id_m)^\dag = \Psi(f)_m^\dag :=(\Psi(f)^\dag)_m = \Psi(f^\dag)_m = f^\dag \vartriangleright \id_m$
and
$(\id_c \vartriangleright g)^\dag = \Psi(c)(g)^\dag = \Psi(c)(g^\dag) = \id_c \vartriangleright g^\dag$.
Thus $\cM$ is a $\scrC$-module dagger category.
\end{proof}

\begin{warn}
We do not yet state 
a $\Cstar$ version of Theorem \ref{thm:AlgOmnibus} as the above theorem implicitly uses rigidity for the statements on Morita equivalence and algebras.
When $\scrC$ is $\Cstar$, it is natural to impose compatibility conditions between the duality functor (implementing rigidity) and the dagger structure.
We will explain this in detail in \S\ref{sec:UnitaryPivotal} below.
\end{warn}

%%%%%%%%%%%%%%%%%%%%%%%%%%%%%%%%%%%%%%%%%%%%%%%
\subsection{Monoidal algebras} 
\label{sec:AlgebraicGPA}

Most algebraic structures have both a biased definition, like the usual definition of an algebra which emphasizes multiplying exactly two elements together, and an unbiased definition, like the definition of an algebra in which you can multiply arbitrary strings.\footnote{
See \cite{1404.0135} for a delightful elementary discussion of the unbiased definition of an algebra. 
}  
The usual definition of monoidal category is biased as it emphasizes tensoring two objects and composing two morphisms.
In Definition \ref{def:MonoidalAlgebra} below, we give an unbiased definition of monoidal category using the graphical calculus; we will see in \S\ref{sec:ModuleEmbedding} below that planar algebras are the analogous unbiased definition of a pivotal monoidal category.

As a warm-up, we first recall the diagrammatic version of the unbiased definitions of an algebra and of a linear category.

\begin{defn}
The $E_1$-\emph{operad} (or \emph{little intervals operad}) consists of 1D \emph{Swiss cheese diagrams} \cite{MR1718089} consisting of a large interval, several removed subintervals called \emph{holes}, all considered up to diffeomorphism.\footnote{In the literature, $E_1$ is an $\infty$-\emph{operad}, which means that instead of being considered up to isotopy, we instead have the isotopies induce higher isomorphisms. 
Since we only care about algebras over operads valued in the ordinary category of vector spaces, we can safely ignore this issue.}  
These interval diagrams can be composed by plugging some new big intervals into the holes to get a new diagram.  
$$
\left(
\begin{tikzpicture}[baseline=-.1cm, scale=.7]
	\draw[very thick] (0,0) -- (1,0);
	\draw[very thick] (2,0) -- (3,0);
	\draw[very thick] (4,0) -- (5,0);
	\filldraw[thick] (0,0) circle (.05cm);
	\filldraw[thick, fill=white] (1,0) circle (.05cm);
	\filldraw[thick, fill=white] (2,0) circle (.05cm);
	\filldraw[thick, fill=white] (3,0) circle (.05cm);
	\filldraw[thick, fill=white] (4,0) circle (.05cm);
	\filldraw[thick] (5,0) circle (.05cm);
	\node at (1.5,0) {\scriptsize{$1$}};
	\node at (3.5,0) {\scriptsize{$2$}};
\end{tikzpicture}
\right)
\circ_2
\left(
\begin{tikzpicture}[baseline=-.1cm, scale=.7]
	\draw[very thick] (0,0) -- (1,0);
	\draw[very thick] (2,0) -- (3,0);
	\draw[very thick] (4,0) -- (5,0);
	\filldraw[thick] (0,0) circle (.05cm);
	\filldraw[thick, fill=white] (1,0) circle (.05cm);
	\filldraw[thick, fill=white] (2,0) circle (.05cm);
	\filldraw[thick, fill=white] (3,0) circle (.05cm);
	\filldraw[thick, fill=white] (4,0) circle (.05cm);
	\filldraw[thick] (5,0) circle (.05cm);
	\node at (1.5,0) {\scriptsize{$1$}};
	\node at (3.5,0) {\scriptsize{$2$}};
\end{tikzpicture}
\right)
=
\left(
\begin{tikzpicture}[baseline=-.1cm, scale=.7]
	\draw[very thick] (0,0) -- (1,0);
	\draw[very thick] (2,0) -- (3,0);
	\draw[very thick] (4,0) -- (5,0);
	\draw[very thick] (6,0) -- (7,0);
	\filldraw[thick] (0,0) circle (.05cm);
	\filldraw[thick, fill=white] (1,0) circle (.05cm);
	\filldraw[thick, fill=white] (2,0) circle (.05cm);
	\filldraw[thick, fill=white] (3,0) circle (.05cm);
	\filldraw[thick, fill=white] (4,0) circle (.05cm);
	\filldraw[thick, fill=white] (5,0) circle (.05cm);
	\filldraw[thick, fill=white] (6,0) circle (.05cm);
	\filldraw[thick] (7,0) circle (.05cm);
	\node at (1.5,0) {\scriptsize{$1$}};
	\node at (3.5,0) {\scriptsize{$2$}};
	\node at (5.5,0) {\scriptsize{$3$}};
\end{tikzpicture}
\right)
$$
An $E_1$-\emph{algebra} in vector spaces is an algebra for this operad, which means it consists of a vector space $A$ together with a linear map $A^{\otimes h} \rightarrow A$ attached to each linear Swiss cheese diagram with $h$ holes.   
These maps must be compatible with the operad structure (i.e., plugging elements of $A$ into holes, and plugging diagrams into larger diagrams, associates).  
Unpacking this definition, an $E_1$-algebra in vector spaces consists of multiplication maps $\mu_n: A^{\otimes n} \rightarrow A$ for every natural number $n$ ($n=0$ gives the unit) which satisfy the appropriate associativity relations.
This is exactly the unbiased definition of a unital associative algebra.
\end{defn}

\begin{defn}
The colored operad of linear tangles with label set $S$ consists of a large interval with several holes removed, together with a labelling by an element of $S$ for each connected component of the diagram modulo diffeomorphism.  
(To match this up with future examples, it helps to think of these components as ``strings" connecting each hole to the next hold or to the outside interval.)  
$$
\begin{tikzpicture}[baseline=-.1cm, scale=.7]
	\draw[very thick] (0,0) -- node[above] {\scriptsize{$w$}} (1,0);
	\draw[very thick] (2,0) -- node[above] {\scriptsize{$x$}} (3,0);
	\draw[very thick] (4,0) -- node[above] {\scriptsize{$y$}} (5,0);
	\draw[very thick] (6,0) -- node[above] {\scriptsize{$z$}} (7,0);
	\filldraw[thick] (0,0) circle (.05cm);
	\filldraw[thick, fill=white] (1,0) circle (.05cm);
	\filldraw[thick, fill=white] (2,0) circle (.05cm);
	\filldraw[thick, fill=white] (3,0) circle (.05cm);
	\filldraw[thick, fill=white] (4,0) circle (.05cm);
	\filldraw[thick, fill=white] (5,0) circle (.05cm);
	\filldraw[thick, fill=white] (6,0) circle (.05cm);
	\filldraw[thick] (7,0) circle (.05cm);
	\node at (1.5,0) {\scriptsize{$1$}};
	\node at (3.5,0) {\scriptsize{$2$}};
	\node at (5.5,0) {\scriptsize{$3$}};
\end{tikzpicture}
$$
Again the operadic structure comes from gluing linear tangles into the holes, but since substitution only makes sense when the labels match, this is a \emph{colored operad}.  
An algebra ${}_\bullet V_\bullet$ for this operad consists of a family of vector spaces $\{{}_{x} V_y\}_{x,y\in S}$ together with an action of linear tangles with holes.
That is, to each linear tangle $T$ with components labelled by $x_1,\dots, x_n$, we get a linear map
 $Z(T) : {}_{x_1}\! V_{x_2} \otimes \cdots  \otimes{}_{x_{n-1}}\!V_{x_n} \rightarrow {}_{x_1}\!V_{x_n}$ which is compatible with composition of linear tangles with holes.  
It is not difficult to see that an algebra for the operad of linear tangles with label set $S$ gives an unbiased definition of a linear category whose set of objects is $S$, whose hom spaces $\Hom(x\to y)$ are the vector spaces ${}_{x} V_{y}$, and whose composition of morphisms is the action of linear tangles.  The identity morphisms come from the tangles with no holes.
\end{defn}

We now give an unbiased definition of monoidal category.  
This is quite similar to the above definition, but we now have two dimensional diagrams where the vertical direction represents composition using labelled strings as before, and the horizontal direction represents tensor product without strings.

\begin{defn}
A \emph{monoidal tangle} with label set $S$ is a rectangle, with several smaller rectangles (with edges parallel to those of the big one) removed, and some non-crossing smooth strings labelled by elements of $S$ which are oriented upward, have no minima or maxima, and begin and end on the tops or bottoms of the rectangles. 
We say a monoidal tangle $T$ has \emph{type} $((s_0,t_0); (s_1,t_1),\dots, (s_k,t_k))$ where $s_0,\dots, s_k, t_0,\dots, t_k$ are finite words on $S$ if the tangle $T$ has $k$ input rectangles, and there are $|s_i|, |t_i|$ strings attached to the bottom and top respectively of the $i$-th rectangle (the zeroth rectangle is the output rectangle and $1\leq i\leq k$ corresponds to the $i$-th input rectangle), which are labelled by the characters in the words $s_i, t_i$ respectively.
Here is an example of a tangle with 
$S = \{
\ColorDot{red}
\,,\,
\ColorDot{blue}
\,,\,
\ColorDot{DarkGreen}
\}$,
where we color the strings instead of labelling them:
$$
\begin{tikzpicture}[baseline = -.1cm, scale=.6]
	\draw[very thick, rounded corners = 5pt] (-2,-2) rectangle (2,2);
	\draw[very thick, rounded corners = 5pt] (-1.5,.5) rectangle (-.5,1.5);
	\draw[very thick, rounded corners = 5pt] (-1,-1.5) rectangle (.5,-.5);
	\draw[very thick, rounded corners = 5pt] (.25,.5) rectangle (1.25,1.5);
	\node at (-1,1) {\scriptsize{$1$}};
	\node at (-.25,-1) {\scriptsize{$2$}};
	\node at (.75,1) {\scriptsize{$3$}};
	\draw[thick, red] (-.8,.5) .. controls ++(270:.5cm) and ++(90:.5cm) .. (-.75,-.5);
	\draw[thick, blue] (-1.5,-2) .. controls ++(90:.5cm) and ++(270:.5cm) .. (-1.2,.5);
	\draw[thick, blue] (.25,-.5) .. controls ++(90:.5cm) and ++(270:.5cm) .. (.5,.5);
	\draw[thick, DarkGreen] (-1,1.5) -- (-1,2);
	\draw[thick, red] (-.5,-2) -- (-.5,-1.5);
	\draw[thick, red] (-.25,-.5) -- (-.25,2);
	\draw[thick, DarkGreen] (0,-2) -- (0,-1.5);
	\draw[thick, DarkGreen] (1,-2) -- (1,.5);
	\draw[thick, blue] (1.5,-2) -- (1.5,2);
\end{tikzpicture}
\qquad
\text{has type}
\qquad
((
\underbrace{\ColorDot{blue}\,\ColorDot{red}\,\ColorDot{DarkGreen}\,\ColorDot{DarkGreen}\,\ColorDot{blue}}_{s_0}
\,,\,
\underbrace{\ColorDot{DarkGreen}\,\ColorDot{red}\,\ColorDot{blue}}_{t_0}
);(
\underbrace{\ColorDot{blue}\,\ColorDot{red}}_{s_1}
\,,\,
\underbrace{\ColorDot{DarkGreen}}_{t_1}
),(
\underbrace{\ColorDot{red}\,\ColorDot{DarkGreen}}_{s_2}
\,,\,
\underbrace{\ColorDot{red}\,\ColorDot{red}\,\ColorDot{blue}}_{t_2}
),(
\underbrace{\ColorDot{blue}\,\ColorDot{DarkGreen}}_{s_3}
\,,\,
\underbrace{\emptyset}_{t_3}
)).
$$
Monoidal tangles are considered up to isotopy (through diagrams that again have no minima or maxima). 
Monoidal tangles form a colored operad, because you can insert monoidal tangles into the rectangles of a large monoidal tangle to get a new monoidal tangle.   
\end{defn}

\begin{defn}
\label{def:MonoidalAlgebra}
A \emph{monoidal algebra} with label set $S$ is an algebra for the operad of monoidal tangles with label set $S$.
Unpacking this definition, a monoidal algebra $\cP_{\bullet \to \bullet}$ consists of a family of finite dimensional vector spaces $\cP_{s\to t}$ where $s,t$ are finite words in $S$, together with an action of monoidal tangles.
To each monoidal tangle $T$ of type $((s_0,t_0); (s_1,t_1,\dots, (s_k,t_k)))$, we associate a multilinear map
$Z(T):\prod_{j=1}^k \cP_{s_j\to t_j}\to \cP_{s_0 \to t_0}$,
and composition of monoidal tangles corresponds to composition of multilinear maps.
Here is an an  example:
$$
Z\left(\,\,
\begin{tikzpicture}[baseline = -.1cm, scale=.6]
	\draw[very thick, rounded corners = 5pt] (-2,-2) rectangle (2,2);
	\draw[very thick, rounded corners = 5pt] (-1.5,.5) rectangle (-.5,1.5);
	\draw[very thick, rounded corners = 5pt] (-1,-1.5) rectangle (.5,-.5);
	\draw[very thick, rounded corners = 5pt] (.25,.5) rectangle (1.25,1.5);
	\node at (-1,1) {\scriptsize{$1$}};
	\node at (-.25,-1) {\scriptsize{$2$}};
	\node at (.75,1) {\scriptsize{$3$}};
	\draw[thick, red] (-.8,.5) .. controls ++(270:.5cm) and ++(90:.5cm) .. (-.75,-.5);
	\draw[thick, blue] (-1.5,-2) .. controls ++(90:.5cm) and ++(270:.5cm) .. (-1.2,.5);
	\draw[thick, blue] (.25,-.5) .. controls ++(90:.5cm) and ++(270:.5cm) .. (.5,.5);
	\draw[thick, DarkGreen] (-1,1.5) -- (-1,2);
	\draw[thick, red] (-.5,-2) -- (-.5,-1.5);
	\draw[thick, red] (-.25,-.5) -- (-.25,2);
	\draw[thick, DarkGreen] (0,-2) -- (0,-1.5);
	\draw[thick, DarkGreen] (1,-2) -- (1,.5);
	\draw[thick, blue] (1.5,-2) -- (1.5,2);
\end{tikzpicture}
\,\,\right)
:
\cP_{
\ColorDot{blue}\,\ColorDot{red}
\,\to\,
\ColorDot{DarkGreen}
}
\times
\cP_{
\ColorDot{red}\,\ColorDot{DarkGreen}
\,\to\,
\ColorDot{red}\,\ColorDot{red}\,\ColorDot{blue}
}
\times
\cP_{
\ColorDot{blue}\,\ColorDot{DarkGreen}
\,\to\,
\emptyset
}
\to
\cP_{
\ColorDot{blue}\,\ColorDot{red}\,\ColorDot{DarkGreen}\,\ColorDot{DarkGreen}\,\ColorDot{blue}
\,\to\,
\ColorDot{DarkGreen}\,\ColorDot{red}\,\ColorDot{blue}
}
%\underbrace{\scrC(X^{\otimes 2}\to X)}_{\cP_{2\to 1}}
%\otimes_\bbC 
%\underbrace{\scrC(X^{\otimes 2} \to X^{\otimes 3})}_{\cP_{2\to 3}}
%\otimes_\bbC
%\underbrace{\scrC(X^{\otimes 2} \to 1_{\scrC})}_{\cP_{2\to 0}}
%\to
%\underbrace{\scrC(X^{\otimes 5} \to X^{\otimes 3})}_{\cP_{5\to 3}}.
$$
A monoidal algebra is called \emph{semisimple} if for every pair of words $s,t$ on $S$, the $2\times2$ \emph{linking algebra}
\begin{equation}
\label{eq:LinkingAlgebra}
\cL(s,t)
:=
\begin{pmatrix}
\cP_{s\to s}
&
\cP_{t\to s}
\\
\cP_{s\to t}
&
\cP_{t\to t}
\end{pmatrix}
\end{equation}
whose multiplication given by matrix multiplication together with the appropriate `stacking' multiplication tangles is a finite dimensional semisimple algebra.
\end{defn}

\begin{example}
\label{example:MonoidalAlgebraViaGraphicalCalculus}
Suppose $\scrC$ is a linear monoidal category with a set of objects $\mathscr{S}:=\{X_s\}_{s\in S}$ which \emph{monoidally generates} $\scrC$, i.e., every object in $\scrC$ is isomorphic to a tensor product of objects in $\mathscr{S}$.
We define a monoidal algebra $\cP(\scrC, \mathscr{S})_{\bullet \to \bullet}$ with label set $S$ as follows.
For $s_1,\dots, s_k,t_1,\dots, t_\ell\in S$, we define 
$$
\cP(\scrC, \mathscr{S})_{s_1\cdots s_k\to t_1\cdots t_\ell} := \scrC(X_{s_1}\otimes \cdots \otimes X_{s_k} \to X_{t_1}\otimes \cdots \otimes X_{t_\ell}).
$$
We use the convention that if $\emptyset$ is the empty word on $S$, then the empty tensor product of objects is $1_{\scrC}$.
The action of tangles is just the graphical calculus for tensor categories.
See \cite{MR0281657,MR1036112,MR1113284} for a summary of many versions of the graphical calculus; additional resources include \cite{MR2767048} and \cite[\S2.1 and 2.3]{MR3578212}.
\end{example}

\begin{remark}
The monoidal algebra $\cP(\scrC, X)_{\bullet \to\bullet}$ is similar in spirit to the way the term `monoidal algebra' is used in the work of Wenzl on constructing and classifying subfactors and fusion categories from quantum groups \cite{MR2132671,MR2988502} which is based on the original towers of algebras approach to subfactor theory \cite{MR0696688,MR936086,MR999799,MR1334479}.
\end{remark}

\begin{thm}
\label{thm:MonoidalAlgebraEquivalence}
There is an equivalence of categories\,\footnote{\label{footnote:Truncation}
Pairs $(\scrC, \{X_{s}\}_{s\in S})$ form a $2$-category where between any two 1-morphisms, there is at most one 2-morphism, which is necessarily invertible when it exists \cite[Lem.~3.5]{1607.06041}.
Hence this 2-category is equivalent to its truncation to a 1-category.
} 
\[
\left\{\, 
\parbox{7cm}{\rm Monoidal algebras $\cP_{\bullet\to \bullet}$ with label set $S$ and finite dimensional box spaces $\cP_{m\to n}$}\,\left\}
\,\,\,\,\cong\,\,
\left\{\,\parbox{7.8cm}{\rm Pairs $(\scrC, \{X_s\}_{s\in S})$ with $\scrC$ a linear monoidal category with generators $X_s\in \scrC$ for $s\in S$}\,\right\}.
\right.\right.
\]
\end{thm}
The equivalence here is given by taking the idempotent category of a monoidal algebra; the details are standard, following e.g. \cite{MR2559686,MR2811311}. We leave as an exercise to the reader to check that the definition of semisimplicity given above corresponds to the usual definition for the Cauchy completion of the idempotent category.

In the case we care most about, $\scrC$ is a semisimple monoidal category, and it is hopeless to expect to have a finite monoidal generating set.  
Instead, one typically has a collection of objects $\mathscr{S}:=\{X_s\}_{s\in S}$ labelled by $S$ which \emph{Cauchy tensor generates} $\scrC$ in the sense that every object in $\scrC$ is a direct summand of a direct sum of tensor products of objects in $\mathscr{S}$. 
This is not a big problem because of the following well-known theorem:

\begin{thm}
\label{thm:CauchyCompletion}
Suppose $\scrC$ is a semisimple monoidal category and $\mathscr{S}:=\{X_s\}_{s\in S}$ is a set of objects that Cauchy tensor generates $\scrC$. 
\begin{itemize}
\item 
Let $\scrC_{\mathscr{S}}$ be the full monoidal subcategory of $\scrC$ whose objects are tensor products of objects in $\mathscr{S}$.  Then $\scrC_{\mathscr{S}}$ is a monoidal category which is monoidally generated by $\mathscr{S}$.
\item 
The monoidal category $\scrC$ is monoidally equivalent to the idempotent completion of the additive envelope (also known as the Cauchy completion, pseudo-abelian envelope, or Karoubi envelope) of $\scrC_{\mathscr{S}}$.
\end{itemize}
\end{thm}

Combining Theorems \ref{thm:MonoidalAlgebraEquivalence} and \ref{thm:CauchyCompletion}, given a semisimple linear monoidal category $\scrC$ and a set $\mathscr{S}:=\{X_s\}_{s\in S}$ of objects that Cauchy tensor generate $\scrC$, we get a semisimple monoidal algebra from $\scrC_{\mathscr{S}}$.
Conversely, given a semisimple monoidal algebra $\cP_{\bullet \to \bullet}$ with label set $S$, we can take the idempotent completion of the corresponding monoidal category to recover the semisimple monoidal category $\scrC$ and a set of Cauchy tensor generating objects $\mathscr{S}:=\{X_s\}_{s\in S}$ corresponding to the strands labelled by $s\in S$.

\begin{cor}
\label{cor:SemisimpleMonoidalAlgebraEquivalence}
There is an equivalence of categories
(see Footnote \ref{footnote:Truncation})
\[
\left\{\, 
\parbox{5.5cm}{\rm Semisimple monoidal algebras $\cP_{\bullet\to \bullet}$ with label set $S$}\,\left\}
\,\,\,\,\cong\,\,
\left\{\,\parbox{7cm}{\rm Pairs $(\scrC, \{X_s\}_{s\in S})$ with $\scrC$ a semisimple linear monoidal category with Cauchy tensor generators $X_s\in \scrC$ for $s\in S$}\,\right\}.
\right.\right.
\]
\end{cor}

\begin{example}
\label{example:GMAfromCX}
Expanding on Example \ref{example:MonoidalAlgebraViaGraphicalCalculus}, given a semisimple linear monoidal category $\scrC$ and an object $X\in \scrC$ which Cauchy tensor generates $\scrC$, we get a semisimple monoidal algebra $\cP(\scrC, X)_{\bullet \to \bullet}$ with label set $\bbN_{\geq 0}$ by defining $\cP(\scrC, X)_{m\to n} := \scrC(X^{\otimes m} \to X^{\otimes n})$.
\end{example}

%%%%%%%%%%%%%%%%%%%%%%%%%%%%%%%%%%%%%%%%%%%%%%%%%%%%%%%%
\subsubsection{Shaded semisimple monoidal algebras and semisimple monoidal categories}

We next extend the discussion in \S\ref{sec:ModulesForMultifusionCats} on $r$-shaded multifusion categories to the case of $r$-shaded semisimple monoidal categories.
Suppose $\scrC$ is a semisimple linear monoidal category, and $1_{\scrC} = \bigoplus_{i=1}^r 1_i$ is a decomposition of $1_\scrC$ into simples.
We call such a $\scrC$ an $r$-\emph{shaded semisimple monoidal category}.
We write $\scrC_{ij} = 1_i \otimes \scrC \otimes 1_j$, and we note that $\scrC = \bigoplus_{i,j=1}^r \scrC_{ij}$.
We also have distinguished idempotents $p_i \in \scrC(1_\scrC \to 1_\scrC)$ corresponding to the summand $1_i$ for $1\leq i\leq r$.
In the graphical calculus, we represent these projections, which freely float about in their regions, as a single shading.
For example, we could denote
$$
\tikz[baseline=.1cm]{\draw[fill=\ShadeOne, rounded corners=5, very thin, baseline=1cm] (0,0) rectangle (.5,.5);}=p_i
\qquad
\tikz[baseline=.1cm]{\draw[fill=\ShadeTwo, rounded corners=5, very thin, baseline=1cm] (0,0) rectangle (.5,.5);}=p_j
$$
Then for objects $a,b\in \scrC_{ij}$, we would denote a morphism $f\in \scrC(a\to b)$ by
$$
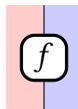
\begin{tikzpicture}
	\fill[\ShadeOne] (-.5,-.7) rectangle (0,.7);
	\fill[\ShadeTwo] (.5,-.7) rectangle (0,.7);
	\draw (0,-.7) -- (0,.7);
	\roundNbox{fill=white}{(0,0)}{.3}{0}{0}{$f$}
\end{tikzpicture}
$$
This motivates the following definition.

\begin{defn}
An $R$-\emph{shaded monoidal tangle} with label set $S$ is 
a monoidal tangle with label set $S$ whose regions are shaded by the elements of $R$
such that 
each element $x\in S$ has a left \emph{source} shading $s_x\in R$ and a right \emph{target} shading $t_y\in R$.
For example, for the shading set
$R = \{
\ColorDot{\ShadeOne}
\,,\,
\ColorDot{\ShadeTwo}
\,,\,
\ColorDot{\ShadeThree}
\}$,
and the label set
$
S= 
\{
\begin{tikzpicture}[baseline=-.1cm]
	\fill[\ShadeOne] (-.2,-.2) rectangle (0,.2);
	\fill[\ShadeTwo] (.2,-.2) rectangle (0,.2);
	\draw (0,-.2) -- (0,.2);
\end{tikzpicture}
\,,\,
\begin{tikzpicture}[baseline=-.1cm, xscale =-1]
	\fill[\ShadeOne] (-.2,-.2) rectangle (0,.2);
	\fill[\ShadeTwo] (.2,-.2) rectangle (0,.2);
	\draw (0,-.2) -- (0,.2);
\end{tikzpicture}
\,,\,
\begin{tikzpicture}[baseline=-.1cm]
	\fill[\ShadeTwo] (-.2,-.2) rectangle (0,.2);
	\fill[\ShadeThree] (.2,-.2) rectangle (0,.2);
	\draw (0,-.2) -- (0,.2);
\end{tikzpicture}
\,,\,
\begin{tikzpicture}[baseline=-.1cm, xscale = -1]
	\fill[\ShadeTwo] (-.2,-.2) rectangle (0,.2);
	\fill[\ShadeThree] (.2,-.2) rectangle (0,.2);
	\draw (0,-.2) -- (0,.2);
\end{tikzpicture}
\}
$,
we have the following $R$-shaded monoidal tangle with label set $S$:
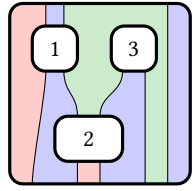
$$
\begin{tikzpicture}[baseline = -.1cm, scale=.6]
	\fill[\ShadeOne, rounded corners = 5pt] (-2,-2) rectangle (1,2);
	\fill[\ShadeTwo, rounded corners = 5pt] (-.5,-2) rectangle (2,2);
	\fill[\ShadeTwo] (-.8,1.5) rectangle (-1.2,2);
	\fill[\ShadeTwo] (-1.5,-2) .. controls ++(90:.5cm) and ++(270:.5cm) .. (-1.2,.5) --
		(-.8,.5) .. controls ++(270:.5cm) and ++(90:.5cm) .. (-.5,-.5) -- (-.5,-1.5) -- (-.5,-2) -- (-1.5,-2);
	\fill[\ShadeThree] (0,-.5) .. controls ++(90:.5cm) and ++(270:.5cm) .. (.5,.5) --
		(1,.5) -- (1,-2) -- (1.5,-2) -- (1.5,2) --(-.8,2) -- (-.8,1.5) -- 
		(-.8,.5) .. controls ++(270:.5cm) and ++(90:.5cm) .. (-.5,-.5) -- (0,-.5);
	\fill[\ShadeOne] (-.5,-2) rectangle (0,-1.5);
	\draw[very thick, rounded corners = 5pt] (-2,-2) rectangle (2,2);
	\draw[fill=white, very thick, rounded corners = 5pt] (-1.5,.5) rectangle (-.5,1.5);
	\draw[fill=white, very thick, rounded corners = 5pt] (-1,-1.5) rectangle (.5,-.5);
	\draw[fill=white, very thick, rounded corners = 5pt] (.25,.5) rectangle (1.25,1.5);
	\node at (-1,1) {\scriptsize{$1$}};
	\node at (-.25,-1) {\scriptsize{$2$}};
	\node at (.75,1) {\scriptsize{$3$}};
	\draw (-.8,.5) .. controls ++(270:.5cm) and ++(90:.5cm) .. (-.5,-.5);
	\draw (-1.5,-2) .. controls ++(90:.5cm) and ++(270:.5cm) .. (-1.2,.5);
	\draw (0,-.5) .. controls ++(90:.5cm) and ++(270:.5cm) .. (.5,.5);
	\draw (-1.2,1.5) -- (-1.2,2);
	\draw (-.8,1.5) -- (-.8,2);
	\draw (-.5,-2) -- (-.5,-1.5);
	\draw (0,-2) -- (0,-1.5);
	\draw (1,-2) -- (1,.5);
	\draw (1.5,-2) -- (1.5,2);
\end{tikzpicture}
$$
\end{defn}

\begin{defn}
An $R$-\emph{shaded monoidal algebra} with label set $S$ is an algebra over the operad of $R$-shaded monoidal tangles with label set $S$.
Notice this means that the spaces $\cP_{x \to y}$ are only well-defined when 
consecutive characters in the words $x$ and $y$ have compatible target and source shadings, and the source and target shadings of the words $x$ and $y$ agree.
\end{defn}

We have the following shaded version Theorem \ref{thm:MonoidalAlgebraEquivalence}.

\begin{cor}
\label{cor:SemisimpleShadedMonoidalAlgebraEquivalence}
There is an equivalence of categories
(see Footnote \ref{footnote:Truncation})
\[
\left\{\, 
\parbox{5cm}{\rm Semisimple $\{1,\dots, r\}$-shaded monoidal algebras $\cP_{\bullet\to \bullet}$ with label set $S$}\,\left\}
\,\,\,\,\cong\,\,
\left\{\,\parbox{9cm}{\rm Pairs $(\scrC, \{X_y\}_{y\in S})$ with $\scrC$ an $r$-shaded semisimple monoidal category with 
decomposition $1 = \bigoplus_{i=1}^r 1_i$ with Cauchy tensor generators $X_y\in \scrC_{s_y, t_y}$ for $y\in S$
}\,\right\}.
\right.\right.
\]
\end{cor}

%%%%%%%%%%%%%%%%%%%%%%%%%%%%%%%%%%%%%%%%%%%%%%%%%%%%%%%%
\subsubsection{Unitary monoidal algebras}

\begin{defn}
A \emph{dagger monoidal algebra} with label set $S$ is a monoidal algebra $\cP_{\bullet \to \bullet}$ with label set $S$ equipped with antilinear maps $\dag: \cP_{s \to t} \to \cP_{t\to s}$ for all words $s,t$ on $S$ such that 
\begin{itemize}
\item
$\dag\circ \dag = \id$ and 
\item
for every monoidal tangle $T$, $T^\dag(x_1^\dag,\dots, x_k^\dag) = T(x_1,\dots, x_k)^\dag$ where $T^\dag$ denotes the vertical reflection of $T$ about the $x$-axis.
\end{itemize}
A dagger monoidal algebra is called a $\Cstar$ \emph{monoidal algebra} if in addition
\begin{itemize}
\item
Every $\dag$-algebra $\cP_{s\to s}$ with the stacking multiplication is a $\Cstar$ algebra,\footnote{
\label{footnote:CstarAlgebraProperty}
Being a $\Cstar$ algebra is a property of a complex $*$-algebra and not extra structure.
Indeed, every $\Cstar$ algebra has a unique $\Cstar$ norm, which can be recovered from the spectral radius, which is defined purely algebraically.
} and
\item
for all $f\in \cP_{s\to t}$, there is a $g\in \cP_{s\to s}$ such that $f^\dag \circ f = g^\dag \circ g$.
\end{itemize}
Finally, a \emph{unitary monoidal algebra} is a semisimple $\Cstar$ monoidal algebra.  
\end{defn}

When $\scrC$ is a $\Cstar$ monoidal category, the unitary Cauchy completion (a.k.a~unitary Karoubi envelope) is the orthogonal projection completion of the orthogonally additive envelope, i.e., we add formal orthogonal direct sums of objects, and then we take the category of orthogonal projections.
Since finite dimensional $\Cstar$ algebras are semisimple, we see that if $\scrC$ is a $\Cstar$ monoidal category whose endomorphisms spaces are finite dimensional, then the unitary Cauchy completion is semisimple.
In this case, we say a set of objects $\mathscr{S} = \{X_s\}_{s\in S}$ \emph{Cauchy tensor generates} $\scrC$ if every object of $\scrC$ is \emph{unitarily} isomorphic to an orthogonal direct summand of an orthogonal direct sum of tensor products of objects in $\mathscr{S}$. 

We have the following unitary version Theorem \ref{thm:MonoidalAlgebraEquivalence}.

\begin{cor}
\label{cor:SemisimpleCstarMonoidalAlgebraEquivalence}
There is an equivalence of categories
(see Footnote \ref{footnote:Truncation})
\[
\left\{\, 
\parbox{6.3cm}{\rm (Semisimple) $\Cstar$ monoidal algebras $\cP_{\bullet\to \bullet}$ with label set $S$}\,\left\}
\,\,\,\,\cong\,\,
\left\{\,\parbox{7cm}{\rm Pairs $(\scrC, \{X_s\}_{s\in S})$ with $\scrC$ a (semisimple) $\Cstar$ monoidal category with Cauchy tensor generators $X_s\in \scrC$ for $s\in S$}\,\right\}.
\right.\right.
\]
\end{cor}

%%%%%%%%%%%%%%%%%%%%%%%%%%%%%%%%%%%%%%%%%%%%%%%
\subsection{Graph monoidal algebra embedding} 
\label{sec:MonoidalAlgebraEmbedding}

In this section we relate endofunctor embeddings $\scrC \rightarrow \End(\cM)$ to embeddings of monoidal algebras into 
graph monoidal algebras, which is the non-pivotal analog of embedding planar algebras into graph planar algebras.
We give a $2$-shaded multifusion version which applies to an algebraic analog of a finite depth subfactor standard invariant.

\begin{defn}
Let $J$ be a finite set.
The tensor category $\Vec(J\times J)$ of \emph{bi-$J$-graded vector spaces} has objects finite dimensional vector spaces which decompose as direct sums $V = \bigoplus_{i,j\in J} V_{ij}$, morphisms linear maps which preserve the bi-grading, i.e., $f: V\to W$ is a sum $f=\sum_{ij} f_{ij}: V_{ij} \to W_{ij}$, and composition the composition of linear maps.
The tensor product of two bi-graded vector spaces is given by convolution
$$
(V\otimes W)_{ik} := \bigoplus_{j\in J} V_{ij} \otimes W_{jk},
$$
as is the tensor product of morphisms, i.e., if $f: V^1 \to V^2$ and $g: W^1 \to W^2$, then
$$
(f\otimes g)_{ik} 
:=
\bigoplus_{j \in J}
f_{ij}\otimes g_{jk}
:
\bigoplus_{j \in J} 
V^1_{ij}\otimes W^1_{jk}
\longrightarrow
\bigoplus_{j\in J} 
V^2_{ij} \otimes W^2_{jk}.
$$
It is straightforward to see that $\Vec(J\times J)$ is a finitely semisimple rigid tensor category.
A set of representatives of the simple objects is given by $\{E_{ij}\}_{i,j\in J}$, where $E_{ij}$ has a copy of $\bbC$ in the $ij$-graded component and the zero vector space everywhere else.
The dual of $V$ is given by $(V^\vee)_{ij} := (V_{ji})^\vee$, the space of linear functionals $V_{ji} \to \bbC$, with obvious evaluation and coevaluation maps.
Indeed, it is straightforward to verify that $\Vec(J\times J)$ is monoidally equivalent to the 
tensor category $\Vec[\cG_r]$ of $\cG_r$-graded vector spaces, where $\cG_r$ is the groupoid with $r:=|J|$ objects and a unique isomorphism between any two object.
In turn, $\Vec[\cG_r]$ is easily seen to be monoidally equivalent to $\End(\cM)$ where $\cM$ is a finitely semisimple category such that a set of representatives of the simple objects $\Irr(\cM)$ is in bijection with $J$.
\end{defn}

\begin{defn}
\label{defn:VecGraph}
Given a bi-graded vector space $V\in \Vec(J\times J)$, we may think of it as a \emph{$\Vec$-enriched graph} $\vec{\Gamma}=(J,V)$, whose vertices are the set $J$, and whose edges are the finite dimensional vector spaces $V_{ij}$.
We call a bi-graded vector space \emph{connected} if given any two vertices $i,k\in J$, there is a sequence of vertices
$(i=j_0, j_1, \dots, j_n=k)$ such that $V_{j_{\ell-1}j_{\ell}}\neq (0)$ for all $\ell=1,\dots, n$.
Observe that if $\Gamma$ is connected, then $\Gamma$ Cauchy tensor generates $\Vec(J\times J)$.

Given a (connected) $\Vec$-graph $\vec{\Gamma} = (J, V)$, 
we get an honest (connected) graph $\Gamma$ with vertex set $J$ and whose edges from $i$ to $j$ is are some choice of basis for the space $V_{ij}$.
Clearly picking different bases yields isomorphic graphs.
\end{defn}

\begin{remark}
This approach is very similar to that in the classification of Temperley-Lieb module categories using weighted graphs from \cite{MR2046203}.
In \cite{MR3420332}, the authors classify unitary Temperley-Lieb module categories using bi-graded Hilbert spaces, which we discuss briefly (with a warning) in \S\ref{sec:EmbeddingUnitaryMultifusionIntoGMA} below.
\end{remark}

\begin{defn}
\label{defn:GraphMonoidalAlgebra}
Suppose $\Gamma$ is a connected directed graph with vertex set $J$ and edge set $V_{ij}$ from $i$ to $j$.
For an edge $\varepsilon \in V_{ij}$, we write $s(\varepsilon) = i$ and $t(\varepsilon) = j$, the \emph{source} and \emph{target} of $\varepsilon$.

We define the \emph{graph monoidal algebra} $\cG\cM\cA(\Gamma)_{\bullet \to \bullet}$ as follows.
For $m,n\geq 0$, we define
$\cG\cM\cA(\Gamma)_{m \to n}$
to be the $\bbC$-vector space with distinguished basis the set of pairs $(p,q)$ where 
$p,q$ are paths on $\Gamma$ of length $m,n$ respectively whose sources and targets agree.

The action of tangles is given by a state-sum model similar to a graph planar algebra:
\begin{equation}
\label{eq:StateSumForGMA}
T((p_1,q_1), \dots, (p_k, q_k))
:=
\sum_{
\text{states $\sigma$ on $T$}
}
\prod_{1\leq i\leq k}
\delta_{\sigma|_i =  (p_i,q_i)}
\sigma|_0
\end{equation}
Here, $T$ is a monoidal tangle with $k$ input disks, and the $(p_i,q_i)$ are basis elements (pairs of paths) in $\cG\cM\cA(\Gamma)_{m_i \to n_i}$.
A \emph{state} $\sigma$ on a monoidal tangle $T$ is an assignment of vertices and edges of $\Gamma$ to the regions and strings of $T$ respectively such that if a string labelled by $\varepsilon$ separates the left region $R_\ell$ from the right region $R_r$, then $R_\ell$ is labelled by $s(\varepsilon)$ and $R_r$ is labelled by $t(\varepsilon)$.
Now $\sigma|_i$ denotes the pair of paths in $\cG\cM\cA(\Gamma)_{m_i \to n_i}$ obtained from reading the bottom and top boundaries of the $i$-th input disk from left to right respectively.
In other words, we only sum over states which are `compatible' with the paths we input.
\end{defn}

\begin{example}
Consider the following directed graph:
$$
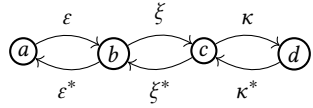
\begin{tikzpicture}[baseline = -.1cm, scale=.6]
	\node[draw, inner sep=.5mm, circle, thick] (a) at (0,0) [below] {\scriptsize{$a$}};
	\node[draw, inner sep=.5mm, circle, thick] (b) at (2,0) [below] {\scriptsize{$b$}};
	\node[draw, inner sep=.5mm, circle, thick] (c) at (4,0) [below] {\scriptsize{$c$}};
	\node[draw, inner sep=.5mm, circle, thick] (d) at (6,0) [below] {\scriptsize{$d$}};
	\draw[->] (a) to[out=30, in=150] node [above] {\scriptsize{$\varepsilon$}} (b);
	\draw[->] (b) to[out=-150, in=-30] node [below] {\scriptsize{$\varepsilon^*$}} (a);
	\draw[->] (b) to[out=30, in=150] node [above] {\scriptsize{$\xi$}} (c);
	\draw[->] (c) to[out=-150, in=-30] node [below] {\scriptsize{$\xi^*$}} (b);
	\draw[->] (c) to[out=30, in=150] node [above] {\scriptsize{$\kappa$}} (d);
	\draw[->] (d) to[out=-150, in=-30] node [below] {\scriptsize{$\kappa^*$}} (c);
\end{tikzpicture}
$$
For the monoidal tangle displayed below on the left, there are exactly two compatible states for the input
$(x_1=(\varepsilon\xi, \varepsilon\xi), x_2=(\xi\xi^*, \varepsilon^*\varepsilon), x_3=(\emptyset, \xi^*\xi))$, which are displayed below on the right. 
$$
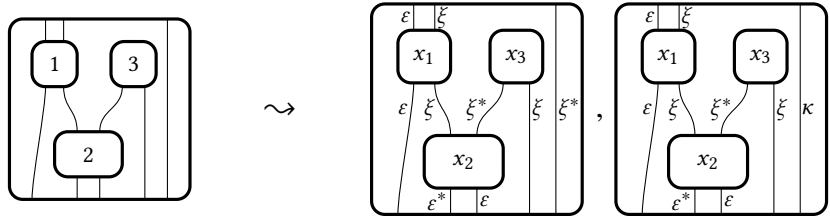
\begin{tikzpicture}[baseline = -.1cm, scale=.6]
	\draw[very thick, rounded corners = 5pt] (-2,-2) rectangle (2,2);
	\draw[fill=white, very thick, rounded corners = 5pt] (-1.5,.5) rectangle (-.5,1.5);
	\draw[fill=white, very thick, rounded corners = 5pt] (-1,-1.5) rectangle (.5,-.5);
	\draw[fill=white, very thick, rounded corners = 5pt] (.25,.5) rectangle (1.25,1.5);
	\node at (-1,1) {\scriptsize{$1$}};
	\node at (-.25,-1) {\scriptsize{$2$}};
	\node at (.75,1) {\scriptsize{$3$}};
	\draw (-.8,.5) .. controls ++(270:.5cm) and ++(90:.5cm) .. (-.5,-.5);
	\draw (-1.5,-2) .. controls ++(90:.5cm) and ++(270:.5cm) .. (-1.2,.5);
	\draw (0,-.5) .. controls ++(90:.5cm) and ++(270:.5cm) .. (.5,.5);
	\draw (-1.2,1.5) -- (-1.2,2);
	\draw (-.8,1.5) -- (-.8,2);
	\draw (-.5,-2) -- (-.5,-1.5);
	\draw (0,-2) -- (0,-1.5);
	\draw (1,-2) -- (1,.5);
	\draw (1.5,-2) -- (1.5,2);
\end{tikzpicture}
\qquad
\leadsto
\qquad
\begin{tikzpicture}[baseline = -.1cm, scale=.7]
	\draw[very thick, rounded corners = 5pt] (-2,-2) rectangle (2,2);
	\draw[fill=white, very thick, rounded corners = 5pt] (-1.5,.5) rectangle (-.5,1.5);
	\draw[fill=white, very thick, rounded corners = 5pt] (-1,-1.5) rectangle (.5,-.5);
	\draw[fill=white, very thick, rounded corners = 5pt] (.25,.5) rectangle (1.25,1.5);
	\node at (-1,1) {\scriptsize{$x_1$}};
	\node at (-.25,-1) {\scriptsize{$x_2$}};
	\node at (.75,1) {\scriptsize{$x_3$}};
	\draw (-.8,.5) .. controls ++(270:.5cm) and ++(90:.5cm) .. (-.5,-.5) ;
	\draw (-1.5,-2) .. controls ++(90:.5cm) and ++(270:.5cm) .. (-1.2,.5);
	\draw (0,-.5) .. controls ++(90:.5cm) and ++(270:.5cm) .. (.5,.5);
	\draw (-1.2,1.5) -- (-1.2,2);
	\draw (-.8,1.5) -- (-.8,2);
	\draw (-.5,-2) -- (-.5,-1.5);
	\draw (0,-2) -- (0,-1.5);
	\draw (1,-2) -- (1,.5);
	\draw (1.5,-2) -- (1.5,2);
	\node at (-1.35,1.75) {\scriptsize{$\varepsilon$}};
	\node at (-.65,1.75) {\scriptsize{$\xi$}};
	\node at (1.75,0) {\scriptsize{$\xi^*$}};
	\node at (-1.4,0) {\scriptsize{$\varepsilon$}};
	\node at (-.9,0) {\scriptsize{$\xi$}};
	\node at (1.15,0) {\scriptsize{$\xi$}};
	\node at (0,0) {\scriptsize{$\xi^*$}};
	\node at (-.75,-1.75) {\scriptsize{$\varepsilon^*$}};
	\node at (.15,-1.75) {\scriptsize{$\varepsilon$}};
\end{tikzpicture}
\,,\,
\begin{tikzpicture}[baseline = -.1cm, scale=.7]
	\draw[very thick, rounded corners = 5pt] (-2,-2) rectangle (2,2);
	\draw[fill=white, very thick, rounded corners = 5pt] (-1.5,.5) rectangle (-.5,1.5);
	\draw[fill=white, very thick, rounded corners = 5pt] (-1,-1.5) rectangle (.5,-.5);
	\draw[fill=white, very thick, rounded corners = 5pt] (.25,.5) rectangle (1.25,1.5);
	\node at (-1,1) {\scriptsize{$x_1$}};
	\node at (-.25,-1) {\scriptsize{$x_2$}};
	\node at (.75,1) {\scriptsize{$x_3$}};
	\draw (-.8,.5) .. controls ++(270:.5cm) and ++(90:.5cm) .. (-.5,-.5) ;
	\draw (-1.5,-2) .. controls ++(90:.5cm) and ++(270:.5cm) .. (-1.2,.5);
	\draw (0,-.5) .. controls ++(90:.5cm) and ++(270:.5cm) .. (.5,.5);
	\draw (-1.2,1.5) -- (-1.2,2);
	\draw (-.8,1.5) -- (-.8,2);
	\draw (-.5,-2) -- (-.5,-1.5);
	\draw (0,-2) -- (0,-1.5);
	\draw (1,-2) -- (1,.5);
	\draw (1.5,-2) -- (1.5,2);
	\node at (-1.35,1.75) {\scriptsize{$\varepsilon$}};
	\node at (-.65,1.75) {\scriptsize{$\xi$}};
	\node at (1.65,0) {\scriptsize{$\kappa$}};
	\node at (-1.4,0) {\scriptsize{$\varepsilon$}};
	\node at (-.9,0) {\scriptsize{$\xi$}};
	\node at (1.15,0) {\scriptsize{$\xi$}};
	\node at (0,0) {\scriptsize{$\xi^*$}};
	\node at (-.75,-1.75) {\scriptsize{$\varepsilon^*$}};
	\node at (.15,-1.75) {\scriptsize{$\varepsilon$}};
\end{tikzpicture}
$$
Hence the output of the tangle on the left applied to the input $(x_1,x_2,x_3)$ is
$(\varepsilon \xi\xi^*,\varepsilon\varepsilon^*\varepsilon\xi\xi^*) + (\varepsilon \xi\kappa,\varepsilon\varepsilon^*\varepsilon\xi\kappa)$.
\end{example}

The graph monoidal algebra of $\Gamma$ is really the non-pivotal analog of the graph planar algebra of $\Gamma$.
The reader is encouraged to compare the above definition with that of the graph planar algebra of a bipartite graph in Definition \ref{defn:GPA} below.

\begin{thm}
\label{thm:StateSumForMonoidalAlgebra}
Given a connected $\Vec$-graph $\vec{\Gamma} = (J, V) \in \Vec(J\times J)$, the semisimple monoidal algebra $\cP(\Vec(J\times J),\vec{\Gamma})_{\bullet \to\bullet}$ from Example \ref{example:GMAfromCX} is isomorphic to the graph monoidal algebra $\cG\cM\cA(\Gamma)_{\bullet \to \bullet}$.
\end{thm}
\begin{proof}
Denoting $\vec{\Gamma}^{\otimes n} = (J, V^{\otimes n})$, we have a canonical isomorphism
$$
V^{\otimes n}_{ik}
\cong
\bigoplus_{j_1,\dots, j_{n-1} \in J} V_{ij_1} \otimes V_{j_1, j_2} \otimes \cdots \otimes V_{j_{n-1} k}.
$$
Observe that an $f\in \Hom_{\Vec(J\times J)}(\vec{\Gamma}^{\otimes m} \to \vec{\Gamma}^{\otimes n})$ is completely determined by its component maps
$$
\left\{
f_{i\ell} 
: 
\bigoplus_{j_1,\dots, j_{m-1} \in J} V_{ij_1} \otimes V_{j_1 j_2} \otimes \cdots \otimes V_{j_{m-1} \ell}
\to 
\bigoplus_{k_1,\dots, k_{n-1} \in J} V_{ik_1} \otimes V_{k_1 k_2} \otimes \cdots \otimes V_{k_{n-1} \ell}.
\right\}_{i,\ell\in J}.
$$
Now fix a basis $\{\varepsilon_{i\ell}^k\}$ for each $V_{i\ell}$, and for each pair of paths on $\Gamma$ from $i$ to $\ell$
$$
p = \varepsilon_{ij_1}^{p_1} \otimes \cdots \otimes \varepsilon_{j_{m-1}\ell}^{p_m}
\qquad
\qquad
q = \varepsilon_{ik_1}^{q_1} \otimes \cdots \otimes \varepsilon_{k_{n-1}\ell}^{q_n},
$$ 
of lengths $m$ and $n$ respectively, we let $F^{i\ell}_{p\to q}\in \Hom_{\Vec(J\times J)}(\vec{\Gamma}^{\otimes m} \to \vec{\Gamma}^{\otimes n})$ be the unique $i\ell$-component map sending $p$ to $q$ and all other paths $p'$ from $i$ to $\ell$ of length $m$ to zero.
We see then that
\begin{equation}
\label{eq:BasisOfVecGMA}
\Vec(J\times J, \vec{\Gamma})_{m \to n}
:=
\Hom_{\Vec(J\times J)}(\vec{\Gamma}^{\otimes m} \to \vec{\Gamma}^{\otimes n})
=
\bigoplus_{i,\ell\in J} \operatorname{span}_\bbC
\left\{
F^{i\ell}_{p\to q}
\right\}_{\text{$p,q$ paths $i$ to $\ell$}}.
\end{equation}
Now it is straightforward to verify that the linear extension 
$$
\Phi_{m\to n} : \Vec(J\times J, \vec{\Gamma})_{m \to n} \to \cG\cM\cA(\Gamma)_{m\to n}
$$
of $F^{i\ell}_{p\to q} \mapsto (p,q)$ is a linear isomorphism for all $m,n\geq 0$.

It remains to see that this isomorphism is compatible with the action of monoidal tangles.
It suffices to show that $\Phi$ intertwines the actions of a single vertical strand with no input disk, vertical stacking tangles, and horizontal concatenation tangles, as these tangles generate the monoidal operad.
The vertical strand in $\cP(\Vec(J\times J), \vec{\Gamma})_{1\to 1}$ is given by
$$
\bigoplus_{i,j\in J} \id_{V_{ij}}
=
\bigoplus_{i,j\in J}
\sum_{k} F^{ij}_{\varepsilon_{ij}^k\to \varepsilon_{ij}^k}
\overset{\Phi_{1\to 1}}{\longmapsto}
\bigoplus_{i,j\in J}
\sum_{k} (\varepsilon_{ij}^k , \varepsilon_{ij}^k)
=
\id_{\cG\cM\cA(\Gamma)_{1\to 1}}
=
\begin{tikzpicture}[baseline = -.1cm]
	\draw[very thick, fill=white, rounded corners = 5pt] (-.3,-.3) rectangle (.3,.3);
	\draw (0,-.3) -- (0,.3);
\end{tikzpicture}\,.
$$
Hence $\Phi_{1\to 1}$ preserves the strand.

To see that $\Phi_{\bullet\to \bullet}$ preserves composition, we check on our basis \eqref{eq:BasisOfVecGMA}.
Suppressing subscripts on edges for simplicity, suppose
$$
p = \varepsilon^{p_1} \otimes \cdots \otimes \varepsilon^{p_k} 
\qquad
\qquad
q = \varepsilon^{q_1} \otimes \cdots \otimes \varepsilon^{q_\ell}
\qquad
\qquad
r = \varepsilon^{r_1} \otimes \cdots \otimes \varepsilon^{r_m}
\qquad
\qquad
s = \varepsilon^{s_1} \otimes \cdots \otimes \varepsilon^{s_n}
$$ 
are paths on $\Gamma$ from $i$ to $j$.
Then
$$
\Phi_{k \to \ell}(F^{ij}_{r\to s})
\circ
\Phi_{m \to n}(F^{ij}_{p\to q})
=
\begin{tikzpicture}[baseline = -.1cm]
	\draw (-.7,-1.2) -- (-.7,1.2);
	\draw (-.3,-1.2) -- (-.3,1.2);
	\draw (.7,-1.2) -- (.7,1.2);
	\draw[very thick, fill=white, rounded corners = 5pt] (-1,-1) rectangle (1,-.4);
	\draw[very thick, fill=white, rounded corners = 5pt] (-1,.4) rectangle (1,1);
	\node at (0,.7) {$(r,s)$};
	\node at (0,-.7) {$(p,q)$};
	\node at (-.9,.17) {\scriptsize{$\varepsilon^{r_1}$}};
	\node at (-.9,-.23) {\scriptsize{$\varepsilon^{q_1}$}};
	\node at (-.5,.17) {\scriptsize{$\varepsilon^{r_2}$}};
	\node at (-.5,-.23) {\scriptsize{$\varepsilon^{q_2}$}};
	\node at (.05,0) {\scriptsize{$\cdots$}};
	\node at (.5,.17) {\scriptsize{$\varepsilon^{r_m}$}};
	\node at (.5,-.23) {\scriptsize{$\varepsilon^{q_\ell}$}};
	\node at (-.7,1.4) {\scriptsize{$\varepsilon^{s_1}$}};
	\node at (-.3,1.4) {\scriptsize{$\varepsilon^{s_2}$}};
	\node at (.2,1.4) {\scriptsize{$\cdots$}};
	\node at (.7,1.4) {\scriptsize{$\varepsilon^{s_n}$}};
	\node at (-.7,-1.4) {\scriptsize{$\varepsilon^{p_1}$}};
	\node at (-.3,-1.4) {\scriptsize{$\varepsilon^{p_2}$}};
	\node at (.2,-1.4) {\scriptsize{$\cdots$}};
	\node at (.7,-1.4) {\scriptsize{$\varepsilon^{p_k}$}};
\end{tikzpicture}
=
\delta_{\ell=m}
\delta_{q=r}
(p,s)
=
\delta_{\ell=m}
\delta_{q=r}
\Phi_{k \to n}(F^{ij}_{p\to s})
=
\Phi_{k \to n}(F^{ij}_{r\to s}\circ F^{ij}_{p\to q}).
$$
As composition is multi-linear, the general case follows by taking linear combinations.

Finally, to show $\Phi_{\bullet\to \bullet}$ preserves tensor product, we again work with our basis \eqref{eq:BasisOfVecGMA}.
Again suppressing subscripts for simplicity, suppose 
$$
p^{gh} = \varepsilon^{p_1} \otimes \cdots \otimes \varepsilon^{p_k} 
\qquad
\qquad
q^{gh} = \varepsilon^{q_1} \otimes \cdots \otimes \varepsilon^{q_\ell}
\qquad
\qquad
r^{ij} = \varepsilon^{r_1} \otimes \cdots \otimes \varepsilon^{r_m}
\qquad
\qquad
s^{ij} = \varepsilon^{s_1} \otimes \cdots \otimes \varepsilon^{s_n}
$$ 
are paths on $\Gamma$,
where $p^{gh}, q^{gh}$ go from $g$ to $h$, and $r^{ij}, s^{ij}$ go from $i$ to $j$.
We calculate
\begin{align*}
\Phi_{k \to \ell}(F^{gh}_{p\to q})
\otimes
\Phi_{m \to n}(F^{ij}_{r\to s})
&=
\begin{tikzpicture}[baseline = -.1cm]
	\draw (-.5,-.5) -- (-.5,.5);
	\draw (-.2,-.5) -- (-.2,.5);
	\draw (.5,-.5) -- (.5,.5);
	\draw (1.1,-.5) -- (1.1,.5);
	\draw (1.4,-.5) -- (1.4,.5);
	\draw (2.1,-.5) -- (2.1,.5);
	\draw[very thick, fill=white, rounded corners = 5pt] (-.7,-.3) rectangle (.7,.3);
	\draw[very thick, fill=white, rounded corners = 5pt] (.9,-.3) rectangle (2.3,.3);
	\node at (0,0) {$(p,q)$};
	\node at (-.5,.7) {\scriptsize{$\varepsilon^{q_1}$}};
	\node at (-.2,.7) {\scriptsize{$\varepsilon^{q_2}$}};
	\node at (.15,.7) {\scriptsize{$\cdots$}};
	\node at (.5,.7) {\scriptsize{$\varepsilon^{q_\ell}$}};
	\node at (-.5,-.7) {\scriptsize{$\varepsilon^{p_1}$}};
	\node at (-.2,-.7) {\scriptsize{$\varepsilon^{p_2}$}};
	\node at (.15,-.7) {\scriptsize{$\cdots$}};
	\node at (.5,-.7) {\scriptsize{$\varepsilon^{p_k}$}};
	\node at (1.6,0) {$(r,s)$};
	\node at (1.1,.7) {\scriptsize{$\varepsilon^{s_1}$}};
	\node at (1.4,.7) {\scriptsize{$\varepsilon^{s_2}$}};
	\node at (1.75,.7) {\scriptsize{$\cdots$}};
	\node at (2.1,.7) {\scriptsize{$\varepsilon^{s_n}$}};
	\node at (1.1,-.7) {\scriptsize{$\varepsilon^{r_1}$}};
	\node at (1.4,-.7) {\scriptsize{$\varepsilon^{r_2}$}};
	\node at (1.75,-.7) {\scriptsize{$\cdots$}};
	\node at (2.1,-.7) {\scriptsize{$\varepsilon^{r_m}$}};
\end{tikzpicture}
=
\delta_{h=i}
\begin{tikzpicture}[baseline = -.1cm]
	\draw (-.5,-.5) -- (-.5,.5);
	\draw (-.2,-.5) -- (-.2,.5);
	\draw (.5,-.5) -- (.5,.5);
	\draw (1.1,-.5) -- (1.1,.5);
	\draw (1.4,-.5) -- (1.4,.5);
	\draw (2.1,-.5) -- (2.1,.5);
	\draw[very thick, fill=white, rounded corners = 5pt] (-.7,-.3) rectangle (2.3,.3);
	\node at (.8,0) {$(pr,qs)$};
	\node at (-.5,.7) {\scriptsize{$\varepsilon^{q_1}$}};
	\node at (-.2,.7) {\scriptsize{$\varepsilon^{q_2}$}};
	\node at (.15,.7) {\scriptsize{$\cdots$}};
	\node at (.5,.7) {\scriptsize{$\varepsilon^{q_\ell}$}};
	\node at (-.5,-.7) {\scriptsize{$\varepsilon^{p_1}$}};
	\node at (-.2,-.7) {\scriptsize{$\varepsilon^{p_2}$}};
	\node at (.15,-.7) {\scriptsize{$\cdots$}};
	\node at (.5,-.7) {\scriptsize{$\varepsilon^{p_k}$}};
	\node at (1.1,.7) {\scriptsize{$\varepsilon^{s_1}$}};
	\node at (1.4,.7) {\scriptsize{$\varepsilon^{s_2}$}};
	\node at (1.75,.7) {\scriptsize{$\cdots$}};
	\node at (2.1,.7) {\scriptsize{$\varepsilon^{s_n}$}};
	\node at (1.1,-.7) {\scriptsize{$\varepsilon^{r_1}$}};
	\node at (1.4,-.7) {\scriptsize{$\varepsilon^{r_2}$}};
	\node at (1.75,-.7) {\scriptsize{$\cdots$}};
	\node at (2.1,-.7) {\scriptsize{$\varepsilon^{r_m}$}};
\end{tikzpicture}
\\&=
\delta_{h=i}
\Phi_{k+m \to \ell+n}(F^{gj}_{pr\to qs})
=
\Phi_{k+m \to \ell+n}(F^{gh}_{p\to q}\otimes F^{ij}_{r\to s}).
\end{align*}
Again, the general case follows by taking linear combinations.

Since the actions of the generating tangles agree, we are finished.
\end{proof}

\begin{defn}
\label{defn:FusionGraph}
Suppose $\scrC$ is a semisimple monoidal category Cauchy tensor generated by $X$, and $\cM$ is a finitely semisimple module category.  Let $\Irr(\cM)=\{m_1,\dots,m_r\}$ be a set of representatives of simple objects of $\cM$, and define $J:= \{1,\dots, r\}$.  The \emph{fusion $\Vec$-graph} $\vec{\Gamma}$ of $\cM$ with respect to $X$ is the $\Vec$-graph whose vertices are $J$ and whose edge spaces are given by
$$
V_{ij} 
:= 
\cM(X \vartriangleright m_i \to m_j).
$$
\end{defn}

\begin{prop}[Graph monoidal algebra embedding]
\label{prop:GMAEmbedding}
Suppose $\scrC$ is a semisimple monoidal category Cauchy tensor generated by $X$, 
$\cM$ is a finitely semisimple  category with isomorphism classes of simple objects indexed by $J$, and
$\vec{\Gamma}$ is a $\Vec$-graph whose vertices are $J$.
Equipping $\cM$ with the structure of an indecomposable left $\scrC$-module category 
whose connected fusion $\Vec$-graph with respect to $X$ is $\vec{\Gamma}$
is equivalent to embedding the monoidal algebra $\cP(\scrC,X)_{\bullet \to \bullet}$ into $\cG\cM\cA(\Gamma)_{\bullet \to \bullet}$.
\end{prop}
\begin{proof}
As we discussed previously, $\scrC$-module structures on $\cM$ are equivalent to tensor functors $\scrC \to \End(\cM)$.
By semisimplicity, $\End(\cM) \cong \Vec(J\times J)$.
Notice that every linear tensor functor $\scrC_X \to \Vec(J\times J)$ uniquely extends to its Cauchy completion $\scrC$, and every linear tensor functor $\scrC_X \to \Vec(J\times J)$ has essential image in $\Vec(J\times J)_{\vec{\Gamma}}$.
The result now follows from Theorem \ref{thm:StateSumForMonoidalAlgebra} together with the equivalence of categories from Corollary \ref{cor:SemisimpleMonoidalAlgebraEquivalence}.
\end{proof}

\begin{remark}
When $\scrC$ is fusion, the Frobenius-Perron dimension of $X$ is the norm of the underlying graph $\Gamma$.
\end{remark}

%%%%%%%%%%%%%%%%%%%%%%%%%%%%%%%%%%%%%%%%%%%%%%%%%%%%%%%%
\subsubsection{Embedding multifusion categories into multishaded graph monoidal algebras}

We now adapt Proposition \ref{prop:GMAEmbedding} to more closely approximate subfactor planar algebras, which have two shadings.
On the $\Vec$-graph side, we will see this translates into our $\Vec$-graphs $\vec{\Gamma} = (J,V)$ being \emph{bipartite}, i.e., $J = J_+\amalg J_-$, and $V_{ij} = (0)$ whenever $i \in J_\pm$ and $j\in J_\mp$.

All the results and definitions in the beginning of this section about the graph tensor category and the (graph) monoidal algebra have straightforward 2-shaded/bipartite generalizations to multifusion categories.

\begin{defn}
\label{defn:BipartiteFusionGraph}
Suppose $\cD$ is a $2$-shaded multifusion category with Cauchy tensor generator $X$ in $\cD_{12}$ and $\cM$ is a finitely semisimple module category.  
As in Definition \ref{defn:FusionGraph}, we define the \emph{fusion $\Vec$-graph} $\vec{\Gamma}$ of $\cM$ with respect to $X$ to have vertices corresponding to simple objects in $\cM$ and edge spaces
$$
V_{ij}
:=
\cM(X \vartriangleright m_i \to m_j)
.$$
Observe that since $\cD$ is 2-shaded and $X\in \cD_{12}$, $\vec{\Gamma}$ is bipartite.
\end{defn}

\begin{prop}[2-shaded graph monoidal algebra embedding]
\label{prop:2ShadedGMAEmbedding}
Suppose $\cD,\cM,\vec{\Gamma}$ are as in Definition \ref{defn:BipartiteFusionGraph}.
Indecomposable left $\cD$-module category structures on $\cM$ whose fusion graph is $\Gamma$ correspond to embeddings of the $2$-shaded monoidal algebra $\cP(\cD,X)_{\bullet \to \bullet}$ into $\cG\cM\cA(\Gamma)_{\bullet \to \bullet}$
\end{prop}

This is the purely algebraic version of our graph planar algebra embedding theorem.

\begin{remark}
It is easy to see that the underlying monoidal category structure on the graph planar algebra $\cG\cP\cA(\Gamma)_\bullet$ agrees with $\cG\cM\cA(\Gamma)_{\bullet \to \bullet}$.  
In particular, it follows just from the results of this section that a graph planar algebra embedding yields a module category.  
This result alone is enough to show the existence of $\cE\cH_3$ and $\cE\cH_4$ as tensor categories from the GPA embeddings constructed in \S\ref{sec:construction}, but not to determine whether these tensor categories are unitary.
\end{remark}

%%%%%%%%%%%%%%%%%%%%%%%%%%%%%%%%%%%%%%%%%%%%%%%%%%%%%%%%
\subsubsection{Embedding unitary multifusion categories}
\label{sec:EmbeddingUnitaryMultifusionIntoGMA}

Similar to the algebraic and multishaded settings, one can adapt to the unitary setting by first considering the unitary multifusion category $\Hilb(J\times J)$ of bi-$J$-graded Hilbert spaces, which is $\dag$-equivalent to $\End^\dag(\cM)$ for any $\Cstar$ category $\cM$ where $|\Irr(\cM)| = |J|$.
(This is the approach to classifying unitary Temperley-Lieb modules in \cite{MR3420332}.)
Analogous to Definition \ref{defn:VecGraph}, we may identify the objects of $\Hilb(J\times J)$ with $\Hilb$-enriched graphs $\vec{\Gamma}=(J, H)$, and we obtain honest graphs by choosing orthonormal bases for the edge Hilbert spaces.

Now the graph monoidal algebra $\cG\cM\cA(\Gamma)_{\bullet \to \bullet}$ carries an obvious $\dag$-structure by the anti-linear extension of $(p,q)\mapsto (q,p)$ where $p,q$ are paths on $\Gamma$ whose sources and targets agree.
It is straightforward to show that this $\dag$-structure is compatible with the vertical reflection of tangles about the $x$-axis, and that it satisfies the positivity axioms, making $\cG\cM\cA(\Gamma)_{\bullet \to \bullet}$ a unitary monoidal algebra.
Similar to Theorem \ref{thm:StateSumForMonoidalAlgebra}, we have a $\dag$-isomorphism of unitary monoidal algebras $\cG\cM\cA(\Gamma)_{\bullet \to \bullet}\cong \cP(\Hilb(J\times J), \vec{\Gamma})_{\bullet \to \bullet}$.

We may pass to the unitary 2-shaded setting by working with bipartite $\Hilb$-graphs.
There is an `obvious' unitary version of Propositions \ref{prop:GMAEmbedding} and \ref{prop:2ShadedGMAEmbedding}.
However, one should consider the following subtlety with this adaptation.

\begin{warn}
While it is true that a $\Cstar$ category is finitely semisimple if and only if it is dagger equivalent to $\Hilb^n$ for some $n\in\bbN$, 
the only way we know to construct such an equivalence requires choosing Hilbert space structures on the hom spaces of $\cM$.
We will see in Remark \ref{rem:OtherCharacterizationsOfTraces} below that this extra structure corresponds to a unitary trace on $\cM$, which is gives a distinguished choice of unitary pivotal structure on $\End^\dag(\cM)$ by Proposition \ref{prop:TracesToPivotalStructures} below.
This is not necessarily a problem, since there is a canonical choice for such inner products which declare all identity morphisms for simple objects to have norm 1.
This corresponds to the canonical spherical structure \cite{MR1444286,MR2091457,MR3342166,1808.00323} on $\End^\dag(\cM)$, which we discuss in \S\ref{sec:UnitaryPivotal} below.

If one does not wish to choose this additional structure, then instead of Hilbert spaces, one can work with Hilbert $\Cstar$ bimodules.
Observe that when $\cM$ is $\Cstar$, 
$\cM(m\to m),\cM(n\to n)$ are $\Cstar$ algebras for all $m,n\in \cM$, and
$\cM(m\to n)$ has the canonical structure of a Hilbert $\Cstar$ $\cM(m\to m) - \cM(n\to n)$ bimodule.
We will not discuss this further as it would take us too far afield.
\end{warn}

The remaining sections of this chapter are dedicated to adapting the above proposition to the pivotal and unitary pivotal settings.
We will see this adaptation naturally becomes the module embedding theorem for graph planar algebras.

%%%%%%%%%%%%%%%%%%%%%%%%%%%%%%%%%%%%%%%%%%%%%%%%%%%%%%%%%%%
%%%%%%%%%%%%%%%%%%%%%%%%%%%%%%%%%%%%%%%%%%%%%%%%%%%%%%%%%%%
%%%%%%%%%%%%%%%%%%%%%%%%%%%%%%%%%%%%%%%%%%%%%%%%%%%%%%%%%%%
\subsection{Planar algebras}
\label{sec:PlanarAlgebras}

In \S\ref{sec:AlgebraicGPA}, we defined the notion of a (shaded) monoidal algebra.
As alluded to earlier, the pivotal analog of a monoidal algebra is a planar algebra.
To simplify the exposition, we will only define (2-)shaded planar algebras with a single strand type following \cite{MR2972458}; we refer the reader to \cite{math.QA/9909027,JonesPANotes} (see also \cite{MR2496052}) for a host of other notions of planar algebra.

\begin{defn}
A (2-)\emph{shaded planar tangle} consists of a disk with smaller internal input disks, together with non-intersecting strings between the disks, a checkerboard shading, and a distinguished interval marked by $\star$ for each disk.
We consider shaded planar tangles up to isotopy.
We say a shaded planar tangle has \emph{type} $((n_0, \pm_0); (n_1,\pm_1) , \dots, (n_k,\pm_k))$ if the $i$-th disk has $2n_i$ strings connected to it, and its distinguished interval is in an unshaded/shaded region corresponding to $\pm_i$.
The collection of shaded planar tangles forms a colored operad by inserting tangles into the input disks to get a new shaded planar tangle, making sure the distinguished intervals align.
We include below an example of a composite of a tangle of type $((4,-);(2,-),(1,+),(3,-))$ with a tangle of type $((3,-); (1,+))$ resulting in a tangle of type $((4,-);(2,-), (1,+), (1,+))$:
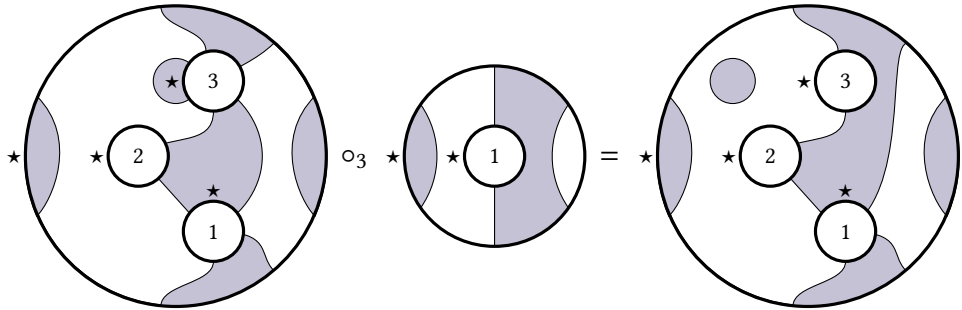
$$
\begin{tikzpicture}[baseline = -.1cm]
\fill[shaded] (2,0) circle (2cm);
 \filldraw[fill=white]  
  (1.8,-2) .. controls ++(90:.3cm) and ++(-90:.3cm) .. (2.5,-1.4) -- (2.4,-1) -- 
  (1.5,0) .. controls ++(40:.5cm) and ++(270:.5cm) .. (2.5,.6) -- 
  (2.4,1.4) .. controls ++(90:.3cm) and ++(270:.3cm) .. (1.8,2) arc (96:264:2cm) ;
 \filldraw[shaded] (2,1) circle (.3cm);
 \filldraw[fill=white] (2.5,1) .. controls ++(30:.3cm) and ++(-135:.3cm) .. (3.3,1.5)
  arc (49:-49:2cm)  .. controls ++(135:.3cm) and ++(0:.8cm) ..  (2.5,-1)
  .. controls ++(30:1cm) and ++(-30:1cm) .. (2.5,1);
 \filldraw[shaded] (3.85,.8) .. controls ++(-135:.8cm) and ++(135:.3cm) .. (3.85,-.8) arc (-23:23:2cm);
 \filldraw[shaded] (.15,.8) .. controls ++(-45:.8cm) and ++(45:.3cm) .. (.15,-.8) arc (203:157:2cm);
 \ncircle{}{(2,0)}{2}{180}{}
 \ncircle{fill=white}{(2.5,1)}{.4}{180}{\scriptsize{$3$}}
 \ncircle{fill=white}{(1.5,0)}{.4}{180}{\scriptsize{$2$}}
 \ncircle{fill=white}{(2.5,-1)}{.4}{90}{\scriptsize{$1$}}
\end{tikzpicture}
\circ_{3}
\begin{tikzpicture}[baseline = -.1cm]
 \filldraw[fill=white] (0,-1.2) -- (0,1.2) arc (90:270:1.2cm);
 \filldraw[shaded] (145:1.2cm) arc (145:215:1.2cm) arc (-35:35:1.2cm);
 \filldraw[shaded] (0,-1.2) -- (0,1.2) arc (90:-90:1.2cm);
 \filldraw[fill=white] (35:1.2cm) arc (35:-35:1.2cm) arc (215:145:1.2cm);
 \ncircle{}{(0,0)}{1.2}{180}{}
 \ncircle{fill=white}{(0,0)}{.4}{180}{\scriptsize{$1$}}
\end{tikzpicture}
=
\begin{tikzpicture}[baseline = -.1cm]
\fill[shaded] (2,0) circle (2cm);
 \filldraw[fill=white]  
  (1.8,-2) .. controls ++(90:.3cm) and ++(-90:.3cm) .. (2.5,-1.4) -- (2.4,-1) -- 
  (1.5,0) .. controls ++(40:.5cm) and ++(270:.5cm) .. (2.5,.6) -- 
  (2.4,1.4) .. controls ++(90:.3cm) and ++(270:.3cm) .. (1.8,2) arc (96:264:2cm) ;
 \filldraw[shaded] (1,1) circle (.3cm);
 \filldraw[fill=white] (3.3,1.5) arc (49:-49:2cm)  .. controls ++(135:.3cm) and ++(0:.8cm) ..  (2.5,-1)
  .. controls ++(30:1cm) and ++(210:.3cm) .. (3.3,1.5);
 \filldraw[shaded] (3.85,.8) .. controls ++(-135:.8cm) and ++(135:.3cm) .. (3.85,-.8) arc (-23:23:2cm);
 \filldraw[shaded] (.15,.8) .. controls ++(-45:.8cm) and ++(45:.3cm) .. (.15,-.8) arc (203:157:2cm);
 \ncircle{}{(2,0)}{2}{180}{}
 \ncircle{fill=white}{(2.5,1)}{.4}{180}{\scriptsize{$3$}}
 \ncircle{fill=white}{(1.5,0)}{.4}{180}{\scriptsize{$2$}}
 \ncircle{fill=white}{(2.5,-1)}{.4}{90}{\scriptsize{$1$}}
\end{tikzpicture}
$$

\end{defn}

\begin{defn}
A (2-)\emph{shaded planar algebra} is an algebra for the shaded planar operad.
Unpacking this definition, we have a vector space $\cP_{n,\pm}$ for each color $(n,\pm)$, and 
for each tangle $T$ of type $((n_0, \pm_0); (n_1,\pm_1) , \dots, (n_k,\pm_k))$, we have 
a multi-linear map $Z(T) : \prod_{i=1}^k \cP_{n_i,\pm_i} \to \cP_{n_0, \pm_0}$. 
Composition of tangles then corresponds to the composition of multi-linear maps.

Notice that any shaded planar algebra gives us a canonical $R$-shaded monoidal algebra with region shadings
$R = \{
\ColorDot{draw=black, fill=white}
\,,\,
\ColorDot{shaded}
\}$
and label set
$
S= 
\{
\begin{tikzpicture}[baseline=-.1cm]
	\fill[white] (-.2,-.2) rectangle (0,.2);
	\fill[shaded] (.2,-.2) rectangle (0,.2);
	\draw (0,-.2) -- (0,.2);
\end{tikzpicture}
\,,\,
\begin{tikzpicture}[baseline=-.1cm]
	\fill[shaded] (-.2,-.2) rectangle (0,.2);
	\fill[white] (.2,-.2) rectangle (0,.2);
	\draw (0,-.2) -- (0,.2);
\end{tikzpicture}
\}$
by 
setting 
\begin{equation}
\label{eq:MAfromPA}
\cP_{(n_1,\pm_1) \to (n_2,\pm_2)} 
:= 
\delta_{\pm_1 = \pm_2} \delta_{n_1\equiv n_2 \operatorname{mod} 2}\cP_{(n_1+n_2)/2, \pm_1}
,
\end{equation}
and the action of monoidal tangles is given by adding a $\star$ to the left of every input rectangle.
(Notice that for this $R$ and $S$, every monoidal tangle must have a checkerboard shading.)
Here is an explicit example:
$$
\underbrace{
Z\left(\,
\begin{tikzpicture}[baseline = -.1cm, scale=.6]
	\fill[white, rounded corners = 5pt] (-2,-2) rectangle (1,2);
	\fill[shaded, rounded corners = 5pt] (-.5,-2) rectangle (2,2);
	\fill[shaded] (-.8,1.5) rectangle (-1.2,2);
	\fill[shaded] (-1.5,-2) .. controls ++(90:.5cm) and ++(270:.5cm) .. (-1.2,.5) --
		(-.8,.5) .. controls ++(270:.5cm) and ++(90:.5cm) .. (-.5,-.5) -- (-.5,-1.5) -- (-.5,-2) -- (-1.5,-2);
	\fill[white] (0,-.5) .. controls ++(90:.5cm) and ++(270:.5cm) .. (.5,.5) --
		(1,.5) -- (1,-2) -- (1.5,-2) -- (1.5,2) --(-.8,2) -- (-.8,1.5) -- 
		(-.8,.5) .. controls ++(270:.5cm) and ++(90:.5cm) .. (-.5,-.5) -- (0,-.5);
	\fill[white] (-.5,-2) rectangle (0,-1.5);
	\draw[very thick, rounded corners = 5pt] (-2,-2) rectangle (2,2);
	\draw[fill=white, very thick, rounded corners = 5pt] (-1.5,.5) rectangle (-.5,1.5);
	\draw[fill=white, very thick, rounded corners = 5pt] (-1,-1.5) rectangle (.5,-.5);
	\draw[fill=white, very thick, rounded corners = 5pt] (.25,.5) rectangle (1.25,1.5);
	\node at (-1,1) {\scriptsize{$1$}};
	\node at (-.25,-1) {\scriptsize{$2$}};
	\node at (.75,1) {\scriptsize{$3$}};
	\draw (-.8,.5) .. controls ++(270:.5cm) and ++(90:.5cm) .. (-.5,-.5);
	\draw (-1.5,-2) .. controls ++(90:.5cm) and ++(270:.5cm) .. (-1.2,.5);
	\draw (0,-.5) .. controls ++(90:.5cm) and ++(270:.5cm) .. (.5,.5);
	\draw (-1.2,1.5) -- (-1.2,2);
	\draw (-.8,1.5) -- (-.8,2);
	\draw (-.5,-2) -- (-.5,-1.5);
	\draw (0,-2) -- (0,-1.5);
	\draw (1,-2) -- (1,.5);
	\draw (1.5,-2) -- (1.5,2);
\end{tikzpicture}
\,\right)
}_{\cP_{(2,+) \to (2,+)}\times \cP_{(2,-) \to (2,-)}\times \cP_{(2,+) \to (0,+)}\to \cP_{(5,+) \to (3,+)}}
:=
\underbrace{
Z_{\cP_\bullet}
\left(
\begin{tikzpicture}[baseline = -.1cm, scale=.6]
	\fill[white, rounded corners = 5pt] (-2,-2) rectangle (1,2);
	\fill[shaded, rounded corners = 5pt] (-.5,-2) rectangle (2,2);
	\fill[shaded] (-.8,1.5) rectangle (-1.2,2);
	\fill[shaded] (-1.5,-2) .. controls ++(90:.5cm) and ++(270:.5cm) .. (-1.2,.5) --
		(-.8,.5) .. controls ++(270:.5cm) and ++(90:.5cm) .. (-.5,-.5) -- (-.5,-1.5) -- (-.5,-2) -- (-1.5,-2);
	\fill[white] (0,-.5) .. controls ++(90:.5cm) and ++(270:.5cm) .. (.5,.5) --
		(1,.5) -- (1,-2) -- (1.5,-2) -- (1.5,2) --(-.8,2) -- (-.8,1.5) -- 
		(-.8,.5) .. controls ++(270:.5cm) and ++(90:.5cm) .. (-.5,-.5) -- (0,-.5);
	\fill[white] (-.5,-2) rectangle (0,-1.5);
	\draw[very thick, rounded corners = 5pt] (-2,-2) rectangle (2,2);
	\draw[fill=white, very thick, rounded corners = 5pt] (-1.5,.5) rectangle (-.5,1.5);
	\draw[fill=white, very thick, rounded corners = 5pt] (-1,-1.5) rectangle (.5,-.5);
	\draw[fill=white, very thick, rounded corners = 5pt] (.25,.5) rectangle (1.25,1.5);
	\node at (-1,1) {\scriptsize{$1$}};
	\node at (-.25,-1) {\scriptsize{$2$}};
	\node at (.75,1) {\scriptsize{$3$}};
	\node at (-1.8, 0) {\scriptsize{$\star$}};
	\node at (-1.7, 1) {\scriptsize{$\star$}};
	\node at (0, 1) {\scriptsize{$\star$}};
	\node at (-1.2, -1) {\scriptsize{$\star$}};
	\draw (-.8,.5) .. controls ++(270:.5cm) and ++(90:.5cm) .. (-.5,-.5);
	\draw (-1.5,-2) .. controls ++(90:.5cm) and ++(270:.5cm) .. (-1.2,.5);
	\draw (0,-.5) .. controls ++(90:.5cm) and ++(270:.5cm) .. (.5,.5);
	\draw (-1.2,1.5) -- (-1.2,2);
	\draw (-.8,1.5) -- (-.8,2);
	\draw (-.5,-2) -- (-.5,-1.5);
	\draw (0,-2) -- (0,-1.5);
	\draw (1,-2) -- (1,.5);
	\draw (1.5,-2) -- (1.5,2);
\end{tikzpicture}\,
\right).
}_{\cP_{2,+}\times \cP_{2,-}\times \cP_{1,+} \to \cP_{4,+}}
$$
A shaded planar algebra is called \emph{semisimple} if the underlying monoidal algebra is semisimple, i.e., every $2\times 2$ linking algebra $\cL((k,\pm), (n,\pm))$ as in \eqref{eq:LinkingAlgebra} is a finite dimensional semisimple algebra.
\end{defn}

The article \cite{MR2811311} provided a dictionary between semisimple shaded planar algebras and triples
$(\scrC, X, \varphi)$ where 
$\scrC$ is a semisimple multitensor category,
$1_\scrC = 1_+\oplus 1_-$ is a decomposition (not necessarily into simples!),
$X= 1_+\otimes \scrC\otimes 1_-$ Cauchy tensor generates $\scrC$, and 
$\varphi: \id \Rightarrow \vee\circ \vee$ is a trivialization of the double dual functor also known as a \emph{pivotal structure}.
This dictionary actually gives an equivalence of categories similar to \cite{1607.06041}.

\begin{thm}
\label{thm:AlgebraicPAEquivalence}
There is an equivalence of categories\,\footnote{
\label{footnote:TruncationTriples}
Similar to Footnote \ref{footnote:Truncation}, triples $(\mathcal{C},\varphi,X)$ form a $2$-category where between any two 1-morphisms, there is at most one 2-morphism, which is necessarily invertible when it exists \cite[Lem.~3.5]{1607.06041}.
Hence this 2-category is equivalent to its truncation to a 1-category.
} 
\[
\left\{\, 
\parbox{3.7cm}{\rm Semisimple 2-shaded planar algebras $\cP_\bullet$}\,\left\}
\,\,\,\,\cong\,\,
\left\{\,\parbox{10cm}{\rm Triples $(\scrC, \varphi, X)$ with $(\scrC,\varphi)$ a pivotal multitensor category, and $X\in \scrC$ a Cauchy tensor generator with a decomposition $1_\scrC = 1_+ \oplus 1_-$ such that $X = 1_+\otimes X \otimes 1_-$}\,\right\}.
\right.\right.
\]
\end{thm}

This theorem is exactly the pivotal analog of Corollary \ref{cor:SemisimpleShadedMonoidalAlgebraEquivalence}, which provides the equivalence of the underlying 2-shaded monoidal algebras and semisimple linear monoidal categories.
Indeed, given such a planar algebra, we see that its category of idempotents $\scrC$ is rigid with 
$$
\ev_p :=
\begin{tikzpicture}[baseline=-.1cm]
	\draw (-.5,-.5) -- (-.5,.2) arc (180:0:.5cm) -- (.5,-.5);
	\roundNbox{fill=white}{(-.5,0)}{.2}{0}{0}{\rotatebox{180}{$p$}}
	\roundNbox{fill=white}{(.5,0)}{.2}{0}{0}{$p$}
	\node at (-.1,0) {\scriptsize{$\star$}};
	\node at (.1,0) {\scriptsize{$\star$}};
	\node at (-.7,.35) {\scriptsize{$n$}};
	\node at (.7,.35) {\scriptsize{$n$}};
	\node at (-.7,-.35) {\scriptsize{$n$}};
	\node at (.7,-.35) {\scriptsize{$n$}};
\end{tikzpicture}
\qquad
\qquad
\coev_p :=
\begin{tikzpicture}[baseline=-.1cm]
	\draw (-.3,.5) -- (-.3,-.2) arc (-180:0:.3cm) -- (.3,.5);
	\roundNbox{fill=white}{(.3,0)}{.2}{0}{0}{\rotatebox{180}{$p$}}
	\roundNbox{fill=white}{(-.3,0)}{.2}{0}{0}{$p$}
	\node at (-.7,0) {\scriptsize{$\star$}};
	\node at (.7,0) {\scriptsize{$\star$}};
	\node at (-.5,.35) {\scriptsize{$n$}};
	\node at (.5,.35) {\scriptsize{$n$}};
	\node at (-.5,-.35) {\scriptsize{$n$}};
	\node at (.5,-.35) {\scriptsize{$n$}};
\end{tikzpicture}
,
$$
and choosing these evaluation and coevaluation maps, choosing $\varphi_p := p \in \Hom(p \to p)$ endows $\scrC$ with a pivotal structure.
Conversely, one passes from triples $(\scrC, \varphi, X)$ to planar algebras via the graphical calculus to produce a monoidal algebra, and one then gets cups and caps by defining
$$
\begin{tikzpicture}[baseline = 0cm]
	\fill[shaded] (-.1,0) -- (-.1,.3) -- (.5,.3) -- (.5,0) -- (.4,0) arc (0:180:.2cm);
	\draw[] (0,0) arc (180:0:.2cm);
\end{tikzpicture}
:=\ev_{X}
\qquad
\qquad
\begin{tikzpicture}[baseline = -.2cm, yscale = -1]
	\filldraw[shaded] (0,0) arc (180:0:.2cm);
\end{tikzpicture}
:=\coev_{X}
\qquad
\qquad
\begin{tikzpicture}[baseline = 0cm]
	\filldraw[shaded] (0,0) arc (180:0:.2cm);
\end{tikzpicture}
:=\ev_{X^\vee} \circ (\varphi_X \otimes \id_{X^\vee})
\qquad
\qquad
\begin{tikzpicture}[baseline = -.2cm, yscale=-1]
	\fill[shaded] (-.1,0) -- (-.1,.3) -- (.5,.3) -- (.5,0) -- (.4,0) arc (0:180:.2cm);
	\draw[] (0,0) arc (180:0:.2cm);
\end{tikzpicture}
:=
(\id_{X^\vee}\otimes \varphi_X^{-1})
\circ
\coev_{X^\vee}
.
$$

%%%%%%%%%%%%%%%%%%%%%%%%%%%%%%%%%%%%%%%%%%%%%%%%%%%%%%%%%%%
\subsubsection{Unitary dual functors for unitary multifusion categories} 
\label{sec:UnitaryPivotal}

In order to discuss the unitary version of Theorem \ref{thm:AlgebraicPAEquivalence}, we rapidly recall the relevant notions for unitary dual functors and unitary pivotal structures from \cite{MR2767048,1808.00323}.
We do so only for \emph{unitary multifusion categories}, which are finitely semisimple multitensor $\Cstar$ categories, which substantially simplifies the presentation.
For $\scrC$ a unitary multifusion category, we have $1_\scrC$ decomposes into an orthogonal direct sum of simples as $\bigoplus_{i=1}^r 1_i$, and we let $p_i\in \scrC(1_\scrC \to 1_\scrC)$ be the minimal projection corresponding to the summand $1_i$ for $i=1,\dots, r$.
In the exposition below, we assume $\scrC$ is \emph{indecomposable}, i.e., $\scrC$ is not equivalent to the direct sum of two non-zero unitary multifusion categories.
We write $\scrC_{ij} = 1_i \otimes \scrC \otimes 1_j$, and we note $\scrC = \bigoplus_{i,j} \scrC_{i,j}$ is a faithful grading of $\scrC$ by the groupoid $\cG_r$ with $r$ objects and a unique isomorphism between any two objects, which can also be viewed as the standard system of matrix units $\{E_{ij}\}$ for $M_r(\bbC)$.

A dual functor $\vee : \scrC \to \scrC$ consists of a choice of dual object $c^\vee$ for each $c\in \scrC$ together with morphisms $\ev_c,\coev_c$ which satisfy the zig-zag axioms.
On morphisms $f\in \scrC(a\to b)$, we define $\vee$ by
$$
f^\vee 
= 
(\ev_b \otimes \id_{a^\vee})\circ (\id_{b^\vee}\otimes f\otimes \id_{a^\vee}) \circ (\id_{b^\vee}\otimes \coev_a).
$$
A dual functor has a canonical anti-tensorator $\nu_{a,b}: a^\vee \otimes b^\vee \to (b\otimes a)^\vee$ built from evaluations and coevaluations.
Any two dual functors are uniquely monoidally naturally isomorphic.

A \emph{pivotal structure} is a pair $(\vee, \varphi)$ consisting of a dual functor $\vee$ and a monoidal natural isomorphism $\varphi : \id_\scrC \Rightarrow \vee\circ \vee$.
If a pivotal structure exists for a multitensor category, the equivalence classes of pivotal structures form a torsor over the group $\Hom(\cU \to \bbC^\times)$ with group law given by pointwise multiplication, where $\cU$ is the universal grading groupoid of $\scrC$ (see \cite[\S4.14]{MR3242743} and \cite[\S3.3]{1808.00323} for more details).

A dual functor is called \emph{unitary} if it is a dagger tensor functor, i.e., $\nu_{a,b}$ is unitary for all $a,b\in \scrC$, and $f^{\vee \dag} = f^{\dag \vee}$ for all $f\in \scrC(a\to b)$.
Each unitary dual functor induces a canonical \emph{unitary pivotal structure} by $\varphi_c := (\coev_c^\dag \otimes \id_{c^{\vee\vee}})\circ (\id_c \otimes \coev_{c^\vee})$, which is unitary.
As in \cite[\S7.3]{MR2767048}, the term `unitary pivotal structure' should be viewed as a synonym for `the canonical unitary pivotal structure induced from a unitary dual functor.'

\begin{remark}
\label{rem:OtherCharacterizationOfUnitaryPivotal}
Equivalently, we can say that a pivotal structure $\varphi_c: c \rightarrow c^{\vee\vee}$ is compatible with the dagger structure if $\coev_c^\dag = \ev_{c^\vee} \circ (\varphi_c \otimes \id_{c^\vee}): c \otimes c^\vee \rightarrow 1_\scrC$, and define a unitary pivotal structure as a pivotal structure which is compatible with the dagger structure.  
It is easy to see that the only compatible pivotal structure is the canonical one.  
Note that this compatibility condition is needed in order for unitary pivotal categories to have the correct diagram calculus where dagger corresponds to reflection of diagrams, since $\coev_c^\dag$ and $\ev_{c^\vee} \circ (\varphi_c \otimes \id_{c^\vee})$ both are represented graphically by the same oriented cap. 
\end{remark}

\begin{remark}
We found the relationship between pivotal structures and unitary pivotal structures very confusing, and so we'd like to pause to explain why it's so confusing.  
In both the algebraic and unitary settings a pivotal structure consists of two parts: a choice of dual functor and a choice of trivialization of the double dual functor subject to a compatibility condition.  
In the algebraic setting, the dual functor is essentially unique (i.e. any two choices are canonically naturally isomorphic) and the compatibility condition is vacuous, so the only interesting part is the trivialization of the double dual functor.  
By contrast, in the unitary setting, once you've chosen a unitary dual functor, the compatibility condition guarantees that there's a unique compatible trivialization of the double dual, so the only interesting part is the choice of unitary dual functor.  
This means even though the two definitions can be made parallel, the {\em interesting parts} of the two definitions are disjoint!
\end{remark}

Note that a unitary pivotal structure $\varphi$ is \emph{pseudounitary}, i.e., all dimensions of simple objects are strictly positive.
Here, the left and right dimensions of a non-simple object $c\in\scrC$ are the matrices in $M_r(\bbC)$ determined by
$$
\Dim_L^\varphi(c)_{ij} \id_{1_j} 
=  
\tr_L^\varphi(p_i \otimes \id_c \otimes p_j)
\qquad\qquad
\Dim_R^\varphi(c)_{ij} \id_{1_i} 
=  
\tr_R^\varphi(p_i \otimes \id_c \otimes p_j).
$$
When $c\in \scrC$ is simple, $\Dim_L^\varphi(c), \Dim_R^\varphi(c)$ have exactly one non-zero entry, which we call $\dim_L^\varphi(c), \dim_R^\varphi(c)$ respectively.

For our indecomposable unitary multifusion category $\scrC$, there exists a canonical spherical structure \cite{MR1444286,MR2091457,MR3342166,1808.00323} which satisfies for all simples $c\in \scrC$, $\dim^\varphi_L(c) = \dim^\varphi_R(c)$. 
By picking this basepoint, we identify the torsor of pivotal structures with the group $\Hom(\cU \to \bbC^\times)$.
Polar decomposition gives us a group isomorphism $\bbC^\times \cong U(1) \times \bbR_{>0}$, which gives us a group isomorphism
$$
\Hom(\cU \to \bbC^\times) \cong \Hom(\cU \to U(1)) \times \Hom(\cU \to \bbR_{>0}).
$$
It follows that the unitary pivotal structures correspond to the subgroup $1 \times \Hom(\cU \to \bbR_{>0})$ as all dimensions must be strictly positive.

Now in the case of a unitary multifusion category, the universal grading groupoid $\cU$ is finite.
If $\cG \subseteq \cU$ is a subgroup (with only one object), then given a $\pi\in \Hom(\cU \to \bbR_{>0})$, we must have $\pi(\cG)=\{1\}$.
Hence for our indecomposable unitary multifusion category $\scrC$ such that $1_\scrC = \bigoplus_{i=1}^r 1_i$ is a decomposition into simples, the relevant grading groupoid to see all unitary pivotal structures is exactly $\cG_r$.

Summarizing, we have:

\begin{thm}
\label{thm:ClassificationOfUnitaryDualFunctors}
Let $\scrC$ be a unitary multifusion category.
There is a bijective correspondence between
\begin{enumerate}
\item
unitary equivalence classes of unitary dual functors and their induced unitary pivotal structures
\item
$\Hom(\cG_r \to \bbR_{>0})$.
\end{enumerate}
\end{thm}

See \cite{1808.00323} for more details.

\begin{remark}
\label{rem:DetermineGroupoidHom}
Notice that a homomorphism $\pi \in \Hom(\cG_r \to \bbR_{>0})$ is uniquely determined by its image on $E_{i+1,i}$ for $1\leq i\leq r-1$.
\end{remark}

Explicitly, starting with a unitary dual functor $\vee$ with its induced unitary pivotal structure $\varphi$, we get our $\pi \in \Hom(\cG_r \to \bbR_{>0})$ by taking the ratio of left to right quantum dimensions of simple objects:
$$
\pi(E_{ij}) 
:=
\frac{\dim^\varphi_L(c)}{\dim^\varphi_R(c)}
\qquad
\qquad
\text{for all simple $c\in \scrC_{ij}$.}
$$
Conversely, we can choose for each $c\in \scrC$ a unique balanced dual $(\overline{c}, \ev_c, \coev_c)$ up to unique isomorphism.
One then obtains all other unitary dual functors from homomorphisms $\pi \in \Hom(\cG_r \to \bbR_{>0})$ by rescaling the evaluations and coevaluations on simple objects $c\in \scrC_{ij}$ by
$$
\ev_c^\pi
:=
\pi(E_{ij})^{1/4} 
\ev_c.
\qquad
\qquad
\coev^\pi_c
:=
\pi(E_{ij})^{-1/4} 
\coev_c.
$$

%%%%%%%%%%%%%%%%%%%%%%%%%%%%%%%%%%%%%%%%%%%%%%%%%%%%%%%%%%%
\subsubsection{Unitary planar algebras} 
\label{sec:UnitaryPAs}

\begin{defn}
A \emph{planar $\dag$-algebra} is a planar algebra equipped with antilinear maps $\dag: \cP_{n,\pm} \to \cP_{n,\pm}$ such that 
\begin{itemize}
\item
$\dag\circ \dag = \id$, and 
\item
for every planar tangle $T$, $T^\dag(x_1^\dag,\dots, x_k^\dag) = T(x_1,\dots, x_k)^\dag$ where $T^\dag$ denotes the reflection of $T$ about \emph{any} axis.
\end{itemize}
A planar $\dag$-algebra is called a $\Cstar$ \emph{planar algebra} if in addition
\begin{itemize}
\item
($\Cstar$)
Every $\dag$-algebra $\cP_{n,\pm}$ with the stacking multiplication is a $\Cstar$ algebra (see Footnote \ref{footnote:CstarAlgebraProperty}), and
\item
(positivity)
for every $x\in \cP_{n,\pm}$ and every $-n\leq k\leq n$, there is a $y\in \cP_{n+ k,\pm}$ such that 
\begin{equation}
\label{eq:CstarPAPositivity}
\begin{tikzpicture}[baseline=-.1cm]
	\draw (0,1.2) -- (0,-1.2);
	\roundNbox{fill=white}{(0,.5)}{.3}{0}{0}{$x^\dag$}
	\roundNbox{fill=white}{(0,-.5)}{.3}{0}{0}{$x$}
	\node at (-.45,.5) {\scriptsize{$\star$}};
	\node at (-.45,-.5) {\scriptsize{$\star$}};
	\node at (.4,0) {\scriptsize{$n+ k$}};
	\node at (.4,1) {\scriptsize{$n- k$}};
	\node at (.4,-1) {\scriptsize{$n+ k$}};
\end{tikzpicture}
=
\begin{tikzpicture}[baseline=-.1cm]
	\draw (0,1.2) -- (0,-1.2);
	\roundNbox{fill=white}{(0,.5)}{.3}{0}{0}{$y^\dag$}
	\roundNbox{fill=white}{(0,-.5)}{.3}{0}{0}{$y$}
	\node at (-.45,.5) {\scriptsize{$\star$}};
	\node at (-.45,-.5) {\scriptsize{$\star$}};
	\node at (.4,0) {\scriptsize{$n+ k$}};
	\node at (.4,1) {\scriptsize{$n+ k$}};
	\node at (.4,-1) {\scriptsize{$n+ k$}};
\end{tikzpicture}\,,
\end{equation}
\end{itemize}
Finally, a \emph{unitary planar algebra} is a semisimple $\Cstar$ planar algebra.\footnote{
Our definition of unitary planar $\dag$- algebra from \cite[Def.~2.3]{MR3402358} was incorrect.
We should have used the above definition, which is satisfied by the planar algebra of a bipartite graph from \cite{MR1865703} (see also Definition \ref{defn:GPA} below).
}
\end{defn}

\begin{lem}
\label{lem:AnotherPositivityAxiom}
Suppose $\cP_\bullet$ is a planar $\dag$-algebra with finite dimensional box spaces, i.e., $\dim(\cP_{n,\pm})<\infty$ for all $n\geq 0$.
Then $\cP_\bullet$ is unitary if and only if there exists a faithful tracial state $\psi_\pm$ on $\cP_{0,\pm}$ such that for every $n\geq 0$, the sesquilinear form
\begin{equation}
\label{eq:PositiveDefiniteInnerProductOnPn}
\langle x,y\rangle_{n,\pm}^\psi := 
\psi_\pm
\left(
\begin{tikzpicture}[baseline=-.1cm]
	\draw (0,.8) arc (180:0:.3cm) -- (.6,-.8) arc (0:-180:.3cm) -- (0,.8);
	\roundNbox{fill=white}{(0,.5)}{.3}{0}{0}{$y^\dag$}
	\roundNbox{fill=white}{(0,-.5)}{.3}{0}{0}{$x$}
	\node at (-.45,.5) {\scriptsize{$\star$}};
	\node at (-.45,-.5) {\scriptsize{$\star$}};
	\node at (-.2,0) {\scriptsize{$n$}};
	\node at (-.2,1) {\scriptsize{$n$}};
	\node at (-.2,-1) {\scriptsize{$n$}};
\end{tikzpicture}
\,\right)
\end{equation}
is a positive definite inner product.
\end{lem}
\begin{proof}
Suppose $\cP_\bullet$ is unitary, so $\cP_{0,\pm}$ is a finite dimensional $\Cstar$ algebra.
Choose a faithful tracial state $\psi_\pm$ on $\cP_{0,\pm}$.
By the positivity axiom, for every $x\in \cP_{n,\pm}$, there is a $y\in \cP_{0,\pm}$ such that
$$
\begin{tikzpicture}[baseline=-.1cm]
	\draw (0,.8) arc (180:0:.3cm) -- (.6,-.8) arc (0:-180:.3cm) -- (0,.8);
	\roundNbox{fill=white}{(0,.5)}{.3}{0}{0}{$x^\dag$}
	\roundNbox{fill=white}{(0,-.5)}{.3}{0}{0}{$x$}
	\node at (-.45,.5) {\scriptsize{$\star$}};
	\node at (-.45,-.5) {\scriptsize{$\star$}};
	\node at (-.2,0) {\scriptsize{$n$}};
	\node at (-.2,1) {\scriptsize{$n$}};
	\node at (-.2,-1) {\scriptsize{$n$}};
\end{tikzpicture}
=
\begin{tikzpicture}[baseline=-.1cm]
	\roundNbox{fill=white}{(0,.5)}{.3}{0}{0}{$y^\dag$}
	\roundNbox{fill=white}{(0,-.5)}{.3}{0}{0}{$y$}
	\node at (-.45,.5) {\scriptsize{$\star$}};
	\node at (-.45,-.5) {\scriptsize{$\star$}};
\end{tikzpicture}\,.
$$
Hence the sesquilinear form \eqref{eq:PositiveDefiniteInnerProductOnPn} above is positive definite by faithfulness of $\psi$.

Conversely, suppose we have $\psi_\pm$ on $\cP_{0,\pm}$ such that \eqref{eq:PositiveDefiniteInnerProductOnPn} is positive definite for all $n\geq 0$.
Then every $\dag$-algebra $\cP_{n,\pm}$ is a unital $\dag$-subalgebra of $B(\cP_{n,\pm})$ with its inner product.
Thus it is is a finite dimensional von Neumann ($\Cstar$) algebra by the finite dimensional bicommutant theorem \cite[Thm.~3.2.1]{JonesVNA}.
Now suppose $x\in \cP_{n,\pm}$, $-n\leq k\leq n$.
We need to find a $y\in \cP_{n+k,\pm}$ such that \eqref{eq:CstarPAPositivity} holds.
For notational simplicity, we denote the left hand side of \eqref{eq:CstarPAPositivity} by $x^\dag \circ_k x$,
and we denote the stacking multiplication in $\cP_{n+k,\pm}$ by $\cdot$.
Notice that for all $z\in \cP_{n+ k, \pm}$,
$$
\langle
(x^\dag \circ_k x) \cdot z
,
z
\rangle^\psi
=
\left\langle
\begin{tikzpicture}[baseline=-.1cm]
	\draw (0,1.7) -- (0,-1.7);
	\roundNbox{fill=white}{(0,1)}{.3}{0}{0}{$x^\dag$}
	\roundNbox{fill=white}{(0,0)}{.3}{0}{0}{$x$}
	\roundNbox{fill=white}{(0,-1)}{.3}{0}{0}{$z$}
	\node at (-.45,1) {\scriptsize{$\star$}};
	\node at (-.45,0) {\scriptsize{$\star$}};
	\node at (-.45,-1) {\scriptsize{$\star$}};
	\node at (.4,1.5) {\scriptsize{$n+ k$}};
	\node at (.4,.5) {\scriptsize{$n- k$}};
	\node at (.4,-.5) {\scriptsize{$n+ k$}};
	\node at (.4,-1.5) {\scriptsize{$n+ k$}};
\end{tikzpicture}
\,,\,
\begin{tikzpicture}[baseline=-.1cm]
	\draw (0,-.7) -- (0,.7);
	\roundNbox{fill=white}{(0,0)}{.3}{0}{0}{$z$}
	\node at (-.45,0) {\scriptsize{$\star$}};
	\node at (.4,.5) {\scriptsize{$n+ k$}};
	\node at (.4,-.5) {\scriptsize{$n+ k$}};
\end{tikzpicture}
\right\rangle\raisebox{.6cm}{\scriptsize{$\psi$}}
=
\left\| 
\begin{tikzpicture}[baseline=-.1cm]
	\draw (0,1.2) -- (0,-1.2);
	\roundNbox{fill=white}{(0,.5)}{.3}{0}{0}{$x$}
	\roundNbox{fill=white}{(0,-.5)}{.3}{0}{0}{$z$}
	\node at (-.45,.5) {\scriptsize{$\star$}};
	\node at (-.45,-.5) {\scriptsize{$\star$}};
	\node at (.4,0) {\scriptsize{$n+ k$}};
	\node at (.4,1) {\scriptsize{$n- k$}};
	\node at (.4,-1) {\scriptsize{$n+ k$}};
\end{tikzpicture}
\right\|^2_\psi
\geq 0,
$$
which shows $x^\dag \circ_k x \geq 0$ in $B(\cP_{n+k,\pm})$.
Now $x^\dag \circ_k x$ commutes with the right $\cP_{n+k,\pm}$-action on $\cP_{n+k,\pm}$ 
$$
((x^\dag \circ_k x) \cdot z) \cdot w
=
(x^\dag \circ_k x) \cdot (z \cdot w)
\qquad\qquad\qquad
\forall w\in \cP_{n+k,\pm},
$$
so $x^\dag \circ_k x$ is a left multiplication operator by a positive operator in the $\Cstar$-algebra $\cP_{n+k,\pm}$.
Thus a $y\in \cP_{n+k,\pm}$ such that $x^\dag \circ_k x = y^\dag \cdot y$ exists.
\end{proof}

\begin{defn}
A \emph{subfactor planar algebra} is a 2-shaded planar $\dag$-algebra satisfying the following axioms:
\begin{itemize}
\item
(connected) $\cP_{0,\pm} \cong \bbC$ via the map which sends the empty diagram to $1_\bbC$,
\item
(finite dimensional) $\dim(\cP_{n,\pm})<\infty$ for all $n\geq 0$.
\item
(positive)
For every $n\geq 0$, the sesquilinear form on $\cP_{n,\pm}$ given by
$$
\langle x,y\rangle_n := 
\begin{tikzpicture}[baseline=-.1cm]
	\draw (0,.8) arc (180:0:.3cm) -- (.6,-.8) arc (0:-180:.3cm) -- (0,.8);
	\roundNbox{fill=white}{(0,.5)}{.3}{0}{0}{$y^\dag$}
	\roundNbox{fill=white}{(0,-.5)}{.3}{0}{0}{$x$}
	\node at (-.45,.5) {\scriptsize{$\star$}};
	\node at (-.45,-.5) {\scriptsize{$\star$}};
	\node at (-.2,0) {\scriptsize{$n$}};
	\node at (-.2,1) {\scriptsize{$n$}};
	\node at (-.2,-1) {\scriptsize{$n$}};
\end{tikzpicture}
$$
is a positive definite inner product, and
\item
(spherical)
For every $x\in \cP_{1,\pm}$, 
$
\begin{tikzpicture}[baseline=-.1cm]
	\filldraw[shaded] (0,.3) arc (180:0:.3cm) -- (.6,-.3) arc (0:-180:.3cm);
	\roundNbox{fill=white}{(0,0)}{.3}{0}{0}{$x$}
	\node at (-.45,0) {\scriptsize{$\star$}};
\end{tikzpicture}
=
\begin{tikzpicture}[baseline=-.1cm, xscale=-1]
	\fill[shaded, rounded corners = 5pt] (-.2,-.8) rectangle (.8,.8);
	\filldraw[fill=white] (0,.3) arc (180:0:.3cm) -- (.6,-.3) arc (0:-180:.3cm);
	\roundNbox{fill=white}{(0,0)}{.3}{0}{0}{$x$}
	\node at (.45,0) {\scriptsize{$\star$}};
\end{tikzpicture}
$.
\end{itemize}
By Lemma \ref{lem:AnotherPositivityAxiom}, a subfactor planar algebra is a $\Cstar$ planar algebra.
\end{defn}

The following result, which appears in \cite[\S4]{1808.00323}, is the unitary analog of Theorem \ref{thm:AlgebraicPAEquivalence}, which uses unitary dual functors instead of a pivotal structure.

\begin{cor}
\label{cor:PAEquivalence}
There is an equivalence of categories (see Footnote \ref{footnote:TruncationTriples})
\[
\left\{\, 
\parbox{4cm}{\rm Semisimple shaded $\Cstar$ planar algebras $\cP_\bullet$}\,\left\}
\,\,\,\,\cong\,\,
\left\{\,\parbox{9.5cm}{\rm Triples $(\scrC, \vee, X)$ with $\scrC$ a unitary multitensor category, $\vee$ a unitary dual functor, and a generator $X\in \scrC$ with an orthogonal decomposition $1_\scrC = 1_+ \oplus 1_-$ such that $X = 1_+\otimes X \otimes 1_-$}\,\right\}.
\right.\right.
\]
Moreover, under this equivalence,
\begin{itemize}
\item
finite depth planar algebras correspond to triples $(\scrC, \vee, X)$ where $\scrC$ is unitary multifusion, and
\item
subfactor planar algebras correspond to triples $(\scrC, \vee, X)$ where $1_\pm$ are simple and $\vee$ is the canonical spherical dual functor.
\end{itemize}
\end{cor}

\begin{remark}
In the unitary setting, Corollary \ref{cor:SemisimpleCstarMonoidalAlgebraEquivalence} gives us an equivalence between the underlying 2-shaded unitary monoidal algebras and unitary multitensor categories.
Starting with a semisimple 2-shaded $\Cstar$ planar algebra $\cP_\bullet$, we get a unitary dual functor on the projection category $\scrC$ by taking the $\pi$-rotation in $\cP_\bullet$.
Conversely, given a tuple $(\scrC, \vee, X)$, by by Remark \ref{rem:OtherCharacterizationOfUnitaryPivotal},
$\coev^\dag_c=\ev_{X^\vee} \circ (\varphi_X \otimes \id_{X^\vee})$
and 
$\ev_c^\dag = (\id_{X^\vee}\otimes \varphi_X^{-1})\circ\coev_{X^\vee}$.
This means the cups and caps be alternately described by
$$
\begin{tikzpicture}[baseline = 0cm]
	\fill[shaded] (-.1,0) -- (-.1,.3) -- (.5,.3) -- (.5,0) -- (.4,0) arc (0:180:.2cm);
	\draw[] (0,0) arc (180:0:.2cm);
\end{tikzpicture}
:=\ev_{X}
\qquad
\qquad
\begin{tikzpicture}[baseline = -.2cm, yscale = -1]
	\filldraw[shaded] (0,0) arc (180:0:.2cm);
\end{tikzpicture}
:=\coev_{X}
\qquad
\qquad
\begin{tikzpicture}[baseline = 0cm]
	\filldraw[shaded] (0,0) arc (180:0:.2cm);
\end{tikzpicture}
:=\coev_{X}^\dag
\qquad
\qquad
\begin{tikzpicture}[baseline = -.2cm, yscale=-1]
	\fill[shaded] (-.1,0) -- (-.1,.3) -- (.5,.3) -- (.5,0) -- (.4,0) arc (0:180:.2cm);
	\draw[] (0,0) arc (180:0:.2cm);
\end{tikzpicture}
:=
\ev_{X}^\dag.
$$
\end{remark}

Now in order for a semisimple shaded $\Cstar$ planar algebra $\cP_\bullet$ to have \emph{scalar} loop modulus, 
we choose the \emph{standard} unitary dual functor $\vee_{\text{standard}}$ on $\scrC$ \emph{with respect to} $X$ following \cite{1805.09234}, which is clarified in \cite{1808.00323}.
First, define $n_\pm :=\dim(\End_\scrC(1_\pm))$, and denote the summands of $1_+$ and $1_-$ by $V_+$ and $V_-$ respectively. 
Let $D_X$ be the $n_+\times n_-$ matrix whose $uv$-th entry is $\dim(u \otimes X \otimes v)$, using the canonical spherical structure.
Let $d_X>0$ such that $d_X^2 =\|D_XD_X^T\|=\|D_X^TD_X\|$, 
and let $\mu$ and $\nu$ be the Frobenius-Perron eigenvectors of $D_XD_X^T$ and $D_X^TD_X$ respectively normalized so that
$$
\sum_{u \in V_+} \mu(u)^2 = 1 = \sum_{v\in V_-} \nu(v)^2.
$$
We denote by $\lambda$ the vector in $\bbR_{>0}^{n_++n_-}$ obtained by concatenating $\mu$ and $\nu$.

\begin{defn}
The \emph{standard} unitary dual functor with respect to $X$ 
corresponds to the standard groupoid homormorphism $\cG_r \to \bbR_{>0}$ given by
\begin{equation}
\label{eq:StandardAndLopsided1808.00323}
\pi^{\text{standard}}(E_{u,v})
:=
\left(\frac{\lambda(u)}{\lambda(v)}\right)^2
\end{equation}
under Theorem \ref{thm:ClassificationOfUnitaryDualFunctors}.
It is straightforward to verify that the shaded planar algebra corresponding to $(\scrC, \vee_{\text{standard}}, X)$ under Corollary \ref{cor:PAEquivalence} has scalar loop moduli given by
\begin{equation}
\label{eq:StandardPA}
\begin{tikzpicture}[baseline=-.1cm]
	\draw[shaded] (0,0) circle (.3cm);
\end{tikzpicture}
=
d_X
\id_{\cP_{0,+}}
\qquad\qquad\qquad
\begin{tikzpicture}[baseline=-.1cm]
	\fill[shaded, rounded corners=5] (-.6,-.6) rectangle (.6,.6);
	\draw[fill=white] (0,0) circle (.3cm);
\end{tikzpicture}
=
d_X 
\id_{\cP_{0,-}}.
\end{equation}
\end{defn}

%%%%%%%%%%%%%%%%%%%%%%%%%%%%%%%%%%%%%%%%%%%%%%%%%%%%%%%%%%%
%%%%%%%%%%%%%%%%%%%%%%%%%%%%%%%%%%%%%%%%%%%%%%%%%%%%%%%%%%%
%%%%%%%%%%%%%%%%%%%%%%%%%%%%%%%%%%%%%%%%%%%%%%%%%%%%%%%%%%%
\subsection{The graph planar algebra module embedding theorem}
\label{sec:GPAModuleEmbedding}

In this section, we finally prove the unitary pivotal module embedding theorem.
We begin by defining the notion of a trace on a semisimple category in \S\ref{sec:Traces}, and then discussing Schaumann's notion of a pivotal module for a pivotal category from \cite{MR3019263} in \S\ref{sec:PivotalModules}.
As both of these concepts have unitary versions, we treat both the algebraic and unitary setting in parallel;
the reader should include the parenthetical statements for the unitary setting, and may omit these statements in the non-unitary setting.
Finally, in \S\ref{sec:ModuleEmbedding}, we see how our Main Theorem \ref{thm:ModuleEmbedding} in this section is a natural generalization of the embedding theorems from \S\ref{sec:MonoidalAlgebraEmbedding}.

%%%%%%%%%%%%%%%%%%%%%%%%%%%%%%%%%%%%%%%%%%%%%%%%%%%%%%%%%%%
\subsubsection{Traces on semisimple categories }
\label{sec:Traces}

In this section we now discuss (unitary) traces on finitely semisimple ($\Cstar$) categories.
Throughout we denote the semisimple category with trace by $\cM$ because in our applications we will be looking at traces on module categories, but nothing in this section uses a module structure.

\begin{defn}
A \emph{trace} on a semisimple category $\cM$ is a family of linear functionals $\Tr_m: \End_\cM(m) \rightarrow \mathbb{C}$ for $m\in\cM$ such that the bilinear forms $\Hom_\cM(m,n) \times \Hom_\cM(n,m) \rightarrow \mathbb{C}$ via $(f,g) \mapsto \Tr_m(g\circ f)$ are non-degenerate ($\Tr_m(g\circ f)=0$ for all $g\in \Hom(n,m)$ implies $f=0$) and satisfy $\Tr_m(g\circ f) = \Tr_n(f\circ g)$.

When $\cM$ is a semisimple ${\rm C^*}$ category, we call a trace \emph{unitary} if in addition for every $m,n\in \cM$, the sesquilinear form $\langle f,g\rangle := \Tr_m(g^\dag\circ f)$ on $\Hom_\cM(m,n)$ is a positive definite inner product.
\end{defn}

\begin{remark}
We do not require $\Tr_m^\cM$ to be normalized; that is $\Tr^\cM_m(\id_m)$ is typically not $1$.  
Instead we think of $\Tr^\cM_m(\id_m)$ as specifiying a notion of the dimension of $m\in\cM$.

For example, any trace on the $n\times n$ matrices $M_{n}(\bbC)$ is a scalar multiples of the standard matrix trace.  
A trace on $\Vec$ is a collection of traces on $M_n(\bbC)$ for all $n$; 
however the condition $\Tr^\cM(f \circ g) = \Tr^\cM(g \circ f)$ applied to maps between vector spaces of different dimensions restricts the normalizations of the different traces. 
In particular, the standard trace ($\Tr_V(\id_V) = \dim V$) on each $\End(V)$ gives a trace on $\Vec$, but the normalized trace ($\tr_V(\id_V) = 1$) on each $\End(V)$ does not give a trace on $\Vec$.
\end{remark}

For the remainder of this section, we simultaneously develop the theory of traces on semisimple categories and on $\Cstar$ categories; the extra adjectives and conditions required in the latter case appear parenthetically.
We denote by $G$ the multiplicative group $\bbC^\times$ in the algebraic setting or $\bbR_{>0}$ in the unitary setting.

\begin{nota}
\label{nota:DecomposeMObjects}
Suppose $\cM$ is finitely semisimple and $\Irr(\cM):=\{x_1, \ldots, x_r\}$ is a choice of representatives of the isomorphism classes of simple objects in $\cM$.
Let $\cE(\cM)$ denote $\End(\cM)$ in the algebraic setting, and $\End^\dag(\cM)$ in the unitary setting. 
If $\cN$ is a ($\Cstar$) category and $y_1, \ldots, y_r$ are objects in $\cN$, then there is a (dagger) functor $\cF_{y_1, \ldots, y_r}: \cM \rightarrow \cN$ that is unique up to unique (unitary) isomorphism such that $\cF(x_i) = y_i$. 
Furthermore, any (dagger) functor out of $\cM$ is of this form.
In particular, we let $E_{ij} \in \cE(\cM)$ denote the (dagger) functor which sends $x_i$ to $x_j$ and sends $x_k$ to the zero object for all $k \neq i$. 
Then $\left\{E_{ij}\middle| 1\leq i,j\leq r\right\}$ is a choice of representatives of isomorphism classes of simple objects in $\cE(\cM)$.  
Thus 
in the algebraic setting,
$\cE(\cM)$ is equivalent to the category of $\cG_r$-graded vector spaces $\Vec[\cG_r]$,
and in the unitary setting, $\cE(\cM)$ is dagger equivalent to $\Hilb[\cG_r]$.
In either case, the universal grading groupoid of $\cE(\cM)$ is $\cG_r$.
\end{nota}

\begin{lem}
\label{lem:TracesOnVec}
Let $\cV$ be $\Vec$ (respectively $\Hilb$).
The function from (unitary) equivalence classes of traces on $\cV$ to $G$ given by $\Tr^\cV \longmapsto \Tr^\cV_{\mathbb{C}}(\id_{\mathbb{C}})$ is a bijection.
\end{lem}
\begin{proof}
For surjectivity, we note that if $\lambda \in G$,
then $(\lambda \Tr)(V):= \lambda \dim(V)$ is a trace on $\cV$ which satisfies $(\lambda \Tr)_\mathbb{C}(\id_\mathbb{C}) = \lambda$.

For injectivity, we prove that $\Tr^\cV$ is determined by $\Tr^\cV_{\mathbb{C}}(\id_{\mathbb{C}})$.  
Let $V \in \cV$ and choose a(n orthormal) basis $v_1, \ldots v_n$ for $V$.  
Let $\pi_j: V \rightarrow \mathbb{C}$ be the projection $\sum_i a_i v_i \mapsto a_j$, and let $\iota_k: \mathbb{C} \rightarrow V$ be the inclusion $\lambda \mapsto \lambda v_j$.  
The composites $\iota_k \pi_j$ span $\End(V)$, and
$$
\Tr^\cV_V(\iota_k \pi_j) = \Tr^\cV_\mathbb{C}(\pi_j \iota_k) = \delta_{j=k} \Tr_\mathbb{C}^\cV(\id_\bbC).
$$ 
Hence $\Tr^\cV$ is completely determined by $\Tr^\cV_{\mathbb{C}}(\id_{\mathbb{C}})$, which proves injectivity.
\end{proof}

\begin{prop}
\label{prop:CharacterizationOfTraces}
The function from (unitary) equivalence classes of traces on $\cM$ to $G^r$ given by 
$$
\Tr^\cM \longmapsto (\Tr^\cM_{x_1}(\id_{x_1}), \dots, \Tr^\cM_{x_r}(\id_{x_r})) 
$$
is a bijection.
\end{prop}
\begin{proof}
If $\cM$ has $r$ distinct isomorphism classes of simple objects, $\cM$ is (dagger) equivalent to a(n orthogonal) direct sum $\bigoplus_{i=1}^r \Vec$ (respectively $\bigoplus_{i=1}^r \Hilb$).  
Since there are no maps between objects in the different summands, a (unitary) trace on $\bigoplus_{i=1}^r \Vec$ 
(respectively $\bigoplus_{i=1}^r \Hilb$)
is equivalent to independently giving a (unitary) trace on each of the $r$ copies of $\Vec$ (respectively $\Hilb$).
The result now follows from Lemma \ref{lem:TracesOnVec}.
\end{proof}

\begin{remark}
\label{rem:OtherCharacterizationsOfTraces}
Similar to \cite{MR3019263}, in the non-unitary setting, traces on $\cM$ are in bijection with families of natural isomorphisms $\cM(m\to n) \cong \cM(n\to m)^*$ for all $m,n\in\cM$.

Unitary traces on $\cM$ are in bijection with  \emph{2-Hilbert space} structures on $\cM$ \cite{MR1448713} (see also  \cite[\S3 and 5.6]{0901.3975}), i.e., for every $m,n\in \cM$, a Hilbert space structure on $\cM(m\to n)$ such that for all $f\in \cM(m \to n)$, $g\in \cM(n\to p)$, and $h\in \cM(m\to p)$,
\begin{equation}
\label{eq:2HilbertSpace}
\langle g\circ f, h\rangle_{\cM(m \to p)}
=
\langle f, g^\dag\circ h\rangle_{\cM(m \to n)}
=
\langle g, h\circ f^\dag\rangle_{\cM(n \to p)}.
\footnote{We note that the second equality in \eqref{eq:2HilbertSpace} holds if and only if for each $m\in \cM$, the linear functor $\cM(- \to m) : \cM^{\rm op}\to \Hilb$ is a dagger functor.
In this case, the first equality in \eqref{eq:2HilbertSpace} holds if and only if the Yoneda embedding $m\mapsto \cM(- \to m)$ is a (fully faithful) dagger functor $\cM \hookrightarrow \Fun^\dag(\cM^{\rm op} \to \Hilb)$.}
\end{equation}
\end{remark}

\begin{prop}
\label{prop:CharacterizePivotalStructures}
The function from (unitary) pivotal structures on $\cE(\cM)$ to $G^{r-1}$ given by
$$
\varphi
\longmapsto
\left(
\dim_L^\varphi(E_{i+1,i})
\right)_{i=1}^{r-1}
$$
is a bijection.
\end{prop}
\begin{proof}
There exists a canonical (unitary) spherical structure on $\cE(\cM)$ where all objects have left and right dimension 1.
Thus the (unitary) pivotal structures on $\cE(\cM)$ form a torsor over $\Hom(\cG_r\to G)$.
Such a homomorphism is uniquely determined by its image on $E_{i+1,i}$ for $1\leq i\leq r-1$ as in Remark \ref{rem:DetermineGroupoidHom}.
\end{proof}

Given a pivotal structure $\varphi$ on $\cE(\cM)$, the left pivotal trace $\tr^\varphi_L$ takes values in $\cE(\cM)(1_{\cE(\cM)} \to 1_{\cE(\cM)}) \cong \bbC^r$.
Choosing a simple object $x_i$ induces a $\bbC$-valued trace on $\cE(\cM)$ by projecting to the $x_i$-component of $1_{\cE(\cM)} := \id_\cM$.
That is, 
if $F\in \cE(\cM)$ and $\eta : F \Rightarrow F$ is a natural transformation, we define
$\Tr^{\cE(\cM), x_i}_F(\eta)$ by the formula
\begin{equation}
\label{eq:CValuedTraceOnE(M)}
\Tr^{\cE(\cM), x_i}_F(\eta)
\cdot
\id_{x_i}
:=
\left(
\begin{tikzpicture}[baseline=-.1cm]
	\draw (0,.9) arc (0:180:.3cm) -- (-.6,-.9) arc (-180:0:.3cm) -- (0,.9);
	\roundNbox{fill=white}{(0,.55)}{.35}{0}{0}{$\eta$}
	\roundNbox{fill=white}{(0,-.55)}{.35}{0}{0}{$\varphi_F^{-1}$}
	\node at (-.9,0) {\scriptsize{$F^\vee$}};
	\node at (.3,0) {\scriptsize{$F$}};
	\node at (.3,1.1) {\scriptsize{$F$}};
	\node at (.3,-1.1) {\scriptsize{$F^{\vee\vee}$}};
\end{tikzpicture}
\right)_{x_i}
=
\tr^\varphi_L(\eta)_{x_i}.
\end{equation}

We define the $j$-th column functor $F_j:\cM \rightarrow \cE(\cM)$ by letting $F_j(m)$ be the (dagger) functor which sends $x_j$ to $m$ and all other simples to the zero object.
We denote by $\cE(\cM)_j$ the essential image of $\cM$ under $F_j$, which consists of (orthogonal) directs sums of the objects $E_{ij}$ for $i=1,\dots, r$.
Notice that $F_i$ is a (dagger) equivalence $\cM \cong \cE(\cM)_j$.
This is the categorical analogue of identifying a vector space with matrices supported on the $j$-th column.

We now choose the simple object $x_1\in \cM$ giving us our scalar-valued trace $\Tr^{\cE(\cM), x_1}$ on $\cE(\cM)$.
By restriction, we get a (unitary) trace on $\cE(\cM)_1\cong \cM$, which we denote by $\Tr^{\cE(\cM)_1}$.
Notice that taking the $x_1$-component of $\tr_L^\varphi$ can be viewed as cutting down $\cE(\cM)(\id \Rightarrow \id)$ by the (orthogonal) projection onto the summand $E_{11} \subset \id_{\cE(\cM)}$. 
Denoting this projection by a shading, we get the following diagrammatic formula
for $\Tr^{\cE(\cM)_1}_{E_{i1}}(\eta)$ for $\eta: E_{i1}\Rightarrow E_{i1}$:
\begin{equation}
\label{eq:CValuedTraceOnE(M)_1}
\Tr^{\cE(\cM)_1}_{E_{i1}}(\eta)
:=
\begin{tikzpicture}[baseline=-.1cm]
	\fill[shaded, rounded corners = 5pt] (-1.3,-1.5) rectangle (.7,1.5); 
	\filldraw[fill=white] (0,.9) arc (0:180:.3cm) -- (-.6,-.9) arc (-180:0:.3cm) -- (0,.9);
	\roundNbox{fill=white}{(0,.55)}{.35}{0}{0}{$\eta$}
	\roundNbox{fill=white}{(0,-.55)}{.35}{0}{0}{$\varphi^{-1}$}
	\node at (-.9,0) {\scriptsize{$E_{i,1}^\vee$}};
	\node at (.3,0) {\scriptsize{$E_{i,1}$}};
	\node at (.3,1.1) {\scriptsize{$E_{i,1}$}};
	\node at (.3,-1.1) {\scriptsize{$E_{i,1}^{\vee\vee}$}};
\end{tikzpicture}
\in \cE(E_{1,1} \to E_{1,1})
\cong\bbC
\qquad\qquad\qquad
\begin{tikzpicture}[baseline=-.1cm]
	\fill[shaded, rounded corners = 5pt] (-.3,-.3) rectangle (.3,.3); 
\end{tikzpicture}
:=
\operatorname{proj}_{E_{1,1}}.
\end{equation}

We have thus proved:

\begin{prop}
\label{prop:PivotalStructuresToTraces}
The function $\varphi \mapsto \tr_L^\varphi|_{\cE(\cM)_1}$ together with the (dagger) equivalence $\cE(\cM)_1 \cong \cM$, induces a function $\Delta$ from the set of (unitary) equivalence classes of pivotal structures on $\cE(\cM)$ to the set of (unitary) equivalence classes of (unitary) traces on $\cM$.
\end{prop}

We now construct a left inverse to the function $\Delta$.

\begin{prop}
\label{prop:TracesToPivotalStructures}
The function $\Lambda$ defined by
$$
\left\{
\text{\rm 
Traces $\Tr^\cM$
}
\right\}
\underset{\text{\rm Prop.~\ref{prop:CharacterizationOfTraces}}}{\cong}
G^r
\ni
(a_1,\dots, a_r)
\mapsto
\left(
\frac{a_2}{a_1}, \dots, \frac{a_{r}}{a_{r-1}}
\right) 
\in 
G^{r-1}
\underset{\text{\rm Prop.~\ref{prop:CharacterizePivotalStructures}}}{\cong}
\left\{
\text{\rm 
Pivotal structures $\varphi^\cE$
}
\right\}
$$
is surjective and provides a left inverse to $\Delta$ from Proposition \ref{prop:PivotalStructuresToTraces}.
Moreover, under $\Lambda$, two (unitary) traces map to the same (unitary) pivotal structure if and only if they are proportional.
\end{prop}
\begin{proof}
Surjectivity of $\Lambda$ is obvious.
Notice that $a_{i+1}/a_i = b_{i+1}/b_i$ for all $1\leq i\leq r-1$
if and only if
$a_i/b_i = a_{i+1}/b_{i+1}$ for $1\leq i\leq r-1$ 
if and only if 
$(a_1,\dots, a_r)$ is proportional to $(b_1,\dots, b_r)$.

Finally, we show $\Lambda \circ \Delta = \id$.
Let $\Tr^\cM$ be the (unitary) trace on $\cM$ corresponding to $(a_1,\dots, a_r)\in G^r$ under Proposition \ref{prop:CharacterizationOfTraces}, 
and let $\varphi$ be the corresponding (unitary) pivotal structure on $\cE$ corresponding to $(a_2/a_1,\dots, a_r/a_{r-1})$.
It suffices to prove that $\tr^\varphi|_{\cN}$ is proportional to $\Tr^\cM$ under the equivalence $\cN \cong \cM$, since proportional traces give rise to (unitarily) equivalent pivotal structures under $\Lambda$.
Indeed, for a fixed $1\leq j\leq r$, by monoidality of $\varphi$, we have
$$
\tr_L^\varphi(\id_{E_{j,1}})
= 
\dim_L^\cE(E_{j,1})
=
\prod_{i=1}^{j-1}
\dim_L^\varphi(\id_{E_{i+1,i}})
=
\prod_{i=1}^{j-1} \frac{a_{i+1}}{a_{i}}
=
\frac{a_j}{a_1}
=
\frac{1}{a_1} \Tr^\cM_{x_j}(\id_{x_j})
$$
as in the proof of Proposition \ref{prop:CharacterizationOfTraces}.
Hence $\tr_L^\varphi = a_1^{-1} \Tr^\cM$ under the (dagger) equivalence $\cN\cong\cM$.
\end{proof}

In particular, if we change our choice of simple object $x_1$ this only rescales the trace on $\cM$.

%%%%%%%%%%%%%%%%%%%%%%%%%%%%%%%%%%%%%%%%%%%%%%%%%%%%%%%%%%
\subsubsection{Pivotal module categories for pivotal categories}
\label{sec:PivotalModules}

We now expand on the previous section to the scenario where $\cM$ is equipped with the structure of a $\scrC$-module ($\Cstar$) category, where $\scrC$ is a (unitary) multitensor category.
Some other interesting results related to the non-unitary multifusion case were recently obtained in \cite[\S2.6]{1805.09395}.

\begin{defn}[\cite{MR3019263}]
If $(\scrC,\varphi)$ is a semisimple (unitary) pivotal multifusion category and $\cM$ is a semisimple left $\scrC$-module ($\Cstar$) category with a (unitary) trace $\Tr^\cM$, then $(\cM,\Tr^\cM)$ is called a \emph{pivotal} $\scrC$-module  ($\Cstar$) category if we have the following compatibility of $\Tr^\cM$ with the left partial trace in $\scrC$:
for all $c\in\scrC$, $m\in\cM$, and $f\in \cM(c\vartriangleright m \to c\vartriangleright m)$,
\begin{equation}
\label{eq:TraceCompatibility}
\Tr^\cM_{c \vartriangleright m}(f) 
= 
\Tr_m^\cM[ (\ev_c \vartriangleright \id_m) \circ f \circ ((\varphi_c)^{-1}\vartriangleright \id_m)\circ (\coev_{c^\vee}\vartriangleright \id_m)]
=
\Tr^\cM_m
\left(
\begin{tikzpicture}[baseline=-.1cm]
	\draw (0,.95) arc (0:180:.3cm) -- (-.6,-.95) arc (-180:0:.3cm) -- (0,.95);
	\draw (.55,-1.4) -- (.55,1.4);
	\roundNbox{fill=white}{(0,-.6)}{.35}{0}{0}{$\varphi_c^{-1}$}
	\roundNbox{fill=white}{(0,.6)}{.35}{0}{.6}{$f$}
	\fill[fill=white] (.75,0) rectangle (1,1);
	\node at (.75,1.2) {\scriptsize{$m$}};
	\node at (.75,0) {\scriptsize{$m$}};
	\node at (.2,0) {\scriptsize{$c$}};
	\node at (.2,1.2) {\scriptsize{$c$}};
	\node at (-.8,0) {\scriptsize{$c^\vee$}};
	\node at (.2,-1.2) {\scriptsize{$c^{\vee\vee}$}};
\end{tikzpicture}
\right).
\end{equation}
Here, we use the diagrammatic convention of \cite{MR3342166} for left $\scrC$-module categories, where the coupons in $\cM$ are drawn cut open on the right hand side to indicate the absence of any right action.
\end{defn}

\begin{remark}
In \cite[\S4.1]{MR3019263}, it is shown that when $(\scrC,\varphi)$ is pivotal fusion and $\cM$ is indecomposable, traces on $\cM$ which satisfy \eqref{eq:TraceCompatibility} are unique up to scaling.
\end{remark}

\begin{remark}
In fact, pivotal structures on the $2$-shaded multifusion category built from $\cC$, $\cM$, and its dual category correspond exactly to module traces on  $\cM$ \emph{not up to rescaling}.  That is rescaling the choice of trace changes the pivotal structure on the odd part of the $2$-shaded multifusion category, but in the even parts this rescaling cancels out.
\end{remark}

\begin{defn}
Given a tensor functor between pivotal categories $(\Psi, \mu) : (\scrC,\varphi^\scrC) \to (\cD, \varphi^\cD)$, where our convention for the tensorator natural isomorphism is $\mu_{a,b} : \Psi(a)\otimes \Psi(b) \to \Psi(a\otimes b)$,
we get a canonical anti-monoidal natural isomorphism $\delta_c : \Psi(c^\vee) \to \Psi(c)^\vee$ given by
\begin{equation}
\label{eq:CanonicalDualIso}
\delta_c := 
([\Psi(\ev_c) \circ \mu_{c^\vee, c}] \otimes \id_{\Psi(c)^\vee})
\circ
(\id_{\Phi(c^\vee)} \otimes \coev_{\Psi(c)}).
\end{equation}
We call $(\Psi, \mu)$ \emph{pivotal} if $\delta_c^\vee\circ \varphi_{\Psi(c)} = \delta_{c^\vee} \circ \Psi(\varphi_c)$ for all $c\in \scrC$.
\end{defn}

\begin{thm}
\label{thm:PivotalModulesAndPivotalFunctors}
Suppose $\cM$ is a finitely semisimple left $\scrC$-module ($\Cstar$) category, and let $(\Psi,\mu) : \scrC \to \cE(\cM)$ be the corresponding (dagger) tensor functor from Lemma \ref{lem:DaggerModules}.
The following are equivalent for a (unitary) trace $\Tr^\cM$ on $\cM$ and its induced (unitary) pivotal structure $\varphi$ on $\cE(\cM)$ from Proposition \ref{prop:TracesToPivotalStructures}.
\begin{enumerate}
\item
Compatibility condition \eqref{eq:TraceCompatibility} holds.
\item
The corresponding (dagger) tensor functor $(\Psi,\mu)$ is pivotal.
\end{enumerate}
\end{thm}

(Note that (2) implies (1) is relatively straightforward since one can use the graphical calculi for pivotal categories and module categories with trace.  But for (2) implies (1) since we do not know that the functor is pivotal we cannot use a standard graphical calculus and need to keep track of all of the structure maps.  This explains why the formulas in the following proof have a lot of explicit structure maps .)

\begin{proof}
As in the discussion right before Proposition \ref{prop:PivotalStructuresToTraces}, there is a (dagger) equivalence 
$$
F_1 : \cM 
\xrightarrow{\sim} 
\cE(\cM)_1 := \operatorname{span}\left\{E_{i,1}\middle| 1\leq i\leq r\right\}\subset \cE(\cM)
\qquad
\qquad
x_j \mapsto E_{j,1}.
$$
In fact, using the tensorator of $\Psi$, we can equip the (dagger) equivalence $F_1$ with a (unitary) \emph{modulator} 
$$
\nu_{a,b, m} : 
\Psi(a) \otimes F_1(b\vartriangleright x_j) 
=
\Psi(a) \otimes (\Psi(b)\otimes E_{j,1})
=
(\Psi(a) \otimes \Psi(b))\otimes E_{j,1}
\cong
\Psi(a\otimes b)\otimes E_{j,1}
\cong
F_1(a\otimes b \vartriangleright x_j),
$$
extending it to a (dagger) equivalence of $\scrC$-module ($\Cstar$) categories.
As in the proof of Proposition \ref{prop:TracesToPivotalStructures}, there is a non-zero scalar $\alpha \in G$ such that for every $f\in \cM(m \to m)$, $\Tr^\cM(f) = \alpha \tr^\varphi_L(F(f))$.
Thus for $1\leq j\leq r$ and 
$f\in \cM(c\vartriangleright x_j \to c\vartriangleright x_j) \cong \cE(\cM)(\Psi(c)\otimes E_{j,1} \to \Psi(c)\otimes E_{j,1})$, 
we always have
\begin{align}
\label{eq:TraceCompatibilityInE}
\alpha^{-1}\Tr^\cM_{x_j}
\left(
\begin{tikzpicture}[baseline=-.1cm]
	\draw (0,.95) arc (0:180:.3cm) -- (-.6,-.95) arc (-180:0:.3cm) -- (0,.95);
	\draw (.55,-1.4) -- (.55,1.4);
	\roundNbox{fill=white}{(0,-.6)}{.35}{0}{0}{$\varphi_c^{-1}$}
	\roundNbox{fill=white}{(0,.6)}{.35}{0}{.6}{$f$}
	\fill[fill=white] (.75,0) rectangle (1,1);
	\node at (.75,1.2) {\scriptsize{$x_j$}};
	\node at (.75,0) {\scriptsize{$x_j$}};
	\node at (.2,0) {\scriptsize{$c$}};
	\node at (.2,1.2) {\scriptsize{$c$}};
%	\node at (-.8,0) {\scriptsize{$c^\vee$}};
	\node at (.2,-1.2) {\scriptsize{$c^{\vee\vee}$}};
\end{tikzpicture}
\right)
&=
\tr^\varphi_L
\left(
\begin{tikzpicture}[baseline=-.6cm]
	\draw (.6,-3.5) -- (.6,2.5);
	\draw (-.4,-2) -- (-.4,1);
	\draw (-1.2,-2) -- (-1.2,1);
	\draw[double] (-.8,1) -- (-.8,2);
	\draw[double] (-.8,-2) -- (-.8,-3);
	\roundNbox{fill=white}{(-.8,2)}{.3}{.3}{.3}{$\Psi(\ev_{c})$}
	\roundNbox{fill=white}{(-.8,1)}{.3}{.5}{.5}{$\mu$}
	\roundNbox{fill=white}{(0,0)}{.3}{.4}{.6}{$F_1(f)$}
	\roundNbox{fill=white}{(-.4,-1)}{.35}{.25}{.25}{$\Psi(\varphi_{c}^{-1})$}
	\roundNbox{fill=white}{(-.8,-2)}{.3}{.5}{.5}{$\mu^{-1}$}
	\roundNbox{fill=white}{(-.8,-3)}{.3}{.55}{.55}{$\Psi(\coev_{c^\vee})$}
%	\node at (-1.5,1.5) {\scriptsize{$\Psi(c^\vee \otimes c)$}};
	\node at (.9,.5) {\scriptsize{$E_{j,1}$}};
	\node at (.9,-.5) {\scriptsize{$E_{j,1}$}};
	\node at (-.1,.5) {\scriptsize{$\Psi(c)$}};
	\node at (-.1,-.5) {\scriptsize{$\Psi(c)$}};
	\node at (-1.6,-.5) {\scriptsize{$\Psi(c^\vee)$}};
	\node at (.1,-1.5) {\scriptsize{$\Psi(c^{\vee\vee})$}};
%	\node at (-1.7,-2.5) {\scriptsize{$\Psi(c^\vee \otimes c^{\vee\vee})$}};
\end{tikzpicture}
\right)
=
\tr^\varphi_L
\left(
\begin{tikzpicture}[baseline=-.1cm]
	\draw (.6,-2) -- (.6,2);
	\draw (-1.2,-1.3) -- (-1.2,1.3) arc (180:0:.4cm) -- (-.4,-1.3) arc (0:-180:.4cm);
	\roundNbox{fill=white}{(-1.2,1)}{.3}{0}{0}{$\delta_c$}
	\roundNbox{fill=white}{(0,1)}{.3}{.4}{.6}{$F_1(f)$}
	\roundNbox{fill=white}{(-.4,0)}{.35}{.25}{.25}{$\Psi(\varphi_{c}^{-1})$}
	\roundNbox{fill=white}{(-.4,-1)}{.3}{.1}{.1}{$\delta_{c^\vee}^{-1}$}
%	\node at (-1.6,1.5) {\scriptsize{$\Psi(c)^\vee$}};
	\node at (.9,1.5) {\scriptsize{$E_{j,1}$}};
	\node at (.9,.5) {\scriptsize{$E_{j,1}$}};
	\node at (-.1,1.5) {\scriptsize{$\Psi(c)$}};
	\node at (-.1,.5) {\scriptsize{$\Psi(c)$}};
%	\node at (-1.6,0) {\scriptsize{$\Psi(c^\vee)$}};
	\node at (.1,-.5) {\scriptsize{$\Psi(c^{\vee\vee})$}};
	\node at (.1,-1.5) {\scriptsize{$\Psi(c^\vee)^{\vee}$}};
\end{tikzpicture}
\right)
=
\tr^\varphi_L
\left(
\begin{tikzpicture}[baseline=-.6cm]
	\draw (.6,-3.7) -- (.6,1.7);
	\draw (-.4,-3.7) -- (-.4,1.7);
	\roundNbox{fill=white}{(-.4,-2)}{.3}{0}{0}{$\delta_c^\vee$}
	\roundNbox{fill=white}{(0,1)}{.3}{.4}{.6}{$F_1(f)$}
	\roundNbox{fill=white}{(-.4,0)}{.35}{.25}{.25}{$\Psi(\varphi_{c}^{-1})$}
	\roundNbox{fill=white}{(-.4,-1)}{.3}{.1}{.1}{$\delta_{c^\vee}^{-1}$}
	\node at (.9,1.5) {\scriptsize{$E_{j,1}$}};
	\node at (.9,.5) {\scriptsize{$E_{j,1}$}};
	\node at (-.1,1.5) {\scriptsize{$\Psi(c)$}};
	\node at (-.1,.5) {\scriptsize{$\Psi(c)$}};
	\node at (.1,-.5) {\scriptsize{$\Psi(c^{\vee\vee})$}};
	\node at (.1,-1.5) {\scriptsize{$\Psi(c^\vee)^{\vee}$}};
	\node at (.1,-2.5) {\scriptsize{$\Psi(c)^{\vee\vee}$}};
	\roundNbox{fill=white}{(-.4,-3)}{.3}{.25}{.25}{$\varphi_{\Psi(c)}$}
	\node at (-0.1,-3.5) {\scriptsize{$\Psi(c)$}};
\end{tikzpicture}
\right)
\end{align}
where $\delta_c \in \cE(\cM)(\Psi(c^\vee) \to \Psi(c)^\vee)$ is the canonical isomorphism from \eqref{eq:CanonicalDualIso}.

\item[\underline{$(1)\Rightarrow (2)$:}]
Suppose \eqref{eq:TraceCompatibility} holds.
Then for all $1\leq j\leq r$ and 
$f\in \cM(c\vartriangleright x_j\to c\vartriangleright x_j)\cong \cE( \Psi(c)\otimes E_{j,1} \to \Psi(c)\otimes E_{j,1})$, 
$\tr_L^\varphi(F_1(f)) = \alpha^{-1}\Tr^\cM_{c\vartriangleright x_j}(f)$, which is equal to the right hand side of \eqref{eq:TraceCompatibilityInE}.
Hence
$$
\tr^\varphi_L( f\circ [(\id_{\Psi(c)} - \Psi(\varphi_c^{-1}) \circ \delta_{c^\vee}^{-1} \circ \delta_c^\vee\circ \varphi_{\Psi(c)})\otimes \id_{E_{j,1}}]) = 0.
$$
Since $\tr_L^\varphi$ is nondegenerate (e.g., see \cite[Lem.~2.6]{1808.00323}), 
we must have $(\id_{\Psi(c)} - \Psi(\varphi_c^{-1}) \circ \delta_{c^\vee}^{-1} \circ \delta_c^\vee\circ \varphi_{\Psi(c)})\otimes \id_{E_{j,1}} = 0$ for all $1\leq j\leq r$.
Now taking right partial traces in $\cE(\cM)$, we must have $\id_{\Psi(c)} = \Psi(\varphi_c^{-1}) \circ \delta_{c^\vee}^{-1} \circ \delta_c^\vee\circ \varphi_{\Psi(c)}$, so $(\Psi,\mu)$ is pivotal.

\item[\underline{$(2)\Rightarrow (1)$:}] 
Suppose that $(\Psi,\mu)$ is pivotal, so that $\delta_c^\vee\circ \varphi_{\Psi(c)} = \delta_{c^\vee} \circ \Psi(\varphi_c)$.
Then for any $f\in \cM(c\vartriangleright m \to c\vartriangleright m)$, the right hand side of \eqref{eq:TraceCompatibilityInE} is equal to $\tr_L^\varphi(F_1(f)) = \alpha^{-1}\Tr^\cM_{c\vartriangleright x_j}(f)$, and thus \eqref{eq:TraceCompatibility} holds.
\end{proof}

We have an analogous omnibus theorem in the pivotal and unitary pivotal settings.

\begin{thm}
\label{thm:Pivotal}
Suppose that $\cC$ is a (unitary) pivotal fusion category then
\begin{itemize}
\item 
Module category structures with a (unitary) trace on a ($\Cstar$) category $\cM$ correspond exactly to (unitary) pivotal tensor functors $\scrC \rightarrow \cE(\cM)$.
\item 
A (unitary) pivotal fusion category $\cD$ is (unitary) pivotal Morita equivalent to $\scrC$ if and only if there is an indecomposable semisimple pivotal ($\Cstar$) module category $(\cM,\Tr^\cM)$ such that $\cD$ is (unitary) pivotal tensor equivalent to $\End_\scrC(\cM)$, the $\scrC$-linear (dagger) endofunctors of $\cM$.  
Furthermore, the pivotal left $\scrC$-module ($\Cstar$) categories $(\cM,\Tr^\cM)$ which realize a (unitary) pivotal Morita equivalence between $\scrC$ and $\cD$ are a torsor for the group of (unitary) pivotal outer automorphisms $\mathrm{Out}(\cD)$.
\item 
Tuples $(\cM, \Tr^\cM, m)$ 
where $(\cM, \Tr^\cM)$ is an indecomposable semisimple pivotal $\scrC$-module ($\Cstar$) category and $m\in \cM$ is a chosen simple object correspond exactly to connected normalized Frobenius algebras (irreducible Q-systems \cite{MR3308880}) $A$ in $\scrC$.  
The dual category corresponding to $A$ is the category of $A$-$A$ bimodules in $\scrC$ \cite{MR1966524}.
\end{itemize}
\end{thm}

\begin{remark}
This theorem is analogous to the purely algebraic Theorem \ref{thm:AlgOmnibus}.  
We warn the reader that if $\scrC$ is a (unitary) pivotal fusion category, the answers to our main problems might in principle be \emph{different} in the algebraic and pivotal (and unitary pivotal) settings.\footnote{
As discussed is \S\ref{sec:UnitaryPivotal}, it makes less sense to ask our main problems in the purely unitary non-pivotal setting.
}
For example, there might be several pivotal $\scrC$-module ($\Cstar$) categories which are equivalent just as algebraic $\scrC$-module categories, or there may be an algebraic module category which cannot be endowed with a (unitary) compatible trace (or even a dagger structure!).
These phenomena do not happen for Extended Haagerup, but it is interesting to ask whether they ever occur.
\end{remark}

%%%%%%%%%%%%%%%%%%%%%%%%%%%%%%%%%%%%%%%%%%%%%%%%%%%%%%%%%%
\subsubsection{The embedding theorem for pivotal module categories}
\label{sec:ModuleEmbedding}

In this section, we finally prove the embedding theorem for pivotal module categories.
We begin with a discussion of the planar algebra of a bipartite graph \cite{MR1865703}.
Our definition will simply use a Frobenius-Perron vertex weighting on our finite graph to extend
the action of 2-shaded monoidal tangles for a bipartite graph monoidal algebra to an action of shaded planar tangles.
We then show how to recover the usual definition of the graph planar algebra from \cite{MR1865703}.

\begin{defn}
\label{defn:GPA}
Let $\Gamma = (V_+, V_-, E)$ be a finite connected bipartite graph with even/$+$ vertices $V_+$, odd/$-$ vertices $V_-$, and edges $E$.
We consider an edge $\varepsilon \in E$ as directed from $+$ to $-$ with source $s(\varepsilon) \in V_+$ and target $t(\varepsilon) \in V_-$.
We write $\varepsilon^*$ for the same edge with the opposite direction.
Let $\lambda$ denote any Frobenius-Perron eigenvector of the adjacency matrix of $\Gamma$.\footnote{
The definition of the graph planar algebra $\cG_\bullet$ does not depend on the normalization of the Frobenius-Perron eigenvector $\lambda$.
In Remark \ref{rem:GPA-NotSpherical} below, we will define a spherical faithful state on $\cG_\bullet$ using a particular normalization of $\lambda$.
}

The 2-shaded graph monoidal algebra $\cG_{\bullet\to \bullet}=\cG\cM\cA(\Gamma)_{\bullet\to \bullet}$ is defined analogously to the unshaded version in Definition \ref{defn:GraphMonoidalAlgebra}.
The $\bbC$-vector spaces $\cG\cM\cA(\Gamma)_{(m\to n),\pm}$ are spanned by pairs of paths $(p,q)$ of length $m,n$ respectively which start at the same $\pm$ vertex and end at the same vertex.
Note that $\cG\cM\cA(\Gamma)_{(m\to n),\pm}$ is non-zero only when $m \equiv n \,\operatorname{mod}\, 2$.
The action of shaded monoidal tangles is given by the state-sum formula \eqref{eq:StateSumForGMA}.
Note that $\cG_{\bullet \to \bullet}$ is unitary with $\dag$-structure given by the conjugate-linear extension of $(p,q)^\dag = (q, p)$.

Now given a shaded \emph{planar} tangle $T$ of type $((t_0,\pm_0); (t_1,\pm_1),\dots, (t_k,\pm_k))$ whose input and output disks are rectangles with the star on the left, where the $i$-th disk has $n_i$ strings emanating from the top and $m_i$ from the bottom with $m_i+n_i = 2t_i$, we describe its action on tuples of basis elements $(p_i,q_i)\in \cG_{m_i\to n_i,\pm}$
by the weighted state-sum formula
\begin{equation}
\label{eq:StateSumFormulaForGPA}
T((p_1,q_1), \ldots, (p_k,q_k))
:=
\sum_{\text{states } \sigma \text { on } T } 
c(T; \sigma) 
\left(
\prod_{1\leq i\leq k}
\delta_{\sigma|_i, (p_i,q_i)}
\right)
\sigma|_0.
\end{equation}
A \emph{state} $\sigma$ on the tangle $T$ is an assignment of even vertices to unshaded regions, odd vertices to shaded regions, and edges to strings such that if a string labelled by $\epsilon$ separates two regions, then $s(\epsilon)$ is assigned to that unshaded region, and $t(\epsilon)$ is assigned to that shaded region.
Now $\sigma|_i$ denotes the pair of paths in $\cG\cM\cA(\Gamma)_{m_i \to n_i}$ obtained from reading the bottom and top boundaries of the $i$-th input disk from left to right.
In other words, we sum only over states that are `compatible' with the loops we start with.
To define the constant $c(T; \sigma)$, we first isotope $T$ so that strings are sufficiently smooth.
Now consider the set $E(T)$ of all local maxima and minima of strings of $T$.  
Then
$$
c(T; \sigma) = \prod_{e \in E(T)} \frac{\lambda(\sigma(e_{\text{convex}}))}{\lambda(\sigma(e_{\text{concave}}))},
$$
where $e_{\text{convex}}$ is the convex region of the extremum $e$, and $e_{\text{concave}}$ is the concave region of $e$.

This definition appears to be highly dependent on the choice of numbers of strings $m_i,n_i$ emanating from the bottom and top of each input and output disk.
However, every space $\cG_{m\to n,\pm}$ is canonically isomorphic to $\cG_{m+n,\pm}:=\cG_{m+n\to 0,\pm}$ by 
\begin{equation}
\label{eq:TurnDownAllTopStrings}
(p,q)\mapsto 
\left(
\frac{\lambda(t(p))}{\lambda(s(p))} 
\right)^{1/2}
pq^*.
\end{equation}
Here, instead of writing the pair of paths $(pq^*,\emptyset)$ where
the second has length zero, we only write the first path $pq^*$, which
is actually a loop of length $2t=m+n$.
Indeed, by post-composing with instances of the above isomorphism and precomposing with instances of its inverse as appropriate, we see that the action of planar tangles does not depend on the decomposition $2t_i = m_i + n_i$.
As in \cite[Th.~3.1]{MR1865703}, changing a tangle by a Morse cancellation or rotating a single input out output disk by $2\pi$ does not change the action of the tangle.
Hence the isomorphisms \eqref{eq:TurnDownAllTopStrings} endow the spaces $\cG_{n,\pm}$ with the structure of a shaded planar algebra called the \emph{graph planar algebra}, denoted $\cG_\bullet$.
%See also \cite[Pf.~of Thm.~4.2.1]{math.QA/9909027}, and \cite[Prop.~4.5]{1607.06041}.
We recover the definition from \cite{MR1865703} by always choosing $m_i=n_i = t_i$ for every input and output disk of $T$.

The $\dag$-structure of $\cG_\bullet$ is inherited from the graph monoidal algebra $\cG_{\bullet \to \bullet}$.
Since $\dag: \cG_{m\to n , \pm} \to \cG_{n\to m, \pm}$, the identification of both spaces with $\cG_{m+n,\pm}$ means that $(pq^*)^\dag = qp^*$, i.e., $\dag$ is the conjugate-linear extension of reversing a loop. 
It is straightforward to see that the $\dag$ structure on $\cG_\bullet$ is compatible with the reflection of planar tangles.
As the underlying monoidal algebra is unitary, so is the graph planar algebra.
\end{defn}

\begin{remark}
\label{rem:GPA-NotSpherical}
The graph planar algebra, and hence its projection category, is in general not spherical.
For example, taking any edge $\varepsilon$ which connects two vertices of distinct weights, the projection $\varepsilon \varepsilon^* \in \cG_{1,+}$ has distinct left and right traces.
However, if we normalize the Frobenius-Perron eigenvector $\lambda$ so that $\sum_{u\in V_+}\lambda_u^2 = 1 = \sum_{v\in V_-}\lambda_v^2$, then $\psi(p_v) := \lambda_v^2$ defines a spherical faithful state on $\cG_\bullet$ \cite[Prop.~3.4]{MR1865703}.
\end{remark}

\begin{nota}
To state the main theorems of this section, we fix the following notation.
\begin{itemize}
\item
$\Gamma = (V_+, V_-, E)$ is a connected bipartite graph
\item
$\lambda$ is any Frobenius-Perron eigenvector of $\Gamma$.
\item
$\cG_\bullet$ is the bipartite graph planar algebra of $\Gamma$
\item
$\cM = \Hilb^{V_+} \oplus \Hilb^{V_-}$ is one copy of $\Hilb$ for each vertex of $\Gamma$. 
\item
$\Tr^\cM$ is the unitary trace on $\cM$ corresponding to $\lambda \in G^{V_+\amalg V_-}$ under Proposition \ref{prop:CharacterizationOfTraces}.
\item
$\End^\dag(\cM)$ is the unitary multifusion category of dagger endofunctors of $\cM$.
\item
$\cU$ is the universal grading groupoid of $\End^\dag(\cM)$, which is the groupoid with $n_+ + n_-$ objects, and a unique isomorphism between any two objects.
\item
$F = \bigoplus_{\varepsilon \in E} E_{t(\varepsilon), s(\varepsilon)}\in \End^\dag(\cM)$.
\item
$\vee_{\text{standard}}$ is the standard unitary dual functor with respect to $F$ from \cite{1808.00323,1805.09234}, which is induced by the standard groupoid homomorphism defined from $\lambda$ as in \eqref{eq:StandardAndLopsided1808.00323}.
\item $\cH_\bullet$ is the planar algebra corresponding to $(\End^\dag(\cM), \vee_{\text{standard}}, F)$ under Corollary \ref{cor:PAEquivalence}.
\end{itemize}
\end{nota}

\begin{thm}
\label{thm:StateSumForPlanarAlgebra}
With the above notation, the $\dag$-isomorphism of the underlying monoidal algebras $\cH_\bullet \cong \cG_\bullet$ from Theorem \ref{thm:StateSumForMonoidalAlgebra} gives a $\dag$-isomorphism of unitary planar algebras.
\end{thm}
\begin{proof}
As the isomorphism from Theorem \ref{thm:StateSumForMonoidalAlgebra} identifies the underlying unitary monoidal algebras, 
we need only check that the actions of cup and cap agree.
Since cup is always the $\dag$ of cap in a unitary planar algebra as discussed in Remark \ref{rem:OtherCharacterizationOfUnitaryPivotal}, we need only check each shading of cap agrees.

First, the standard evaluation and coevaluation with respect to $F$ are given by
$$
\ev^{\text{standard}}_{E_{u,v}} 
:= 
\left(\frac{\lambda(u)}{\lambda(v)}\right)^{1/2}
\qquad
\qquad
\coev^{\text{standard}}_{E_{u,v}} 
:= 
\left(\frac{\lambda(v)}{\lambda(u)}\right)^{1/2}.
$$
Indeed, it is straightforward to check that the ratio  $\pi^{\text{standard}}(E_{u,v})$ of the left to right standard pivotal dimension of $E_{u,v}$ is given exactly by \eqref{eq:StandardAndLopsided1808.00323}.
Thus we see from the graphical calculus for $\End^\dag(\cM)$ that under the isomorphism of underlying monoidal algebras from Theorem \ref{thm:StateSumForMonoidalAlgebra} where we suppress the second empty loop, the formula for each shading of cap is given by
$$
\begin{tikzpicture}[baseline = 0cm]
	\filldraw[shaded] (0,0) arc (180:0:.2cm);
\end{tikzpicture}
=
\sum_{\varepsilon\in E}
\left(\frac{\lambda(t(\varepsilon))}{\lambda(s(\varepsilon))}\right)
\varepsilon\varepsilon^*
\qquad\qquad
\begin{tikzpicture}[baseline = 0cm]
	\fill[shaded] (-.1,0) -- (-.1,.3) -- (.5,.3) -- (.5,0) -- (.4,0) arc (0:180:.2cm);
	\draw[] (0,0) arc (180:0:.2cm);
\end{tikzpicture}
=
\sum_{\varepsilon\in E}
\left(\frac{\lambda(s(\varepsilon))}{\lambda(t(\varepsilon))}\right)
\varepsilon^*\varepsilon.
$$
These are exactly the same formulas for each shading of cap given by the state-sum formula \eqref{eq:StateSumFormulaForGPA}.
\end{proof}

The following corollary follows immediately follows from \ref{cor:PAEquivalence}.

\begin{cor}
\label{cor:GPAandEnd(M)}
Let $\Gamma = (V_+, V_-, E)$ be a finite connected bipartite graph, and let $\cG_\bullet$ be its graph planar algebra.
Let $\cM=\Hilb^{n_+}\oplus \Hilb^{n_-}$ where $n_\pm = |V_\pm|$.
\begin{itemize}
\item The idempotent category of $\cG_\bullet$ is equivalent to $\End(\cM)$, as multifusion categories.
\item The projection $\Cstar$ category of $\cG_\bullet$ is dagger equivalent to $\End^\dag(\cM)$, as unitary multifusion categories.
\end{itemize}
\end{cor}

We now prove a version of the graph planar algebra embedding theorem \cite{MR2812459} for module categories.
Below, we fix 
a finite depth subfactor planar algebra $\cP_\bullet$, and we denote by 
$(\scrC,X)$ the unitary multifusion category of projections of $\cP_\bullet$ with distinguished object $X$ corresponding to the unshaded-shaded strand of $\cP_\bullet$.
We endow $\scrC$ with the canonical spherical structure from \cite{MR1444286,MR2091457,MR3342166,1808.00323}.

\begin{thm}
\label{thm:ModuleEmbedding}
The following are equivalent:
\begin{enumerate}
\item
An embedding of shaded planar $\dag$-algebras $\cP_\bullet \hookrightarrow \cG_\bullet$
\item
A pivotal dagger tensor functor 
$(\Psi,\mu):(\scrC,\vee_{\rm spherical}) \to (\End^\dag(\cM), \vee_{\rm{standard}})$ such that $\Psi(X) = F$ and $\Psi(\id_X) = \id_F$, and
\item
an indecomposable left $\scrC$-module structure on $\cM$, compatible with the dagger structures of $\scrC$ and $\cM$, together with a unitary trace $\Tr^\cM$ defined up to scalar satisfying the compatibility condition \eqref{eq:TraceCompatibility}, whose fusion graph with respect to $X$ is $\Gamma$.
\end{enumerate}
\end{thm}
\begin{proof}
The equivalence of (1) and (2) follows from Corollary \ref{cor:PAEquivalence} together with Corollary \ref{cor:GPAandEnd(M)}.
The equivalence of (2) and (3) follows from Lemma \ref{lem:DaggerModules} together with Proposition \ref{prop:TracesToPivotalStructures} and Theorem \ref{thm:PivotalModulesAndPivotalFunctors}.
\end{proof}

%auto-ignore
%this ensures the arxiv doesn't try to start TeXing here.
%!TEX root =../EH3.tex

%%%%%%%%%%%%%%%%%%%%%%%%%%%%%%%%%%%%%%%%%%%%%%%%%%%%%%%%%%%%
%%%%%%%%%%%%%%%%%%%%%%%%%%%%%%%%%%%%%%%%%%%%%%%%%%%%%%%%%%%%
%%%%%%%%%%%%%%%%%%%%%%%%%%%%%%%%%%%%%%%%%%%%%%%%%%%%%%%%%%%%
\section{Combinatorics of potential (bi)module categories for Extended Haagerup} \label{sec:combinatorics}

%%%%%%%%%%%%%%%%%%%%%%%%%%%%%%%%%%%%%%%%%%%%%%%%%%%%%%%%%%%%
\subsection{Summary of the combinatorial techniques for classifying module and bimodule categories}

In this section we prove a partial classification of all fusion categories Morita equivalent to the Extended Haagerup fusion categories and all Morita equivalences between them.  
Specifically, we show that there are at most four fusion categories in the Morita equivalence class, show that there is exactly one Morita equivalence between any two that actually exist, and determine the fusion rules for all possible fusion categories and bimodule categories.  
This argument closely follows the outline of \cite{MR2909758,MR3449240}, so we begin by briefly summarizing the techniques of these articles.

Given a fusion category $\scrC$, one gets a fusion ring $C:=K_0(\scrC)$ with basis consisting of the isomorphism classes of simple objects in $\scrC$ and non-negative structure constants $N_{ij}^k$ for multiplication coming from the fusion rules $X_i \otimes X_j \cong \bigoplus N_{ij}^k X_k$. 
This fusion ring has an involution corresponding to taking duals. 
It is natural to wonder: given a candidate fusion ring, is it categorifiable into a fusion category, and if so, how many fusion categories categorify our fusion ring?
Typically, each of these questions is quite difficult \cite{MR1981895, MR3427429, MR3229513}; combinatorics alone tells you very little about a single fusion category.

Given several fusion categories $\scrC_i$ and some Morita equivalences $\cM_{ij}^k$ between $\scrC_i$ and $\scrC_j$, one gets several fusion rings $C_i = K_0(\scrC_i)$, several fusion bimodules $M_{ij}^k = K_0(\cM_{ij}^k)$, and many ``composition rules" $M_{ij}^k \otimes_{C_j} M_{j\ell}^{k'} \rightarrow M_{i\ell}^{k''}$.  
This collection of data satisfies many combinatorial constraints.
It is again natural to wonder:
given a collection of fusion rings, fusion bimodules, and composition rules, are they categorifiable, and if so, in how many different ways?
In general, this question is again quite difficult.
However, in a small handful of examples coming from the small index subfactor classification program, we have seen that candidates which satisfy the many combinatorial constraints have been uniquely categorified.
In contrast to the situation for a single fusion category, combinatorics often tells you quite a lot about the full Morita equivalence class of a known fusion category with a few known Morita equivalences.

Here is the outline in more detail.  
We start with some fusion categories $\scrC_i$ with fusion algebras $C_i$, and some Morita equivalences between them which we understand well. 
We first use a computer to list the fusion modules over the fusion rings $C_i$. 
(By `fusion modules' we mean based modules satisfying some additional properties --- see \cite{MR3449240}. Sometimes the term `NIMrep' is used in the literature; this is an abbreviation for non-negative integer matrix representations  \cite{MR1895282}.)
We identify a few of these fusion modules as coming from the known Morita equivalences, and we use some additional arguments to see that the known categorification is the only possible realization of these modules.  
(In our case, this step uses the uniqueness of the Extended Haagerup subfactor \cite{MR2979509}, which is much easier than existence.)

Second, we try to determine the possible fusion rings of the dual categories for each (real or hypothetical) module category. 
Using a computer, we can sometimes uniquely determine the dual fusion ring from a fusion module combinatorially, or at least produce a relatively small list.
We then compute the fusion modules over each of these new fusion rings, as well as the fusion bimodules between each pair of rings in our collection. 

At this point, we have a collection of rings, bimodules between them, and some information about categorification of some of the bimodules (coming from known algebra objects).
We now use the following key fact: given a triple of fusion categories $\mathcal{A}, \mathcal{B},\mathcal{C} $, invertible bimodule categories ${}_\mathcal{A} \mathcal{K} {}_\mathcal{B},  {}_\mathcal{B} \mathcal{L} {}_\mathcal{C},  {}_\mathcal{A} \mathcal{M} {}_\mathcal{C}$, and a tensor equivalence 
$$
{}_\mathcal{A} \mathcal{K} {}_\mathcal{B} \boxtimes_{\mathcal{B}} {}_\mathcal{B} \mathcal{L} {}_\mathcal{C} 
\cong  
{}_\mathcal{A} \mathcal{M} {}_\mathcal{C} ,
$$
we get an induced map on the decategorified bimodules over the fusion rings:
$$
{}_A K {}_B \otimes_B {}_B L {}_ C \rightarrow {}_A M {}_C .
$$
This induced map preserves positivity of coefficients and Frobenius-Perron dimensions.
Moreover, the existence of such a map can be checked with a computer. 
If such a map does not exist, we say that the triple of fusion bimodules is \emph{not multiplicatively compatible}. 
Thus categorification of many fusion modules or bimodules can be ruled out due to not being multiplicatively compatible with those fusion bimodules which have known categorifications. 
A similar argument can be used to compare the number of categorifications of different bimodules.  
This stage of the argument is a bit similar to playing Sudoku, since each time you rule out one possible bimodule, then the composites which were only compatible with the eliminated one are now themselves incompatible.

Following this outline, we can often deduce a lot of information about the Brauer-Picard groupoid from a relatively small amount of input data (such as existence of a few small objects which are known to have unique algebra structures). 
In particular, for the Extended Haagerup fusion categories, we start with our two fusion categories $\cE\cH_1$ and $\cE\cH_2$, the existence and uniqueness of the Extended Haagerup Morita equivalence between them, and the lack of automorphisms of the Extended Haagerup planar algebra.
This data is sufficient to successfully run the above procedure to obtain the entire Morita equivalence class, as evidenced by Theorem \ref{thm:bpg} below.

%%%%%%%%%%%%%%%%%%%%%%%%%%%%%%%%%%%%%%%%%%%%%%%%%%%%%%%%%%%%
\subsection{The Brauer-Picard groupoid of Extended Haagerup}

\label{sec:BrPic}
The Extended Haagerup subfactor gives a Morita equivalence between two fusion categories which are not tensor equivalent. 
The fusion rules for these two categories are given in \S\ref{sec:StructureOfEH1} and \ref{sec:StructureOfEH2}; one of the categories has commuting fusion rules and the other one does not. 
We will call the category with commuting fusion rules $\mathcal{EH}_1 $ and the other category $\mathcal{EH}_2 $.

We refer the reader to \cite{MR3449240} for precise definitions of fusion modules and bimodules and multiplicative compatibility of triples of modules/bimodules. 
Detailed descriptions of the computer algorithms used to search for fusion (bi)modules and to check for multiplicative compatibility are also described there.

\begin{lemma}
There are exactly $7$ fusion modules over $EH_1 $ and exactly $5$ fusion modules over $EH_2$.
\end{lemma}
\begin{proof}
Checked with computer.
\end{proof}

The data of the (right) fusion modules are presented in accompanying text files \texttt{EH1modules.txt} and \texttt{EH2modules.txt}. 
Each fusion module of rank $r$ over the fusion ring of rank $s$ is described by a list of
$r$ non-negative integer matrices of size $s \times r$. 
The $(i,j)$-th entry of the $k$-th matrix is the coefficient of the module basis element $m_j$ in the product $x_i m_k $ (where $x_i$ is a ring basis element). 
The bases for the fusion rings and modules are ordered with increasing Frobenius-Perron dimension.

From the list of matrices for a given fusion module, one can read off a corresponding list of objects in the fusion category which are of the form $\underline{\End}_{\cE\cH_i}(m)$, the internal endomorphism object associated to a simple object $m$ in a module category categorification.
Such an internal endomorphism object necessarily admits an algebra structure if the module can be categorified.
The multiplicity vector of the simple objects in $\cE\cH_i$ in each such (hypothetical) algebra object is given by the
$j^{th} $ column of the $j^{th}$ matrix. 
In particular, the first column of the first matrix in the data of the fusion module gives the multiplicity vector of the internal endomorphism algebra object with the smallest Frobenius-Perron dimension for any module category realization. 
Therefore if we can classify algebra objects with the given multiplicity vector, we can classify module category realizations of the fusion module.

We refer to the five fusion modules over $EH_2$ as $EH_2$-Modules 1-5, using the same order as in the text file. 
In the notation of Section \S\ref{sec:StructureOfEH2}, the corresponding (hypothetical) smallest algebra objects are given by
$1+2f_2+2f_4+f_6 $, $ 1+f_6$, $1+f_4+P $ (or $1+f_4+Q) $, $1+f_2$, and $1$.

\begin{lemma}
$EH_2$-Modules 4 and 5 are each realized by a unique right $\mathcal{EH}_2 $-module category.
\end{lemma}

\begin{proof}
In any fusion category, the object $1$ has a unique algebra structure. 
The object $1+f_2$ has a (necessarily unique by $3$-supertransitivity \cite[Lemma 3.13]{MR2909758}) algebra structure by the existence of the Extended Haagerup subfactor \cite{MR2979509}.
\end{proof}

To go further, we consider fusion bimodules, which we again enumerate with a computer. 
The full data is in the accompanying text file \texttt{EHBimodules.txt}.
There are two $EH_1$-$EH_1 $ fusion bimodules. 
There are three $EH_1$-$EH_2 $ fusion bimodules, exactly one of which corresponds to the algebra $1+f_2$ in $\mathcal{EH}_2 $ (i.e. the Extended Haagerup subfactor). There are three $EH_2$-$EH_2$ fusion bimodules, one of which has rank $3$, and the other two of which each contain basis elements with Frobenius-Perron dimension $1$.

\begin{lemma} \label{elim}
The rank $3$ $EH_2$-$EH_2$ fusion bimodule is not realized by an $\mathcal{EH}_2 $-$\mathcal{EH}_2 $ bimodule category.
\end{lemma}
\begin{proof}
Looking at the (computer-generated) lists of multiplicatively compatible modules and bimodules in the accompanying text file \texttt{EHbimodulecomposition.txt}, we find that there is no possible way to tensor a realization of the rank $3$ $EH_2$-$EH_2$ fusion bimodule (which is the first one on the list of $ EH_2$-$EH_2$ bimodules) with any invertible $\mathcal{EH}_1 $-$\mathcal{EH}_2 $ bimodule category.
\end{proof}

\begin{lemma}
The automorphism group of $\mathcal{EH}_2 $ is trivial.
\end{lemma}
\begin{proof}
This argument is the same as the corresponding ones for Haagerup and Asaeda--Haagerup \cite{MR2909758, MR3449240}.  There is a unique algebra object in $\mathcal{EH}_2 $ giving the Extended Haagerup planar algebra and this algebra tensor generates $\mathcal{EH}_2 $.  Thus automorphisms of $\mathcal{EH}_2$ correspond to automorphisms of the Extended Haagerup planar algebra (see \cite[Thm A]{1607.06041} for details).  Any automorphism of the Extended Haagerup planar algebra must send the generator to a multiple of itself (because it is uncappable) and the quadratic relation says that this scalar must be be one.  Thus the Extended Haagerup planar algebra does not admit non-trivial automorphisms.
\end{proof}

\begin{theorem}
The Brauer-Picard group of the Extended Haagerup fusion categories is trivial.
\end{theorem}
\begin{proof}
Since by Lemma \ref{elim} the only realizable $EH_2$-$EH_2 $-bimodules each contain a basis element of Frobenius-Perron dimension $1$, any bimodule category 
realization of one of these bimodules is equivalent to the trivial module category as either a left or right module category. 
Since $\mathcal{EH}_2 $ has no outer automorphisms, any such bimodule category is in fact the trivial bimodule category. 
Thus $\mathcal{EH}_2 $ does not admit any non-trivial invertible bimodule categories, and the Brauer-Picard group is trivial.
\end{proof}

\begin{corollary}
\label{cor:HomotopyTypeOfBPG}
The Brauer--Picard $3$-groupoid has the homotopy type of $K(\mathbb{C}^\times, 3)$.   
Any $G$-graded extension of an Extended Haagerup fusion category is of the form $\scrC \boxtimes \Vec(G, \omega)$ for $\omega \in H^3(G, \mathbb{C}^\times)$.  
\end{corollary}
\begin{proof}
The Brauer--Picard $3$-groupoid is connected, has trivial $\pi_1$ (since the Brauer--Picard group is trivial), has trivial $\pi_2$ (by \cite[Cor. 3.7]{MR3354332} since $\cE\cH_1$ has no invertible objects and no non-trivial gradings), and has $\pi_3 = \mathbb{C}^\times$ (by \cite[Prop. 7.1]{MR2677836}).  
Hence it is a $K(\mathbb{C}^\times, 3)$.

The classification of obstructions follows from the main result of \cite{MR2677836}. 
Since the Brauer--Picard group is trivial, the obstructions $O_3$ and $O_4$ vanish.  
Since $\pi_2$ is trivial, extensions are classified by $H^3(G, \mathbb{C}^\times)$ and it is easy to see that $\scrC \boxtimes \Vec(G, \omega)$ realizes these extensions.  
\end{proof}

\begin{corollary}
Exactly one of the three $EH_1$-$EH_2 $ fusion bimodules is realized by a bimodule category (the one corresponding to the Extended Haagerup subfactor).
\end{corollary}

\begin{lemma}
$EH_2$-Module 1 is not realized by any module category. 
\end{lemma}
\begin{proof}
Again looking at the lists of multiplicatively compatible modules and bimodules in the file \texttt{EHbimodulecomposition.txt}, we find that there is no possible way to tensor a right $\mathcal{EH}_2$-module category realizing $EH_2$-Module 1 with the known existing $\mathcal{EH}_2$-$\mathcal{EH}_1 $ bimodule category (which corresponds to the third $EH_2-EH_1$ bimodule on the list in the text files). This implies that $EH_2$-Module 1 is not realized by a module category.
\end{proof}

We are now left to classify categorifications of $EH_2$-Modules 2 and 3. 
For each of $EH_2$-Module 2/3, we can use multiplicative compatibility with the realized $EH_1$-$EH_2 $-bimodule to 
uniquely identify a corresponding fusion module over $EH_1$ which would have to be realized as well for any realization of $EH_2$-Module 2/3. From the lists in \texttt{EHbimodulecomposition.txt}, we see that $EH_2$-Module 2 corresponds to $EH_1$-Module 6 and $EH_2$-Module 3 corresponds to $EH_1$-Module 7.

We now introduce fusion rings $EH_3$ and $EH_4$ (whose multiplication tables were described in the preceding section). 
We compute the lists of fusion modules over $EH_3$ and $EH_4$; fusion bimodules over $EH_i$-$EH_j$, $1 \leq i,j, \leq 4 $; and multiplicative compatibility between all of these modules and bimodules. This data is all included in the accompanying text files.

The reason for introducing these rings is the following:

\begin{lemma}
\label{lem:DualCategory}
If $EH_2$-Module 2 is realized by a right $\mathcal{EH}_2 $-module category, then the fusion ring of the dual category is $EH_3$. 
If $EH_2$-Module 3 is realized by a right $\mathcal{EH}_2 $-module category, then the 
fusion ring of the dual category is $EH_4$.
\end{lemma}
\begin{proof}
We use a computer to find the fusion rings of the dual categories of realizations of these fusion modules. It turns out that is easier to compute the dual rings for the corresponding $EH_1$-Modules 6 and 7. Since any module category $\mathcal{K}_{\mathcal{EH}_2} $ realizing $EH_2$-Module 2 can be tensored with the  $\mathcal{EH}_2$-$\mathcal{EH}_1 $ Morita equivalence  to give a module category $\mathcal{L}_{\mathcal{EH}_1} $ realizing $EH_1$-Module 6 and having the same dual category as $\mathcal{K} $ (and similarly for $EH_2$-Module 3), this is sufficient.
\end{proof}

\begin{lemma}
$EH_2$-Module 2 and $EH_2$-Module 3 are each realized by at most one module category.
\end{lemma}
\begin{proof}
Let $\mathcal{K}{}_{\mathcal{EH}_2} $ and $\mathcal{L} {}_{\mathcal{EH}_2} $ be realizations of $EH_2$-Module 2 with dual categories $ \mathcal{C}$ and $\mathcal{D} $. 
Then by the previous lemma $ \mathcal{C}$ and $\mathcal{D} $
each have fusion ring $EH_3$. 
Then 
$$
\mathcal{M}= \mathcal{K}{}_{\mathcal{EH}_2} \boxtimes_{\mathcal{EH}_2 } {}_{\mathcal{EH}_2} \mathcal{L}^{\text{op}}
$$ 
is an invertible $\mathcal{C} $-$\mathcal{D} $
bimodule category with realizes some $ EH_3$-$EH_3$ fusion bimodule. 
Looking at the list of $EH_3$-$EH_3 $-fusion bimodules, we see that every such bimodule has a basis element with Frobenius-Perron dimension $1$. 
Therefore $\mathcal{M}$ is trivial as a left 
$\mathcal{C} $ module category. 
This means that $\mathcal{C} \cong \mathcal{D}$. 
Since the Brauer-Picard group is trivial, this implies that $\mathcal{K}_{\mathcal{EH}_2}  \cong \mathcal{L}_{\mathcal{EH}_2} $.

The proof for $EH_2$-Module 3 is similar.
\end{proof}  

Since ${}_\cA \cM_\cB$ is a Morita equivalence if and only if $\cB$ is isomorphic to the dual category $\End_{\cA}(\cM)$, we have the following corollary.

\begin{cor}
There is at most one fusion category Morita equivalent to $\cE\cH_2$ with fusion ring $EH_3$ and at most one fusion category Morita equivalent to $\cE\cH_2$ with fusion ring $EH_4$.
\end{cor}

Putting this all together, we obtain the following result.

\begin{theorem}
\label{thm:bpg}
In addition to $\mathcal{EH}_1 $ and $\mathcal{EH}_2 $, the Morita equivalence class of the Extended Haagerup fusion categories contains:
\begin{itemize}
\item at most one fusion category with fusion ring $ EH_3$;
\item at most one fusion category with fusion ring $ EH_4$;
\item and no other fusion categories.
\end{itemize}

\end{theorem}

The main result of this paper, Theorem \ref{thm:Main} asserts the existence of fusion categories $\mathcal{EH}_3 $ and $\mathcal{EH}_4 $ in the Extended Haagerup Morita equivalence class with fusion rings $EH_3$ and $EH_4$, respectively.

\begin{remark}
There are analogous versions of Theorem \ref{thm:bpg} for the unitary, pivotal, and unitary pivotal settings with the analogous conclusion as Theorem \ref{thm:bpg}.
\begin{itemize}
\item
In the unitary setting, the unitary Morita equivalence class of the Extended Haagerup unitary fusion categories contains at most one unitary fusion category with each of the fusion rings $EH_3$ and $EH_4$ and no other unitary fusion categories.
\item
In the pivotal setting, the pivotal Morita equivalence class of the Extended Haagerup pivotal fusion categories contains at most one pivotal fusion category with each of the fusion rings $EH_3$ and $EH_4$ and no other pivotal fusion categories.
%Here, $\cC$ and $\cD$ are \emph{pivotal Morita equivalent} if there is a pivotal $\cC-\cD$ bimodule category $\cM$ (pivotal as both a left and right module category) such that $\cM\boxtimes_\cD \cM^{\op}$ is equivalent to $\cC$ as a $\cC-\cC$ pivotal bimodule category, and  $\cM^{\op}\boxtimes_\cC \cM$ is equivalent to $\cD$ as pivotal bimodule categories.
\item
In the unitary pivotal setting, the unitary pivotal Morita equivalence class of the Extended Haagerup unitary fusion categories contains at most one unitary fusion category with each of the fusion rings $EH_3$ and $EH_4$ and no other unitary fusion categories.
\end{itemize}

The proofs of these theorems are completely analogous to the above argument inserting adjectives as necessary.  
The key point is that we already know that the Extended Haagerup subfactor is unique in all contexts (algebraically, unitarily, pivotally, and unitary pivotally).  That is, we need to know that not only is there a unique algebra structure on $1+f^{(2)}$ in $\cE\cH_3$, but we also have a unique $\Cstar$ algebra structure, a unique normalized Frobenius structure, and a unique Q-system structure.  
It is important to note that there is no obvious way to derive these theorems from each other; rather we must use the same argument separately in each setting.

In principal, we might still have that $\cE\cH_3$ or $\cE\cH_4$ exists as say a fusion category, but not as a unitary pivotal fusion category.  
However, note that existence of $\cE\cH_3$ and $\cE\cH_4$ in the unitary pivotal setting (which is what we actually prove!) implies existence in all settings.  
\end{remark}

%%%%%%%%%%%%%%%%%%%%%%%%%%%%%%%%%%%%%%%%%%%%%%%%%%%%%%%%%%%%
\subsection{The fusion ring of the groupoid}

Suppose $\mathcal{A}$, $ \mathcal{B}$, and $\mathcal{C} $ are fusion categories and ${}_\mathcal{A} \mathcal{K} {}_\mathcal{B}$, $  {}_\mathcal{B} \mathcal{L} {}_\mathcal{C}$,  and $  {}_\mathcal{A} \mathcal{M} {}_\mathcal{C}$ are Morita equivalences such that there is a bimodule equivalence
$$\Phi:
{}_\mathcal{A} \mathcal{K} {}_\mathcal{B} \boxtimes_{\mathcal{B}} {}_\mathcal{B} \mathcal{L} {}_\mathcal{C} 
\cong  
{}_\mathcal{A} \mathcal{M} {}_\mathcal{C}.
$$
In general there may be multiple such equivalences $\Phi $, which are parametrized by invertible objects in the (common) Drinfeld center $Z(\mathcal{A}) $. If the Drinfeld center has no non-trivial invertible objects then the equivalence $\Phi$ is uniquely determined by $\mathcal{K} $, $\mathcal{L} $, and $\mathcal{M} $. 
There are no invertible central objects for the Extended Haagerup categories, as can be read off from the complete description of $Z(\cE\cH)$ in \cite{MR3611056} or can seen more directly following \cite{MR3354332}.
Therefore it makes sense to define the tensor product of simple objects in $\mathcal{K} $ and $\mathcal{L} $ as a direct sum of simple objects of $\mathcal{M}$. 
Thus for Extended Haagerup, we can define the fusion ring of the Brauer-Picard groupoid, with basis consisting of isomorphism classes of simple objects in each invertible bimodule category.

\begin{nota}
\label{nota:ExtendedHaagerupFusionRules}
For $1\leq i,j\leq 4$, we denote by $C_{ij}$ the unique $EH_i - EH_j$ fusion bimodule which was calculated using a computer and discussed in the last section.
The rank of $C_{ij}$ is the $ij$-th entry of the following matrix:
\begin{equation}
\label{eq:RanksOfModules}
R:=\left(
\begin{array}{cccc}
 6 & 6 & 6 & 6 \\
 6 & 8 & 5 & 5 \\
 6 & 5 & 8 & 5 \\
 6 & 5 & 5 & 8 \\
\end{array}
\right).
\end{equation}
We may view $(C_{ij})_{i,j=1}^4$ as one fusion ring whose basis consists of the union of the distinguished bases of each $C_{ij}$.
Multiplication of basis elements is determined by the relative tensor product of the ambient bimodules (and defined to be zero when the ambient bimodules don't compose).

We describe the fusion ring in the Mathematica notebook \texttt{EHmult.nb}, which is a wrapper for the data file \texttt{EHmult.txt}, both of which are bundled with the arXiv sources of this article.
Therein, we supply a 6-dimensional tensor $T$ whose $(i,j,k,x,y,z)$-entry is the coefficient of 
$z$-th basis element of $C_{i,k}$ in the product of the
$x$-th basis element of $C_{ij}$ 
and the
$y$-th basis element of $C_{jk}$.
That is,
$$
{}_i X_j \otimes {}_jY_k  = \sum_{Z} T(i,j,k,x,y,z) {}_iZ_k
\qquad\qquad
\langle {}_i X_j \otimes {}_jY_k , {}_i Z_k\rangle
:=
T(i,j,k,x,y,z)
$$
where ${}_iX_j$ is the $x$-th basis of $EH_{ij}$, and similarly for ${}_jY_k$ and ${}_iZ_k$.
\end{nota}

\begin{nota}
\label{nota:ExtendedHaagerupMultifusion}
We denote by  
$(\cE\cH_{ij})_{i,j=1}^2$ and $(EH_{ij})_{i,j=1}^2$
the projection unitary multifusion category of the Extended Haagerup subfactor planar algebra and its fusion ring, where the 2 corresponds to an unshaded region and a 1 corresponds to a shaded region.
\end{nota}

\begin{remark}
By Theorem \ref{thm:bpg}, there is at most one way to extend the unitary multifusion category $(\cE\cH_{ij})_{i,j=1}^2$ to a unitary multifusion category $(\cE\cH_{ij})_{i,j=1}^4$ such that
$\cE\cH_{ij}$ categorifies $EH_{ij}$ for all $1\leq i,j\leq 4$.
\end{remark}

%%%%%%%%%%%%%%%%%%%%%%%%%%%%%%%%%%%%%%%%%%%%%%%%%%%%%%%%%%%%
\subsection{Fusion graphs from \texorpdfstring{$EH_2$}{EH2}-Modules}
\label{sec:FusionGraphsForEH2Modules}

We continue using Notations \ref{nota:ExtendedHaagerupFusionRules} and \ref{nota:ExtendedHaagerupMultifusion} from the previous section.
Notice that for $1\leq k\leq 4$, we get a left $C$-module $M_k$ given by
$$
M_k := 
\begin{pmatrix}
EH_{1k} 
\\
EH_{2k}
\end{pmatrix}.
$$

\begin{defn}
For $1\leq k\leq 4$, the \emph{fusion graph} $\Gamma_k$ for $M_k$ with respect to $X$ is the bipartite graph consisting of
\begin{itemize}
\item
odd, shaded vertices given by the basis elements of $EH_{1k}$,
\item
even, unshaded vertices given by the basis elements of $EH_{2k}$, and
\item
$\langle {}_2X_1\otimes {}_1Y_k , {}_2Z_k\rangle = T(2,1,k,1,y,z)$ edges between the $y$-th basis element ${}_1Y_k \in EH_{1k}$ and the $z$-th basis element ${}_2Z_k \in EH_{2k}$. 
\end{itemize}
Note this convention is opposite to the one used for principal graphs of subfactors and fusion graphs for fusion categories in \S\ref{sec:StructureOfEH3} (see Notation \ref{notation:TensorOnRight}). 
\end{defn}

\begin{prop}
\label{prop:fusion-graphs}
The fusion graphs $\Gamma_k$ for $1\leq k\leq 4$ are given by
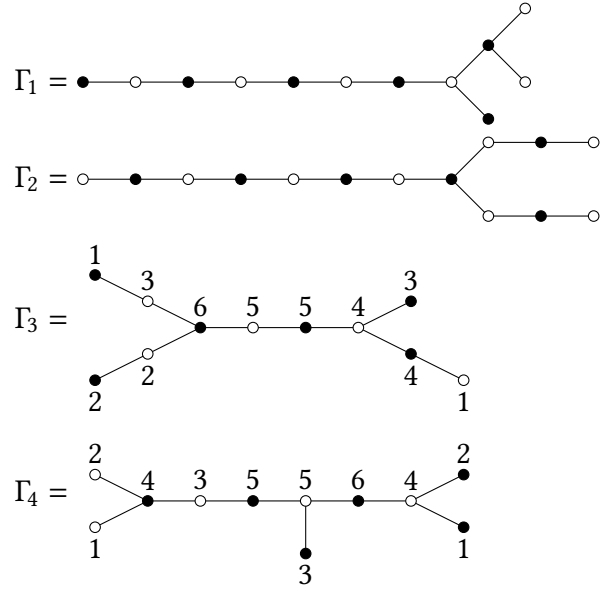
\begin{align*}
\Gamma_1 &= 
\begin{tikzpicture}[baseline=-1mm, scale=.7]
	\draw (0,0) -- (7,0);
	\draw (7,0)--(7.7,.7);
	\draw (7,0)--(7.7,-.7);
	\draw (7.7,.7)--(8.4,1.4);
	\draw (7.7,.7)--(8.4,0);
	\filldraw [fill]             (0,0) circle (1mm);
	\filldraw [fill=white] (1,0) circle (1mm);
	\filldraw [fill]             (2,0) circle (1mm);
	\filldraw [fill=white] (3,0) circle (1mm);
	\filldraw [fill]             (4,0) circle (1mm);
	\filldraw [fill=white] (5,0) circle (1mm);
	\filldraw [fill]             (6,0) circle (1mm);
	\filldraw [fill=white] (7,0) circle (1mm);		
	\filldraw [fill]             (7.7,.7) circle (1mm);
	\filldraw [fill]             (7.7,-.7) circle (1mm);
	\filldraw [fill=white] (8.4,1.4) circle (1mm);
	\filldraw [fill=white] (8.4,0) circle (1mm);
\end{tikzpicture}
\\\
\Gamma_2 &= 
\begin{tikzpicture}[baseline=-1mm, scale=.7]
	\draw (0,0) -- (7,0);
	\draw (7,0)--(7.7,.7)--(9.7,.7);
	\draw (7,0)--(7.7,-.7)--(9.7,-.7);
	\filldraw [fill=white]             (0,0) circle (1mm);
	\filldraw [fill] (1,0) circle (1mm);
	\filldraw [fill=white]             (2,0) circle (1mm);
	\filldraw [fill] (3,0) node {}  circle (1mm);
	\filldraw [fill=white]             (4,0) circle (1mm);
	\filldraw [fill] (5,0) node {}  circle (1mm);
	\filldraw  [fill=white]            (6,0) circle (1mm);
	\filldraw [fill] (7,0) node {}  circle (1mm);
	\filldraw [fill=white]             (7.7,.7) circle (1mm);
	\filldraw [fill] (8.7,.7) circle (1mm);
	\filldraw [fill=white]             (9.7,.7) circle (1mm);
	\filldraw [fill=white]            (7.7,-.7) circle (1mm);
	\filldraw [fill] (8.7,-.7) circle (1mm);
	\filldraw [fill=white]            (9.7,-.7)  circle (1mm);
\end{tikzpicture}
\\
\Gamma_3 &= 
\begin{tikzpicture}[baseline, scale=.7]
	\draw (-2,-1) -- (0,0);
	\draw (-2,1)--(0,0);
	\draw (0,0)--(3,0);
	\draw (3,0)--(4,.5);
	\draw (3,0)--(5,-1);
	\filldraw [fill]  (-2,-1) circle (1mm) node[below] {$2$};
	\filldraw [fill]  (-2,1) circle (1mm) node[above] {$1$};
	\filldraw [fill=white] (-1,-.5) circle (1mm) node[below] {$2$};
	\filldraw [fill=white] (-1,.5) circle (1mm) node[above] {$3$};
	\filldraw [fill]  (0,0) circle (1mm)node[above] {$6$};
	\filldraw [fill=white] (1,0) circle (1mm)node[above] {$5$};
	\filldraw [fill]  (2,0) circle (1mm) node[above] {$5$};
	\filldraw [fill=white] (3,0) circle (1mm) node[above] {$4$};
	\filldraw [fill]  (4,-.5) circle (1mm) node[below] {$4$};
	\filldraw [fill]  (4,.5) circle (1mm) node[above] {$3$};
	\filldraw [fill=white] (5,-1) circle (1mm) node[below] {$1$};
\end{tikzpicture}
\\
\Gamma_4 &= 
\begin{tikzpicture}[baseline, scale=.7]
	\draw (-1,-.5) -- (0,0);
	\draw (-1,.5)--(0,0);
	\draw (0,0)--(5,0);
	\draw (5,0)--(6,.5);
	\draw (5,0)--(6,-.5);
	\draw (3,0)--(3,-1);
	\filldraw [fill=white] (-1,-.5) circle (1mm) node[below] {$1$};
	\filldraw [fill=white] (-1,.5) circle (1mm) node[above] {$2$};
	\filldraw [fill]  (0,0) circle (1mm) node[above] {$4$};
	\filldraw [fill=white] (1,0) circle (1mm) node[above] {$3$};
	\filldraw [fill]  (2,0) circle (1mm) node[above] {$5$};
	\filldraw [fill=white] (3,0) circle (1mm) node[above] {$5$};
	\filldraw [fill] (3,-1) circle (1mm) node[below] {$3$};
	\filldraw [fill]  (4,0) circle (1mm) node[above] {$6$} ;
	\filldraw [fill=white] (5,0) circle (1mm) node[above] {$4$};
	\filldraw [fill]  (6,-.5) circle (1mm) node[below] {$1$};
	\filldraw [fill]  (6,.5) circle (1mm) node[above] {$2$};
\end{tikzpicture}
\end{align*}
\end{prop}

\begin{remark}
The labelings on $\Gamma_3$ and $\Gamma_4$ match the indexing of objects in \texttt{EHmult.nb}.
As we only need labelings on $\Gamma_3$ and $\Gamma_4$ in the following section, we have not labelled $\Gamma_1$ and $\Gamma_2$.

Our convention for shading the above vertices is that all vertices in $\cE\cH_{1k}$ are shaded, whereas all vertices in $\cE\cH_{2k}$ are unshaded.
This corresponds to the fact that the unshaded region of the Extended Haagerup planar algebra $\cE\cH_\bullet$ corresponds to $\cE\cH_2$, and the shaded region corresponds to $\cE\cH_1$.
\end{remark}

\begin{proof}[Proof of Proposition \ref{prop:fusion-graphs}]
The first two are exactly the definition of the dual principal graph and the principal graph of the extended Haagerup subfactor.
The second two are obtained via computer in the Mathematica notebook \texttt{EHmult.nb} included with the arXiv sources of this article.
\end{proof}

\begin{remark}
By the complete classification of possible module categories for $\cE\cH_1$ and $\cE\cH_2$ in Theorem \ref{thm:bpg} together with Corollary \ref{cor:EmbeddingGivesModule}, 
the graphs in Proposition \ref{prop:fusion-graphs} are the only bipartite graphs which could accept a planar algebra embedding map from the Extended Haagerup subfactor planar algebra.
\end{remark}

\begin{cor}
\label{cor:ExistenceFromPAEmbeddings}
If the extended Haagerup subfactor planar algebra embeds into the graph planar algebra of $\Gamma_k$ for $k=3,4$, then $M_{k}$ is categorifiable as a $(\cE\cH_{ij})_{i,j=1}^2$-module $\Cstar$-category, and $\cE\cH_k$ exists.
\end{cor}
\begin{proof}
Fix $3\leq k\leq 4$.
By Corollary \ref{cor:EmbeddingGivesModule}, the embedding of shaded planar algebras gives us a $(\cE\cH_{ij})_{i,j=1}^2$-module $\Cstar$-category $\cM_k$ which categorifies $M_k$ and whose fusion graph with respect to the unshaded-shaded strand is given by $\Gamma_k$.
We see that $\cM_k$ is equivalent to a direct sum $\cE\cH_{1k} \oplus \cE\cH_{2k}$ where $\cE\cH_{jk}$ is a left module category over $\cE\cH_{jj}:=\cE\cH_j$ for $j=1,2$.
By analyzing the fusion rules with $X$, by Theorem \ref{thm:bpg}, we may conclude that $\cE\cH_{jk}$ categorifies the fusion bimodule $EH_{jk}$ for $j=1,2$. 
Specializing to $j=2$,
since $EH_{2k}$ is a $EH_{22} - EH_{kk}$ bimodule, by Theorem \ref{thm:bpg}, the dual category $\cE\cH_{kk}$ of the $\cE\cH_{22}$-module $\cE\cH_{2k}$ must categorify $EH_{kk}$.
Again by Theorem \ref{thm:bpg}, $\cE\cH_{kk}$ is equivalent to $\cE\cH_k$.
\end{proof}

\begin{remark}
We can perform a similar (simpler) calculation for the Haagerup fusion categories. It was shown in \cite{MR2909758} that there are exactly three fusion categories in the Morita equivalence class of the Haagerup subfactor, which we will denote by $\mathcal{H}_k, \ k=1,2,3 $; and a unique Morita equivalence between each pair. The category $\mathcal{H}_2 $ has six simple objects, labeled $1$, $g$, $g^2$, $X $, $g X $, and $g^2X $, which satisfy the fusion rules
$$g^3=1, \quad X^2 = 1+ X +gX+g^2X,\quad  gX=Xg^2 .$$
(Here we have used decategorified notation, and suppressed tensor product, direct sum, and isomorphism symbols).

The category $\mathcal{H}_3 $ is the category of bimodules in $\mathcal{H}_2 $ over the algebra $ 1+g+g^2$; it has the same fusion ring as $\mathcal{H}_2 $, and we will label its simple objects  by $ 1$, $h$, $h^2$, $Y$, $gY$, and $g^2Y$. The category $\mathcal{H}_1 $ can described as the category of bimodules over the algebra $1+X$ in $\mathcal{H}_2 $, or as the category of bimodules over $1+Y+hY $ in $\mathcal{H}_3 $. The Haagerup planar algebra is the planar algebra corresponding to the generator $K$ of the $\mathcal{H}_1 $-$\mathcal{H}_2 $ Morita equivalence whose right internal end $\overline{K} K$ is $1+X $. Let $L$ be the object in the $\mathcal{H}_2 $-$\mathcal{H}_3 $ Morita equivalence whose left internal end $L \overline{L}$ is $1+g+g^2$ (and whose right internal end $\overline{L} L$ is $1+h+h^2$). Let $M$ be the object in the $\mathcal{H}_1 $-$\mathcal{H}_3 $ Morita equivalence whose right internal end $\overline{M} M$ is $1+Y+hY$.

The  $\mathcal{H}_2 $-$\mathcal{H}_3 $ Morita equivalence has rank two, with simple objects $ L=gL=Lh$ and $ XL=LY$. The $\mathcal{H}_1 $-$\mathcal{H}_3 $ Morita equivalence has rank four, with simple objects $KL$, $M$, $Mh$, and $Mh^2$. The fusion graph for the module corresponding to $\mathcal{EH}_3 $ is then determined by tensoring each of the two simple $1 $-$2$ objects  on the left by $K$ and decomposing into simple $2$-$3$ objects. Clearly there is a single edge from $L$ to $KL$ and no other edges out of $L$.  We now want to find the vertices adjacent to $XL$, i.e. the summands of $KXL$.
By Frobenius reciprocity, using $(\cdot, \cdot)$ to denote the dimension of the hom space,
$$(KXL, KXL) =  (\overline{K}K,XL\overline{L}\overline{X}) = (1+X, X(1+g+g^2)\overline{X}) = (1+X, 1+g+g^2+3X+3gX+3g^2X) = 4.$$
So $KXL$ has either four distinct simple summands or a single simple summand with multiplicity two.  But 
$$(KXL, KL) = (\overline{K}K X, L \overline{L}) = ((1+X)X,1+g+g^2) = (1+2X+gX+g^2X,1+g+g^2)=1,$$ 
so $KL$ appears with multiplicity one in $KXL$.  Thus $KXL$ has four distinct distinct summands and there is a single edge from $XL$ to each of the four simple $2 $-$3$ objects. This gives the broom graph of Corollary \ref{cor:hrp_graphs}.

\end{remark}

%auto-ignore
%this ensures the arxiv doesn't try to start TeXing here.
%!TEX root =../EH3.tex

%%%%%%%%%%%%%%%%%%%%%%%%%%%%%%%%%%%%%%%%%%%%%%%%%%%%%%%%%%%%
%%%%%%%%%%%%%%%%%%%%%%%%%%%%%%%%%%%%%%%%%%%%%%%%%%%%%%%%%%%%
%%%%%%%%%%%%%%%%%%%%%%%%%%%%%%%%%%%%%%%%%%%%%%%%%%%%%%%%%%%%
\section{Graph planar algebra embeddings for Extended Haagerup}
\label{sec:construction}

To specify a map out of a planar algebra presented by generators and relations, we need only assign values to the the generators and check the relations.  
In particular, once we have a nice presentation of a planar algebra, we can easily calculate all pivotal ($\Cstar$) module categories over it.   
For example, if we want to calculate all pivotal ($\Cstar$) module categories over the Termperley-Lieb-Jones planar algebra, we have no generators, and the only relation is the loop modulus, so we get a unique module category for every planar graph with the correct Frobenius--Perron eigenvector \cite{MR2046203,MR3420332}.
The $SU(3)_q$ planar algebra is presented by two trivalent vertices satisfying certain relations using 
Kuperberg's spider description \cite{MR1403861}, and finding elements in a graph planar algebra corresponding to these two trivalent vertices is exactly solving Ocneanu's cell conditions \cite{MR1907188,MR2545609,MR1839381}.

One of the main results of \cite{MR2979509} is to give a similar characterization of maps out of the Extended Haagerup planar algebra denoted $\cE\cH_\bullet$, which we recall in Proposition \ref{prop:maps-out-of-EH}.  
Using this result, we give the embeddings of the extended Haagerup subfactor planar into the graph planar algebras of $\Gamma_3$ and $\Gamma_4$, by solving the equations specified in Proposition \ref{prop:maps-out-of-EH} in the appropriate graph planar algebras.  
This is closely analogous to the original construction of $\cE\cH_\bullet$ by embedding it in the graph planar algebra of its principal graph, and we are able to reuse the same code.  
There are associated Mathematica notebooks (\texttt{module-GPAs-EH3.nb} and \texttt{module-GPAs-EH4.nb}) which demonstrate the messy process of solving these equations.   Here we simply exhibit particular solutions.  Thus by Corollary \ref{cor:ExistenceFromPAEmbeddings}, $M_3$ and $M_4$ are categorifiable as $(\cE\cH_{ij})_{i,j=1}^2$-module $\Cstar$-categories, and $\cE\cH_3$ and $\cE\cH_4$ exist.

%%%%%%%%%%%%%%%%%%%%%%%%%%%%%%%%%%%%%%%%%%%%%%%%%%%%%%%%%%%%%%%
\subsection{The lopsided graph planar algebra convention}
\label{sec:LopsidedConvention}

Suppose $\cP_\bullet$ is a semisimple shaded planar algebra with pivotal projection multitensor category $(\scrC,X,\varphi)$ where $X\in \cP_{1,+}$ is the shaded-unshaded strand.
By just rescaling cups and caps in $\scrC$ for $X$ as in \cite[\S1.1]{MR3254427}, 
\begin{equation}
\label{eq:RescaleCupsAndCaps}
\begin{tikzpicture}[baseline = 0cm]
	\fill[shaded] (-.1,0) -- (-.1,.3) -- (.5,.3) -- (.5,0) -- (.4,0) arc (0:180:.2cm);
	\draw[] (0,0) arc (180:0:.2cm);
\end{tikzpicture}
\mapsto
x\,
\begin{tikzpicture}[baseline = 0cm]
	\fill[shaded] (-.1,0) -- (-.1,.3) -- (.5,.3) -- (.5,0) -- (.4,0) arc (0:180:.2cm);
	\draw[] (0,0) arc (180:0:.2cm);
\end{tikzpicture}
\qquad
\qquad
\begin{tikzpicture}[baseline = -.2cm, yscale = -1]
	\filldraw[shaded] (0,0) arc (180:0:.2cm);
\end{tikzpicture}
\mapsto
x^{-1}\,
\begin{tikzpicture}[baseline = -.2cm, yscale = -1]
	\filldraw[shaded] (0,0) arc (180:0:.2cm);
\end{tikzpicture}
\qquad
\qquad
\begin{tikzpicture}[baseline = 0cm]
	\filldraw[shaded] (0,0) arc (180:0:.2cm);
\end{tikzpicture}
\mapsto
y\,
\begin{tikzpicture}[baseline = 0cm]
	\filldraw[shaded] (0,0) arc (180:0:.2cm);
\end{tikzpicture}
\qquad
\qquad
\begin{tikzpicture}[baseline = -.2cm, yscale=-1]
	\fill[shaded] (-.1,0) -- (-.1,.3) -- (.5,.3) -- (.5,0) -- (.4,0) arc (0:180:.2cm);
	\draw[] (0,0) arc (180:0:.2cm);
\end{tikzpicture}
\mapsto
y^{-1}\,
\begin{tikzpicture}[baseline = -.2cm, yscale=-1]
	\fill[shaded] (-.1,0) -- (-.1,.3) -- (.5,.3) -- (.5,0) -- (.4,0) arc (0:180:.2cm);
	\draw[] (0,0) arc (180:0:.2cm);
\end{tikzpicture}
\end{equation}
we obtain another semisimple shaded planar algebra $\cP_\bullet^{\cap x,y}$ with the same underlying projection multitensor category.
To describe the action of tangles, we first write the tangles in standard form, where each box has the same number of strings emanating from the top and bottom.
The action of tangles is obtained from the action of tangles for $\cP_\bullet$, where in addition, we multiply by factors of 
$x,y,x^{-1},y^{-1}$ corresponding to appearances of cups and caps as in \eqref{eq:RescaleCupsAndCaps} in the standard form for the tangle.

It is straightforward to verify that this is a well-defined action of planar tangles which is independent of the choice of standard form of a tangle.
One first verifies that the zig-zag relations hold and $2\pi$-rotation is still the identity.
One then appeals to the folklore theorem (\cite[Proof of Thm.~4.2.1]{math.QA/9909027}, similar to \cite[Prop.~4.5]{1607.06041}) that any two standard forms of a tangle are related by a finite number of moves including Morse cancelation, $2\pi$-rotation, and exchanging the heights of two input boxes.
Thus  $\cP_\bullet^{\cap x,y}$ is a shaded planar algebra.

While the underlying projection multitensor category $\scrC$ has not changed, the pivotal structure $\varphi^{\cap x,y}$ on $\scrC$ corresponding to $\cP_\bullet^{\cap x,y}$ has changed!
Indeed, pivotal structures on a semisimple multitensor category are completely determined by the left and right pivotal dimensions \cite[Lem.~2.12]{1808.00323}.
The left and right $\varphi^{\cap x,y}$ pivotal dimensions on $\scrC$, denoted $\dim^{\cap x,y}_{L/R}$, are related to the left and right $\varphi$ pivotal dimensions, denoted $\dim^\varphi_{L/R}$, as follows:
\begin{equation}
\label{eq:LopsidedDimensions}
(\dim_L^{\cap x,y}(c) , \dim_R^{\cap x,y}(c))
=
\begin{cases}
(\dim_L^{\varphi}(c) , \dim_R^{\varphi}(c))
&
\text{if }c\in \scrC_{00}
\\
(xy^{-1}\dim_L^{\varphi}(c) , yx^{-1} \dim_R^{\varphi}(c))
&
\text{if }c\in \scrC_{01}
\\
(yx^{-1}\dim_L^{\varphi}(c) , xy^{-1}\dim_R^{\varphi}(c))
&
\text{if }c\in \scrC_{10}
\\
(\dim_L^{\varphi}(c) , \dim_R^{\varphi}(c))
&
\text{if }c\in \scrC_{11}
\end{cases} 
\end{equation}
Notice that we may write \eqref{eq:LopsidedDimensions} as simply one equation:
$$
(\dim_L^{\cap x,y}(c) , \dim_R^{\cap x,y}(c))
=
(x^jx^{-i}y^iy^{-j}\dim_L^{\varphi}(c) , x^ix^{-j}y^jy^{-i}\dim_R^{\varphi}(c))
\qquad\qquad
\forall c\in \scrC_{ij}.
$$

\begin{defn}
\label{defn:LopsidedShadedPlanarAlgebra}
Suppose $\cP_\bullet$ is a semisimple shaded planar algebra in which the shaded/unshaded closed loops are multiplicative scalars $\delta_\pm \in \cP_{0,\pm}$ respectively.
We call $\cP_\bullet$ \emph{lopsided} if $\delta_+ = 1$.
\end{defn}

Given a semisimple shaded planar algebra $\cP_\bullet$ with scalar loop moduli $\delta_\pm$ as in Definition \ref{defn:LopsidedShadedPlanarAlgebra}, we can always obtain a lopsided planar algebra $\cP_\bullet^{\text{lopsided}}:=\cP_\bullet^{\cap \delta_+, 1}$.
Notice that the shaded/unshaded loop moduli in $\cP_\bullet^{\text{lopsided}}$ are now $1$ and $\delta_+\delta_-$ respectively.

\begin{example}[{\cite[\S1.1]{MR3254427}}]
Let $\cG_\bullet$ be the graph planar algebra of a finite bipartite graph $\Gamma=(V_+,V_-,E)$, whose shaded and unshaded loop moduli are both $\delta = \|\Gamma\|$.
The \emph{lopsided} graph planar algebra is $\cG_\bullet^{\text{lopsided}}:=\cG_\bullet^{\cap \delta, 1}$.
Notice that the lopsided pivotal structure is obtained from the standard pivotal structure by \emph{only} rescaling cups and caps which are shaded above by a multiplicative factor of $\delta^{\pm 1}$, where the sign is the sign of the critical point ($+1$ for caps and $-1$ for cups).
\end{example}

\begin{warn}
The corresponding projection unitary multifusion category of $\cG_\bullet$ is $\End^\dag(\Hilb^{|V_+|+|V_-|})$, which is equipped with the standard unitary dual functor $\vee_{\text{standard}}$ with respect to the object $X$ representing $\Gamma$.
The lopsided pivotal structure on $\End^\dag(\Hilb^{|V_+|+|V_-|})$ induced by $\cG_\bullet^{\text{lopsided}}$ is not unitary as noted in the first paragraph of \cite[\S1.1]{MR3254427}, as $y^{-1} =1\neq \delta=\overline{x}$.
However, it is computationally easier to work with the non-unitary lopsided pivotal structure as introducing square roots increases the degree of the number fields involved.
Moreover, by \cite{MR3254427}, one can pass back and forth between the non-unitary lopsided convention and the unitary standard convention, so we do not lose any examples.
\end{warn}

%%%%%%%%%%%%%%%%%%%%%%%%%%%%%%%%%%%%%%%%%%%%%%%%%%%%%%%%%%%%%%%
\subsection{The Extended Haagerup subfactor planar algebra}
\label{sec:ExtendedHaagerupSubfactor}

The Extended Haagerup subfactor planar algebra $\cE\cH_\bullet$ is a shaded planar $\dag$-algebra, generated by an 8-box called $S$ which satisfies the relations given below.

The presentation given in \cite{MR2979509} uses the spherical pivotal structure, and here we also give a presentation with the lopsided pivotal structure, as this is necessary for computations later. 
The translation follows the discussion on p. 3 of \cite{MR3254427}.
\begin{itemize}
\item 
\emph{modulus:}
With $[2]$ the largest root of $x^6 - 8 x^4 + 17 x^2 - 5 = 0$, approximately 2.09218,
in the lopsided pivotal structure we have the shaded loop equal to 1 and the unshaded loop equal to $[2]^2$, while in the spherical pivotal structure both loops are equal to $[2]$.

(In the remainder of these formulas, coefficients are given using quantum numbers defined in the usual way, $[n] = \frac{q^n - q^{-n}}{q-q^{-1}}$.)
\item
\emph{self-adjoint:}
$S=S^*$
\item
\emph{rotational eigenvector:}
$ \begin{tikzpicture}[annular]
	\clip (0,0) circle (2cm);

	\filldraw[shaded] (158:4cm)--(0,0)--(112:4cm);
	\filldraw[shaded] (-158:4cm)--(0,0)--(-112:4cm);
	
	\draw[shaded] (68:4cm)--(0,0)--(-68:4cm)--(0:10cm);
	
	\draw[ultra thick] (0,0) circle (2cm);
	
	\node (T) at (0,0) [circle, fill=white, thick, draw] {$S$};
	\node at (T.180) [left] {$\star$};
	\node at (-90:2cm) [above] {$\star$};
	\node at (0:1cm) {$\cdot$};
	\node at (20:1cm) {$\cdot$};
	\node at (-20:1cm) {$\cdot$};
\end{tikzpicture}
= - S
$
\item
\emph{uncappable:}
$
\begin{tikzpicture}[annular]
	\clip (0,0) circle (2cm);

	\filldraw[shaded] (0,0) .. controls ++(170:2cm) and ++(100:2cm) .. (0,0);
	\filldraw[shaded] (-158:4cm)--(0,0)--(-112:4cm);

	\draw[shaded] (68:4cm)--(0,0)--(-68:4cm)--(0:10cm);
	
	\draw[ultra thick] (0,0) circle (2cm);
	
	\node (T) at (0,0) [circle, fill=white, thick, draw] {$S$};
	\node at (T.180) [left] {$\star$};
	\node at (180:2cm) [right] {$\star$};
	\node at (0:1cm) {$\cdot$};
	\node at (20:1cm) {$\cdot$};
	\node at (-20:1cm) {$\cdot$};
	
\end{tikzpicture}
=0
$
and
$
\begin{tikzpicture}[annular]
	\clip (0,0) circle (2cm);

	\filldraw[shaded] (158:4cm) -- (0,0) .. controls ++(130:2cm) and ++(50:2cm) .. (0,0)--(-68:4cm) arc (-68:158:4cm);
	\filldraw[shaded] (-158:4cm)--(0,0)--(-112:4cm);

	\draw[ultra thick] (0,0) circle (2cm);
	
	\node (T) at (0,0) [circle, fill=white, thick, draw] {$S$};
	\node at (T.180) [left] {$\star$};
	\node at (180:2cm) [right] {$\star$};
	\node at (0:1cm) {$\cdot$};
	\node at (20:1cm) {$\cdot$};
	\node at (-20:1cm) {$\cdot$};
\end{tikzpicture}
=
0
$ 
\\
(and in combination with rotation, all placements of a cap on a generator $S$ are zero).
\item
\emph{multiplication relation:}
$S^2 =
\begin{tikzpicture}[baseline=-.1cm]
	\draw (-1,0)--(3,0);
	\node (T1) at (0,0) [circle, fill=white, thick, draw] {$S$};
	\node (T2) at (2,0) [circle, fill=white, thick, draw] {$S$};
	\node [above=1pt] at (T1.north) {$\star$};
	\node [above=1pt] at (T2.north) {$\star$};
	\node at (-.5,0) [above] {\footnotesize$8$};
	\node at (1,0) [above] {\footnotesize$8$};
	\node at (2.5,0) [above] {\footnotesize$8$};	
\end{tikzpicture}
= \jw{8}$
\item
\emph{one strand jellyfish relation:}
$$
	\scalebox{0.8}{
		\begin{tikzpicture}[STrain]
			\RainbowOne;
		        \draw (0,0)--(0,-0.5);
		        \node[anchor=west] at (0,-0.35) {\footnotesize$18$};
			\JWPlusTwo;
		\end{tikzpicture}
	}
%	        &
	         = \alpha
	\scalebox{0.8}{
		\begin{tikzpicture}[STrain]
			\STrainStrings{$9$}{$9$} \STrainOne
        			\draw (0,0)--(0,-0.5);
        			\node[anchor=west] at (0,-0.35) {\footnotesize$18$};
			\JWPlusTwo
		\end{tikzpicture}
	},
$$
with $\alpha = i \dfrac{\sqrt{[8][10]}}{[2]^4 [9]}$ in the lopsided case, or $\alpha = i \dfrac{\sqrt{[8][10]}}{[9]}$ in the spherical case.
\item 
\emph{two strand jellyfish relation:}
$$
	\scalebox{0.8}
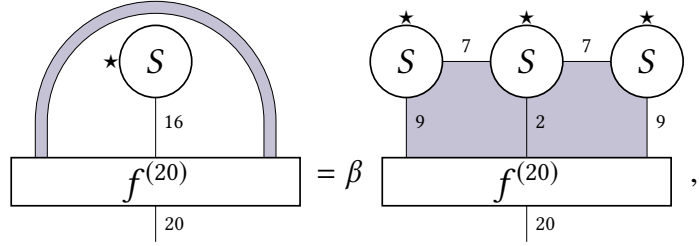
{
		\begin{tikzpicture}[STrain]
			\RainbowTwo
        			\draw (0,0)--(0,-0.5);
        			\node[anchor=west] at (0,-0.35) {\footnotesize$20$};
			\JWPlusFour
		\end{tikzpicture}
	} 
	         = \beta
        \scalebox{0.8}{
	        \begin{tikzpicture}[STrain]
			\STrainThreeStrings{$9$}{$2$}{$9$} \STrainOneOne
		        \draw (0,0)--(0,-0.5);
		        \node[anchor=west] at (0,-0.35) {\footnotesize$20$};
			\JWPlusFour
		\end{tikzpicture}
        },
$$
with $\beta = \frac{[20]}{[2]^6[9][10]}$ in the lopsided case, or $\beta = \frac{[2][20]}{[9][10]}$ in the spherical case.
\end{itemize}

These relations are sufficient to evaluate all closed diagrams in $S$, via the `jellyfish algorithm' which pulls copies of $S$ to the exterior and then cancels them in pairs.  
Note that in addition to the above relations, to give a complete description of the Extended Haagerup subfactor planar algebra we also quotient by the negligible elements.  
Moreover, there is a non-zero representation of this abstract planar $\dag$-algebra in the graph planar algebra of the principal graph, which proves the existence of the Extended Haagerup subfactor planar algebra.
We refer the reader to \cite{MR2979509} for more details.

Below we will use the constant $\lambda$ for the largest purely imaginary root of $\lambda^6 + 2 \lambda^4 - 3 \lambda^2 - 5 = 0$, approximately $1.54i$.

\begin{lem}[Variation of {\cite[Prop.~3.12]{MR2979509}}]
\label{lem:EHMomementsIdentifyEH}
Let $\Gamma$ be a finite bipartite graph with norm $[2]$ as above.
Suppose $S\in \cG\cP\cA(\Gamma)_{8,+}$ is a self-adjoint, uncappable, rotational eigenvector with eigenvalue $-1$, and has the Extended Haagerup moments
\begin{equation}
\label{eq:ExtendedHaagerupMoments}
\tr(S^2) = [9] 
\qquad\qquad
\tr(S^3) = 0
\qquad\qquad
\tr(S^4) = [9]
\qquad\qquad
\tr(\rho^{1/2}(S)^3) = i \frac{[18]}{\sqrt{[8][10]}}.
\end{equation}
Let $\cP\cA(S)_\bullet$ be the planar $\dag$-subalgebra of $\cG\cP\cA(\Gamma)_\bullet$ generated by $S$.
Then $\cP\cA(S)_\bullet\cong \cE\cH_\bullet$.
\end{lem}
\begin{proof}
The proof that $\cP\cA(S)_\bullet$ is an irreducible subfactor planar algebra with principal graph $\Gamma_2$ from Proposition \ref{prop:fusion-graphs} is identical to the proof of \cite[Prop.~3.12]{MR2979509}, which never used that $\Gamma=\Gamma_2$.
The final claim that $\cP\cA(S)_\bullet\cong \cE\cH_\bullet$ follows by uniqueness of the Extended Haagerup subfactor planar algebra \cite{MR1317352}.
\end{proof}

\begin{remark}
\label{rem:EHInsideUnitaryShadedPA}
In fact, Lemma \ref{lem:EHMomementsIdentifyEH} holds if we replace $\cG\cP\cA(\Gamma)_\bullet$ with any unitary shaded planar algebra $\cP_\bullet$ with a spherical faithful tracial state $\psi_\pm$ on $\cP_{0,\pm}$ (see Remark \ref{rem:GPA-NotSpherical} or \cite[\S5]{1808.00323})
whose shaded and unshaded loop values are both $[2]$ as above. 
\end{remark}

\begin{remark}\label{rem:UnitaryEssential}
We want to emphasize that the proof  \cite[Prop.~3.12]{MR2979509} uses unitarity in an essential way.  The key step, following \cite{MR2679382}, is that using only the moments you can prove the Jellyfish relations by checking that the inner product of each relation with itself is $0$.
\end{remark}

\begin{prop}
\label{prop:maps-out-of-EH}
Suppose $\cP_\bullet$ is any unitary shaded planar algebra with a spherical faithful tracial state $\psi_\pm$ on $\cP_{0,\pm}$ whose shaded and unshaded loop values are both $[2]$ as above. 
Planar $\dag$-algebra morphisms
$\cE\cH_\bullet \to \cP_\bullet$ are in bijection with choices of self-adjoint uncappable elements $S' \in \cP_{8,+}$ with rotational eigenvalue $-1$, satisfying
\begin{align}
\label{eq:S2}
S'^2 & = f^{(8)} \\
\rho^{-1/2}(S')^2 & = 
\frac{2}{5} \left(-\lambda^5 - 2 \lambda^3 + 3 \lambda\right) [2]^{-1} \rho^{-1/2}(S') + \left(\lambda^2 - 2\right) [2]^{-2} f^{(8)} \notag \\
& = i (\check{r}^{1/2}-\check{r}^{-1/2}) \rho^{-1/2}(S) - f^{(8)}
\label{eq:S2c-spherical}
\end{align}
where $\check{r} = \frac{[10]}{[8]}$.
\end{prop}
\begin{proof}
By Lemma \ref{lem:EHMomementsIdentifyEH} and Remark \ref{rem:EHInsideUnitaryShadedPA}, we need only show that $S'$ satisfies the Extended Haagerup moments \eqref{eq:ExtendedHaagerupMoments} if and only if 
\eqref{eq:S2} and \eqref{eq:S2c-spherical} hold.
Clearly if \eqref{eq:S2} and \eqref{eq:S2c-spherical} hold, then $S'$ satisfies the Extended Haagerup moments \eqref{eq:ExtendedHaagerupMoments}.
Conversely, suppose $S'$ satisfies the Extended Haagerup moments \eqref{eq:ExtendedHaagerupMoments}.
By \cite[Prop.~3.7]{MR2979509}, $S'$ so \eqref{eq:S2} holds, together with the one and two strand jellyfish relations.
As the principal graphs must be those of Extended Haagerup, again by Lemma \ref{lem:EHMomementsIdentifyEH}, we can apply \cite[Eq.~(3.3)]{MR2979509} (essentially from \cite{MR2972458}), which gives \eqref{eq:S2c-spherical} above for $S'$.
\end{proof}

\begin{cor}
\label{cor:LopsidedHomomorphismsOutOfEH}
Planar algebra homomorphisms $\cE\cH_\bullet^{\cap \delta, 1} \to \cP_\bullet^{\cap \delta, 1}$ between the lopsided planar algebras are in bijection with choices of uncappable elements $S'\in \cP_{8,+}$ with rotational eigenvalue $-1$ satisfying \eqref{eq:S2} and
\begin{equation}
\label{eq:S2c}
\rho^{-1/2}(S)^2 = 
\frac{2}{5} \left(-\lambda^5 - 2 \lambda^3 + 3 \lambda\right) \rho^{-1/2}(S) + \left(\lambda^2 - 2\right) f^{(8)}
\end{equation}
rather than \eqref{eq:S2c-spherical}.
\end{cor}

%%%%%%%%%%%%%%%%%%%%%%%%%%%%%%%%%%%%%%%%%%%%%%%%%%%%%%%%%%%%%%%%%%%%%%%%%%%
\subsection{\texorpdfstring{$\mathcal{EH}_3$}{EH3}}

In this section, we use Corollary \ref{cor:LopsidedHomomorphismsOutOfEH} to find $\cE\cH_\bullet$ in the graph planar algebra of the bipartite graph
$$
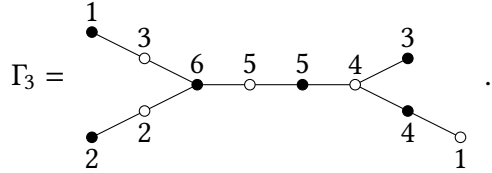
\Gamma_3 
= 
\begin{tikzpicture}[baseline, scale=.7]
	\draw (-2,-1) -- (0,0);
	\draw (-2,1)--(0,0);
	\draw (0,0)--(3,0);
	\draw (3,0)--(4,.5);
	\draw (3,0)--(5,-1);
	\filldraw [fill]  (-2,-1) circle (1mm) node[below] {$2$};
	\filldraw [fill]  (-2,1) circle (1mm) node[above] {$1$};
	\filldraw [fill=white] (-1,-.5) circle (1mm) node[below] {$2$};
	\filldraw [fill=white] (-1,.5) circle (1mm) node[above] {$3$};
	\filldraw [fill]  (0,0) circle (1mm)node[above] {$6$};
	\filldraw [fill=white] (1,0) circle (1mm)node[above] {$5$};
	\filldraw [fill]  (2,0) circle (1mm) node[above] {$5$};
	\filldraw [fill=white] (3,0) circle (1mm) node[above] {$4$};
	\filldraw [fill]  (4,-.5) circle (1mm) node[below] {$4$};
	\filldraw [fill]  (4,.5) circle (1mm) node[above] {$3$};
	\filldraw [fill=white] (5,-1) circle (1mm) node[below] {$1$};
\end{tikzpicture}\,.
$$
The lowest weight eigenspace with 16 boundary points, and rotational eigenvalue -1, is 18 dimensional. 
An element in this eigenspace is determined by its values $c_i$ on the following loops $\ell_i$ based at unshaded/even vertices:
\begin{align*}
\ell_{1} & = 5655434556554345 & \ell_{2} & = 5543455622263626 \displaybreak[1] \\
\ell_{3} & = 4556265626365543 & \ell_{4} & = 5636265626554345 \displaybreak[1] \\
\ell_{5} & = 2636265626362636 & \ell_{6} & = 4556263626265543 \displaybreak[1] \\
\ell_{7} & = 4556222655554345 & \ell_{8} & = 2636263626263622 \displaybreak[1] \\
\ell_{9} & = 4345562626313655 & \ell_{10} & = 4345563626363655 \displaybreak[1] \\
\ell_{11} & = 4556313655454345 & \ell_{12} & = 5631365636554345 \displaybreak[1] \\
\ell_{13} & = 2636265626263136 & \ell_{14} & = 2226362631362636 \displaybreak[1] \\
\ell_{15} & = 2631362636362636 & \ell_{16} & = 5631362626554345 \displaybreak[1] \\
\ell_{17} & = 2631362626362636 & \ell_{18} & = 2226313622263136
\end{align*}

There are exactly two solutions to the equations, and these are related by $S' = - \overline{S}$, or by applying the unique graph automorphism.
The element $S$ has coefficients in $\mathbb{Q}(\mu)$, where $\mu$ is the root of $\mu ^{12}+718 \mu ^{10}+679145 \mu ^8+43340550 \mu ^6+43588750 \mu ^4-625000 \mu ^2+390625 = 0$ which is approximately $-0.229025 - 0.202916 i$. The values of $c_i$ written as polynomials in $\mu$ are quite horrific (coefficients rational numbers with numerators and denominators having up to 30 digits), so we instead express them directly in terms of their minimal polynomials. (The associated Mathematica notebook contains their values in the number field.) We use the notation $\lambda^x_{a_0,\ldots,a_k}$ to denote the root of $a_0 + a_1 \lambda + \cdots + a_k \lambda^k = 0$ which is closest to the approximate number $x$ (and we're careful to write $x$ with enough precision that this is unambiguous).
\begin{align*}
c_{1} & = \lambda_{1,0,112942,0,-1940695,0,-125}^{(0.0080256 i)} \displaybreak[1] \\
c_{2} & = \lambda_{\begin{subarray}{l}625,0,58550,0,1877265,0,24363782,0,119192086,0,-4303080,0,172225\\\mbox{}\end{subarray}}^{(0.1672-0.0995 i)} \displaybreak[1] \\
c_{3} & = \lambda_{\begin{subarray}{l}9765625,0,822187500,0,5692096250,0,704926450,0,34457185,0,774362,0,6889\\\mbox{}\end{subarray}}^{(0.03538+0.16258 i)} \displaybreak[1] \\
c_{4} & = \lambda_{\begin{subarray}{l}15625,0,47736250,0,11814953125,0,1219921150,0,49538050,0,927928,0,6889\\\mbox{}\end{subarray}}^{(0.03272-0.15038 i)} \displaybreak[1] \\
c_{5} & = \lambda_{25,0,4235,0,26582,0,-1}^{(0.0061335)} \displaybreak[1] \\
c_{6} & = \lambda_{\begin{subarray}{l}9765625,0,100312500,0,287121250,0,166019450,0,31036785,0,-421822,0,6889\\\mbox{}\end{subarray}}^{(0.10306+0.06133 i)} \displaybreak[1] \\
c_{7} & = \lambda_{5,0,183,0,-422,0,-1}^{(-0.048654 i)} \displaybreak[1] \\
c_{8} & = \lambda_{125,0,1490,0,137,0,-5}^{(-0.1672)} \displaybreak[1] \\
c_{9} & = \lambda_{\begin{subarray}{l}625,0,30300,0,164710,0,6266122,0,18530421,0,-2194130,0,70225\\\mbox{}\end{subarray}}^{(0.24287-0.03754 i)} \displaybreak[1] \\
c_{10} & = \lambda_{\begin{subarray}{l}15625,0,1045000,0,25515750,0,222706550,0,624079625,0,-1976682,0,6889\\\mbox{}\end{subarray}}^{(0.049520-0.029468 i)} \displaybreak[1] \\
c_{11} & = \lambda_{125,0,-205,0,-362,0,-1}^{(0.05260 i)} \displaybreak[1] \\
c_{12} & = \lambda_{\begin{subarray}{l}15625,0,5448750,0,470120625,0,259808550,0,42457870,0,-493928,0,6889\\\mbox{}\end{subarray}}^{(-0.09532-0.05673 i)} \displaybreak[1] \\
c_{13} & = \lambda_{\begin{subarray}{l}625,0,17950,0,679145,0,1733622,0,69742,0,-40,0,1\\\mbox{}\end{subarray}}^{(-0.045805+0.040583 i)} \displaybreak[1] \\
c_{14} & = \lambda_{25,0,-622,0,-543,0,-5}^{(0.09647 i)} \displaybreak[1] \\
c_{15} & = \lambda_{625,0,17450,0,365,0,-1}^{(-0.049520)} \displaybreak[1] \\
c_{16} & = \lambda_{\begin{subarray}{l}15625,0,3842500,0,55831750,0,-4013550,0,7389525,0,-273698,0,2809\\\mbox{}\end{subarray}}^{(-0.138433-0.021397 i)} \displaybreak[1] \\
c_{17} & = \lambda_{\begin{subarray}{l}390625,0,24156250,0,2220203125,0,1165172950,0,9182770,0,-608,0,1\\\mbox{}\end{subarray}}^{(0.013563+0.012017 i)} \displaybreak[1] \\
c_{18} & = \lambda_{5,0,-222,0,-279,0,-25}^{(-0.3117 i)} \displaybreak[1]
\end{align*}

It is then a simple matter to directly verify the equations (this takes less than a minute on a modern CPU); this verification can be found in \texttt{module-GPAs-EH3.nb}.

%%%%%%%%%%%%%%%%%%%%%%%%%%%%%%%%%%%%%%%%%%%%%%%%%%%%%%%%%%%%%%%%%%%%%%%%%%%

\subsection{\texorpdfstring{$\mathcal{EH}_4$}{EH4}}
In this section, we use Corollary \ref{cor:LopsidedHomomorphismsOutOfEH} to find $\cE\cH_\bullet$ in the graph planar algebra of the bipartite graph
$$
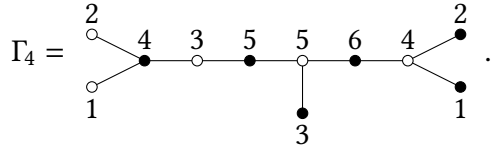
\Gamma_4
= 
\begin{tikzpicture}[baseline, scale=.7]
	\draw (-1,-.5) -- (0,0);
	\draw (-1,.5)--(0,0);
	\draw (0,0)--(5,0);
	\draw (5,0)--(6,.5);
	\draw (5,0)--(6,-.5);
	\draw (3,0)--(3,-1);
	\filldraw [fill=white] (-1,-.5) circle (1mm) node[below] {$1$};
	\filldraw [fill=white] (-1,.5) circle (1mm) node[above] {$2$};
	\filldraw [fill]  (0,0) circle (1mm) node[above] {$4$};
	\filldraw [fill=white] (1,0) circle (1mm) node[above] {$3$};
	\filldraw [fill]  (2,0) circle (1mm) node[above] {$5$};
	\filldraw [fill=white] (3,0) circle (1mm) node[above] {$5$};
	\filldraw [fill] (3,-1) circle (1mm) node[below] {$3$};
	\filldraw [fill]  (4,0) circle (1mm) node[above] {$6$} ;
	\filldraw [fill=white] (5,0) circle (1mm) node[above] {$4$};
	\filldraw [fill]  (6,-.5) circle (1mm) node[below] {$1$};
	\filldraw [fill]  (6,.5) circle (1mm) node[above] {$2$};
\end{tikzpicture}\,.
$$
The lowest weight eigenspace with 16 boundary points, and rotational eigenvalue -1, is 20 dimensional. 
An element in this eigenspace is determined by its values $c_i$ on the following loops $\ell_i$ based at unshaded/even vertices:
\begin{align*}
\ell_{1} & = 3553565355553424 & \ell_{2} & = 5653555356535653 \displaybreak[1] \\
\ell_{3} & = 3555535646553414 & \ell_{4} & = 5646535653564653 \displaybreak[1] \\
\ell_{5} & = 4146535534243556 & \ell_{6} & = 5553564146535653 \displaybreak[1] \\
\ell_{7} & = 4146535641465356 & \ell_{8} & = 4246553424355356 \displaybreak[1] \\
\ell_{9} & = 4246535534243556 & \ell_{10} & = 5553564246535653 \displaybreak[1] \\
\ell_{11} & = 5646535646424653 & \ell_{12} & = 4146535642465356 \displaybreak[1] \\
\ell_{13} & = 4246535642465356 & \ell_{14} & = 3556424146553424 \displaybreak[1] \\
\ell_{15} & = 5646535653564146 & \ell_{16} & = 5356424146535646 \displaybreak[1] \\
\ell_{17} & = 5642414646535653 & \ell_{18} & = 5641424646414653 \displaybreak[1] \\
\ell_{19} & = 5642414246424653 & \ell_{20} & = 4142414641424142
\end{align*}

There are four solutions to these equations, and the graph automorphism group acts freely and transitively on them. The solutions have
coefficients in $\mathbb{Q}(\mu)$, where $\mu$ is the root of $\mu ^{12}-74510 \mu ^{10}+1753550625 \mu ^8-8889717968750 \mu ^6+23050129394531250 \mu ^4+42850952148437500 \mu ^2+95367431640625 = 0$ which is approximately $-0.0472042 i$. One of the four solutions has coefficients:
\begin{align*}
c_{1} & = \lambda_{\begin{subarray}{l}3125,49250,56580,53520,1597,-200,53\\\mbox{}\end{subarray}}^{(0.04828+0.07374 i)} \displaybreak[1] \\
c_{2} & = \lambda_{\begin{subarray}{l}125,0,-1285982,0,-1789244179,0,-2699449\\\mbox{}\end{subarray}}^{(-0.038842 i)} \displaybreak[1] \\
c_{3} & = \lambda_{\begin{subarray}{l}3125,-18750,31575,-20540,4443,186,25\\\mbox{}\end{subarray}}^{(-0.02632-0.06233 i)} \displaybreak[1] \\
c_{4} & = \lambda_{5,0,2882,0,-249683,0,-625}^{(0.05003 i)} \displaybreak[1] \\
c_{5} & = \lambda_{\begin{subarray}{l}48828125,-195312500,386718750,-344687500,126334375,-3725000,-6388300,43560,201947,23420,1230,36,1\\\mbox{}\end{subarray}}^{(-0.063152-0.039778 i)} \displaybreak[1] \\
c_{6} & = \lambda_{\begin{subarray}{l}125,0,150048,0,92084512,0,8056764288,0,285286080768,0,296306688,0,20480\\\mbox{}\end{subarray}}^{(0.0086287 i)} \displaybreak[1] \\
c_{7} & = \lambda_{\begin{subarray}{l}125,0,197208,0,81755664,0,-661557632,0,3025487360,0,515469312,0,20480\\\mbox{}\end{subarray}}^{(0.40535 i)} \displaybreak[1] \\
c_{8} & = \lambda_{\begin{subarray}{l}125,3750,27250,-64700,141035,-2100,103848,105108,29242,2034,-122,-10,1\\\mbox{}\end{subarray}}^{(-0.30264+0.07970 i)} \displaybreak[1] \\
c_{9} & = \lambda_{\begin{subarray}{l}48828125,-195312500,386718750,-344687500,126334375,-3725000,-6388300,43560,201947,23420,1230,36,1\\\mbox{}\end{subarray}}^{(0.38771+0.10211 i)} \displaybreak[1] \\
c_{10} & = \lambda_{\begin{subarray}{l}125,0,150048,0,92084512,0,8056764288,0,285286080768,0,296306688,0,20480\\\mbox{}\end{subarray}}^{(-0.031052 i)} \displaybreak[1] \\
c_{11} & = \lambda_{5,0,1282,0,-4739,0,-5}^{(0.032477 i)} \displaybreak[1] \\
c_{12} & = \lambda_{1,0,-293,0,-118,0,-5}^{(-0.2194 i)} \displaybreak[1] \\
c_{13} & = \lambda_{\begin{subarray}{l}125,0,197208,0,81755664,0,-661557632,0,3025487360,0,515469312,0,20480\\\mbox{}\end{subarray}}^{(0.0063040 i)} \displaybreak[1] \\
c_{14} & = \lambda_{\begin{subarray}{l}15625,-37500,2375,-850,-832,-156,13\\\mbox{}\end{subarray}}^{(-0.1850-0.1190 i)} \displaybreak[1] \\
c_{15} & = \lambda_{\begin{subarray}{l}125,0,-9582,0,821981,0,-28226758,0,2200643514,0,16166708,0,26645\\\mbox{}\end{subarray}}^{(-0.069636 i)} \displaybreak[1] \\
c_{16} & = \lambda_{125,0,-2047,0,-50809,0,-5}^{(0.0099201 i)} \displaybreak[1] \\
c_{17} & = \lambda_{\begin{subarray}{l}125,0,455738,0,13472487051,0,-26481195508,0,28428109059,0,58134938,0,26645\\\mbox{}\end{subarray}}^{(-0.026344 i)} \displaybreak[1] \\
c_{18} & = \lambda_{\begin{subarray}{l}625,0,-74510,0,2805681,0,-22757678,0,94413330,0,280828,0,1\\\mbox{}\end{subarray}}^{(-0.00188817 i)} \displaybreak[1] \\
c_{19} & = \lambda_{\begin{subarray}{l}625,0,-74510,0,2805681,0,-22757678,0,94413330,0,280828,0,1\\\mbox{}\end{subarray}}^{(0.0544863 i)} \displaybreak[1] \\
c_{20} & = \lambda_{1,0,4982,0,-2155,0,-25}^{(0.1063 i)}
\end{align*}
Again, it is easy to verify this gives a solution, shown in \texttt{module-GPAs-EH4.nb}.

%%%%%%%%%%%%%%%%%%%%%%%%%%%%%%%%%%%%%%%%%%%%%%%%%%%%%%%%%%%%
\appendix
%auto-ignore
%this ensures the arxiv doesn't try to start TeXing here.
%!TEX root =../EH3.tex

%%%%%%%%%%%%%%%%%%%%%%%%%%%%%%%%%%%%%%%%%%%%%%%%%%%%%%%%%%%%
%%%%%%%%%%%%%%%%%%%%%%%%%%%%%%%%%%%%%%%%%%%%%%%%%%%%%%%%%%%%
%%%%%%%%%%%%%%%%%%%%%%%%%%%%%%%%%%%%%%%%%%%%%%%%%%%%%%%%%%%%
\section{Constructing \texorpdfstring{$EH3$}{EH3} via a Q-system}

We now show that $1\oplus \jw{6}\in \cE\cH_2$ can be endowed with the structure of a Q-system.
To do so, we prove a result similar to \cite[Lemma 3.3]{MR2418197}.

\begin{defn}[{\cite[Def.~3.8]{MR1444286,MR3308880}}]
A Q-system in a rigid $\rm C^*$-tensor category $\scrC$ is an algebra object $(A,\mu,i)$ which satisfies the following properties:
\begin{itemize}
\item ($\rm C^*$-Frobenius) 
$(\id_A \otimes \mu)\circ (\mu^*\otimes \id_A) = \mu^*\circ \mu = (\mu\otimes \id_A)\circ (\id_A \otimes \mu^*)$,
\item (special)
$\mu\circ \mu^*$ is a (non-zero) multiple of $\id_A$, and
\item (standard) 
$i^*\circ i = \sqrt{\dim_\scrC(A)} \id_{1_\scrC}$ and $\mu\circ \mu^* = \sqrt{\dim_\scrC(A)} \id_A$.
\end{itemize}
\end{defn}

\begin{prop}
\label{prop:CanonicalQSystem}
Suppose $\sigma$ is a symmetrically self-dual object in a unitary fusion category $\scrC$ with $\dim(\sigma)>1$.
If $T \in \Hom(\sigma\otimes \sigma \to \sigma)$ and $b>0$ such that
\begin{itemize}
\item
(bigon)
$
\begin{tikzpicture}[baseline=-.6cm]
	\node (T1) [circle, fill=white, thick, draw]  at (0,0) {$T$};
	\node (T2) [circle, fill=white, thick, draw]  at (0,-1) {$T$};
	\node [above=1pt] at (T1.west) {$\star$};
	\node [below=1pt] at (T2.west) {$\star$};
	\draw (T1) [in=45,out=-45] to (T2);
	\draw (T1) [in=135,out=-135] to (T2);
	\draw (T1) -- (0,1);
	\draw (T2) -- (0,-2);
	\node at (-.6,-.5) {\scriptsize{$\sigma$}};
	\node at (.6,-.5) {\scriptsize{$\sigma$}};
	\node at (0,1.2) {\scriptsize{$\sigma$}};
	\node at (0,-2.2) {\scriptsize{$\sigma$}};
\end{tikzpicture}
=
b
\begin{tikzpicture}[baseline=-.1cm]
	\draw (0,-1) -- (0,1);
	\node at (0,-1.2) {\scriptsize{$\sigma$}};
\end{tikzpicture}
$
\item
(rotational invariance)
$
\begin{tikzpicture}[baseline=-.1cm]
	\node (S) [circle, fill=white, thick, draw]  at (0,0) {$T$};
	\node [below=1pt] at (S.west) {$\star$};
	\draw (S) -- (0,1);
	\draw (S) -- (-.6,-1);
	\draw (S) -- (.6,-1);
	\node at (-.6,-1.2) {\scriptsize{$\sigma$}};
	\node at (.6,-1.2) {\scriptsize{$\sigma$}};
	\node at (0,1.2) {\scriptsize{$\sigma$}};
\end{tikzpicture}
=
\begin{tikzpicture}[baseline=-.1cm]
	\node (S) [circle, fill=white, thick, draw]  at (0,0) {$T$};
	\node [below=1pt] at (S.east) {$\star$};
	\draw (S) -- (0,1);
	\draw (S) -- (-.6,-1);
	\draw (S) -- (.6,-1);
	\node at (-.6,-1.2) {\scriptsize{$\sigma$}};
	\node at (.6,-1.2) {\scriptsize{$\sigma$}};
	\node at (0,1.2) {\scriptsize{$\sigma$}};
\end{tikzpicture}
$
\item
(self-adjoint)
$
\begin{tikzpicture}[baseline=-.1cm, yscale=-1]
	\node (S) [circle, fill=white, thick, draw]  at (0,0) {$T^*$};
	\node [below=1pt] at (S.west) {$\star$};
	\draw (S) -- (0,1);
	\draw (S) -- (-.6,-1);
	\draw (S) -- (.6,-1);
	\node at (-.6,-1.2) {\scriptsize{$\sigma$}};
	\node at (.6,-1.2) {\scriptsize{$\sigma$}};
	\node at (0,1.2) {\scriptsize{$\sigma$}};
\end{tikzpicture}
=
\begin{tikzpicture}[baseline=-.1cm]
	\node (S) [circle, fill=white, thick, draw]  at (0,0) {$T$};
	\node [below=1pt] at (S.west) {$\star$};
	\draw (S) -- (0,1);
	\draw (S) -- (-.6,-1);
	\draw (S) [out=-60,in=-90] to (1,-.2) -- (1,1);
	\node at (-.6,-1.2) {\scriptsize{$\sigma$}};
	\node at (1,1.2) {\scriptsize{$\sigma$}};
	\node at (0,1.2) {\scriptsize{$\sigma$}};
\end{tikzpicture}
$, and
\item
(I=H)
$
\displaystyle
\begin{tikzpicture}[baseline=-.4cm]
\node (T1) [circle, fill=white, thick, draw]  at (0,.2) {$T$};
\node (T2) [circle, fill=white, thick, draw]  at (.8,-1) {$T$};
\node [below=1pt] at (T1.west) {$\star$};
\node [below=1pt] at (T2.west) {$\star$};
\draw (T1) -- (0,1);
\draw (T1) -- (-1.2,-1.6);
\draw (T1) -- (T2);
\draw (T2) -- (.4,-1.6);
\draw (T2) -- (1.2,-1.6);
\node at (-1.2,-1.8) {\scriptsize{$\sigma$}};
\node at (.4,-1.8) {\scriptsize{$\sigma$}};
\node at (1.2,-1.8) {\scriptsize{$\sigma$}};
\node at (.4,-.6) {\scriptsize{$\sigma$}};
\node at (0,1.2) {\scriptsize{$\sigma$}};
\end{tikzpicture}
-
\begin{tikzpicture}[baseline=-.4cm, xscale=-1]
\node (T1) [circle, fill=white, thick, draw]  at (0,.2) {$T$};
\node (T2) [circle, fill=white, thick, draw]  at (.8,-1) {$T$};
\node [below=1pt] at (T1.west) {$\star$};
\node [below=1pt] at (T2.west) {$\star$};
\draw (T1) -- (0,1);
\draw (T1) -- (-1.2,-1.6);
\draw (T1) -- (T2);
\draw (T2) -- (.4,-1.6);
\draw (T2) -- (1.2,-1.6);
\node at (-1.2,-1.8) {\scriptsize{$\sigma$}};
\node at (.4,-1.8) {\scriptsize{$\sigma$}};
\node at (1.2,-1.8) {\scriptsize{$\sigma$}};
\node at (.4,-.6) {\scriptsize{$\sigma$}};
\node at (0,1.2) {\scriptsize{$\sigma$}};
\end{tikzpicture}
=
\frac{b}{\dim(\sigma)-1}
\left(
\begin{tikzpicture}[baseline=-.1cm]
	\draw (-.4,-1) -- (-.4,-.7) arc (180:0:.4cm) -- (.4,-1);
	\draw (.8,-1) .. controls ++(90:1cm) and ++(270:1cm) .. (.2,1);
	\node at (-.4,-1.2) {\scriptsize{$\sigma$}};
	\node at (.4,-1.2) {\scriptsize{$\sigma$}};
	\node at (.8,-1.2) {\scriptsize{$\sigma$}};
	\node at (.2,1.2) {\scriptsize{$\sigma$}};
\end{tikzpicture}
-
\begin{tikzpicture}[baseline=-.1cm]
	\draw (-.4,-1) -- (-.4,-.7) arc (180:0:.4cm) -- (.4,-1);
	\draw (-.8,-1) .. controls ++(90:1cm) and ++(270:1cm) .. (-.2,1);
	\node at (-.8,-1.2) {\scriptsize{$\sigma$}};
	\node at (-.4,-1.2) {\scriptsize{$\sigma$}};
	\node at (.4,-1.2) {\scriptsize{$\sigma$}};
	\node at (-.2,1.2) {\scriptsize{$\sigma$}};
\end{tikzpicture}
\right)
$
\end{itemize}
then $1\oplus \sigma$ can be canonically endowed with the structure of a Q-system.
If moreover $\dim(\Hom_\scrC(1, \sigma)) =0$, then $1\oplus \sigma$ is irreducible.
\end{prop}
\begin{remark}
If $\sigma\in\scrC$ is simple, then the existence of $b>0$ such that the bigon axiom holds for $T$ follows formally from the rotational invariance and self-adjoint axioms.
\end{remark}
\begin{proof}
Let $A=1\oplus \sigma$.
We need to construct maps $\mu: A\otimes A \to A$ and $i: 1\to A$ such that $(A,\mu,i)$ is a Q-system.
We define $\mu$ as follows: 
$$
\begin{array}{c|cc}
1 & 1 & \sigma
\\\hline
1 & \id_1 & 0
\\
\sigma & 0 &\ev_\sigma
\end{array}
\qquad\qquad
\begin{array}{c|cc}
\sigma & 1 & \sigma
\\\hline 
1 & 0 & \id_\sigma
\\
\sigma & \id_\sigma & c T
\end{array}
$$
We now solve for $c$ using associativity.
The only interesting diagram to check is when the boundaries are all $\sigma$, which simplifies to
$$
\displaystyle
\begin{tikzpicture}[baseline=-.4cm]
\node (T1) [circle, fill=white, thick, draw]  at (0,.2) {$T$};
\node (T2) [circle, fill=white, thick, draw]  at (.8,-1) {$T$};
\node [below=1pt] at (T1.west) {$\star$};
\node [below=1pt] at (T2.west) {$\star$};
\draw (T1) -- (0,1);
\draw (T1) -- (-1.2,-1.6);
\draw (T1) -- (T2);
\draw (T2) -- (.4,-1.6);
\draw (T2) -- (1.2,-1.6);
\node at (-1.2,-1.8) {\scriptsize{$\sigma$}};
\node at (.4,-1.8) {\scriptsize{$\sigma$}};
\node at (1.2,-1.8) {\scriptsize{$\sigma$}};
\node at (.4,-.6) {\scriptsize{$\sigma$}};
\node at (0,1.2) {\scriptsize{$\sigma$}};
\end{tikzpicture}
-
\begin{tikzpicture}[baseline=-.4cm, xscale=-1]
\node (T1) [circle, fill=white, thick, draw]  at (0,.2) {$T$};
\node (T2) [circle, fill=white, thick, draw]  at (.8,-1) {$T$};
\node [below=1pt] at (T1.west) {$\star$};
\node [below=1pt] at (T2.west) {$\star$};
\draw (T1) -- (0,1);
\draw (T1) -- (-1.2,-1.6);
\draw (T1) -- (T2);
\draw (T2) -- (.4,-1.6);
\draw (T2) -- (1.2,-1.6);
\node at (-1.2,-1.8) {\scriptsize{$\sigma$}};
\node at (.4,-1.8) {\scriptsize{$\sigma$}};
\node at (1.2,-1.8) {\scriptsize{$\sigma$}};
\node at (.4,-.6) {\scriptsize{$\sigma$}};
\node at (0,1.2) {\scriptsize{$\sigma$}};
\end{tikzpicture}
=
\frac{1}{c^2}
\left(
\begin{tikzpicture}[baseline=-.1cm]
	\draw (-.4,-1) -- (-.4,-.7) arc (180:0:.4cm) -- (.4,-1);
	\draw (.8,-1) .. controls ++(90:1cm) and ++(270:1cm) .. (.2,1);
	\node at (-.4,-1.2) {\scriptsize{$\sigma$}};
	\node at (.4,-1.2) {\scriptsize{$\sigma$}};
	\node at (.8,-1.2) {\scriptsize{$\sigma$}};
	\node at (.2,1.2) {\scriptsize{$\sigma$}};
\end{tikzpicture}
-
\begin{tikzpicture}[baseline=-.1cm]
	\draw (-.4,-1) -- (-.4,-.7) arc (180:0:.4cm) -- (.4,-1);
	\draw (-.8,-1) .. controls ++(90:1cm) and ++(270:1cm) .. (-.2,1);
	\node at (-.8,-1.2) {\scriptsize{$\sigma$}};
	\node at (-.4,-1.2) {\scriptsize{$\sigma$}};
	\node at (.4,-1.2) {\scriptsize{$\sigma$}};
	\node at (-.2,1.2) {\scriptsize{$\sigma$}};
\end{tikzpicture}
\right)
$$
Hence we must have $c = \pm \left(\frac{\dim(\sigma)-1}{b}\right)^{1/2}$.
Notice that these two choices for $c$ give algebras $A_\pm$ which are isomorphic via $(\id_1,-\id_\sigma): A_{\pm}\to A_{\mp}$.
Without loss of generality, we take $A = A_+$.

We see that $i = (\id_1, 0) : 1_\scrC \to A= 1_\scrC \oplus \sigma$ is the unique unit such that $(A, \mu, i)$ is an algebra.
By construction, $\mu$ is rotationally invariant and self-adjoint (as in hypotheses of this propostion).
Together with the facts that $(A,\mu,i)$ is an algebra and $(A, \mu^*, i^*)$ is a coalgebra, this immediately implies the $\rm C^*$-Frobenius property.
It is now straightforward to check that $(A, \mu, i)$ is actually special.
Indeed, this follows automatically as $\Hom(1_\scrC, A)$ is one dimensional c.f.~\cite[Lem.~3.3]{MR3308880}.
Finally, another straightforward calculation shows that $i^*\circ i = \id_{1_\scrC}$ and 
$$
\mu\circ \mu^*= (1+\dim(\sigma)) \id_{1_\scrC} + (2+c^2b)\id_\sigma = \dim_\scrC(A) \id_A.
$$
Hence normalizing by a factor of $\sqrt{\dim_\scrC(A)}$ yields a (standard) Q-system.

The final statement now follows since $\dim(\Hom_\scrC(1,1\oplus \sigma))=1$.
\end{proof}

%%%%%%%%%%%%%%%%%%%%%%%%%%%%%%%%%%%%%%%%%%%%%%%%%%%%%%%%%%%%
\subsection{Constructing the Q-system}
\label{sec:ConstructingQSystem}

Below, we use the notational convention that a solid bar across a number of strings denotes the Jones-Wenzl idempotent on that number of strands.
For example,
$$
\begin{tikzpicture}[baseline=-.1cm]
	\draw (-.2,-.4) -- (-.2,.4);
	\draw (.2,-.4) -- (.2,.4);
	\draw[very thick] (-.4,0) -- (.4,0);
	\node at (-.2,-.6) {\scriptsize{$3$}};
	\node at (.2,-.6) {\scriptsize{$3$}};
\end{tikzpicture}
=
\begin{tikzpicture}[baseline=-.1cm]
	\draw (0,-.6) -- (0,.6);
	\filldraw[fill=white,thick] (-.4,-.3) rectangle (.4,.3);
	\node at (0,0) {$\JW{6}$};
	\node at (0,-.8) {\scriptsize{$6$}};
	\node at (0,.8) {\scriptsize{$6$}};
\end{tikzpicture}
\,.
$$

We define the following elements in $\Hom(\jw{6}\otimes \jw{6} \to \jw{6})$:
$$
T_0
=
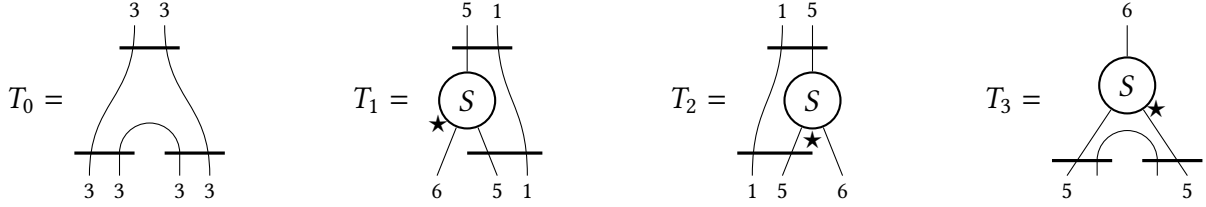
\begin{tikzpicture}[baseline=-.1cm]
	\draw (-.4,-1) -- (-.4,-.7) arc (180:0:.4cm) -- (.4,-1);
	\draw (-.8,-1) .. controls ++(90:1cm) and ++(270:1cm) .. (-.2,1);
	\draw (.8,-1) .. controls ++(90:1cm) and ++(270:1cm) .. (.2,1);
	\draw[very thick] (-1,-.7) -- (-.2,-.7);
	\draw[very thick] (1,-.7) -- (.2,-.7);
	\draw[very thick] (-.4,.7) -- (.4,.7);
	\node at (-.8,-1.2) {\scriptsize{$3$}};
	\node at (-.4,-1.2) {\scriptsize{$3$}};
	\node at (.4,-1.2) {\scriptsize{$3$}};
	\node at (.8,-1.2) {\scriptsize{$3$}};
	\node at (.2,1.2) {\scriptsize{$3$}};
	\node at (-.2,1.2) {\scriptsize{$3$}};
\end{tikzpicture}
\qquad\qquad
T_1
=
\begin{tikzpicture}[baseline=-.1cm]
	\node (S) [circle, fill=white, thick, draw]  at (0,0) {$S$};
	\node [below=1pt] at (S.west) {$\star$};
	\draw (S) -- (0,1);
	\draw (S) -- (-.4,-1);
	\draw (S) -- (.4,-1);
	\draw (.8,-1) .. controls ++(90:1cm) and ++(270:1cm) .. (.4,1);
	\draw[very thick] (0,-.7) -- (1,-.7);
	\draw[very thick] (-.2,.7) -- (.6,.7);
	\node at (-.4,-1.2) {\scriptsize{$6$}};
	\node at (.4,-1.2) {\scriptsize{$5$}};
	\node at (.8,-1.2) {\scriptsize{$1$}};
	\node at (0,1.2) {\scriptsize{$5$}};
	\node at (.4,1.2) {\scriptsize{$1$}};
\end{tikzpicture}
\qquad\qquad
T_2
=
\begin{tikzpicture}[baseline=-.1cm, xscale=-1]
	\node (S) [circle, fill=white, thick, draw]  at (0,0) {$S$};
	\node [below=-4pt] at (S.south) {$\star$};
	\draw (S) -- (0,1);
	\draw (S) -- (-.4,-1);
	\draw (S) -- (.4,-1);
	\draw (.8,-1) .. controls ++(90:1cm) and ++(270:1cm) .. (.4,1);
	\draw[very thick] (0,-.7) -- (1,-.7);
	\draw[very thick] (-.2,.7) -- (.6,.7);
	\node at (-.4,-1.2) {\scriptsize{$6$}};
	\node at (.4,-1.2) {\scriptsize{$5$}};
	\node at (.8,-1.2) {\scriptsize{$1$}};
	\node at (0,1.2) {\scriptsize{$5$}};
	\node at (.4,1.2) {\scriptsize{$1$}};
\end{tikzpicture}
\qquad\qquad
T_3
=
\begin{tikzpicture}[baseline=-.1cm]
	\node (S) [circle, fill=white, thick, draw]  at (0,.2) {$S$};
	\node [below=1pt] at (S.east) {$\star$};
	\draw (S) -- (0,1);
	\draw (S) -- (-.8,-1);
	\draw (S) -- (.8,-1);
	\draw (-.4,-1) -- (-.4,-.8) arc (180:0:.4cm) -- (.4,-1);
	\draw[very thick] (.2,-.8) -- (1,-.8);
	\draw[very thick] (-.2,-.8) -- (-1,-.8);
	\node at (-.8,-1.2) {\scriptsize{$5$}};
	\node at (.8,-1.2) {\scriptsize{$5$}};
	\node at (0,1.2) {\scriptsize{$6$}};
\end{tikzpicture}
$$

\begin{fact}
Recall the following formula for coefficients in the Jones-Wenzl idempotent $\jw{k}$: 
$$
\underset{\in f^{(k)}}{\coeff}
\left(
\begin{tikzpicture}[baseline = -.1cm, yscale=.5]
	\draw (-.5,-.8)--(-.5,.8);
	\draw (-.3,.8) arc (-180:0:.2cm);
	\draw (-.3,-.8)--(.3,.8);
	\draw (-.1,-.8) arc (180:0:.2cm);
	\draw (.5,-.8)--(.5,.8);
	\node at (-.7,0) {{\scriptsize{$a$}}};
	\node at (-.2,0) {{\scriptsize{$b$}}};
	\node at (.7,0) {{\scriptsize{$c$}}};
\end{tikzpicture}
\right)
=
(-1)^{b+1}\frac{[a+1][c+1]}{[k]}.
$$
\end{fact}

\begin{lem}
The set $\{T_0,T_1,T_2,T_3\}$ is linearly independent and is thus a basis for $\Hom(f^{(6)}\otimes f^{(6)}\to f^{(6)})$.
\end{lem}
\begin{proof}
Suppose $\sum_{j=0}^3 \lambda_j T_j = 0$.
We calculate the coefficient of the following diagrams in $\sum_{j=0}^3 \lambda_j T_j$:
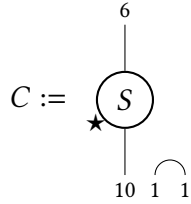
$$
C:=
\begin{tikzpicture}[baseline=-.1cm]
\node (S) [circle, fill=white, thick, draw]  at (0,0) {$S$};
\node [below=1pt] at (S.west) {$\star$};
\draw (S) -- (0,1);
\draw (S) -- (0,-1);
\draw (.4,-1) arc (180:0:.2cm);
\node at (0,-1.2) {\scriptsize{$10$}};
\node at (0,1.2) {\scriptsize{$6$}};
\node at (.4,-1.2) {\scriptsize{$1$}};
\node at (.8,-1.2) {\scriptsize{$1$}};
\end{tikzpicture}
$$
By expanding the Jones Wenzls in $\sum_{j=0}^3 \lambda_j T_j$ and noticing that almost all positions of cups or caps put a cap on $S$ (which is uncappable), we calculate that the coefficient in $\sum_{j=0}^3 \lambda_j T_j$ of $C$ is given by
$$
\lambda_1
\underset{\in f^{(6)}}{\coeff}\left(\,
\begin{tikzpicture}[baseline = -.1cm]
	\draw[thick] (-.6,-.4)--(-.6,.4)--(.6,.4)--(.6,-.4)--(-.6,-.4);
	\draw (0,.4) arc (-180:0:.2cm);
	\draw (-.4,-.4)--(-.4,.4);
	\draw (0,-.4) arc (180:0:.2cm);
	\node at (-.2,0) {{\scriptsize{$4$}}};
\end{tikzpicture}
\,\right)
+
\lambda_3
\omega_S^3
\underset{\in f^{(6)}}{\coeff}\left(\,
\begin{tikzpicture}[baseline = -.1cm]
	\draw[thick] (-.6,-.4)--(-.6,.4)--(.8,.4)--(.8,-.4)--(-.6,-.4);
	\draw (-.2,-.4)--(.4,.4);
	\draw (.2,-.4) arc (180:0:.2cm);
	\draw (-.4,.4) arc (-180:0:.2cm);
	\node at (.3,0) {{\scriptsize{$4$}}};
\end{tikzpicture}
\,\right)
=-\lambda_1
\frac{[5]}{[6]}
+(-1)^3 \lambda_3 \frac{-1}{[6]}
=
\frac{\lambda_3-[5]\lambda_1}{[6]}.
$$
This quantity must be zero as $\sum_{j=0}^3 \lambda_j T_j=0$.
Now considering the 6-click rotations of $\sum_{j=0}^3 \lambda_j T_j$ and $C$, we also see that
$$
\frac{\lambda_1-[5]\lambda_2}{[6]}=0
\qquad\text{and}\qquad
\frac{\lambda_2-[5]\lambda_3}{[6]}=0.
$$
This is only possible when $\lambda_1 = \lambda_2=\lambda_3 = 0$.
We now conclude that $\lambda_0=0$ as well.
\end{proof}

\begin{cor}
\label{cor:FormOfT}
Any rotationally invariant map $\jw{6}\otimes \jw{6}\to \jw{6}$ is of the form $\alpha T_0 + \beta (T_1+T_2+T_3)$.
\end{cor}
\begin{proof}
This follows from the previous lemma together with $T_1, T_2, T_3$ being rotations of each other.
\end{proof}

\begin{cor}
\label{cor:FormOfT-SelfAdjoint}
The element $T=\alpha  T_0 + \beta (T_1+T_2+T_3)$ is self-adjoint if and only if $\alpha$ is real and $\beta$ is purely imaginary.
\end{cor}
\begin{proof}
Since $\omega_S^3 = (-1)^3=-1$,
denoting $\sigma = \jw{6}$, 
we have
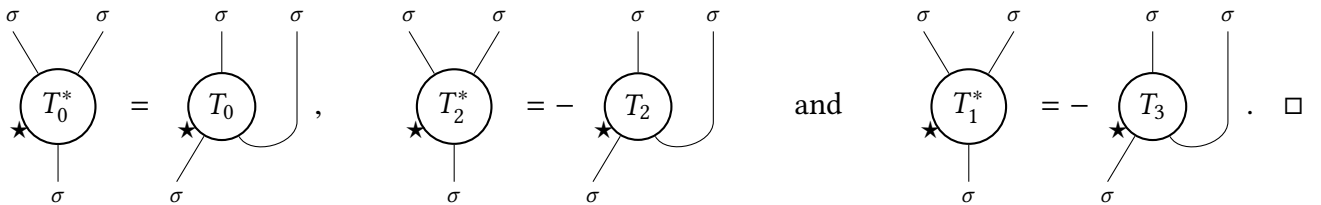
\begin{equation*}
\begin{tikzpicture}[baseline=-.1cm, yscale=-1]
	\node (S) [circle, fill=white, thick, draw]  at (0,0) {$T_0^*$};
	\node [below=1pt] at (S.west) {$\star$};
	\draw (S) -- (0,1);
	\draw (S) -- (-.6,-1);
	\draw (S) -- (.6,-1);
	\node at (-.6,-1.2) {\scriptsize{$\sigma$}};
	\node at (.6,-1.2) {\scriptsize{$\sigma$}};
	\node at (0,1.2) {\scriptsize{$\sigma$}};
\end{tikzpicture}
=
\begin{tikzpicture}[baseline=-.1cm]
%	\draw (0,0) arc (-180:0:.6cm) -- (1.2,1);
	\node (S) [circle, fill=white, thick, draw]  at (0,0) {$T_0$};
	\node [below=1pt] at (S.west) {$\star$};
	\draw (S) -- (0,1);
	\draw (S) -- (-.6,-1);
	\draw (S) [out=-60,in=-90] to (1,-.2) -- (1,1);
	\node at (-.6,-1.2) {\scriptsize{$\sigma$}};
	\node at (1,1.2) {\scriptsize{$\sigma$}};
	\node at (0,1.2) {\scriptsize{$\sigma$}};
\end{tikzpicture}
,
\qquad
\begin{tikzpicture}[baseline=-.1cm, yscale=-1]
	\node (S) [circle, fill=white, thick, draw]  at (0,0) {$T_2^*$};
	\node [below=1pt] at (S.west) {$\star$};
	\draw (S) -- (0,1);
	\draw (S) -- (-.6,-1);
	\draw (S) -- (.6,-1);
	\node at (-.6,-1.2) {\scriptsize{$\sigma$}};
	\node at (.6,-1.2) {\scriptsize{$\sigma$}};
	\node at (0,1.2) {\scriptsize{$\sigma$}};
\end{tikzpicture}
=
-\begin{tikzpicture}[baseline=-.1cm]
	\node (S) [circle, fill=white, thick, draw]  at (0,0) {$T_2$};
	\node [below=1pt] at (S.west) {$\star$};
	\draw (S) -- (0,1);
	\draw (S) -- (-.6,-1);
	\draw (S) [out=-60,in=-90] to (1,-.2) -- (1,1);
	\node at (-.6,-1.2) {\scriptsize{$\sigma$}};
	\node at (1,1.2) {\scriptsize{$\sigma$}};
	\node at (0,1.2) {\scriptsize{$\sigma$}};
\end{tikzpicture}
\qquad
\text{and}
\qquad
\begin{tikzpicture}[baseline=-.1cm, yscale=-1]
	\node (S) [circle, fill=white, thick, draw]  at (0,0) {$T_1^*$};
	\node [below=1pt] at (S.west) {$\star$};
	\draw (S) -- (0,1);
	\draw (S) -- (-.6,-1);
	\draw (S) -- (.6,-1);
	\node at (-.6,-1.2) {\scriptsize{$\sigma$}};
	\node at (.6,-1.2) {\scriptsize{$\sigma$}};
	\node at (0,1.2) {\scriptsize{$\sigma$}};
\end{tikzpicture}
=
-\begin{tikzpicture}[baseline=-.1cm]
	\node (S) [circle, fill=white, thick, draw]  at (0,0) {$T_3$};
	\node [below=1pt] at (S.west) {$\star$};
	\draw (S) -- (0,1);
	\draw (S) -- (-.6,-1);
	\draw (S) [out=-60,in=-90] to (1,-.2) -- (1,1);
	\node at (-.6,-1.2) {\scriptsize{$\sigma$}};
	\node at (1,1.2) {\scriptsize{$\sigma$}};
	\node at (0,1.2) {\scriptsize{$\sigma$}};
\end{tikzpicture}.\qedhere
\end{equation*}
\end{proof}

\begin{cor}
\label{cor:FormOfT-Bigon}
The element $T=\alpha T_0 +  \beta (T_1+T_2+T_3)$ for $\alpha \in \mathbb R$  and $\beta \in i \mathbb R$ satisfies the bigon axiom if and only if at least one of $\alpha$ and $\beta$ are nonzero.
\end{cor}
\begin{proof}
Since $T$ is self-adjoint, the bigon axiom is equivalent to $T^*T\in \End(\jw{6})=\bbC$ being non-zero and positive, which holds if $T\neq 0$.
\end{proof}

%%%%%%%%%%%%%%%%%%%%%%%%%%%%%%%%%%%%%%%%%%%%%%%%%%%%%%%%%%%%
\subsection{Calculations in the graph planar algebra}
\label{sec:CalculationsInGPA}

We still need to find $\alpha \in \mathbb R$ and $\beta \in i \mathbb R$, not both zero, so that $T=\alpha T_0+\beta (T_1+T_2+T_3)$ satisfies the $I=H $ relation.

This calculation is performed in the lopsided graph planar algebra \cite[\S1]{MR3254427} of the $\cE\cH_2-\cE\cH_1$ subfactor principal graph. 
Recall that in \cite{MR2979509} we gave explicit formulas for the image of the generator $S$ of the Extended Haagerup subfactor planar algebra in the (spherical) graph planar algebra. 
It is just a small normalization issue to obtain this element in the lopsided graph planar algebra. 
This lopsided graph planar algebra is built out of vector spaces over the number field $\bbQ(\lambda)$, where $\lambda$ is the largest imaginary root of $\lambda^6 + 2\lambda^4 - 3\lambda^2 -5$.

We first prepare the elements $T_0, T_1, T_2, T_3$ by direct calculation in the graph planar algebra (first constructing the 6-strand Jones-Wenzl idempotent using the usual Wenzl recursion formula).

In a world with bigger and faster computers, we would next calculate the 32 compositions 
$$
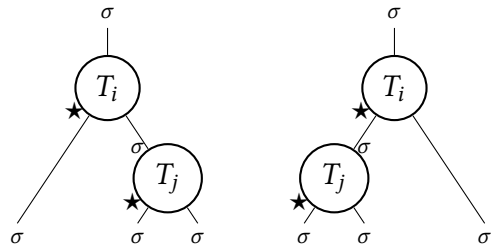
\displaystyle
\begin{tikzpicture}[baseline=-.4cm]
\node (T1) [circle, fill=white, thick, draw]  at (0,.2) {$T_i$};
\node (T2) [circle, fill=white, thick, draw]  at (.8,-1) {$T_j$};
\node [below=1pt] at (T1.west) {$\star$};
\node [below=1pt] at (T2.west) {$\star$};
\draw (T1) -- (0,1);
\draw (T1) -- (-1.2,-1.6);
\draw (T1) -- (T2);
\draw (T2) -- (.4,-1.6);
\draw (T2) -- (1.2,-1.6);
\node at (-1.2,-1.8) {\scriptsize{$\sigma$}};
\node at (.4,-1.8) {\scriptsize{$\sigma$}};
\node at (1.2,-1.8) {\scriptsize{$\sigma$}};
\node at (.4,-.6) {\scriptsize{$\sigma$}};
\node at (0,1.2) {\scriptsize{$\sigma$}};
\end{tikzpicture}
\qquad
\begin{tikzpicture}[baseline=-.4cm, xscale=-1]
\node (T1) [circle, fill=white, thick, draw]  at (0,.2) {$T_i$};
\node (T2) [circle, fill=white, thick, draw]  at (.8,-1) {$T_j$};
\node [below=1pt] at (T1.west) {$\star$};
\node [below=1pt] at (T2.west) {$\star$};
\draw (T1) -- (0,1);
\draw (T1) -- (-1.2,-1.6);
\draw (T1) -- (T2);
\draw (T2) -- (.4,-1.6);
\draw (T2) -- (1.2,-1.6);
\node at (-1.2,-1.8) {\scriptsize{$\sigma$}};
\node at (.4,-1.8) {\scriptsize{$\sigma$}};
\node at (1.2,-1.8) {\scriptsize{$\sigma$}};
\node at (.4,-.6) {\scriptsize{$\sigma$}};
\node at (0,1.2) {\scriptsize{$\sigma$}};
\end{tikzpicture}
$$
for $i,j=0,1,2,3$ and $\sigma = \jw{6}$, as well as the diagrams
$$
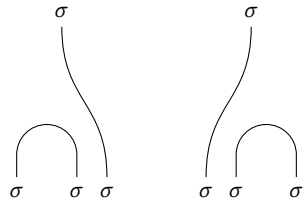
\begin{tikzpicture}[baseline=-.1cm]
	\draw (-.4,-1) -- (-.4,-.7) arc (180:0:.4cm) -- (.4,-1);
	\draw (.8,-1) .. controls ++(90:1cm) and ++(270:1cm) .. (.2,1);
	\node at (-.4,-1.2) {\scriptsize{$\sigma$}};
	\node at (.4,-1.2) {\scriptsize{$\sigma$}};
	\node at (.8,-1.2) {\scriptsize{$\sigma$}};
	\node at (.2,1.2) {\scriptsize{$\sigma$}};
\end{tikzpicture}
\qquad
\begin{tikzpicture}[baseline=-.1cm]
	\draw (-.4,-1) -- (-.4,-.7) arc (180:0:.4cm) -- (.4,-1);
	\draw (-.8,-1) .. controls ++(90:1cm) and ++(270:1cm) .. (-.2,1);
	\node at (-.8,-1.2) {\scriptsize{$\sigma$}};
	\node at (-.4,-1.2) {\scriptsize{$\sigma$}};
	\node at (.4,-1.2) {\scriptsize{$\sigma$}};
	\node at (-.2,1.2) {\scriptsize{$\sigma$}};
\end{tikzpicture}
.$$
These are each elements in the 24 boundary point space of the graph planar algebra, which is unfortunately rather large, at 50,996,510-dimensional. 
Each of the elements would take several gigabytes to store, and in practice this approach runs into tedious memory management issues. 
Rather than preparing the actual element $I-H$ on the left hand side of the $I=H$ relation, as a quadratic expression in $\alpha$ and $\beta$, we attempt to solve the $I=H$ equation component by component in the graph planar algebra. 

Although in many components this equation is trivially satisfied, unsurprisingly we quickly obtain enough equations to solve for $\alpha$ and $\beta$, finding that $\beta = i$ and $\alpha$ is the largest root (all are real) of $5 \lambda^6 + 17\lambda^4 - 18\lambda^2 + 1$, approximately 0.8932. 
By Corollaries \ref{cor:FormOfT}, \ref{cor:FormOfT-SelfAdjoint}, and \ref{cor:FormOfT-Bigon}, there is a unique such $T \in \Hom(\sigma\otimes \sigma \to \sigma)$, which by Proposition \ref{prop:CanonicalQSystem}, gives two isomorphic Q-system structures on $1\oplus \sigma$.

Finally, we still need to check that these values gives a solution to the $I=H$ equation. 
To do this, we directly evaluate $T$ in the graph planar algebra with these values of $\alpha$ and $\beta$, and compute the two compositions appearing on the left hand side of the $I=H$ relation. 
This is a lengthy computation (about an hour, but as it eventually checks relations in a $\sim 51 \times 10^6$ dimensional vector space over a number field, the unoptimised calculation requires at least 64gb of RAM), and eventually verifies that we have found a Q-system structure on $1 \oplus \sigma$.
%auto-ignore
%this ensures the arxiv doesn't try to start TeXing here.
%!TEX root =../EH3.tex

%%%%%%%%%%%%%%%%%%%%%%%%%%%%%%%%%%%%%%%%%%%%%%%%%%%%%%%%%%%%%%%%%%%%%%%%%%%
%%%%%%%%%%%%%%%%%%%%%%%%%%%%%%%%%%%%%%%%%%%%%%%%%%%%%%%%%%%%%%%%%%%%%%%%%%%
%%%%%%%%%%%%%%%%%%%%%%%%%%%%%%%%%%%%%%%%%%%%%%%%%%%%%%%%%%%%%%%%%%%%%%%%%%%

\section{Possible skein theoretic approaches to \texorpdfstring{$\cE\cH_3$}{EH3} and \texorpdfstring{$\cE\cH_4$}{EH4}} \label{sec:YBR}

For Asaeda--Haagerup it turned out that one of the fusion categories in its Morita equivalence class was an Izumi quadratic category, which illuminated why the Asaeda--Haagerup category exists.  We would love for $\cE\cH_3$ or $\cE\cH_4$ to fit into a nicer story in a similar way thereby explaining the currently mysterious Extended Haagerup categories.  However, we have not yet found any direct approaches to understanding these new fusion categories.  In particular, none of these categories have any invertible objects or additional symmetries that were hidden in the original Extended Haagerup categories.  There are two skein theoretic approaches to understanding $\cE\cH_4$ that are potentially promising.

First, we saw in the \S\ref{sec:intermediate} that there is a $(3,3)$-supertransitive quadrilateral which is non-commuting and cocommuting where the lower inclusions are the subfactor of index roughly $8.0$ from Equation \eqref{eq:S8} and the upper inclusions are the subfactor of index roughly $7.0$ from Equation \eqref{eq:S7}.  One could try to construct this quadrilateral directly.  Following \cite{MR2418197} and unpublished planar algebra approaches of Grossman and Snyder, one can try to give a skein theoretic construction of highly supertransitive quadrilaterals.  Unfortunately this approach remains unpublished because it has not yet succeeded for any important examples.  Nonetheless, it doesn't seem entirely hopeless.

Second, there's an intriguing pattern that $\cE\cH_3$ comes from a Q-system on $1 \oplus (X \otimes X)$ in $\cE\cH_2$ while $\cE\cH_4$ would come from a Q-system on $1 \oplus (Y \otimes Y)$.  So it would be natural to try to classify triples of a tensor category, a non-self-dual simple object $X$, and a Q-system structure on $1 \oplus (X \otimes X)$.  By self-duality of the Q-system we see that $X \otimes X \cong \overline{X} \otimes \overline{X}$, and there is a very nice skein theoretic category which is universal for tensor categories with an object $X$ satisfying $X \otimes X \cong \overline{X} \otimes \overline{X}$

\begin{defn}
Let $\zeta$ be a fourth root of unity, let $\cU_\zeta$ be the oriented planar algebra given by the following generators 
$$
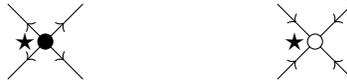
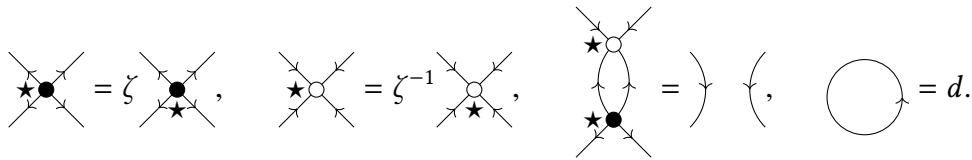
\begin{tikzpicture}[baseline=-.1cm]
	\filldraw (0,0) node [left] {$\star$} circle (1mm);
	\draw[->] (0,0) --(.25,.25);
	\draw (.25,.25) --(.5,.5);
	\draw[->] (0,0) --(.25,.-.25);
	\draw (.25,-.25) --(.5,-.5);
	\draw[->] (0,0) --(-.25,.25);
	\draw (-.25,.25) --(-.5,.5);
	\draw[->] (0,0) --(-.25,.-.25);
	\draw (-.25,-.25) --(-.5,-.5);
\end{tikzpicture}
\hspace{1in}
\begin{tikzpicture}[baseline=-.1cm]
	\draw (0,0) --(.25,.25);
	\draw[<-] (.25,.25) --(.5,.5);
	\draw (0,0) --(.25,.-.25);
	\draw[<-] (.25,-.25) --(.5,-.5);
	\draw (0,0) --(-.25,.25);
	\draw[<-] (-.25,.25) --(-.5,.5);
	\draw (0,0) --(-.25,.-.25);
	\draw[<-] (-.25,-.25) --(-.5,-.5);
	\filldraw [fill=white] (0,0) node [left] {$\star$} circle (1mm);
\end{tikzpicture}
$$

and relations
$$\begin{tikzpicture}[baseline=-.1cm]
	\filldraw (0,0) node [left] {$\star$} circle (1mm);
	\draw[->] (0,0) --(.25,.25);
	\draw (.25,.25) --(.5,.5);
	\draw[->] (0,0) --(.25,.-.25);
	\draw (.25,-.25) --(.5,-.5);
	\draw[->] (0,0) --(-.25,.25);
	\draw (-.25,.25) --(-.5,.5);
	\draw[->] (0,0) --(-.25,.-.25);
	\draw (-.25,-.25) --(-.5,-.5);
\end{tikzpicture}
= \zeta \begin{tikzpicture}[baseline=-.1cm]
	\filldraw (0,0) node [below] {$\star$} circle (1mm);
	\draw[->] (0,0) --(.25,.25);
	\draw (.25,.25) --(.5,.5);
	\draw[->] (0,0) --(.25,.-.25);
	\draw (.25,-.25) --(.5,-.5);
	\draw[->] (0,0) --(-.25,.25);
	\draw (-.25,.25) --(-.5,.5);
	\draw[->] (0,0) --(-.25,.-.25);
	\draw (-.25,-.25) --(-.5,-.5);
\end{tikzpicture}
,\hspace{.25in}
\begin{tikzpicture}[baseline=-.1cm]
	\draw (0,0) --(.25,.25);
	\draw[<-] (.25,.25) --(.5,.5);
	\draw (0,0) --(.25,.-.25);
	\draw[<-] (.25,-.25) --(.5,-.5);
	\draw (0,0) --(-.25,.25);
	\draw[<-] (-.25,.25) --(-.5,.5);
	\draw (0,0) --(-.25,.-.25);
	\draw[<-] (-.25,-.25) --(-.5,-.5);
	\filldraw [fill=white] (0,0) node [left] {$\star$} circle (1mm);
\end{tikzpicture}
= \zeta^{-1} \begin{tikzpicture}[baseline=-.1cm]
	\draw (0,0) --(.25,.25);
	\draw[<-] (.25,.25) --(.5,.5);
	\draw (0,0) --(.25,.-.25);
	\draw[<-] (.25,-.25) --(.5,-.5);
	\draw (0,0) --(-.25,.25);
	\draw[<-] (-.25,.25) --(-.5,.5);
	\draw (0,0) --(-.25,.-.25);
	\draw[<-] (-.25,-.25) --(-.5,-.5);
	\filldraw [fill=white] (0,0) node [below] {$\star$} circle (1mm);
\end{tikzpicture}
,\hspace{.25in}
\begin{tikzpicture}[baseline=.8cm]
	\draw[mid>] (0,.5) to (-.5,0);
	\draw[mid>] (0,.5) to (.5,0);
	\draw[mid>] (0,.5) to [out=45, in=-45] (0,1.5);
	\draw[mid>] (0,.5) to [out=135, in=-135] (0,1.5);
	\draw[mid<] (0,1.5) to (-.5,2);
	\draw[mid<] (0,1.5) to (.5,2);
	\filldraw (0,.5) node [left] {$\star$} circle (1mm);
	\filldraw[fill=white] (0,1.5) node [left] {$\star$} circle (1mm);
\end{tikzpicture}
= 
\begin{tikzpicture}[baseline=.4cm]
	\draw[mid>] (-.5, 1) to [out=-45, in = 45] (-.5,0);
	\draw[mid>] (.5, 1) to [out=-135, in = 135] (.5,0);
\end{tikzpicture}
,\hspace{.25in}
\begin{tikzpicture}[baseline=0cm]
	\draw[->] (0,0) arc (0:360:.5);
\end{tikzpicture} = d.
$$
\end{defn}

Using that the diagrams are bipartite, an easy Euler characteristic argument shows that you can evaluate all closed diagrams in $\cU_\zeta$, and that it has finite dimensional box spaces.  The simple objects in $\cU_\zeta$ are the tensor products of two Jones-Wenzls 
$\jw{i} \otimes \jw{j}$ where each Jones-Wenzl alternates up-down, with the first one starting with an up arrow, and the second one starting with the same up/down that the previous one ended with.  (In the degenerate case where $i=0$ the second Jones-Wenzl starts with a down arrow.)  The fusion graph for $\cU_\zeta$ is given by the following:

$$
\begin{tikzpicture}[baseline=-1mm, scale=.7]
	\node (00) at (0,2*0)  {\scriptsize{$\jw{0} \otimes \jw{0}$}};
	\node (01) at (0,2*1)  {\scriptsize{$\jw{0} \otimes \jw{1}$}};
	\node (02) at (0,2*2)  {\scriptsize{$\jw{0} \otimes \jw{2}$}};
	\node (03) at (0,2*3) {\scriptsize{$\vdots$}};
	\node (10) at (3*1,2*0)  {\scriptsize{$\jw{1} \otimes \jw{0}$}};
	\node (11) at (3*1,2*1)  {\scriptsize{$\jw{1} \otimes \jw{1}$}};	
	\node (12) at (3*1,2*2)  {\scriptsize{$\jw{1} \otimes \jw{2}$}};
	\node (13) at (3*1,2*3)  {\scriptsize{$\vdots$}};		
	\node (20) at (3*2,2*0)  {\scriptsize{$\jw{2} \otimes \jw{0}$}};		
	\node (21) at (3*2,2*1)  {\scriptsize{$\jw{2} \otimes \jw{1}$}};
	\node (22) at (3*2,2*2)  {\scriptsize{$\jw{2} \otimes \jw{2}$}};
	\node (23) at (3*2,2*3)  {\scriptsize{$\vdots$}};
	\node (30) at (3*3,2*0)  {\scriptsize{$\hdots$}};
	\node (31) at (3*3,2*1)  {\scriptsize{$\hdots$}};
	\node (32) at (3*3,2*2)  {\scriptsize{$\hdots$}};
	\node (33) at (3*3,2*3)  {\scriptsize{$\iddots$}};
	\draw[<-] (00) -- (01);
	\draw[->] (01) -- (02);
	\draw[<-] (02) -- (03);
	\draw[->] (10) -- (11);
	\draw[<-] (11) -- (12);
	\draw[->] (12) -- (13);
	\draw[<-] (20) -- (21);
	\draw[->] (21) -- (22);
	\draw[<-] (22) -- (23);

	\draw[->] (00) -- (10);
	\draw[<-] (10) -- (20);
	\draw[->] (20) -- (30);
	\draw[<-] (01) -- (11);
	\draw[->] (11) -- (21);
	\draw[<-] (21) -- (31);
	\draw[->] (02) -- (12);
	\draw[<-] (12) -- (22);
	\draw[->] (22) -- (32);
\end{tikzpicture}
$$

\begin{remark} $\cU_\zeta$ can be constructed algebraically as a $\mathbb{Z}/4\mathbb{Z}$ extension of $\mathrm{SO}(3)_q \boxtimes \mathrm{SO}(3)_q$ for $d=q+q^{-1}$. Such extensions can be classified following \cite{MR2677836}.  The Brauer--Picard group of $\mathrm{SO}(3)_q$ is $\mathbb{Z}/2\mathbb{Z}$ where the nontrivial bimodule is the the category of representations of $\mathrm{SU}(2)_q$ with odd highest weight.  The Brauer--Picard group of $\mathrm{SO}(3)_q\boxtimes \mathrm{SO}(3)_q$ is richer because of the outer automorphism interchanging the two factors.  
There are four module categories with the correct dual category (picking a parity for the highest weight in each factor), but each gives two bimodule categories using the outer automorphism.  So the whole group is the $8$-element dihedral group.  In particular, there's an element of order $4$ in the Brauer group which we use to build our map from $\mathbb{Z}/4\mathbb{Z}$.  One then checks that two obstructions vanish, the first because there are no non-trivial invertible objects or gradings, and the second because $H^4(\mathbb{Z}/4, \mathbb{C}^\times)$ vanishes because the group is cyclic.  Finally there's a choice of element of $H^3(\mathbb{Z}/4\mathbb{Z}, \mathbb{C}^\times)$, which corresponds exactly to $\zeta$.
\end{remark}

We now consider what it would mean to in addition have a Q-system structure on $1 \oplus (X \otimes X)$.  First, since Q-systems are symmetrically self-dual we see that $\zeta = \pm 1$.  Second, having a Q-system lets us build a shaded planar algebra where the $2$-strand shaded Jones-Wenzl corresponds to $(X \otimes X)$.  That is we have an oriented strand type corresponding to $X$ and an shaded strand type corresponding to $1 \oplus (X \otimes X)$ as a module over itself.  This gives us the following additional generators 

$$
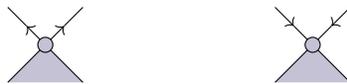
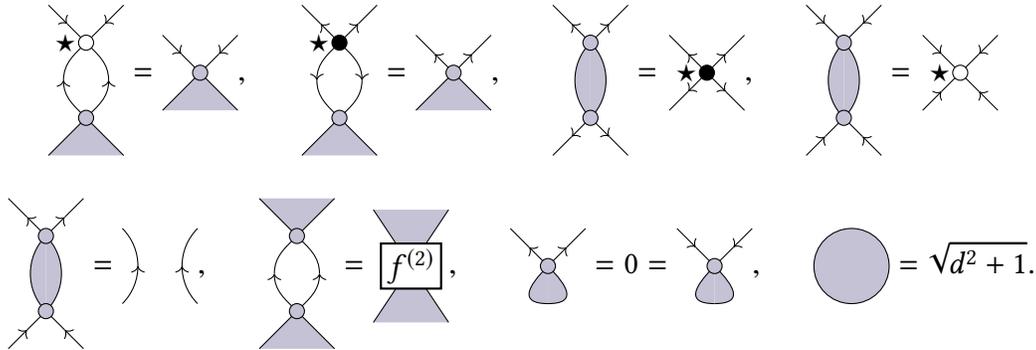
\begin{tikzpicture}[baseline=-.1cm]
	\draw[->] (0,0) -- (.25,.25);
	\draw (.25,.25) -- (.5,.5);
	\draw[->] (0,0) -- (-.25,.25);
	\draw (-.25,.25) -- (-.5,.5);
	\filldraw[shaded] (-.5,-.5) -- (0,0) -- (.5,-.5);
	\filldraw[shaded] (0,0) circle (1mm);
\end{tikzpicture}
\hspace{1in}
\begin{tikzpicture}[baseline=-.1cm]
	\draw (0,0) -- (.25,.25);
	\draw[<-] (.25,.25) -- (.5,.5);
	\draw (0,0) -- (-.25,.25);
	\draw[<-] (-.25,.25) -- (-.5,.5);
	\filldraw[shaded] (-.5,-.5) -- (0,0) -- (.5,-.5);
	\filldraw[shaded] (0,0) circle (1mm);
\end{tikzpicture}
$$

and relations
$$
\begin{tikzpicture}[baseline=.5cm]
	\filldraw[shaded] (-.5,-.5) -- (0,0) -- (.5,-.5);
	\draw (0,1) --(.25,1.25);
	\draw[<-] (.25,1+.25) --(.5,1+.5);
	\draw (0,1+0) --(-.25,1+.25);
	\draw[<-] (-.25,1+.25) --(-.5,1+.5);
	\draw[->] (0,.0)[out=135,in=270] to (-.3,.5);
	\draw[->] (0,.0)[out=45,in=270] to (.3,.5);
	\draw (-.3,.5)[out=90,in=225] to (0,1);
	\draw (.3,.5)[out=90,in=315] to (0,1);
	\filldraw [fill=white] (0,1) node [left] {$\star$} circle (1mm);	
	\filldraw[shaded] (0,0) circle (1mm);
\end{tikzpicture} =
\begin{tikzpicture}[baseline=-.1cm]
	\draw (0,0) -- (.25,.25);
	\draw[<-] (.25,.25) -- (.5,.5);
	\draw (0,0) -- (-.25,.25);
	\draw[<-] (-.25,.25) -- (-.5,.5);
	\filldraw[shaded] (-.5,-.5) -- (0,0) -- (.5,-.5);
	\filldraw[shaded] (0,0) circle (1mm);
\end{tikzpicture}
,\hspace{.25in}
\begin{tikzpicture}[baseline=.5cm]
	\filldraw[shaded] (-.5,-.5) -- (0,0) -- (.5,-.5);
	\draw[->] (0,1) --(.25,1.25);
	\draw (.25,1+.25) --(.5,1+.5);
	\draw[->] (0,1+0) --(-.25,1+.25);
	\draw (-.25,1+.25) --(-.5,1+.5);
	\draw (0,.0)[out=135,in=270] to (-.3,.5);
	\draw (0,.0)[out=45,in=270] to (.3,.5);
	\draw[<-] (-.3,.5)[out=90,in=225] to (0,1);
	\draw[<-] (.3,.5)[out=90,in=315] to (0,1);
	\filldraw (0,1) node [left] {$\star$} circle (1mm);	
	\filldraw[shaded] (0,0) circle (1mm);
\end{tikzpicture} =
\begin{tikzpicture}[baseline=-.1cm]
	\draw[->] (0,0) -- (.25,.25);
	\draw (.25,.25) -- (.5,.5);
	\draw[->] (0,0) -- (-.25,.25);
	\draw (-.25,.25) -- (-.5,.5);
	\filldraw[shaded] (-.5,-.5) -- (0,0) -- (.5,-.5);
	\filldraw[shaded] (0,0) circle (1mm);
\end{tikzpicture}
,\hspace{.25in}
\begin{tikzpicture}[baseline=.5cm]
	\draw[->] (0,0) -- (.25,-.25);
	\draw (.25,-.25) -- (.5,-.5);
	\draw[->] (0,0) -- (-.25,-.25);
	\draw (-.25,-.25) -- (-.5,-.5);
	\draw[->] (0,1) --(.25,1.25);
	\draw (.25,1+.25) --(.5,1+.5);
	\draw[->] (0,1+0) --(-.25,1+.25);
	\draw (-.25,1+.25) --(-.5,1+.5);
	\filldraw[shaded] (0,0)[out=135,in=225] to (0,1);
	\filldraw[shaded] (0,0)[out=45,in=315] to (0,1);
	\filldraw[shaded] (0,1) circle (1mm);	
	\filldraw[shaded] (0,0) circle (1mm);
\end{tikzpicture}
= 
\begin{tikzpicture}[baseline=-.1cm]
	\draw[->] (0,0) --(.25,.25);
	\draw (.25,.25) --(.5,.5);
	\draw[->] (0,0) --(.25,.-.25);
	\draw (.25,-.25) --(.5,-.5);
	\draw[->] (0,0) --(-.25,.25);
	\draw (-.25,.25) --(-.5,.5);
	\draw[->] (0,0) --(-.25,.-.25);
	\draw (-.25,-.25) --(-.5,-.5);
	\filldraw (0,0) node [left] {$\star$} circle (1mm);
\end{tikzpicture}
,\hspace{.25in}
\begin{tikzpicture}[baseline=.5cm]
	\draw (0,0) -- (.25,-.25);
	\draw[<-] (.25,-.25) -- (.5,-.5);
	\draw (0,0) -- (-.25,-.25);
	\draw[<-] (-.25,-.25) -- (-.5,-.5);
	\draw (0,1) --(.25,1.25);
	\draw[<-] (.25,1+.25) --(.5,1+.5);
	\draw (0,1+0) --(-.25,1+.25);
	\draw[<-] (-.25,1+.25) --(-.5,1+.5);
	\filldraw[shaded] (0,0)[out=135,in=225] to (0,1);
	\filldraw[shaded] (0,0)[out=45,in=315] to (0,1);
	\filldraw[shaded] (0,1) circle (1mm);	
	\filldraw[shaded] (0,0) circle (1mm);
\end{tikzpicture}
=
\begin{tikzpicture}[baseline=-.1cm]
	\draw (0,0) --(.25,.25);
	\draw[<-] (.25,.25) --(.5,.5);
	\draw (0,0) --(.25,.-.25);
	\draw[<-] (.25,-.25) --(.5,-.5);
	\draw (0,0) --(-.25,.25);
	\draw[<-] (-.25,.25) --(-.5,.5);
	\draw (0,0) --(-.25,.-.25);
	\draw[<-] (-.25,-.25) --(-.5,-.5);
	\filldraw [fill=white] (0,0) node [left] {$\star$} circle (1mm);
\end{tikzpicture}$$
 
 $$
\begin{tikzpicture}[baseline=.5cm]
	\draw (0,0) -- (.25,-.25);
	\draw[<-] (.25,-.25) -- (.5,-.5);
	\draw (0,0) -- (-.25,-.25);
	\draw[<-] (-.25,-.25) -- (-.5,-.5);
	\draw[->] (0,1) --(.25,1.25);
	\draw (.25,1+.25) --(.5,1+.5);
	\draw[->] (0,1+0) --(-.25,1+.25);
	\draw (-.25,1+.25) --(-.5,1+.5);
	\filldraw[shaded] (0,0)[out=135,in=225] to (0,1);
	\filldraw[shaded] (0,0)[out=45,in=315] to (0,1);
	\filldraw[shaded] (0,1) circle (1mm);	
	\filldraw[shaded] (0,0) circle (1mm);
\end{tikzpicture}
= 
\begin{tikzpicture}[baseline=-.1cm]
	\draw[mid>](-.5,-.5) [out=45, in=315] to (-.5,.5);
	\draw[mid>](.5,-.5) [out=135, in=225] to (.5,.5);
\end{tikzpicture}
,\hspace{.25in}
\begin{tikzpicture}[baseline=.5cm]
	\filldraw[shaded] (-.5,-.5) -- (0,0) -- (.5,-.5);
%	\draw (0,1) --(.25,1.25);
%	\draw[<-] (.25,1+.25) --(.5,1+.5);
%	\draw (0,1+0) --(-.25,1+.25);
%	\draw[<-] (-.25,1+.25) --(-.5,1+.5);
	\draw[->] (0,.0)[out=135,in=270] to (-.3,.5);
	\draw[->] (0,.0)[out=45,in=270] to (.3,.5);
	\draw (-.3,.5)[out=90,in=225] to (0,1);
	\draw (.3,.5)[out=90,in=315] to (0,1);
	\filldraw[shaded] (-.5,1.5) -- (0,1) -- (.5,1.5);
	\filldraw[shaded] (0,1) circle (1mm);	
	\filldraw[shaded] (0,0) circle (1mm);
\end{tikzpicture} = 
\begin{tikzpicture}[baseline=-.1cm]
	\filldraw[shaded] (-.5,-.75) -- (0,0) -- (.5,-.75);
	\filldraw[shaded] (-.5,.75) -- (0,0) -- (.5,.75);
	\filldraw[fill=white,thick] (-.4,-.3) rectangle (.4,.3);
	\node at (0,0) {$\jw{2}$};
\end{tikzpicture}
,\hspace{.25in}
\begin{tikzpicture}[baseline=-.1cm]
	\draw[->] (0,0) -- (.25,.25);
	\draw (.25,.25) -- (.5,.5);
	\draw[->] (0,0) -- (-.25,.25);
	\draw (-.25,.25) -- (-.5,.5);
	\filldraw[shaded] (0,-.5) [out=180,in=225,looseness=2] to (0,0);
	\filldraw[shaded] (0,-.5) [out=0,in=315,looseness=2] to (0,0);
	\filldraw[shaded] (0,0) circle (1mm);
\end{tikzpicture}
= 0 =
\begin{tikzpicture}[baseline=-.1cm]
	\draw (0,0) -- (.25,.25);
	\draw[<-] (.25,.25) -- (.5,.5);
	\draw (0,0) -- (-.25,.25);
	\draw[<-] (-.25,.25) -- (-.5,.5);
	\filldraw[shaded] (0,-.5) [out=180,in=225,looseness=2] to (0,0);
	\filldraw[shaded] (0,-.5) [out=0,in=315,looseness=2] to (0,0);
	\filldraw[shaded] (0,0) circle (1mm);
\end{tikzpicture}
,\hspace{.25in}
\begin{tikzpicture}[baseline=-.1cm]
	\filldraw[shaded] (0,0) arc (0:360:.5);
\end{tikzpicture} = \sqrt{d^2+1}.
$$

In the case of $X \in \cE\cH_3$ the space $\Hom(X^{\otimes 3},\overline{X}^{\otimes 3})=\Hom(\overline{X} \oplus W, X \oplus W)$ is only $1$ dimensional, so unless the coefficient magically vanishes, we have a Yang-Baxter relation (in the sense of \cite{1507.06030}) saying the following diagrams are proportional (and the same if you reverse all arrows).

$$
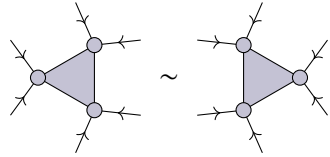
\begin{tikzpicture}[baseline=-.1cm]
	\draw[mid>] (30:1) to (60:.5);
	\draw[mid>] (90:1) to (60:.5);
	\draw[mid>] (150:1) to (180:.5);
	\draw[mid>] (210:1) to (180:.5);
	\draw[mid>] (270:1) to (300:.5);
	\draw[mid>] (330:1) to (300:.5);
	\filldraw[shaded] (60:.5) -- (180:.5) -- (300:.5) -- (60:.5);
	\filldraw[shaded] (60:.5) circle (1mm);	
	\filldraw[shaded] (180:.5) circle (1mm);
	\filldraw[shaded] (300:.5) circle (1mm);
\end{tikzpicture} \; \sim \;
\begin{tikzpicture}[baseline=-.1cm]
	\draw[mid>] (30:1) to (0:.5);
	\draw[mid>] (90:1) to (120:.5);
	\draw[mid>] (150:1) to (120:.5);
	\draw[mid>] (210:1) to (240:.5);
	\draw[mid>] (270:1) to (240:.5);
	\draw[mid>] (330:1) to (0:.5);
	\filldraw[shaded] (0:.5) -- (120:.5) -- (240:.5) -- (0:.5);
	\filldraw[shaded] (0:.5) circle (1mm);	
	\filldraw[shaded] (120:.5) circle (1mm);
	\filldraw[shaded] (240:.5) circle (1mm);
\end{tikzpicture}
$$

Similarly (unless $\alpha$ is miraculously zero) there is a Yang-Baxter relation of the following form:
$$
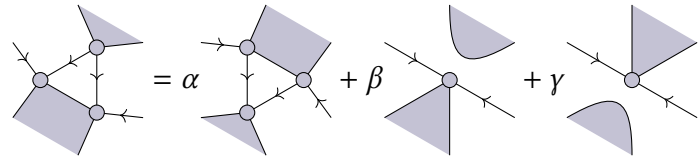
\begin{tikzpicture}[baseline=-.1cm]
	\filldraw[shaded] (30:1) -- (60:.5) -- (90:1);
	\draw[mid>] (150:1) to (180:.5);
	\draw[mid>] (330:1) to (300:.5);
	\draw[mid>] (60:.5) -- (180:.5) ;
	\draw[mid>] (60:.5) -- (300:.5) ;
	\filldraw[shaded] (210:1) -- (180:.5) -- (300:.5) -- (270:1);
	\filldraw[shaded] (60:.5) circle (1mm);	
	\filldraw[shaded] (180:.5) circle (1mm);
	\filldraw[shaded] (300:.5) circle (1mm);
\end{tikzpicture}
= \alpha
\begin{tikzpicture}[baseline=-.1cm]
	\filldraw[shaded] (30:1) to (0:.5) to (120:.5) to (90:1);
	\draw[mid>] (150:1) to (120:.5);
	\filldraw[shaded] (210:1) to (240:.5) to (270:1);
	\draw[mid>] (330:1) to (0:.5);
	\draw[mid>] (120:.5) -- (240:.5);
	\draw[mid>] (0:.5) -- (240:.5);
	\filldraw[shaded] (0:.5) circle (1mm);	
	\filldraw[shaded] (120:.5) circle (1mm);
	\filldraw[shaded] (240:.5) circle (1mm);
\end{tikzpicture}
+ \beta
\begin{tikzpicture}[baseline=-.1cm]
	\filldraw[shaded] (30:1) to [out=210,in=270,looseness=2] (90:1);
	\draw[mid>] (150:1) to (0,0);
	\draw[mid>] (330:1) to (0,0);
	\filldraw[shaded] (210:1) to (0,0) to (270:1);
	\filldraw[shaded] (0,0) circle (1mm);	
\end{tikzpicture}
+ \gamma
\begin{tikzpicture}[baseline=-.1cm, rotate=180]
	\filldraw[shaded] (30:1) to [out=210,in=270,looseness=2] (90:1);
	\draw[mid>] (150:1) to (0,0);
	\draw[mid>] (330:1) to (0,0);
	\filldraw[shaded] (210:1) to (0,0) to (270:1);
	\filldraw[shaded] (0,0) circle (1mm);	
\end{tikzpicture}
$$

Unfortunately, these relations do not suffice to simplify all closed diagrams following \cite{1507.06030}, because several diagrams do not satisfy any Yang-Baxter relations.  Specifically the following two diagrams (together with their rotations, and switching the directions of all arrows) do not satisfy Yang-Baxter relations because there are no $4$-boxes with the appropriate shading/orientation on the other side.

$$
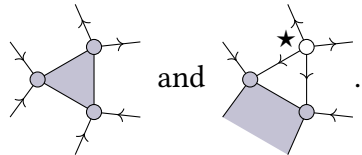
\begin{tikzpicture}[baseline=-.1cm]
	\draw[mid<] (30:1) to (60:.5);
	\draw[mid<] (90:1) to (60:.5);
	\draw[mid>] (150:1) to (180:.5);
	\draw[mid>] (210:1) to (180:.5);
	\draw[mid>] (270:1) to (300:.5);
	\draw[mid>] (330:1) to (300:.5);
	\filldraw[shaded] (60:.5) -- (180:.5) -- (300:.5) -- (60:.5);
	\filldraw[shaded] (60:.5) circle (1mm);	
	\filldraw[shaded] (180:.5) circle (1mm);
	\filldraw[shaded] (300:.5) circle (1mm);
\end{tikzpicture}
\; \text{ and } \;
\begin{tikzpicture}[baseline=-.1cm]
	\draw[mid<] (30:1) to (60:.5);
	\draw[mid<] (90:1) to (60:.5);
	\draw[mid>] (150:1) to (180:.5);
	\draw[mid>] (330:1) to (300:.5);
	\draw[mid>] (60:.5) -- (180:.5) ;
	\draw[mid>] (60:.5) -- (300:.5) ;
	\filldraw[shaded] (210:1) -- (180:.5) -- (300:.5) -- (270:1);
	\filldraw[fill=white] (60:.5) node [above left = -.15 and 0] {$\star$} circle (1mm);	
	\filldraw[shaded] (180:.5) circle (1mm);
	\filldraw[shaded] (300:.5) circle (1mm);
\end{tikzpicture}
.$$

There are closed diagrams which cannot be easily simplified using on the specific Yang-Baxter relations that we have.  The simplest such diagram we found involves $16$-vertices.  Nonetheless this approach seems promising, even if it's beyond current skein theory techniques.

\renewcommand*{\bibfont}{\small}
\setlength{\bibitemsep}{0pt}
\raggedright
\printbibliography

\end{document}